\documentclass[11pt,reqno]{amsart}
\usepackage{amsmath,amssymb,latexsym,esint,cite,mathrsfs}
\usepackage{verbatim,wasysym}
\usepackage[left=1.8cm,right=1.8cm,top=2.4cm,bottom=2.4cm]{geometry}
\usepackage{tikz,enumitem,graphicx, subfig, microtype, color}
\usepackage{epic,eepic}

\usepackage[colorlinks=true,urlcolor=blue, citecolor=red,linkcolor=blue,
linktocpage,pdfpagelabels, bookmarksnumbered,bookmarksopen]{hyperref}
\usepackage[hyperpageref]{backref}
\usepackage[english]{babel}
\usepackage{appendix}

\numberwithin{equation}{section}

\newtheorem{thm}{Theorem}[section]
\newtheorem{lem}[thm]{Lemma}
\newtheorem{cor}[thm]{Corollary}
\newtheorem{Prop}[thm]{Proposition}
\newtheorem{Def}[thm]{Definition}
\newtheorem{Rem}[thm]{Remark}

\newtheorem{step}[thm]{Step}

\renewcommand{\i}{{\rm i}}

\begin{document}
	\title[Stability for nonlocal Sobolev inequality]
{Stability estimates for critical points of a nonlocal Sobolev-type inequality}

\author[M. Yang]{Minbo Yang}
\address{\noindent Minbo Yang  \newline
School of Mathematical Sciences, Zhejiang Normal University,
Jinhua 321004, Zhejiang, People's Republic of China.}\email{mbyang@zjnu.edu.cn}

\author[S. Zhao]{Shunneng Zhao$^\dag$}
\address{\noindent Shunneng Zhao  \newline
School of Mathematical Sciences, Zhejiang Normal University,
Jinhua 321004, Zhejiang, People's Republic of China.
Dipartimento di Matematica, Universit\`{a} degli Studi di Bari Aldo Moro, Via Orabona 4, 70125 Bari, Italy.
}\email{snzhao@zjnu.edu.cn}



\thanks{Minbo Yang was partially supported by the National Key Research and Development Program of China (No. 2022YFA1005700), National Natural Science Foundation of China (12471114) and Natural Science Foundation of Zhejiang Province (LZ22A010001).}
\thanks{$^\dag$Shunneng Zhao was partially supported by National Natural Science Foundation of China (12401146, 12261107), Natural Science Foundation of Zhejiang Province (LMS25A010007) and PNRR MUR project PE0000023 NQSTI - National Quantum Science and Technology Institute (CUP H93C22000670006).}

\allowdisplaybreaks

\begin{abstract}
{\small
In this paper, we study the stability of the following nonlocal Soblev-type inequality
	\begin{equation*}
		C_{HLS}\big(\int_{\mathbb{R}^n}\big(|x|^{-\mu} \ast u^{p}\big)u^{p} dx\big)^{\frac{1}{p}}\leq\int_{\mathbb{R}^n}|\nabla u|^2 dx , \quad \forall~u\in D^{1,2}(\mathbb{R}^n),
	\end{equation*}
which is induced by the classical Sobolev inequality and the Hardy-Littlewood-Sobolev inequality, where $p=\frac{2n-\mu}{n-2}$, $n\geq3$ and $\mu\in(0,n)$, is energy-critical exponent and $C_{HLS}$ is the best constant depending on $n$ and $\mu$.
Up to translation and scaling, the best constant of the nonlocal Soblev inequality can be achieved by  a unique family of positive and radially symmetric extremal function $W(x)$ that satisfies, up to a suitable scaling, the classical critical Hartree equation
 \begin{equation*}
    \Delta u+(|x|^{-\mu}\ast u^{p})u^{p-1}=0 \quad \mbox{in}\quad \mathbb{R}^n.
    \end{equation*}
Recently, Piccione, Yang and Zhao in \cite{p-y-z24} established a nonlocal
version of Struwe's profile decomposition and they only proved the nonlocal version of the quantitative stability for the one bubble case without dimension restriction and the multiple bubbles case $\kappa\geq2$ if dimension $3\leq n<6-\mu$ and $\mu\in(0,n)$ with $\mu\in(0,4]$ in Ciraolo-Figalli-Maggi \cite{CFM18} and Figalli-Glaudo \cite{FG20}. We establish the quantitative stability estimates for critical point of the nonlocal Soblev inequality for $n\geq6-\mu$ and $\mu\in(0,4)$, which is an extension of the recent works by Deng-Sun-Wei in \cite{DSW21} for the classical
Sobolev inequality.
}

\end{abstract}

\vspace{3mm}

\maketitle
\section{Introduction and main results}

\subsection{A nonlocal Sobolev-type inequality}
It is well known that the Sobolev inequlaity \cite{S1963}
\begin{equation}\label{bsic-1}
\big(\int_{\mathbb{R}^n}|u(x)|^{\frac{nq}{n-q}}dx\big)^{\frac{n-q}{n}}\leq S_{n,q}\int_{\mathbb{R}^n}|\nabla u(x)|^qdx,\quad\text{for }\quad1<q<n,\quad  \forall~u\in W^{1,q}(\mathbb{R}^n)
\end{equation}
plays an important role in PDE and functional analysis. Aubin \cite{T-Aubin} and Talenti \cite{Ta76} derived the sharp constants of HLS inequality by classifying the extremal of classical Sobolev inequality, namely that,
   for $q=2$ in inequality \eqref{bsic-1},
the classical Sobolev inequality states, for any $n\geq 3$, there exists a dimensional constant $S=S(n)>0$ such that
\begin{equation}\label{bsic}
\big(\int_{\mathbb{R}^n}|u|^{2^*}dx\big)^{\frac{2}{2^*}}\leq S(n)\int_{\mathbb{R}^n}|\nabla u|^2dx,\quad  \forall~u\in D^{1,2}(\mathbb{R}^n),
\end{equation}
where $2^*=\frac{2n}{n-2}$ and $D^{1,2}(\mathbb{R}^n)$ denotes the closure of $C^\infty_c(\mathbb{R}^n)$ with respect to the norm $\|u\|_{D^{1,2}(\mathbb{R}^n)}=\|\nabla u\|_{L^2}$. It is well known that the Euler-Lagrange equation associated to (\ref{bsic}) is given by
\begin{equation}\label{bec}
\Delta u+|u|^{2^*-2}u=0\quad \mbox{in}\ \ \mathbb{R}^n.
\end{equation}
The best constant $S$ in the Sobolev inequality is achieved by the Aubin-Talanti bubbles \cite{Ta76} $u=U[\xi,\lambda](x)$ defined by
\[U[\xi,\lambda](x)=(n(n-2))^{\frac{n-2}{4}}\big(\frac{\lambda}{1+\lambda^2|x-\xi|^2}\big)^{\frac{n-2}{2}},\hspace{4mm}\lambda\in\mathbb{R}^{+},\hspace{4mm}z\in\mathbb{R}^n.\]
In fact, Caffarelli et al. \cite{CGS89} and Gidas et al. \cite{GNN79} proved that all the positive solutions to equation (\ref{bec}) are the Aubin-Talanti bubbles. In other words, the smooth  manifold of extremal function in \eqref{bsic}
$$\mathcal{M}:=\{cU[\xi,\lambda]: c\in\mathbb{R}, \xi\in\mathbb{R}^n, \lambda>0\}$$
is all nonnegative solution to equation \eqref{bec}.
Moreover, in the general diagonal case $t=r=\frac{2n}{2n-\mu}$, Lieb in \cite{Lieb83} classified the extremal function of HLS inequality with sharp constant by rearrangement and symmetrisation, and obtained the best constant
    \begin{equation}\label{defhlsbc}
    C_{n,\mu}=\pi^{\mu/2}\frac{\Gamma((n-\mu)/2)}{\Gamma(n-\mu/2)}\left(\frac{\Gamma(n)}{\Gamma(n/2)}\right)^{1-\frac{\mu}{n}},
    \end{equation}
   and the equality holds if and only if
    \begin{equation*}
   u(x)=cv(x)=a\big(\frac{1}{1+\lambda^2|x-x_0|^2}\big)^{\frac{2n-\mu}{2}}
    \end{equation*}
    for some $a\in \mathbb{C}$, $\lambda\in \mathbb{R}\backslash\{0\}$ and $x_0\in \mathbb{R}^n$.

To introduce the nonlocal version of Sobolev inequlaity, we may recall that the classical Hardy-Littlewood-Sobolev (HLS for short) inequality, firstly introduced in \cite{H-L-1928,S1963}, states that
\begin{equation}\label{hlsi}
	\int_{\mathbb{R}^n}\int_{\mathbb{R}^n}u(x)|x-y|^{-\mu} v(y)dxdy\leq C_{n,r,t,\mu}\|u\|_{L^r(\mathbb{R}^n)}\|v\|_{L^t(\mathbb{R}^n)}.
\end{equation}
with $\mu\in(0,n)$, $1<r,t<\infty$ and $\frac{1}{r}+\frac{1}{t}+\frac{\mu}{n}=2$.  Lieb and Loss in \cite{Lieb-Loss} applied the layer cake representation formula to give an explicit upper bound for the sharp constant $C_{n,r,t,\mu}$ of HLS inequality.
Invoking the dual form of \eqref{hlsi}, in the special diagonal case $r=t=(2^{\ast})^{\prime}=\frac{2n}{n+2}~(\mu=n-2)$. Define the Coulomb space $\mathcal{C}^{\mu,s}(\mathbb{R}^n)$ as the vector space of measurable functions $u:\mathbb{R}^n\rightarrow\mathbb{R}$ satisfying
\begin{equation*}
\|u\|_{\mathcal{C}^{\mu,s}}=\big(\int_{\mathbb{R}^n}\big(|x|^{-\mu} \ast u^{s}\big)u^{s} dx\big)^{\frac{1}{2s}}<\infty.
\end{equation*}
 It is easy to see that for every measurable function $u\in \mathcal{C}^{\mu,s}(\mathbb{R}^n)$ if and only if $|u|^{s}\in \mathcal{C}^{\mu,1}(\mathbb{R}^n)$. HLS inequality gives that
$$
V(x):=\int_{\mathbb{R}^n}\big(|x|^{-\mu} \ast u^{s}\big)u^{s}\leq C\big(\int_{\mathbb{R}^n}|u|^{\frac{2ns}{2n-\mu}}\big)^{\frac{2n-\mu}{n}},
$$
and thus $L^{\frac{2ns}{2n-\mu}}(\mathbb{R}^n)\subset \mathcal{C}^{\mu,s}(\mathbb{R}^n)$. It is not difficult to check that the Coulomb space  $\mathcal{C}^{\mu,s}(\mathbb{R}^n)$ is a Banach space under the norm $
\|\cdot\|_{\mathcal{C}^{\mu,s}}$ (cf. \cite{Mercuri}). Then it follows from HLS inequality and Sobolev embedding inequality that, for any $u\in D^{1,2}(\mathbb{R}^n)$, the critical case, $p=\frac{2n-\mu}{n-2}$,
\begin{equation}\label{Prm}
C_{HLS}\left(\int_{\mathbb{R}^n}(|x|^{-\mu} \ast|u|^{p})|u|^{p}dx\right)^{\frac{1}{p}}\leq\int_{\mathbb{R}^n}|\nabla u|^2 dx , \quad u\in D^{1,2}(\mathbb{R}^n)
\end{equation}
holds for some positive constant $C_{HLS}$ depending only on $n$ and $\mu$, which can be
characterized by the following minimization problem
\begin{equation}\label{mini-p}
\frac{1}{C_{HLS}}:=\inf\big\{\|\nabla u\|_{L^2}\big|u\in D^{1,2}(\mathbb{R}^n),\hspace{2mm}\Big(\int_{\mathbb{R}^n}\big(|x|^{-\mu} \ast u^{p}\big)u^{p} dx\Big)^{\frac{1}{p}}=1\big\}.
 \end{equation}
It is well-known that the critical nonlocal Hartree equation, up to suitable scaling,
\begin{equation}\label{ele-1.1}
    \Delta u+(|x|^{-\mu}\ast u^{p})u^{p-1}=0 \quad \mbox{in}\quad \mathbb{R}^n,
    \end{equation}
is the Euler-Lagrange equation for the minimization problem \eqref{mini-p}.
 What's more, the authors in \cite{GY18, DY19, GHPS19} independently computed the optimal $C_{HLS}>0$ and classified all positive solutions of \eqref{Prm} are functions of the form
\begin{equation}\label{defU}
W[\xi,\lambda](x)=\alpha_{n,\mu}\big(\frac{\lambda}{1+\lambda^2|x-\xi|^2}\big)^{\frac{n-2}{2}},\hspace{1mm}\lambda\in\mathbb{R}^{+},\hspace{1mm}\xi\in\mathbb{R}^n,
\end{equation}
where $\alpha_{n,\mu}:=(n(n-2))^{\frac{n-2}{4}}S^{\frac{(n-\mu)(2-n)}{4(n-\mu+2)}}C_{n,\mu}^{\frac{2-n}{2(n-\mu+2)}}$
 is a dimension and parameter $\mu$ dependent constant. Furthermore, Gao et al. \cite{GMYZ22} proved the nondegeneracy result at $W[\xi,\lambda]$ forthe linearized operator of equation \eqref{ele-1.1} when $n=6$ and $\mu=4$. Later on, Li et al.\cite{XLi} extended this nondegeneracy theorem to all $\mu\in(0,n)$ with $\mu\in(0,4)$ for $n\geq 3$ by using the spherical harmonic decomposition and the Funk-Hecke formula of the spherical harmonic functions and Morse iteration theorem.
\subsection{The stability problem:}
The start point of the present paper goes back to the study of the classical Sobolev inequality. Brezis and Lieb \cite{BE198585} firstly raised the question of stability for the Sobolev inequality, that is, whether
a remainder term proportional to the quadratic distance of the function $u$ to be the manifold $\mathcal{M}$ - can be added to the right hand side of (\ref{bsic}).
 This open problem was solved few years later by Bianchi and Egnell \cite{BE91}: they established the local stability and the global stability for Sobolev
inequality, for all $u\in D^{1,2}(\mathbb{R}^n)\setminus\mathcal{M}$, there exits $c=c_{BE}>0$ such that
 \begin{equation}\label{BE-24}
 \inf\limits_{f\in\mathcal{M}}\big\|\nabla\big(u-f\big)\big\|_{L^2}^2\leq c_{BE}^{-1}\big(\big\|\nabla u\big\|_{L^2}^2-S^2\big\|u\big\|_{L^{2^\ast}}^2\big).
 \end{equation}
Recently, K\"{o}nig proved that the sharp constant $ c_{BE}$
must be strictly smaller than $\frac{4}{n+4}$ and inequality \eqref{BE-24} always admits a minimizer by establishing two crucial energy estimates of $c_{BE}$ (cf. \cite{T-K-23}, \cite{T-K-24}). After Bianchi-Egnell's seminal work, one would like to know whether Brezis-Lieb's problem could also be solved for the fractional Sobolev inequality or the quasilinear case with general values of $p$. In fact The approach of Bianchi-Egnell in \cite{BE91} depended  heavily on the Hilbert structure of $D^{1,2}(\mathbb{R}^n)$ and on the eigenvalue
properties of a weighted Laplacian in $\mathbb{R}^n$, so it does not seem suitable for the general case (for $p\neq2$) (cf. \cite{C-F-M-P,F-N-18,F-Z-22}).
For the general fractional Sobolev inequalities(cf.\cite{CFW13}), the first-order Sobolev inequality with the explicit lower bound (cf.\cite{D-E-F-F-L}), the hight-order Sobolev inequality (cf.\cite{L-W-99,B-W-W,G-W10}) and
the Hardy-Littlewood-Sobolev (HLS) inequalities (cf. \cite{Ca17, C-L-T, Dolbeault-1, DE22}). Another approach has been developed for studying the f quantitative stability of functional/geometric inequalities: the isoperimetric inequalities (cf. \cite{F-M-P10,F-F-P08,C-E-S}), and conformally invariant Sobolev inequalities on Riemannian manifolds (cf. \cite{E-N-S22,F22}) and so on.

A challenging problem is to study the qualitative stability for the critical points of the Euler-Lagrange equations. Informally, whether a function $u$ almost solves \eqref{bec} must be quantitatively close to Aubin-Talenti bubbles?
In fact, a seminal work of Struwe in \cite{Struwe-1984} proved the
well-known stability of profile decompositions of \eqref{bec}, that is $u$ should be sufficiently close to a sum of weakly-interacting bubbles even if we restrict to nonnegative functions. After the poineering paper \cite{Struwe-1984},
Ciraolo, Figalli, and Maggi in \cite{CFM18} studied the stability results of the Euler-Lagrange equations for the embedding $D^{1,2}(\mathbb{R}^n)\hookrightarrow L^{\frac{2n}{n-2}}(\mathbb{R}^n)$. They
proved firstly a sharp quantitative stability result around one-bubble case $\kappa=1$ and $n\geq3$. Later, Figalli and Glaudo in \cite{FG20} gave a positive result for any dimension $n\in\mathbb{N}$ and $\kappa=1$ and for dimension $3\leq n\leq5$ when the multi-bubbles case $\kappa\geq2$.
In addition, Figalli and Glaudo constructed counter-examples showing that when $n\geq6$, one may have
\begin{equation*}
dist(u,\mathcal{T})\gg\Gamma(u)\hspace{2mm}\mbox{with}\hspace{2mm}\Gamma(u):=\big\|\Delta u+u^{\frac{n+2}{n-2}}\big\|_{(D^{1,2}(\mathbb{R}^n))^{-1}},
\end{equation*}
where $\mathcal{T}$ denotes the manifold of sums of Aubin-Talenti bubbles and
\begin{equation*}
dist(u,\mathcal{T})=\inf\limits_{z_1,z_2,\cdots,z_\kappa\in\mathbb{R}^n,
\lambda_1,\lambda_2,\cdots,\lambda_{\nu}\in\mathbb{R}}\big\|\nabla u-\nabla\big(\sum_{i=1}^\kappa U[z_i,\lambda_i]\big)\big\|_{L^2},
\end{equation*}
and they also propose some conjectures in higher dimension $n\geq6$.
Very recently, Deng, Sun and Wei in \cite{DSW21} proved the optimal estimates for $n\geq6$ by applying finite-dimensional reduction method.
Following the rigidity result \cite{Chen-ou}, an analog of Struwe's result was given by Palatucci and Pisante in \cite{P-P-15} and \cite{T-Konig23}.
 Aryan \cite{Aryan} studied the fractional case $(-\Delta)^{s}$ with $s\in(0,\frac{n}{2})$ and proved the sharp quantitative
results for the critical points of the fractional Sobolev inequalities with $s\in(0,1)$. Chen et al. in \cite{C-K-L-24} investigated the one bubble case $\kappa=1$ and the multi-bubbles case $\kappa\geq2$ for fractional Sobolev inequalities for all $n\in\mathbb{N}$ and $s\in(0,\frac{n}{2})$. We would also like to refer to \cite{Z} for the stability of fractional Sobolev trace inequality within both the functional and critical point settings, and references therein.

Inspired by the above mentioned results for the local Sobolev inequality, it is natural to wonder that if we can investigate the nonlocal case and establish similar stablity results for the nonlocal Sobolev inequalities. Recently, Deng et al. \cite{DSB213} obtained the first result on the gradient-type remainder term of inequality \eqref{Prm}, which can be expressed as follow:
 \[
    C_2{\rm dist}(u,\mathcal{Z})^2
    \geq \int_{\mathbb{R}^n}|\nabla u|^2 dx
    -C_{HLS}\left(\int_{\mathbb{R}^n}\big(|x|^{-\mu} \ast |u|^{p}\big)|u|^{p} dx\right)^{\frac{1}{p}}
    \geq C_1 {\rm dist}(u,\mathcal{Z})^2,
    \]
    where
    $$
    \mathcal{Z}=\big\{cW[\xi,\lambda]:~c\in\mathbb{R},~\xi\in\mathbb{R}^n,~\lambda>0\big\}
    $$
    is an $(n+2)$-dimensional manifold, and ${\rm dist}(u,\mathcal{Z }):=\inf\limits_{c\in\mathbb{R}, \lambda>0, z\in\mathbb{R}^N}\|u-cW[\xi,\lambda]\|_{D^{1,2}(\mathbb{R}^n)}$.
Another way to investigate the stability issue of the nonlocal Sobolev inequality\eqref{Prm} is to study the
stability of profile decompositions to \eqref{ele-1.1} for nonnegative functions
 of nonlocal equation \eqref{ele-1.1}. Inspired by the spirit of Struwe in \cite{Struwe-1984}, a nonlocal version of the stability of profile decompositions to \eqref{ele-1.1} for nonnegative functions can be formulated as follows:\\
{\it {\bf Theorem A.}(\cite{p-y-z24})
   Let $n\geq3$ and $\kappa\geq1$ be positive integers. Let $(u_m)_{m\in\mathbb{N}}\subseteq D^{1,2}(\mathbb{R}^n)$ be a sequence of nonnegative functions such that $$(\kappa-\frac{1}{2})C_{HLS}^{\frac{2n-\mu}{n+2-\mu}}\leq\big\|u_m\big\|_{D^{1,2}(\mathbb{R}^n)}^2\leq(\kappa+\frac{1}{2})C_{HLS}^{\frac{2n-\mu}{n+2-\mu}}$$ with $C_{HLS}=C_{HLS}(n,\mu)$ as in \eqref{Prm}, and assume that
   $$
    \Big\|\Delta u_m+\left(|x|^{-\mu}\ast |u_m|^{p}\right)|u_m|^{p-2}u_m\Big\|_{(D^{1,2}(\mathbb{R}^n))^{-1}}\rightarrow0\quad \mbox{as}\quad m\rightarrow\infty.
    $$
Then, there exist $\kappa$ tuples of points $(\xi_{1}^{(m)},\cdots, \xi_{\kappa}^{(m)})_{m\in\mathbb{N}}$ in $\mathbb{R}^n$ and $\kappa$ tuples of positive real numbers $(\lambda_{1}^{(m)},\cdots,\lambda_{\kappa}^{(m)})_{m\in\mathbb{N}}$  such that
\[
\Big\|\nabla\Big(u_m-\sum_{i=1}^{\kappa}W[\xi_i^{(m)},\lambda_{i}^{(m)}]\Big)\Big\|_{L^2}\rightarrow0\quad \mbox{as}\quad m\rightarrow\infty.
\]}
The above theorem gives us a qualitative answer to the almost rigidity problem posed earlier.
Let $(W_i)_{i=1}^{\kappa}:=(W[\xi_i,\lambda_i])_{i=1}^{\kappa}$ be a family of bubbles. Define the quantity by
\begin{equation*}		Q_{ij}(\xi_i,\xi_j,\lambda_i,\lambda_j)=\min\Big(\frac{\lambda_i}{\lambda_j}+\frac{\lambda_j}{\lambda_i}+\lambda_i\lambda_j|\xi_i-\xi_j|^2\Big)^{-\frac{n-2}{2}}.
	\end{equation*}
It is noticing that $Q_{ij}\rightarrow0$ for any $i,j=1,\cdots,\kappa$ as $m\rightarrow\infty$,
\begin{equation*}
dist_{D^{1,2}}\big(u,\mathcal{M}_0\big)= o\big(\big\|\hat{f}\big\|_{(D^{1,2}(\mathbb{R}^n))^{-1}}\big),
\end{equation*}
where
$\hat{f}:=\Delta u+\left(|x|^{-\mu}\ast u^{p}\right)u^{p-1}$ and the manifold of sums of bubbles $W_i$ given by
$$\mathcal{M}_0=\Big\{\sum_{i=1}^{\kappa}W_i:~\xi_i\in\mathbb{R}^n, ~\lambda_i>0\Big\}.$$
And we say that the family of bubbles is $\delta$-interacting in the following sense:
\begin{equation*}
\mathscr{Q}:=\max\big\{Q_{ij}(\xi_i,\xi_j,\lambda_i,\lambda_j): \quad i,~j=1,\cdots,\kappa\big\}\leq\delta.
\end{equation*}

A quantitative stability problem is the following: Suppose that $u$ is in the neighborhood of a sum of weakly-interacting bubbles $(W_i)_{i=1}^{\kappa}$, can we show that
\begin{equation*}
dist_{D^{1,2}}\big(u,\mathcal{M}_0\big)\lesssim \big\|\hat{f}\big\|_{(D^{1,2}(\mathbb{R}^n))^{-1}},
\end{equation*}
or some kind of quantitative control?
Only very recently the first qualitative stability for
critical points of equation \eqref{ele-1.1} was established in \cite{p-y-z24}(cf. \cite{L}). They generalized the results by Figalli and Glaudo in \cite{FG20} for the Sobolev inequality \eqref{bsic} to a nonlocal version of Sobolev inequality \eqref{Prm} based on the spectrum analysis. {\bf However, the authors in \cite{p-y-z24} only proved the stability results for the one bubble case without dimension restriction and the multiple bubbles case $\kappa\geq2$ if dimension $3\leq n<6-\mu$ and $\mu\in(0,n)$ with $\mu\in(0,4]$.  The aim of the present paper is to establish the quantitative stability estimates for the critical points of inequality \eqref{Prm} for $n\geq6-\mu$ and $\mu\in(0,4]$.}

Let $n\geq6-\mu$ and $\mu\in(0,4]$, we introdcue function $\tau_{n,\mu}(x)$ as
\begin{equation}\label{7-h-0-0}
\tau_{n,\mu}(x):=
\left\lbrace
\begin{aligned}
&
x^{\frac{2^{\ast}-1}{2}}\hspace{9mm}\hspace{8mm}\mbox{for}\hspace{2mm}0<\mu<\frac{n+\mu-2}{2},\\
&
x^{2-\frac{\mu}{n-2}}\hspace{5mm}\hspace{4mm}\hspace{6mm}\mbox{for}\hspace{2mm}\frac{n+\mu-2}{2}\leq\mu<4,
\end{aligned}
\right.
\end{equation}
and $\tau_{n,\mu}(x)$ is increasing near $0$.

The main results are stated in the following theorem:
    \begin{thm}\label{Figalli}
Suppose that $n\geq6-\mu$, $\mu\in(0,n)$ and $\mu\in(0,4]$ satisfy the following assumption
\begin{equation*}
(\sharp)\hspace{4mm}\frac{n^2-6n}{n-4}<\mu<4\hspace{2mm}\mbox{and}\hspace{2mm}n\geq6-\mu,
\end{equation*}
and the number of bubbles $\kappa\geq2$. Then there exist a small constant $\delta=\delta(n,\mu,\kappa)>0$ and a large constant $C=C(n,\mu,\kappa)>0$ such that the following statement holds. If $u\in D^{1,2}(\mathbb{R}^n)$ satisfies
  \begin{equation}\label{w-tittle}
\Big\|\nabla u-\sum_{i=1}^{\kappa}\nabla \widetilde{W}_i\Big\|_{L^2}\leq\delta
\end{equation}
for some $\delta$-interacting family
$(\widetilde{W}_i)_{i=1}^{\kappa}$,
then there is a family $(W_i)_{i=1}^{\kappa}$ of bubbles such that
\begin{equation}\label{tnu}
dist_{D^{1,2}}\big(u,\mathcal{M}_0\big)\leq C\tau_{n,\mu}\big(\big\|\hat{f}\big\|_{(D^{1,2}(\mathbb{R}^n))^{-1}}\big)
\end{equation}
Moreover, for any $i\neq j$, the interaction between the bubbles can be estimated as
\begin{equation}\label{moreover}
\int_{\mathbb{R}^n}\big(|x|^{-\mu}\ast W_i^{p}\big)W_i^{p-1}W_j\leq C\big\|\hat{f}\big\|_{(D^{1,2}(\mathbb{R}^n))^{-1}}.
\end{equation}
    \end{thm}
As a direct consequence of Theorem A and Theorem \ref{Figalli}, we can conclude the following corollary, proving the desired quantitative estimates of profile decomposition.
 \begin{cor}\label{Figalli2}
  For any dimension $n\geq6-\mu$, $\mu\in(0,n)$ and $\mu\in(0,4]$ satisfying $(\sharp)$, and the number of bubbles $\kappa\geq2$, there exist a constant constant $C=C(n,\mu,\kappa)>0$ such that the following statement holds. For any nonnegative function $u\in D^{1,2}(\mathbb{R}^n)$ such that
\begin{equation*}
\big(\kappa-\frac{1}{2})C_{HLS}^{\frac{2n-\mu}{n+2-\mu}}\leq\|u\|_{D^{1,2}(\mathbb{R}^n)}^2\leq\big(\kappa+\frac{1}{2}\big)C_{HLS}^{\frac{2n-\mu}{n+2-\mu}},
\end{equation*}
then there exist $\kappa$ bubbles $(W_i)_{i=1}^{\kappa}$ such that
\begin{equation*}
dist_{D^{1,2}}\big(u,\mathcal{M}_0\big)\leq C\big\|\hat{f}\big\|_{(D^{1,2}(\mathbb{R}^n))^{-1}}.
\end{equation*}
Furthermore, for any $i\neq j$, the interaction between the bubbles can be estimated as
\begin{equation*}
\int_{\mathbb{R}^n}\big(|x|^{-\mu}\ast W_i^{p}\big)W_i^{p-1}W_j\leq C\big\|\Delta u+\left(|x|^{-\mu}\ast u^{p}\right)u^{p-1}\big\|_{(D^{1,2}(\mathbb{R}^n))^{-1}}.
\end{equation*}
    \end{cor}
    We conclude this subsection with some remarks.
\begin{Rem}
For $\kappa\geq2$, the authors in \cite{p-y-z24} proved the first stability estimates when $3\leq n<6-\mu$ and $0<\mu<4$, namely that the parameters $n$, $\mu$ are chosen in the range: $$n=3\hspace{2mm}\mbox{and}\hspace{2mm}\mu\in(0,3),\hspace{2mm}\mbox{or}\hspace{2mm}n=4\hspace{2mm}\mbox{and}\hspace{2mm}\mu\in(0,2),\hspace{2mm}\mbox{or}\hspace{2mm}n=5\hspace{2mm}\mbox{and}\hspace{2mm} \mu\in(0,1).$$ It would be worthwhile to consider such issues for quantitative estimates of nonlocal Sobolev inequality in the case $n$ and $\mu$ satisfying $(\sharp)$, namely that
\begin{itemize}
\item[$\bullet$] $n=4$ and $\mu\in[2,4)$, or $n=5$ and $\mu\in[3,4)$.
\item[$\bullet$] $n=5$ and $\mu\in[1,3)$, or $n=6$ and $\mu\in(0,4)$, or $n=7$ and $\mu\in(\frac{7}{3},4)$.
\end{itemize}
\end{Rem}
\begin{Rem}\label{re1.4}
It is interesting to observe that the above parameters is essentially optimal for all $n\geq6-\mu$ and $0<\mu\leq4$. To see this,
 let us define the Riesz potentials by
\begin{equation*}
\mathcal{I}_{\mu}\{f\}(x):=\int|x-y|^{-\mu}f(y)dy.
\end{equation*}
Due to the appearance of nonlocal interaction
part,  the arguments become much more complicated. We need to calculate carefully the order of weight functions $s_{i,l}$, $t_{i,l}$, $\hat{s}_{i,l}$, $\hat{t}_{i,l}$, $\tilde{s}_{i,l}$ and $\tilde{t}_{i,l}$ (where $l=1,2$) for the new total weights $S_j$ and $T_j(j=1,2,3)$ defiend in the section 3. Furthermore, in order to carry out the reduction arguments in suitable weighted spaces, one has to establish new entire convolution estimates and $C^0$ estimates to overcome the difficulties caused by nonlocal interactions.
Moreover, to apply Contraction Mapping Theorem to solve problem \eqref{c19}, we need suitable Lipschitz property of the operator $\mathcal{A}$ defined in section 6, which requires that the integral $\mathcal{I}_{\mu}\{(1+|\lambda_i(x-\xi_i)|^2)^p\}(x)$ is finite and well-defined. Hence the parameters $n$ and $\mu$ need to satisfy the following restrictions
$$0<\mu\leq4\hspace{2mm}\mbox{and}\hspace{2mm}\mu>\frac{n^2-6n}{n-4}.$$
Still, it is important to notice that the parameters $\mu=4$ does not valid for achieving the desired estimates in Lemma \ref{ww101}, Lemma \ref{cll-1-0-1}-Lemma \ref{cll-1}.
\end{Rem}
\subsection{Strategy of the proof}\
\newline
The proof of Theorem \ref{Figalli} for the case $n\geq6-\mu$ follows by adapting the strategy in \cite{DSW21}. However we need to overcome some difficulties caused by the Hartree term.
 Denote the error between $u$ and the best approximation $\sigma=\sum_{i=1}^{\kappa}W_i$ by $\rho$, i.e.
$u=\sigma+\rho.$
We are lead to the following decomposition, i.e.
\begin{equation}\label{u-0}
\Delta \rho+\Phi_{n,\mu}[\sigma,\rho]+\hbar+\mathscr{N}(\rho)-\Delta u-\big(|x|^{-\mu}\ast u^{p}\big)u^{p-1}=0,
\end{equation}
where
\begin{equation}\label{I-FAI-1}
\Phi_{n,\mu}[\sigma,\rho]:=p\Big(|x|^{-\mu}\ast \sigma^{p-1}\rho\Big)
\sigma^{p-1}+(p-1)\Big(|x|^{-\mu}\ast\sigma^{p}\Big)
\sigma^{p-2}\rho
\end{equation}
\begin{equation}\label{u-1}
\begin{split}
\hbar:=\Big(|x|^{-\mu}\ast\sigma^{p}\Big)\sigma^{p-1}-\sum_{i=1}^{\kappa}\Big(|x|^{-\mu}\ast W_{i}^{p}\Big)W_{i}^{p-1},
\end{split}
\end{equation}
and
\begin{equation}\label{u-2}
\begin{split}
\mathscr{N}(\rho)
&:=\Big(|x|^{-\mu}\ast(\sigma+\rho)^{p}\Big)(\sigma+\rho)^{p-1}
-\Big(|x|^{-\mu}\ast\sigma^{p}\Big)\sigma^{p-1}\\&
~~~-p
\Big(|x|^{-\mu}\ast\sigma^{p-1}\rho\Big)
\sigma^{p-1}-(p-1)
\Big(|x|^{-\mu}\ast\sigma^{p}\Big)
\sigma^{p-2}\rho.
\end{split}
\end{equation}
From \eqref{w-tittle} we know $\|\rho\|_{(D^{1,2}(\mathbb{R}^n))^{-1}}\leq\delta$. Moreover, we will further decompose $\rho=\rho_0+\rho_1$ and prove the existence of the first approximation $\rho_0$ by solving the following system
\begin{equation}\label{coefficients}
	\left\{\begin{array}{l}
		\displaystyle \Delta \phi+\Phi_{n,\mu}[\sigma,\phi]=\hbar+	\sum_{i=1}^{\kappa}\sum_{a=1}^{n+1}c_{a}^{i}\Phi_{n,\mu}[W_{i},\Xi^{a}_i]\hspace{2mm}\mbox{in}\hspace{4mm} \mathbb{R}^n,\\
		\displaystyle 	\int\Phi_{n,\mu}[W_{i},\Xi^{a}_i]\phi=0,\hspace{4mm}i=1,\cdots, \kappa; ~a=1,\cdots,n+1,
	\end{array}
	\right.
\end{equation}
where $c_a^i$ is family of scalars and $\Xi_i^a$ are the rescaled derivative of $W[\xi,\lambda_i]$ defined as follows
\begin{equation}\label{qta}
\begin{split}
&\Xi_i^a=\frac{1}{\lambda_i}\frac{\partial W[\xi,\lambda_i]}{\partial \xi^{a}}\Big|_{\xi=\xi_i}=(2-n)W[\xi,\lambda_i] \frac{\lambda_i(\cdot^{a}-\xi^a)}{1+\lambda^2|\cdot-x|^2}, \hspace{3mm}\mbox{for}\hspace{2mm}i=1,\cdots,\kappa,\\&
\Xi_{i}^{n+1}=\lambda_{i}\frac{\partial W[\xi_{i},\lambda]}{\partial \lambda}\Big|_{\lambda=\lambda_i}=\frac{n-2}{2}W[\xi,\lambda_i]
			\frac{1-\lambda_i^2|\cdot-x|^2}{1+\lambda_i^2|\cdot-x|^2},\hspace{2mm}\mbox{for}\hspace{2mm}i=1,\cdots,\kappa,
\end{split}
\end{equation}
 $\xi^{a}$ is $a$-th component of $\xi$ for $a=1,\cdots,n$.

We need to point out that a key idea in our proof is to set up weight functions and norms (see the section 3.1). Due to this, we can prove the behavior of interaction of bubbles $\hbar$ (see Lemma \ref{estimate1}): there exists a dimensional constant $C$ (depending only on $n$, $\kappa$ and $\mu$) such that
$$\hbar\lesssim \Big( T_1(x) \hspace{2mm}\mbox{if}\hspace{2mm}0<\mu<\frac{n+\mu-2}{2};\hspace{2mm}T_2(x)\hspace{2mm}\mbox{if}\hspace{2mm}\frac{n+\mu-2}{2}\leq\mu<4\Big).$$
From this inequality and together with some estimates of integral quantities, we can deduce a upper bound of coefficients $c_b^j$ in \eqref{coefficients} (see Lemma \ref{cll}). Furthermore, our second goal consists in showing that the existence and point-wise estimate of $\rho_0$ in section 6,
 \begin{equation*}
|\rho_{0}|\lesssim
\left\lbrace
\begin{aligned}
& S_{1}(x),\hspace{4mm}n=5\hspace{2mm}\mbox{and}\hspace{2mm}\mu\in[1,3),\mbox{or}\hspace{2mm}n=6\hspace{2mm}\mbox{and}\hspace{2mm}\mu\in(0,4),\mbox{or}\hspace{2mm}n=7\hspace{2mm}\mbox{and}\hspace{2mm}\mu\in(\frac{7}{3},4),\\
& S_{2}(x),\hspace{4mm}n=4\hspace{2mm}\mbox{and}\hspace{2mm}\mu\in[2,4),\mbox{or}\hspace{2mm}n=5\hspace{2mm}\mbox{and}\hspace{2mm}\mu\in[3,4).
\end{aligned}
\right.
\end{equation*}
To do this, we shall establish a priori estimate $\|\phi\|_{\ast}\lesssim\|\hbar\|_{\ast\ast}$ in Lemma \ref{estimate2}, which
is indeed the core of our argument in this paper. Finally, we show that $L^2$ estimate for $\nabla\rho_0$ and $\nabla\rho_1$ where $\rho=\rho_0+\rho_1$.
Combining all these ingredients, we prove Theorem \ref{Figalli}.

Deducing Theorem \ref{Figalli} is based on delicate the energy method, the
reduction and blow-up argument, which relies on a series of a-priori estimates in sections 4-7. In particular, the proof of Lemma \ref{estimate2} is most challenging part of the entire proof, which relies on a series of estimates:
\begin{itemize}
\item[$(1)$] We define several kinds of weight functions (depending on $n$ and $\mu$) $s_{i,1}$, $s_{i,2}$, $t_{i,1}$, $t_{i,2}$ and $\hat{s}_{i,1}$ $\hat{s}_{i,2}$, $\hat{t}_{i,1}$ and $\hat{t}_{i,2}$ in different ranges that follow by choosing suitable parameters $n$ and $\mu$.
\item[$(2)$] Proving a rough upper bound of solution of \eqref{coefficients} which is depended on the comparative relationship between $\widetilde{H}_{j}(x)s_{i,1}$, $\widetilde{H}_{j}(x)s_{i,2}$, $\widetilde{H}_{j}(x)\hat{s}_{i,1}$, $\widetilde{H}_{j}(x)\hat{s}_{i,2}$, Hartree-type cross terms (as described in Lemma \ref{cll-2} and Lemma \ref{cll-3} below) and $t_{i,1}$, $t_{i,2}$, $\hat{t}_{i,1}$, $\hat{t}_{i,2}$ separately.
   \item[$(3)$]
    To conclude the proof of step \ref{step5.2} in Lemma \ref{estimate2} by proving some useful convergence results in Proposition \ref{converges-0} and Proposition \ref{converges-4}, and the removability of singularities of a solution in Proposition \ref{converges-3}.
\end{itemize}
\subsection{Structure of the paper}\
\newline
The paper is organized as follows. Theorem~\ref{Figalli} is proved in Section 3. We first establish some the key estimates for integral quantities involving two Talenti bubbles and choose appropriate parameters to deduce the convolution terms. Later on, we give a new definition of weight spaces and norms in Section 2.  In order to get existence of $\rho_0$ in Section 6 and the point-wise estimate in Section 7, we analysis that bubbles with weak interaction, the structure of bubbles tree and some $C^0$ estimates of the error function $\rho$ by Green's representation in Section 4. In Section 5, we deduce a important priori estimate for $\|\cdot\|_{\ast}$.

Throughout this paper, $C$ are indiscriminately used to denote various absolutely positive constants. We say that $a\lesssim b$ if $a\leq Cb$, $a\approx b$ if $a\lesssim b$ and $\gtrsim b$.

\section{Estimates of some integral quantities}
In this section we begin by proving a series of estimates by choosing suitable parameters in convolution terms and computing the integral quantities involving two bubbles. The estimates obtained in Lemmas \ref{FPU1}-\ref{B4-1} are crucial in proving our main results. Let $I=\{1,\cdots,\kappa\}$, throughout this paper we denote by $z_i=\lambda_i(x-\xi_i)$, $i\in I$.
We recall some known results that will be used in the next Lemmas.
\begin{lem}\label{FPU1} Given $n\geq6-\mu$, $\mu\in(0,n)$ and $0<\mu\leq4$, let $W_i$ and $W_j$ be two bubbles. Then, for any fixed $\varepsilon>0$ and any nonnegative exponents such that $\tilde{s}+\tilde{t}=2^\ast$, it holds that
\begin{equation*}
\int W_i^{\tilde{s}}W_j^{\tilde{t}}\approx
\left\lbrace
\begin{aligned}
&Q_{ij}^{\min(\tilde{s},\tilde{t})},\quad\hspace{6mm}\hspace{3mm}\quad|\tilde{s}-\tilde{t}|\geq\varepsilon,\\&
Q_{ij}^{\frac{n}{n-2}}\log(\frac{1}{Q_{ij}}),\quad\quad \tilde{s}=\tilde{t},
\end{aligned}
\right.
\end{equation*}
where the quantity
\begin{equation*}	Q_{ij}=\min\Big(\frac{\lambda_i}{\lambda_j}+\frac{\lambda_j}{\lambda_i}+\lambda_i\lambda_j|\xi_i-\xi_j|^2\Big)^{-\frac{n-2}{2}}.
	\end{equation*}
\end{lem}
\begin{proof}
See the proof of Proposition B.2 in \cite{FG20}.
\end{proof}
\begin{lem}\cite{FG20}\label{FPU2}
Let $\tilde{s}>\tilde{t}>1$ and $\tilde{s}+\tilde{t}=2^{\ast}$. It holds
$$\int W_i^t\inf(W_i^{\tilde{s}},W_i^{\tilde{t}})=O\big(Q_{ij}^{\frac{n}{n-2}}|\log{Q_{ij}}|\big).$$
\end{lem}
x
\begin{lem}\label{FPU3}
Suppose that $n\geq6-\mu$, $\mu\in(0,n)$ and $0<\mu\leq4$, let $W_1$, $W_2$ and $W_3$ be three bubbles with $\delta$-interaction. Moreover we denote $r=\frac{2n}{2n-\mu}$. It holds
\begin{itemize}
\item[$\bullet$] For $n=6-\mu$, we have
\begin{equation}\label{w-w-w-1}
\int W_{1}^{r}W_{2}^rW_{3}^r\lesssim \mathscr{Q}^{\frac{6-\mu}{4-\mu}}|\log{\mathscr{\mathscr{Q}}}|\hspace{2mm}\mbox{with}\hspace{2mm}\mathscr{Q}:=\max\{Q_{12},Q_{13},Q_{23}\}.
\end{equation}
\item[$\bullet$] For $n>6-\mu$, we have
\begin{equation}\label{2-ww-1}
\int W_{1}^{(p-2)r}W_{2}^rW_{3}^r=\left\lbrace
\begin{aligned}
&\mathscr{Q}^{\frac{6(4-\mu)}{8-\mu}}\big|\log{(\frac{1}{\mathscr{Q}})}\big|^{\frac{3(4-\mu)}{8-\mu}},\hspace{7mm}\quad\quad\quad\quad n=4,\\&
\mathscr{Q}^{\frac{(4-\mu)r}{n-2}+\frac{n\mu}{2(2n-\mu)}}\big|\log{(\frac{1}{\mathscr{Q}})}\big|^{\frac{\mu(n-2)}{2(2n-\mu)}}, \quad\quad n>4.
\end{aligned}
\right.
\end{equation}
\end{itemize}
\end{lem}
\begin{proof}
For $n=6-\mu$, $3r=2^{\ast}$, by the H\"{o}lder inequality, we have
$$\int W_{1}^{r}W_{2}^rW_{3}^r\leq\Big(\int W_{1}^{\frac{3r}{2}}W_{2}^\frac{3r}{2}\Big)^{\frac{1}{3}}\Big(\int W_{1}^{\frac{3r}{2}}W_{3}^\frac{3r}{2}\Big)^{\frac{1}{3}}\Big(\int W_{2}^{\frac{3r}{2}}W_{2}^\frac{3r}{2}\Big)^{\frac{1}{3}}.$$
Combining Lemma \ref{FPU1}, we get
$$\int W_{1}^{r}W_{2}^rW_{3}^r\lesssim Q_{12}^{\frac{6-\mu}{3(4-\mu)}}|\log{Q_{12}}|^{\frac{1}{3}}Q_{13}^{\frac{6-\mu}{3(4-\mu)}}|\log{Q_{13}}|^{\frac{1}{3}}Q_{23}^{\frac{6-\mu}{3(4-\mu)}}|\log{Q_{23}}|^{\frac{1}{3}}.$$
Choosing $\delta$ small and the conclusion \eqref{w-w-w-1} follows.

For $n>6-\mu$, since
$$\frac{2}{\tilde{p}}+\frac{1}{\tilde{q}}=1\hspace{2mm}\mbox{with}\hspace{2mm}\tilde{p}:=\frac{4(2n-\mu)}{n(4-\mu)},\hspace{2mm}\tilde{q}:=\frac{2(2n-\mu)}{\mu(n-2)},$$
the H\"{o}lder inequality implies that
\begin{equation*}
\begin{split}
\widetilde{\Re}&=:\int W_{1}^{(p-2)r}W_{2}^rW_{3}^r\\&\leq\Big(\int W_{1}^{\frac{2(2n-\mu)r}{n(n-2)}}W_{2}^\frac{(2n-\mu)r}{n}\Big)^{\frac{1}{\tilde{p}}}\Big(\int W_{1}^{\frac{2(2n-\mu)r}{n(n-2)}}W_{3}^\frac{(2n-\mu)r}{n}\Big)^{\frac{1}{\tilde{p}}}\Big(\int W_{2}^{\frac{(2n-\mu)r}{2(n-2)}}W_{3}^\frac{(2n-\mu)r}{2(n-2)}\Big)^{\frac{1}{\tilde{q}}}.
\end{split}
\end{equation*}
If $N=4$, we have that
$$
\widetilde{\Re}\lesssim Q_{12}^{\frac{2(4-\mu)}{8-\mu}}\big|\log{(\frac{1}{Q_{12}})}\big|^{\frac{4-\mu}{8-\mu}}Q_{13}^{\frac{2(4-\mu)}{8-\mu}}\big|\log{(\frac{1}{Q_{13}})}\big|^{\frac{4-\mu}
{8-\mu}}Q_{23}^{\frac{2\mu}{8-\mu}}\big|\log{(\frac{1}{Q_{23}})}\big|^{\frac{\mu}{8-\mu}}
$$
by Lemma \ref{FPU1}. If $N>4$, we have that
$$
\widetilde{\Re}\lesssim Q_{12}^{\frac{(4-\mu)r}{2(n-2)}}Q_{13}^{\frac{(4-\mu)r}{2(n-2)}}Q_{23}^{\frac{n\mu}{2(2n-\mu)}}\big|\log{(\frac{1}{Q_{23}})}\big|^{\frac{\mu(n-2)}{2(2n-\mu)}}
$$
by Lemma \ref{FPU1}. Noting that the functions $$f(x):=x^{\frac{8-\mu}{2(4-\mu)}}|\log{x}|^{\frac{4-\mu}{8-\mu}}\hspace{2mm}\mbox{and}
\hspace{2mm}g(x):=x^{\frac{n\mu}{2(2n-\mu)}}|\log{x}|^{\frac{\mu(n-2)}{2(2n-\mu)}}$$
 increasing near zero. Combining the above inequalities yield the conclusion provided that $\delta$ is small enough.
\end{proof}
The following lemma holds true:
\begin{lem}\label{wwc101}
Suppose that the exponents $\mu$ and $n$ satisfy  $n\geq6-\mu$, $\mu\in(0,n)$ and $0<\mu\leq4$.  There exists a $\Gamma_1^a>0$ such that
\begin{equation*}
\begin{split}
 \int\Phi_{n,\mu}[W_{j},\Xi^{a}_j]\Xi_{j}^{b}=
\left\lbrace
\begin{aligned}
&0,\hspace{7.3mm}\hspace{2mm}a\neq b,\\
&\Gamma_0^a,\hspace{5mm}\hspace{2mm}1\leq a=b\leq n+1.
\end{aligned}
\right.
\end{split}
\end{equation*}
If $i\neq j$, there holds
\begin{equation*}
\Big|\int\Big(|x|^{-\mu}\ast W_{i}^{p}\Big)W_{i}^{p-2}\Xi_{i}^{a}\Xi_{j}^{b}+\int\Big(|x|^{-\mu}\ast (W_{i}^{p-1}\Xi_{i}^{a})\Big)W_{i}^{p-1}\Xi_{j}^{b}\Big|\lesssim Q_{ij}, \hspace{2mm}\mbox{if}\hspace{2mm}1\leq a,b\leq n+1.
\end{equation*}
Here $\Gamma_0^1=\cdots=\Gamma_0^n$ and $\Gamma^{n+1}_0$ are composed of some $\Gamma$ functions. Moreover,
we have that the following equality holds
\begin{equation*}
\begin{split}
p\int
&\Big(|x|^{-\mu}\ast W_m^{p-1}\partial_{\lambda_m} W_m\Big)
W_m^{p-1}W_i+(p-1)\int
\Big(|x|^{-\mu}\ast W_m^{p}\Big)
W_m^{p-2}W_i\partial_{\lambda_m} W_m\\&
=\int\Big(|x|^{-\mu}\ast W_i^{p}\Big)
W_i^{p-1}\partial_{\lambda_m} W_m.
\end{split}
\end{equation*}
\end{lem}
\begin{proof}
 It holds:
$$
 \Delta \Xi_j^a+p\Big(|x|^{-\mu}\ast W_{j}^{p-1}\Xi_j^a\Big)W_{j}^{p-1}+(p-1)\Big(|x|^{-\mu}\ast W_{j}^{p}\Big)W_{j}^{p-2}\Xi_j^a=0
$$
for $1\leq a\leq N+1$.
Therefore, it is easy to check that
\begin{equation*}
\begin{split}
p&\int\Big(|x|^{-\mu}\ast W_{j}^{p-1}\Xi_j^a\Big)W_{j}^{p-1}\Xi_j^b+(p-1)\int\Big(|x|^{-\mu}\ast W_{j}^{p}\Big)W_{j}^{p-2}\Xi_j^a\Xi_j^b\\&=\int\nabla\Xi_j^a\nabla\Xi_j^b=C\int U_i^{2^{\ast}-1}\nabla\big(\frac{1}{\lambda_{i}}\frac{\partial U[\xi,\lambda_i]}{\partial \xi^a}\Big|_{\xi=\xi_i}\big)\nabla\big(\frac{1}{\lambda_{i}}\frac{\partial U[\xi,\lambda_i]}{\partial \xi^b}\Big|_{\xi=\xi_i}\big).
\end{split}
\end{equation*}
Moreover, by lemma \ref{p1-00} gives that
\begin{equation*}
|x|^{-\mu}\ast W_m^{p}
=\widetilde{\alpha}_{n,\mu}W_m^{2^{\ast}-p}(x).
\end{equation*}
Hence one can follow from the same argument as in the proof of \cite{FG20}.
Note that it holds
$$\Delta(\partial_{\lambda_m}W_m)+p\Big(|x|^{-\mu}\ast W_{m}^{p-1}\partial_{\lambda_m} W_m\Big)W_{m}^{p-1}+(p-1)\Big(|x|^{-\mu}\ast W_{m}^{p}\Big)W_{m}^{p-2}\partial_{\lambda_m} W_m=0.$$
Then, the conclusion follow by simple integration by parts. The Lemma is obtained.
\end{proof}

\subsection{Some estimates of the convolution terms.}\
\newline
In this subsection we begin by proving a series of convolution estimates by choosing suitable parameters $n$ and $\mu$. Then, we shall list all the constraints of the constants $n$, $\mu$ which are sufficient for the reduction argument scheme to work. We denote in what follows
\begin{equation*}
\begin{split}
&\mathscr{J}_{\mu}^{(1)}(x):=\frac{1}{|x|^{\mu}}\ast \frac{\lambda_i^{n-\mu/2}}{\tau( z_i)^{2p}}\hspace{4mm}\hspace{3mm}\hspace{3mm}\hspace{2mm}\hspace{2mm}\hspace{2mm}\mbox{and}
\hspace{2mm}\mathscr{J}_{\mu}^{(2)}(x):=\frac{1}{|x|^{\mu}}\ast \frac{\lambda_i^{n-\mu/2}}{\tau(z_i)^{(n-2-\epsilon_0)p}},\\& \mathscr{J}_{\mu}^{(3)}(x):=\frac{1}{|x|^{\mu}}\ast \frac{\lambda_i^{n-\mu/2}}{\tau( z_i)^{(n-2-\Theta)p}}\hspace{5mm}\mbox{and}\hspace{2mm}\mathscr{J}_{\mu}^{(4)}(x):=\frac{1}{|x|^{\mu}}\ast \frac{\lambda_i^{n-\mu/2}}{\tau(z_i)^{(n+\mu-2)p/2}}.
\end{split}
\end{equation*}
Here $\Theta>0$ satisfying the following restrictions
\begin{equation}\label{seta}
\Theta<\left\lbrace
\begin{aligned}
&\min\big\{\frac{\mu}{2},\frac{n+2-2\mu}{2}\big\},\hspace{6mm}\quad\hspace{2mm} 0<\mu<\frac{n+\mu-2}{2},\\&
\min\big\{\frac{\mu}{2},2-\frac{\mu}{2}\big\}, \quad\quad\quad\hspace{7mm}\hspace{2mm}\frac{n+\mu-2}{2}\leq\mu<4,\\&
2,\quad\quad\quad\quad\quad\quad\quad\quad\quad\quad\hspace{5mm}\hspace{2mm}\mu=4.
\end{aligned}
\right.
\end{equation}
In order to conclude the Lipschitz property of $\mathscr{N}$ defined in \eqref{u-2}, we establish now the
following key estimate:
\begin{lem}\label{B4}
Suppose that $n\geq6-\mu$, $\mu\in(0,n)$ and $0<\mu\leq4$,
there exist positive small numbers  $\theta_1,\theta_2,\theta_3$ satisfying the follow restrictions
\begin{equation}\label{ceta}
\left\lbrace
\begin{aligned}
&0<\theta_1<4-\mu,\hspace{4mm}\hspace{4mm}\hspace{7mm}\hspace{2mm}n=6,\\
&0<\theta_2\leq\frac{9-4\mu}{5},\hspace{4mm}\hspace{5mm}\hspace{3.5mm}\hspace{2mm}n=7,\\
&0<\theta_3<\frac{\mu(4-\mu)}{8},\hspace{5mm}\hspace{4mm}\hspace{2mm}n=6
\end{aligned}
\right.
\end{equation}
such that, for $i\in\{1,\cdots,\kappa\}$, there holds
\begin{equation*}
\mathscr{J}_{\mu}^{(1)}(x)\lesssim
\left\lbrace
\begin{aligned}
& \frac{\lambda_{i}^{\mu/2}}{\tau( z_i)^{\mu}},\hspace{4mm}\hspace{4mm}\hspace{4mm}\hspace{4mm}\hspace{4mm}\hspace{4mm}\hspace{1mm}\hspace{2mm}\hspace{2mm}n=4\hspace{2mm}\mbox{and}\hspace{2mm}2\leq\mu<4,\\
& \frac{\lambda_{i}^{\mu/2}}{\tau( z_i)^{\min\{\mu,(5+\mu)/3\}}},\hspace{4mm}\hspace{1mm}\hspace{2mm}\hspace{2mm}n=5\hspace{2mm}\mbox{and}\hspace{2mm}1\leq\mu\leq4,\\
& \frac{\lambda_{i}^{\mu/2}}{\tau( z_i)^{2p-n+\mu-\theta_1}},\hspace{3mm}\hspace{4mm}\hspace{3mm}\hspace{2mm}\hspace{2mm}n=6\hspace{2mm}\mbox{and}\hspace{2mm}0<\mu<4,\\
& \frac{\lambda_{i}^{\mu/2}}{\tau( z_i)^{2p-n+\mu-\theta_2}},\hspace{3mm}\hspace{4mm}\hspace{3mm}\hspace{2mm}\hspace{2mm}n=7\hspace{2mm}\mbox{and}\hspace{2mm}\frac{7}{3}<\mu\leq4,
	\end{aligned}
\right.
\end{equation*}
\begin{equation*}
\mathscr{J}_{\mu}^{(2)}(x)\lesssim \frac{\lambda_i^{\mu/2}}{\tau(z_i)^{\mu}},\hspace{4mm}\hspace{2mm}n=4\hspace{2mm}\mbox{and}\hspace{2mm}2\leq\mu<4,\hspace{2mm}\mbox{or}\hspace{2mm}n=5\hspace{2mm}\mbox{and}\hspace{2mm}3\leq\mu<4,
\end{equation*}
\begin{equation*}
\mathscr{J}_{\mu}^{(3)}(x)\lesssim\frac{\lambda_i^{\mu/2}}{\tau( z_i)^{4}},\hspace{4mm}\hspace{2mm}\mu=4\hspace{2mm}\mbox{and}\hspace{2mm}n=5,\hspace{2mm}\mbox{or}\hspace{2mm}n=7,
\end{equation*}
and that, if $\hspace{2mm}n=5\hspace{2mm}\mbox{and}\hspace{2mm}1\leq\mu<3$ or $n=6\hspace{2mm}\mbox{and}\hspace{2mm}0<\mu<4$, or $n=7\hspace{2mm}\mbox{and}\hspace{2mm}\frac{7}{3}<\mu\leq4$.
Then we have that
\begin{equation*}
\mathscr{J}_{\mu}^{(4)}(x)\lesssim\frac{\lambda_{i}^{\mu/2}}{\tau( z_i)^{\mu}},.
\end{equation*}
\end{lem}
\begin{proof}
Let $d=\frac{\lambda_i}{2}|x-\xi_i|>1$. Then
	\begin{equation}\label{zzi3}
\aligned
	\frac{\mathscr{J}_{\mu}^{(1)}}{\lambda_i^{\mu/2}}=\int_{B_{d}(0)}\frac{1}{|y|^{\mu}}\frac{1}{\tau( \lambda_ix-\lambda_i\xi_{i}-y)^{2p}}&
\lesssim\frac{ 1}{\tau(d)^{2p-(n-\mu)}}.
	\endaligned
\end{equation}
Hence, together with the following restrictions
\begin{equation*}
\left\lbrace
\begin{aligned}
&2\leq\mu<4\hspace{2mm},\hspace{2mm}\hspace{2mm}\hspace{6mm}\hspace{6mm}\hspace{2mm}\hspace{6mm}\hspace{4mm}\hspace{6mm}\hspace{6mm}\hspace{6mm}\hspace{2mm}\hspace{2mm}n=4,\\
&1\leq\mu\leq4\hspace{2mm}\mbox{and}\hspace{2mm}\mu>\frac{n^2-6n}{n-4},\hspace{2mm}\hspace{2mm}\hspace{2mm}\hspace{6mm}\hspace{4mm}\hspace{1mm}\hspace{2mm}\hspace{2mm}n=5,\\
&0<\mu\leq4\hspace{2mm}\mbox{and}\hspace{2mm}\mu>\frac{n^2-6n}{n-4},\hspace{2mm}\hspace{2mm}\hspace{2mm}\hspace{6mm}\hspace{4mm}\hspace{1mm}\hspace{2mm}\hspace{2mm}n\geq6,
\end{aligned}
\right.
\end{equation*}
we obtain
	\begin{equation}\label{zzi1-1}
\begin{split}
	\int_{B_{d}(\lambda_ix-\lambda_i\xi_{i})}\frac{1}{|y|^{\mu}}\frac{1}{\tau( \lambda_ix-\lambda_i\xi_{i}-y)^{2p}}
	&\leq \frac{1}{d^{\mu}}\int_{B_{d}(0)}\frac{1}{\tau(y)^{2p}}\\&
\lesssim
\left\lbrace
\begin{aligned}
& \frac{1}{\tau( d)^{\mu}},\hspace{4mm}\hspace{2mm}\hspace{4mm}\hspace{4mm}\hspace{6mm}\hspace{4mm}\hspace{2mm}n=4 \hspace{2mm}\mbox{and}\hspace{2mm}2\leq\mu<4,\\
& \frac{1}{\tau( d)^{\min\{\mu,(5+\mu)/3\}}},\hspace{2mm}\hspace{2mm}\hspace{2mm}n=5\hspace{2mm}\mbox{and}\hspace{2mm}1\leq\mu\leq4,\\
& \frac{1}{\tau( d)^{2p-n+\mu}},\hspace{4mm}\hspace{4mm}\hspace{4mm}\hspace{2mm}\hspace{2mm}n\geq6\hspace{2mm}\mbox{and}\hspace{2mm}\frac{n^2-6n}{n-4}<\mu\leq4.
	\end{aligned}
\right.
\end{split}
\end{equation}
Assume that $y\in\mathbb{R}^{n}\backslash (B_{d}(0)\cup B_{d}(\lambda_ix-\lambda_i\xi_{i}))$. Then
	\begin{equation}\label{zzzi-2}
	\big|\lambda_ix-\lambda_i\xi_{i}-y\big|\geq\frac{1}{2}\big|\lambda_ix-\lambda_i\xi_{i}\big|,\hspace{2mm}\big|y\big|\geq\frac{1}{2}\big|\lambda_ix-\lambda_i\xi_{i}\big|,
	\end{equation}
We consider three cases separately.\\
$\bullet$
For $n=5$ and $1\leq\mu\leq4$, one can infer from \eqref{zzzi-2} that
	$$
	\frac{1}{|y|^{\mu}}\frac{1}{\tau(\lambda_ix-\lambda_i\xi_{i}-y)^{2p}}\leq \frac{ C}{\tau( d)^{\min\{\mu,(5+\mu)/3\}}}\frac{1}{|y|^{\mu}}\frac{1}{\tau(\lambda_ix-\lambda_i\xi_{i}-y)^{2p-\min\{\mu,(5+\mu)/3\}}}.
	$$
	If $|y|\leq2|\lambda_ix-\lambda_i\xi_{i}|$, then we can also compute
\begin{equation}\label{zzi2-3}
\begin{split}
	\frac{1}{|y|^{\mu}}\frac{1}{\tau(\lambda_i x-\lambda_i \xi_{i}-y)^{2p-\min\{\mu,(5+\mu)/3\}}}&\leq
	\frac{1}{|y|^{\mu}}\frac{C}{\tau(\lambda_ix-\lambda_i\xi_{i})^{2\cdot2_{\mu}^{\ast}-\min\{\mu,(5+\mu)/3\}}}\\&\leq \frac{1}{|y|^{\mu}}\frac{C}{\tau(y)^{2p-\min\{\mu,(5+\mu)/3\}}}.
\end{split}
\end{equation}
	If $|y|\geq 2|\lambda_ix-\lambda_i\xi_{i}|$, then we have
\begin{equation}\label{xx1-4}
\big|\lambda_i x-\lambda_i \xi_{i}-y\big|\geq\big|y\big|-\big|\lambda_i x-\lambda_i \xi_{i}\big|\geq\frac{1}{2}\big|y\big|.
\end{equation}
As a consequence, combining with \eqref{xx1-4} entails that
\begin{equation*}
	\frac{1}{|y|^{\mu}}\frac{1}{\tau(\lambda_i x-\lambda_i \xi_{i}-y\big)^{2p-\min\{\mu,(5+\mu)/3\}}}\leq
	\frac{1}{|y|^{\mu}}\frac{C}{\tau(y)^{2p-\min\{\mu,(5+\mu)/3\}}}.
	\end{equation*}
	Therefore, we eventually get that
\begin{equation}\label{n=5}
\begin{split}
\int_{\mathbb{R}^{n}\backslash (B_{d}(0)\cup B_{d}(\lambda_i x-\lambda_i \xi_{i}))}&\frac{1}{|y|^{\mu}}\frac{1}{\tau(\lambda_i x-\lambda_i \xi_{i}-y\big)^{2p}}
\\&\lesssim \frac{ 1}{\tau( d)^{\min\{\mu,(5+\mu)/3\}}}\int\frac{1}{|y|^{\mu}}\frac{1}{\tau( y)^{2p-\min\{\mu,(5+\mu)/3\}}}
\lesssim \frac{1}{\tau(d)^{\min\{\mu,(5+\mu)/3\}}}.
\end{split}
\end{equation}
$\bullet$
For $n=4$ and $2\leq\mu<4$, we have the following estimate
\begin{equation}\label{n=4}
\begin{split}
\int_{\mathbb{R}^{N}\backslash (B_{d}(0)\cup B_{d}(\lambda_i x-\lambda_i \xi_{i}))}&\frac{1}{|y|^{\mu}}\frac{1}{\tau(\lambda_i x-\lambda_i  \xi_{i}-y)^{2p}}
\lesssim \frac{ 1}{\tau( d)^{\mu}}\int\frac{1}{|y|^{\mu}}\frac{1}{\tau( y)^{2p-\mu}}
\lesssim \frac{1}{\tau(d)^{\mu}}.
\end{split}
\end{equation}
$\bullet$
For $n\geq6$ and $0<\mu\leq4$, we similarly compute and get
\begin{equation}\label{n=6}
\begin{split}
\int_{\mathbb{R}^{n}\backslash (B_{d}(0)\cup B_{d}(\lambda_i x-\lambda_i \xi_{i}))}&\frac{1}{|y|^{\mu}}\frac{1}{\tau(\lambda_i x-\lambda_i  \xi_{i}-y\big)^{2p}}
\lesssim \frac{ 1}{\tau( d)^{2p-n+\mu-\theta}}\int\frac{1}{|y|^{\mu}}\frac{1}{\tau(y)^{n-\mu+\theta}}
\lesssim \frac{1}{\tau(d)^{2p-n+\mu-\theta}},
\end{split}
\end{equation}
where $\theta>0$ is a small parameter.
Combining this inequality with \eqref{zzi3}, \eqref{zzi1-1}, \eqref{n=5}, \eqref{n=6} and \eqref{n=4} yield the estimate-$\mathscr{J}_{\mu}^{(1)}$.

The proof of $\mathscr{J}_{\mu}^{(2)}$ is very similar to that of the conclusion-$\mathscr{J}_{\mu}^{(1)}$. Therefore, we will just sketch it.
Firstly, we note that $0<\epsilon_0<\frac{(n-2)p-n}{p}$ which means  $(n-2-\epsilon_0)p-n>0$. Therefore
	\begin{equation}\label{zz-i3}
\aligned
	\int_{B_{d}(0)}\frac{1}{|y|^{\mu}}\frac{1}{\tau(\lambda_i x-\lambda_i \xi_{i}-y)^{(n-2-\epsilon_0)p}}&\leq \frac{C}{\tau(d)^{(n-2-\epsilon_0)p}}\int_{B_{d}(0)}\frac{1}{|y|^{\mu}}
\leq\frac{ C}{\tau(d)^{(n-2-\epsilon_0)p-n+\mu}},
	\endaligned
\end{equation}
Moreover, we compute
	\begin{equation}\label{zzi1}
\aligned
\int_{B_{d}(\lambda_i x-\lambda_i \xi_{i})}\frac{1}{|y|^{\mu}}\frac{1}{\tau(\lambda_i x-\lambda_i\xi_{i}-y)^{(n-2-\epsilon_0)p}}
\leq \frac{C}{\tau(d)^{\mu}}.
	\endaligned
\end{equation}
	If $|y|\leq2|\lambda_i x-\lambda_i \xi_{i}|$, then we can also compute
\begin{equation}\label{zzi2}
\begin{split}
\frac{1}{|y|^{\mu}}\frac{1}{\tau(\lambda_i x-\lambda_i \xi_{i}-y)^{(n-2-\epsilon_0)p}}\leq\frac{ C}{\tau(d)^{\mu}}
\frac{1}{|y|^{\mu}}\frac{1}{\tau(y)^{(n-2-\epsilon_0)p-\mu}}.
\end{split}
\end{equation}
If $|y|\geq 2|\lambda_i x-\lambda_i\xi_{i}|$,
then we have
\begin{equation*}
	\frac{1}{|y|^{\mu}}\frac{1}{\tau(\lambda_i x-\lambda_i \xi_{i}-y)^{(n-2-\epsilon_0)p}}\leq
	\frac{ C}{\tau(d)^{\mu}}
\frac{1}{|y|^{\mu}}\frac{1}{\tau(y)^{(n-2-\epsilon_0)p-\mu}}
	\end{equation*}
	Therefore, we eventually get that
\begin{equation*}
\begin{split}
\int_{\mathbb{R}^{n}\backslash (B_{d}(0)\cup B_{d}(\lambda_i x-\lambda_i \xi_{i}))}&\frac{1}{|y|^{\mu}}\frac{1}{\tau(\lambda_i x-\lambda_i \xi_{i}-y)^{(n-2-\epsilon_0)p}}
\leq \frac{ C}{\tau(d)^{\mu}}\int\frac{1}{|y|^{\mu}}\frac{1}{\tau(y)^{(n-2-\epsilon_0)p-\mu}}
\leq \frac{C}{\tau(d)^{\mu}}.
\end{split}
\end{equation*}
Combining this inequality with \eqref{zz-i3}-\eqref{zzi1} and \eqref{zzi2} yields the conclusion. $\mathscr{J}_{\mu}^{(2)}$, $\mathscr{J}_{\mu}^{(3)}$ and $\mathscr{J}_{\mu}^{(4)}$ is derived by simple computations similar to the estimate of $\mathscr{J}_{\mu}^{(1)}$. Thus, the result easily follows.
\end{proof}
The following important estimates will be used later to conclude the proofs of Lemma \ref{estimate1} and Lemmas \ref{cll-0}-\ref{cll-1}.
\begin{lem}\label{B4-1}
	For $n\geq6-\mu$, $\mu\in(0,n)$ and $0<\mu\leq4$, there exist constants $C>0$
and $\Theta>0$ satisfying the restrictions \eqref{seta} such that, for all $i\in\{1,\cdots,\kappa\}$
\begin{equation*}
\frac{1}{\big|x\big|^{\mu}}\ast \frac{\lambda_i^{n-\mu/2}}{\tau(z_i)^{2n-\mu}}
\lesssim\frac{\lambda_i^{\mu/2}}{\tau(z_i)^{\mu}},\hspace{2mm}\frac{1}{\big|x\big|^{\mu}}\ast \frac{\lambda_i^{n-\mu/2}}{\tau(z_i)^{n-\mu+4}}
\lesssim\frac{\lambda_i^{\mu/2}}{\tau(z_i)^{\mu}},\hspace{2mm}\hspace{2mm}\hspace{4mm}\hspace{2mm}\hspace{2mm}n\geq4,
\end{equation*}
\begin{equation*}
\frac{1}{\big|x\big|^{\mu}}\ast \frac{\lambda_i^{n-\mu/2}}{\tau(z_i)^{(3n-\mu+2)/2}}
\lesssim\frac{\lambda_i^{\mu/2}}{\tau(z_i)^{\mu}},\hspace{2mm}\hspace{2mm}n\geq4;\hspace{2mm}\frac{1}{\big|x\big|^{\mu}}\ast \frac{\lambda_i^{n-\mu/2}}{\tau(z_i)^{2n-\mu-\epsilon_0}}
\lesssim\frac{\lambda_i^{\mu/2}}{\tau(z_i)^{\mu}},\hspace{2mm}\hspace{2mm}n=4\hspace{2mm}\mbox{or}\hspace{2mm}n=5,
\end{equation*}
\begin{equation*}
\frac{1}{\big|x\big|^{\mu}}\ast \frac{\lambda_i^{n-\mu/2}}{\tau(
\tilde{z}_i)^{2n-2-\mu/2}}
\lesssim\frac{\lambda_i^{\mu/2}}{\tau(\tilde{z}_i)^{\mu}},\hspace{2mm}\hspace{2mm}\frac{1}{\big|x\big|^{\mu}}\ast \frac{\lambda_i^{n-\mu/2}}{\tau(\tilde{z}_i)^{n-\frac{\mu}{2}}}
\lesssim\frac{\lambda_i^{\mu/2}}{\tau(\tilde{z}_i)^{\frac{\mu}{2}-\Theta}},\hspace{2mm}\tilde{z}_i=z_i/\mathscr{R}_{ij},\hspace{2mm}\hspace{2mm}n\geq4,
\end{equation*}
\begin{equation*}
\frac{1}{\big|x\big|^{\mu}}\ast \frac{\lambda_i^{n-\mu/2}}{\tau(\tilde{z}_i)^{n-2+(n-\mu+2)/2}}
\lesssim\frac{\lambda_i^{\mu/2}}{\tau(\tilde{z}_i)^{\min\{(n+\mu-2)/2,\mu\}}},\hspace{4mm}\tilde{z}_i=z_i/\mathscr{R}_{12},\hspace{4mm}\hspace{2mm}n\geq4.
\end{equation*}
\end{lem}
\begin{proof}
The argument is similar to the proof of Lemma \ref{B4}. Here we omit the details.
\end{proof}

\section{Proof of the main result}
The aim of the section is to prove the main theorem. As a preparation step for the proof of Theorem \ref{Figalli}, we first perform a Lyapunov-Schmidt reduction in a weighted space. In this way, we can establish good $C^0$ estimates for the error term $\rho$ in sections 4-5 and a point-wide estimate of $\rho_0$ in section 7.
\begin{Def}\label{del-1}
Let $W_i$ and $W_j$ be two bubbles, if $\mathscr{R}_{ij}=\sqrt{\lambda_i\lambda_j}|\xi_i-\xi_j|$, then we call them a bubble cluster, other call them a bubble tower.
We also set
\begin{equation}\label{R1}
\mathscr{R}_{ij}=\max\Big\{\sqrt{\lambda_i/\lambda_j}, \sqrt{\lambda_j/\lambda_i}, \sqrt{\lambda_i\lambda_j}|\xi_i-\xi_j|\Big\}=\Re_{ij}^{-1}\quad\text{if}\quad i\neq j\in I,
\end{equation}
and
 $$\mathscr{R}:=\frac{1}{2}\min\limits_{i\neq j}\big\{\mathscr{R}_{ij}:~i,j=1,\cdots,\kappa,~i\neq j\big\}.$$
\end{Def}
\subsection{The weighted space and its norm}
We carry out the Lyapunov-Schmidt reduction argument in a weighted space so that we can establish some good estimates of error terms.
For simplicity of notation, we define three kinds of weight functions depending on $n$ and $\mu$ as follows:
\begin{itemize}
\item[$\bullet$]
For any $0<\mu<\frac{n+\mu-2}{2}$,
\begin{equation*}
\begin{split}
&s_{i,1}(x,\mathscr{R}):=\frac{\lambda_i^{(n-2)/2}}{\tau(z_i)^{2}\mathscr{R}^{n-2}}\chi_{\{|z_i|\leq\mathscr{R}\}},\hspace{6mm}
s_{i,2}(x,\mathscr{R})=\frac{\lambda_{i}^{(n-2)/2}}{\tau(z_i)^{(n+\mu-2)/2}\mathscr{R}^{(n-\mu+2)/2}}\chi_{\{|z_i|>\mathscr{R}\}},\\&
t_{i,1}(x,\mathscr{R}):=\frac{\lambda_i^{(n+2)/2}}{\tau(z_i)^{4}\mathscr{R}^{n-2}}\chi_{\{|z_i|\leq\mathscr{R}\}},\hspace{7mm}
t_{i,2}(x,\mathscr{R}):=\frac{\lambda_{i}^{(n+2)/2}}{\tau(z_i)^{(n+\mu+2)/2}\mathscr{R}^{(n-\mu+2)/2}}\chi_{\{|z_i|>\mathscr{R}\}}.
\end{split}
\end{equation*}
\item[$\bullet$]
For any $\frac{n+\mu-2}{2}\leq\mu<4$, and $0<\epsilon_0<\frac{(n-2)p-n}{p}$ is a parameter,
\begin{equation*}
\begin{split}
&\hat{s}_{i,1}(x,\mathscr{R}):=\frac{\lambda_i^{(n-2)/2}}{\tau(z_i)^{2}\mathscr{R}^{2n-4-\mu}}\chi_{\{|z_i|\leq\mathscr{R}\}},\hspace{6mm}\hat{s}_{i,2}(x,\mathscr{R}):=\frac{\lambda_{i}^{(n-2)/2}}{\tau(z_i)^{n-2-\epsilon_0}\mathscr{R}^{n-\mu+\epsilon_0}}\chi_{\{|z_i|>\mathscr{R}\}},\\&
\hat{t}_{i,1}(x,\mathscr{R}):=\frac{\lambda_i^{(n+2)/2}}{\tau(z_i)^{4}\mathscr{R}^{2n-4-\mu}}\chi_{\{|z_i|\leq\mathscr{R}\}},\hspace{6mm}\hat{t}_{i,2}(x,\mathscr{R}):=\frac{\lambda_{i}^{(n+2)/2}}{\tau(z_i)^{n-\epsilon_0}\mathscr{R}^{n-\mu+\epsilon_0}}\chi_{\{|z_i|>\mathscr{R}\}}.
\end{split}
\end{equation*}
\item[$\bullet$]
When $\mu=4$ and the parameter $\Theta<2$,
\begin{equation*}
\begin{split}
&\tilde{s}_{i,1}(x,\mathscr{R}):=\frac{\lambda_i^{(n-2)/2}}{\tau(z_i)^{2}\mathscr{R}^{n-2}}\chi_{\{|z_i|\leq\mathscr{R}\}},\hspace{6mm}\tilde{s}_{i,2}(x,\mathscr{R}):=\frac{\lambda_{i}^{(n-2)/2}}{\tau(z_i)^{n-2-\Theta}\mathscr{R}^{n-\mu}}\chi_{\{|z_i|>\mathscr{R}\}},\\&
\tilde{t}_{i,1}(x,\mathscr{R}):=\frac{\lambda_i^{\frac{n+2}{2}}}{\tau(z_i)^{4}\mathscr{R}^{n-2}}\chi_{\{|z_i|\leq\mathscr{R}\}},\hspace{6mm}\tilde{t}_{i,2}(x,\mathscr{R}):=\frac{\lambda_{i}^{(n+2)/2}}{\tau(z_i)^{n-\Theta}\mathcal{R}^{n-\mu}}\chi_{\{|z_i|>\mathscr{R}\}}.
\end{split}
\end{equation*}
\end{itemize}
Here the parameter $\Theta>0$ from \eqref{seta}, $\tau(z_i)=(1+|z_i|^2)^{1/2}$ with $z_i=\lambda_i(x-\xi_i)$ throughout the paper.
For the functions $\phi$ and $\mathscr{R}$,
we define the following weighted norms $\|\cdot\|_{\ast}$ and $\|\cdot\|_{\ast\ast}$ that will help us to characterize the behavior of the interaction term $\hbar$.
\begin{Def}\label{st-11}
Define the norm $\|\cdot\|_{\ast}$ as
\begin{equation}\label{f-ai}
\|\phi\|_{\ast}=\sup_{x\in\mathbb{R}^n}\big|\phi(x)\big|S_{j}^{-1}(x),\hspace{2mm}j=1,2, 3,
\end{equation}
and the norm $\|\cdot\|_{\ast\ast}$ as
\begin{equation}\label{h}
\|\hbar\|_{\ast\ast}=\sup_{x\in\mathbb{R}^n}\big|\hbar(x)\big|T_{j}^{-1}(x),\hspace{2mm}j=1,2, 3,
\end{equation}
with the weights
\begin{equation}\label{zv4}
\left\lbrace
\begin{aligned}
 &S_1(x)=\sum_{i=1}^{\kappa}\big[s_{i,1}(x,\mathscr{R})
+s_{i,2}(x,\mathscr{R})\big]\hspace{3mm}\quad\quad 0<\mu<\frac{n+\mu-2}{2},\\
&S_2(x)=\sum_{i=1}^{\kappa}\big[\hat{s}_{i,1}(x,\mathscr{R})
+\hat{s}_{i,2}(x,\mathscr{R})\big]\quad\quad \hspace{3mm}\frac{n+\mu-2}{2}\leq\mu<4,\\
&S_3(x)=\sum_{i=1}^{\kappa}\big[\tilde{s}_{i,1}(x,\mathscr{R})
+\tilde{s}_{i,2}(x,\mathscr{R})\big]\quad\quad\quad  \mu=4,
\end{aligned}
\right.
\end{equation}
and
\begin{equation}\label{zv5}
\left\lbrace
\begin{aligned}
&T_1(x)=\sum_{i=1}^{\kappa}\big[t_{i,1}(x,\mathscr{R})
+t_{i,2}(x,\mathscr{R})\big]\hspace{5mm}\quad\quad 0<\mu<\frac{n+\mu-2}{2},\\
&T_2(x)=\sum_{i=1}^{\kappa}\big[\hat{t}_{i,1}(x,\mathscr{R})
+\hat{t}_{i,2}(x,\mathscr{R})\big]\quad\quad\hspace{5mm}\frac{n+\mu-2}{2}\leq\mu<4,\\
&T_3(x)=\sum_{i=1}^{\kappa}\big[\tilde{t}_{i,1}(x,\mathscr{R})
+\tilde{t}_{i,2}(x,\mathscr{R})\big]\quad\quad\hspace{5mm} \mu=4.
\end{aligned}
\right.
\end{equation}
\end{Def}
\begin{Rem}\label{stxqi}
It is worth mentioning that, the weight functions $S_j(x)$ and $T_j(x)$ defined in \eqref{zv4} and \eqref{zv5} depend on the parameters $n$ and $\mu$. However, the case $n=4$ is not valid for achieving the desired estimates in Lemma \ref{ww101}. With the above choice of $n$ and $\mu$, we indicate all the parameters used in different norms.
\begin{itemize}
\item[$\bullet$] The weight functions $S_1(x)$ and $T_1(x)$ are defined in \eqref{zv4} and \eqref{zv5} if the parameters $n$, $\mu$ are chosen in the range: $n=5$ and $\mu\in[1,3)$, or $n=6$ and $\mu\in(0,4)$,  or $n=7$ and $\mu\in(\frac{7}{3},4)$.
\item[$\bullet$] The weight functions $S_2(x)$ and $T_2(x)$ are defined in \eqref{zv4} and \eqref{zv5} if the parameters $n$, $\mu$ are chosen in the range: $n=4$ and $\mu\in[2,4)$, or $n=5$ and $\mu\in[3,4)$.
\end{itemize}
\end{Rem}
\begin{Rem}\label{stxqi-1}
In Lemma \ref{cll-1}, we need to establish the behavior of $\phi$ at infinity can be bounded by $\|\hbar\|_{\ast\ast}$ up to a multiplicative constant and some small error. Therefore, it is necessary to be given that $\tilde{\mathbf{s}}_{i,in}$ and $\tilde{\mathbf{s}}_{i,out}$ decay faster $s_{i,1}$ and $s_{i,2}$,
$\bar{\mathbf{s}}_{i,in}$ and $\bar{\mathbf{s}}_{i,out}$ decay faster $\hat{s}_{i,1}$ and $\hat{s}_{i,2}$ but when the power of $\tau(z_i)$ in weight functions of the outer region is larger than or equal to $n-2$, the behavior of $s_{i,2}$ and $\hat{s}_{i,2}$ like that is disrupted. Moreover, the critical decay of $\hat{s}_{i,2}$ (in the sense of $\|\cdot\|_{2^{\ast}}$) does not occur which means $\epsilon_0\neq1$ in dimension $n=4$.
\end{Rem}
By utilizing the above norms, we will find a function $\rho_0$ and a set of scalars $\{c_{a}^{i}\}$ such that the following system is satisfied
\begin{equation}\label{c1}
	\left\{\begin{array}{l}
		\displaystyle \Delta \phi+\Phi_{n,\mu}[\sigma,\phi]=\hbar+
		\sum_{i=1}^{\kappa}\sum_{a=1}^{n+1}c_{a}^{i}\Phi_{n,\mu}[W_{i},\Xi^{a}_i]\hspace{4mm}\mbox{in}\hspace{2mm} \mathbb{R}^n,\\
		\displaystyle 	\int\Phi_{n,\mu}[W_{i},\Xi^{a}_i]\phi=0,\hspace{4mm}i=1,\cdots, \kappa; ~a=1,\cdots,n+1.
	\end{array}
	\right.
\end{equation}
To this end, our purpose in what follows is to prove the following results.
\begin{lem}\label{ww10}
There exist $m_0 > 0$ and a constant $C > 0$, independent of $m$, such that for all
$\delta_0\leq\delta$ and all $\hbar\in L^{\infty}(\mathbb{R}^n)$, the system \eqref{c1} has a unique solution $\phi \equiv \mathcal{L}_\delta(\hbar)$ such that for
\begin{equation*}
	\|\mathcal{L}_\delta(\hbar)\|_{\ast}\leq C\|\hbar\|_{\ast\ast},\quad |c_a^i|\leq C\delta\|\hbar\|_{\ast\ast}.
\end{equation*}
\end{lem}
 Lemma \ref{ww10} is established by the following lemma \ref{estimate1}, lemma \ref{cll} and lemma \ref{estimate2} with different right hand sides.
\begin{lem}\label{estimate1}
Let $n\geq4$ and $\kappa\in\mathbb{N}$. There exist a positive constant $\delta=\delta(n,\kappa)>0$ and large constant $C(n,\kappa,\mu)$ such that
\begin{equation}\label{eq3.7}
\Big\|\Big(|x|^{-\mu}\ast\sigma^{p}\Big)\sigma^{p-1}-\sum_{i=1}^{\kappa}\Big(|x|^{-\mu}\ast W_{i}^{p}\Big)W_{i}^{p-1}\Big\|_{\ast\ast}\leq C(n,\kappa,\mu),
\end{equation}
where $C$ depending only on $n$, $\kappa$ and $\mu$.
\end{lem}
The proof will make us of a series of the key lemmas, which is based on the estimates of the convolution terms in the section 2.

\subsection{Behavior of the interaction bubbles}\
\newline
For all $i\in\{1,\cdots,\kappa\}$ and given two bubbles $W_i$ and $W_j$.
Since the bubbles are weakly interacting, if the bubbles $W_i$ and $W_j$ form a bubble tower, then we denote by $\mathcal{C}(W_i)$ the core region of $W_i$ by
\begin{equation*}
\mathcal{C}(W_i):=\Big\{x\in\mathbb{R}^n:\hspace{2mm}|x-\xi_i|<\frac{1}{\sqrt{\lambda_i\lambda_j}}\Big\},
\hspace{2mm}\mathcal{C}(W_j):=\Big\{x\in\mathbb{R}^n:\hspace{2mm}|x-\xi_j|<\frac{1}{\lambda_j}\sqrt{\frac{\lambda_i}{\lambda_j}}\Big\}.
\end{equation*}
If the bubbles $W_i$ and $W_j$ form a bubble cluster, we denote by $\mathcal{D}(W_i)$ the core region of $W_i$ by
\begin{equation*}
\mathcal{D}(W_i):=\Big\{x\in\mathbb{R}^n:\hspace{2mm}|x-\xi_i|<\sqrt{\frac{\lambda_j}{\lambda_i}}|\xi_i-\xi_j|\Big\},
\hspace{2mm}\mathcal{D}(W_j):=\Big\{x\in\mathbb{R}^n:\hspace{2mm}|x-\xi_j|<\sqrt{\frac{\lambda_i}{\lambda_j}}|\xi_i-\xi_j|\Big\}.
\end{equation*}
Now we are in a position to prove Lemma \ref{estimate1}.\\
\emph{\textbf{Proof of Lemma \ref{estimate1}.}}
Let us consider the simplest two-bubble case.\\
$\bullet$
\textbf{The case of $\kappa=2$.}\\
Let's discuss the two cases of bubble tower case and bubble cluster case separately.
Assume, without loss of generality, that $W_1=W[\xi_1,\lambda_1]$ and $W_2=W[\xi_2,\lambda_2]$ be two bubbles. Then we know that
$$
\hbar=\big[|x|^{-\mu}\ast(W_1+W_2)^{p}\big](W_1+W_2)^{p-1}-\big(|x|^{-\mu}\ast W_{1}^{p}\big)W_{1}^{p-1}-\big(|x|^{-\mu}\ast W_{2}^{p}\big)W_{2}^{p-1}.
$$
In view of \eqref{h}, to conclude the proof of \eqref{eq3.7}, it is sufficient to obtain $\hbar(x)\lesssim T_j(x)(j=1,2,3)$.

We divide our argument into several cases.\\		
$\mathbf{Case\ 1.}$ \textbf{Bubble tower}:\\
Without loss of generality, we may assume that $\lambda_1>\lambda_2$ and $\mathscr{R}_{12}=\sqrt{\lambda_1/\lambda_2}=\Re_{12}^{-1}\gg1$. Then $\mathcal{Q}\approx \mathscr{R}_{12}^{2-n}$. Let $z_1=\lambda_1(x-\xi_1)$, by scaling we find that $W_1(x)=\lambda_1^{(n-2)/2}W[0,1](z_1)=:\lambda_1^{(n-2)/2}W(z_1)$.
\begin{itemize}
\item[$(1)$] \textbf{Core region of $W_1$}:
in this case $|z_1|\leq\frac{1}{2}\mathscr{R}_{12}$.
\end{itemize}
Then we have $W_2\lesssim W_1$, a straightforward computation
\begin{equation*}
W_2(x)=\tilde{\alpha}_{n,\mu}\frac{\lambda_1^{(n-2)/2}}{\mathscr{R}_{12}^{n-2}}\Big(\frac{1}{1+\mathscr{R}_{12}^4|z_1+\alpha_{12}|^2}\Big)^{\frac{n-2}{2}}
\end{equation*}
where $|\alpha_{12}|=|\lambda_{1}(\xi_1-\xi_2)|\leq\mathscr{R}_{12}^2$.
Using the elementary inequality
we get
\begin{equation}\label{R-0}
\aligned
\hbar&\lesssim\big(|x|^{-\mu}\ast W_1^{p}\big)\big[W_1^{p-2}W_2+W_2^{p-1}\big]+\big(|x|^{-\mu}\ast W_2^{p}\big)\big[W_1^{p-1}+W_1^{p-2}W_2\big]\\&
\hspace{2mm}+\big[|x|^{-\mu}\ast \big(W_1^{p-1}W_2+W_1^{p-2}W_2^2\big)\big]\big[W_1^{p-1}+W_1^{p-2}W_2+W_2^{p-1}\big].
\endaligned
\end{equation}
Now we estimate every term of the previous decomposition. Combining Lemma \ref{B4-1}, we get
\begin{equation}\label{1-h}
\begin{split}
\big(|x|^{-\mu}\ast W_1^{p}\big)W_1^{p-2}W_2&=\Big[\frac{1}{|x|^{\mu}}\ast \Big(
\frac{\lambda_1^{\frac{n-2}{2}}}{\tau(z_1)^{n-2}}\Big)^{p}\Big]\Big(
\frac{\lambda_1^{\frac{n-2}{2}}}{\tau(z_1)^{n-2}}\Big)^{p-2}\frac{\lambda_2^{\frac{n-2}{2}}}{\tau(z_2)^{n-2}}\\&\lesssim
\frac{\lambda_1^{2}}{\tau(z_1)^{4}}\frac{\lambda_1^{(n-2)/2}}{\mathscr{R}_{12}^{n-2}}\Big(\frac{1}{1+\mathscr{R}_{12}^4|z_1+\alpha_{12}|^2}\Big)^{\frac{n-2}{2}}
\approx\frac{\lambda_{1}^{(n+2)/2}}{\mathscr{R}_{12}^{n-2}\tau(z_1)^{4}}.
\end{split}
\end{equation}
As in the previous case, combined with $W_2\lesssim W_1$, we infer that
\begin{equation}\label{2-h}
\big(|x|^{-\mu}\ast W_1^{p}\big)W_2^{p-1}+\big(|x|^{-\mu}\ast W_2^{p}\big)\big[W_1^{p-1}+|W_1|^{p-2}W_2\big]\lesssim\frac{\lambda_{1}^{(n+2)/2}}{\mathscr{R}_{12}^{n-2}\tau(z_1)^{4}},
\end{equation}
and
\begin{equation}\label{3-h}
\big[|x|^{-\mu}\ast\big(W_1^{p-1}W_2+W_1^{p-2}W_2^2\big)\big]\big[W_1^{p-1}+|W_1|^{p-2}W_2+
W_2^{p-1}\big]\lesssim\frac{\lambda_{1}^{(n+2)/2}}{\mathscr{R}_{12}^{n-2}\tau(z_1)^{4}}.
\end{equation}
Putting \eqref{1-h}-\eqref{3-h} it follows that
\begin{equation}\label{5-h}
\hbar\lesssim
\left\lbrace
\begin{aligned}
&\frac{\lambda_{1}^{(n+2)/2}}{\mathscr{R}_{12}^{n-2}\tau(z_1)^{4}},
\hspace{6mm}\hspace{4mm}\hspace{6mm}\hspace{6mm}\hspace{2mm}0<\mu<\frac{n+\mu-2}{2},\hspace{2mm}|z_1|\leq\frac{1}{2}\mathscr{R}_{12},\\
&\frac{\lambda_{1}^{(n+2)/2}}{\mathscr{R}_{12}^{2n-4-\mu}\tau(z_1)^{4}},
\hspace{4mm}\hspace{6mm}\hspace{1mm}\hspace{6mm}\hspace{2mm}\frac{n+\mu-2}{2}\leq\mu\leq4,\hspace{2mm}|z_1|\leq\frac{1}{2}\mathscr{R}_{12}.
\end{aligned}
\right.
\end{equation}
\begin{itemize}
\item[$(2)$] \textbf{Outside the Core region of $W_1$}: in this case we consider $\frac{1}{3}\mathscr{R}_{12}\leq|z_1|\leq2\mathscr{R}_{12}^2$.
\end{itemize}
We compute $\mathscr{R}\lesssim I+II+III$ with
\begin{equation}\label{III}
\begin{split}
&I:=\big[|x|^{-\mu}\ast\big((W_1+W_2)^{p}-W_1^{p}-W_2^{p}\big)\big](W_1^{p-1}+W_2^{p-1})\\&
II:=\big(|x|^{-\mu}\ast W_1^{p}\big)\big((W_1+W_2)^{p-1}-W_1^{p-1}-W_2^{p-1}\big)\\&
III:=\big(|x|^{-\mu}\ast W_1^{p}\big)W_2^{p-1}+\big(|x|^{-\mu}\ast W_2^{p}\big)W_1^{p-1}.
\end{split}
\end{equation}
To estimate $I$, $II$ and $III$, we first introduce two inequalities. Without loss of generality, we may assume $A,B>0$. Then for all $p\geq2$,
we have  that
    \begin{equation}\label{a0b0}
    \begin{split}
    (A+B)^{p}-A^{p}-B^{p}= A^{p}\big[(1+\frac{B}{A})^{p}-1-(\frac{B}{A})^{p}\big]\lesssim A^{p-1}B+  A^{p-2}B^2,
    \end{split}
    \end{equation}
    and
    \begin{equation}\label{a0b1}
    \begin{split}
    (A+B)^{p-1}-A^{p-1}-B^{p-1}= A^{p-1}\big[(1+\frac{B}{A})^{p-1}-1-(\frac{B}{A})^{p-1}\big]
    \lesssim   A^{p-2}B.
    \end{split}
    \end{equation}
\textbf{Estimate of $I$.} For $I$,
along with
\begin{equation}\label{ww-1}
W_1\approx\frac{\lambda_1^{(n-2)/2}}{|z_1|^{n-2}}~~\text{and} ~~W_2\approx\frac{\lambda_1^{(n-2)/2}}{\mathscr{R}_{12}^{n-2}}
\lesssim\frac{\lambda_1^{(n-2)/2}}{|z_1|^{(n-2)/2}},
\end{equation}
so we get form \eqref{a0b0} that
\begin{equation*}
\begin{split}
\big(W_1+W_2\big)^{p}-W_1^{p}-W_2^{p}
\lesssim\frac{1}{\mathscr{R}_{12}^{2n-2-\mu/2}}\frac{\lambda_1^{\frac{2n-\mu}{2}}}{|\tilde{z}_1|^{2n-2-\mu/2}}+\frac{1}{\mathscr{R}_{12}^{n-2+(n-\mu+2)/2}}\frac{\lambda_1^{\frac{2n-\mu}{2}}}{|\tilde{z}_1|^{n-2+(n-\mu+2)/2}},
\end{split}
\end{equation*}
where $\tilde{z}_1=z_1/\mathscr{R}_{12}$.
Therefore
\begin{equation}\label{0h}
\begin{split}
\Big[|x|^{-\mu}\ast\Big(\big(W_1+W_2\big)^{p}&-W_1^{p}-W_2^{p}\Big)\Big]W_1^{p-1}
\lesssim\frac{\lambda_1^{\frac{n-\mu+2}{2}}}{\mathcal{R}_{12}^{3n-3\mu/2}}\Big(|x|^{-\mu}\ast\big(\frac{\lambda_1^{\frac{2n-\mu}{2}}}{|\tilde{z}_1|^{2n-2-\mu/2}}\big)\Big)\frac{1}{|\tilde{z}_1|^{n-\mu+2}}
\\&+\frac{\lambda_1^{\frac{2n-\mu}{2}}}{\mathscr{R}_{12}^{n-2+3(n-\mu+2)/2}}\Big(|x|^{-\mu}\ast\big(\frac{\lambda_1^{\frac{2n-\mu}{2}}}{|\tilde{z}_1|^{n-2+(n-\mu+2)/2}}\big)\Big)\frac{1}{|\tilde{z}_1|^{n-\mu+2}}
\\&
\lesssim\frac{1}{\mathscr{R}_{12}^{2n-2-3\mu/2}}\frac{\lambda_{1}^{\frac{N+2}{2}}}{|z_1|^{n+2}}
+\frac{\lambda_{1}^{\frac{n+2}{2}}}{\mathscr{R}_{12}^{n-2+(n-\mu+2)/2}}\frac{\mathscr{R}_{12}^{\min\{\frac{n+\mu-2}{2},\mu\}}}{|z_1|^{\min\{\frac{n+\mu-2}{2},\mu\}+n-\mu+2}}\\&
\approx\frac{1}{\mathscr{R}_{12}^{2n-2-3\mu/2}}\frac{\lambda_{1}^{\frac{n+2}{2}}}{\tau(z_1)^{n+2}}
+\frac{\lambda_{1}^{\frac{n+2}{2}}}{\mathscr{R}_{12}^{n-2+(n-\mu+2)/2}}\frac{\mathscr{R}_{12}^{\min\{\frac{n+\mu-2}{2},\mu\}}}{\tau(z_1)^{\min\{\frac{n+\mu-2}{2},\mu\}+n-\mu+2}}.
\end{split}
\end{equation}
where we applied the estimate similar to Proposition \ref{B4-1}.
Similarly, we have
\begin{equation}\label{1h}
\begin{split}
\big[|x|^{-\mu}\ast\big((W_1&+W_2)^{p}-W_1^{p}-W_2^{p}\big)\big]W_2^{p-1}\\&
\lesssim\frac{1}{\mathscr{R}_{12}^{2n-2-3\mu/2}}\frac{\lambda_{1}^{\frac{N+2}{2}}}{\tau(z_1)^{(n+\mu+2)/2}}
+\frac{\lambda_{1}^{\frac{n+2}{2}}}{\mathscr{R}_{12}^{n-2+(n-\mu+2)/2}}\frac{\mathscr{R}_{12}^{\min\{\frac{n+\mu-2}{2},\mu\}}}{\tau(z_1)^{\min\{\frac{n+\mu-2}{2},\mu\}+\frac{n-\mu+2}{2}}}.
\end{split}
\end{equation}
As a result,
\begin{equation}\label{1h000}
\begin{split}
I\lesssim \mbox{RHS}\hspace{2mm}\mbox{of}\hspace{2mm}\eqref{0h}+\mbox{RHS}\hspace{2mm}\mbox{of}\hspace{2mm}\eqref{1h}.
\end{split}
\end{equation}
\textbf{Estimate of $II$.}
By \eqref{a0b1}
$$\big(W_1+W_2\big)^{p-1}-W_1^{p-1}-W_2^{p-1}=W_2^{p-1}\big[(1+\frac{W_1}{W_2})^{p-1}-1-(\frac{W_1}{W_2})^{p-1}\big].$$
Directly computing, one has
\begin{equation}\label{2h}
\begin{split}
II&\lesssim
\frac{\lambda_{1}^{\frac{n+2}{2}}}{\mathscr{R}_{12}^{3n-2\mu+2}}\frac{1}{|\tilde{z}_1|^{\mu}}\Big|\big(1+\frac{1}{|\tilde{z}_1|^{n-2}}\big)^{p-1}-1-\frac{1}{|\tilde{z}_1|^{n-\mu+2}}\Big|\\&
\lesssim\frac{1}{\mathscr{R}_{12}^{3n-2\mu+2}}\frac{\lambda_{1}^{\frac{n+2}{2}}}{|\tilde{z}_1|^{\mu}}\frac{1}{|\tilde{z}_1|^{n-2}}\approx\frac{1}{\mathscr{R}_{12}^{2n-3\mu+4}}\frac{\lambda_{1}^{\frac{n+2}{2}}}{\tau(z_1)^{n-2+\mu}}
\end{split}
\end{equation}
by Lemma \ref{B4-1}. \\
\textbf{Estimate of $III$.} As in the previous case, we get
\begin{equation}\label{3h}
\begin{split}
III&
\lesssim\frac{\lambda_{1}^{\frac{n+2}{2}}}{\mathscr{R}_{12}^{2n-\mu+(n-\mu+2)/2}}\frac{1}{| \tilde{z}_1|^{(n+\mu+2)/2}}+\frac{\lambda_{1}^{\frac{n+2}{2}}}{\mathscr{R}_{12}^{2n-3\mu/2+2}}\frac{1}{|\tilde{z}_1|^{n-\mu/2+2-\Theta}}\\&
\approx\frac{\lambda_{1}^{\frac{n+2}{2}}}{\mathscr{R}_{12}^{2n-2\mu}}\frac{1}{\langle z_1\rangle^{(n+\mu+2)/2}}+\frac{\lambda_{1}^{\frac{n+2}{2}}}{\mathscr{R}_{12}^{n-\mu+\Theta}}\frac{1}{\langle z_1\rangle^{n-\mu/2+2-\Theta}}.
\end{split}
\end{equation}
We now divide our argument into two cases.
\begin{itemize}
\item[$(1)$] For $0<\mu<\frac{n+\mu-2}{2}$. In view of \eqref{0h}, \eqref{1h}, \eqref{2h} and \eqref{3h},
\begin{equation}\label{h-8}
\begin{split}
\hbar&\lesssim\Big(\frac{1}{\mathscr{R}_{12}^{2n-2-3\mu/2}}+\frac{1}{\mathscr{R}_{12}^{n-2-\mu+(n-\mu+2)/2}}\Big)\Big|\frac{\lambda_{1}^{\frac{n+2}{2}}}{\tau(z_1)^{n+2}}+\frac{\lambda_{1}^{\frac{n+2}{2}}}{\langle z_1\rangle^{(n+\mu+2)/2}}\bigg|\\&+\frac{\lambda_{1}^{\frac{n+2}{2}}}{\mathscr{R}_{12}^{2n-3\mu+4}\tau(z_1)^{n-2+\mu}}+
\frac{\lambda_{1}^{\frac{n+2}{2}}}{\mathscr{R}_{12}^{2n-2\mu}\tau(z_1)^{(n+\mu+2)/2}}+\frac{\lambda_{1}^{\frac{n+2}{2}}}{\mathscr{R}_{12}^{n-\mu+\Theta}\tau(z_1)^{n-\mu/2+2-\Theta}}.
\end{split}
\end{equation}
\item[$(2)$] For $\frac{n+\mu-2}{2}\leq\mu\leq4$. Combining \eqref{0h}, \eqref{1h}, \eqref{2h} and \eqref{3h}, we get
\begin{equation}\label{0h-8}
\begin{split}
\hbar&\lesssim\frac{1}{\mathscr{R}_{12}^{2n-2-3\mu/2}}\Big|\frac{\lambda_{1}^{\frac{n+2}{2}}}{\tau(z_1)^{n+2}}+\frac{\lambda_{1}^{\frac{n+2}{2}}}{\tau(z_1)^{(n+\mu+2)/2}}\Big|+\frac{1}{\mathscr{R}_{12}^{n-\mu}}\Big|\frac{\lambda_{1}^{\frac{n+2}{2}}}{\langle z_1\rangle^{(3n-\mu+2)/2}}+\frac{\lambda_{1}^{\frac{n+2}{2}}}{\tau(z_1)^{n}}\Big|\\&+\frac{\lambda_{1}^{\frac{n+2}{2}}}{\mathscr{R}_{12}^{2n-3\mu+4}\tau(z_1)^{n-2+\mu}}+
\frac{\lambda_{1}^{\frac{n+2}{2}}}{\mathscr{R}_{12}^{2n-2\mu}\tau(z_1)^{(n+\mu+2)/2}}+\frac{\lambda_{1}^{\frac{n+2}{2}}}{\mathscr{R}_{12}^{n-\mu+\Theta}\tau(z_1)^{n-\mu/2+2-\Theta}}.
\end{split}
\end{equation}
\end{itemize}
Because the parameters are controlled by
\begin{equation*}
\begin{split}
2n-3\mu+4\geq2n-2\mu\geq 2n-2-\frac{3\mu}{2}\geq\frac{3n-3\mu-2}{2}\hspace{2mm}\mbox{for}\hspace{2mm}0<\mu<\frac{n+\mu-2}{2};
\end{split}
\end{equation*}
\begin{equation*}
\begin{split}
n+2\geq n-2+\mu\geq\frac{n+\mu+2}{2}\leq n-\frac{\mu}{2}+2-\Theta\hspace{4mm}\mbox{for}\hspace{2mm}0<\mu<\frac{n+\mu-2}{2};
\end{split}
\end{equation*}
\begin{equation*}
\begin{split}
2n-3\mu+4\geq2n-2\mu\geq 2n-2-\frac{3\mu}{2}\geq n-\mu\hspace{4mm}\mbox{for}\hspace{2mm}\frac{n+\mu-2}{2}<\mu\leq4
\end{split}
\end{equation*}
and
\begin{equation*}
\begin{split}
n+2\geq n-2+\mu\geq\frac{n+\mu+2}{2}\geq n,\hspace{2mm}\frac{3n-\mu+2}{2}\geq n-\frac{\mu}{2}+2-\Theta\geq n\hspace{2mm}\mbox{for}\hspace{2mm}\frac{n+\mu-2}{2}<\mu<4.
\end{split}
\end{equation*}
Here $\Theta<\min\{\mu/2,(n+2-2\mu)/2\}$ is a small positive number.
Collecting \eqref{h-8} and \eqref{0h-8}, we conclude that for $|z_1|\geq\frac{1}{3}\mathscr{R}_{12}$,
\begin{equation}\label{3h-8}
\hbar\lesssim
\left\lbrace
\begin{aligned}
&\frac{1}{\mathscr{R}_{12}^{\min\{n-\mu+\Theta,n-2-\mu+(n-\mu+2)/2\}}}\frac{\lambda_{1}^{\frac{n+2}{2}}}{\tau(z_1)^{(n+\mu+2)/2}},\hspace{6mm}\hspace{4mm}\hspace{1mm}\hspace{2mm}\hspace{2mm}0<\mu<\frac{n+\mu-2}{2},\\
&\frac{1}{\mathscr{R}_{12}^{n-\mu}}\frac{\lambda_{1}^{\frac{n+2}{2}}}{\tau(z_1)^{\min\{n,n-\frac{\mu}{2}+2-\Theta\}}},\hspace{4mm}\hspace{6mm}\hspace{6mm}\hspace{6mm}\hspace{4mm}\hspace{2mm}\hspace{4mm}\hspace{6mm}\hspace{6mm}\hspace{2mm}\frac{n+\mu-2}{2}\leq\mu\leq4.
\end{aligned}
\right.
\end{equation}

Let $z_2=\lambda_2(x-\xi_2)$, by scaling we find that $W_2(x)=\lambda_2^{(n-2)/2}W[0,1](z_2)=:\lambda_2^{(n-2)/2}W(z_2)$,  and
\begin{equation*}
W_1(x)=\tilde{\alpha}_{n,\mu}\frac{\lambda_2^{(n-2)/2}}{\mathscr{R}_{12}^{n-2}}\Big(\frac{1}{\mathscr{R}_{12}^4+|z_2+\beta_{21}|^2}\Big)^{\frac{n-2}{2}}
\end{equation*}
with $|\beta_{21}|=|\lambda_{2}(\xi_2-\xi_1)|<1$.
\begin{itemize}
\item[$(3)$] \textbf{Core region of $W_2$}: In this case we consider
 $1\leq|z_2+\beta_{21}|\leq\mathscr{R}_{12}$, it means $\mathscr{R}_{12}^2\leq|z_1|\leq\mathscr{R}_{12}^3$, then we have $W_1\lesssim W_2$. Similar to the proof of \eqref{5-h}, we have that
\end{itemize}
\begin{equation}\label{6-h}
\hbar\lesssim
\left\lbrace
\begin{aligned}
&\frac{\lambda_{1}^{(n+2)/2}}{\mathscr{R}_{12}^{n-2}\tau(z_2)^{4}},
\hspace{6mm}\hspace{4mm}\hspace{6mm}\hspace{6mm}\hspace{2mm}0<\mu<\frac{n+\mu-2}{2},\hspace{2mm}|z_2|\leq\frac{1}{2}\mathscr{R}_{12},\\
&\frac{\lambda_{1}^{(n+2)/2}}{\mathscr{R}_{12}^{2n-4-\mu}\tau z_2)^{4}}.
\hspace{4mm}\hspace{6mm}\hspace{3mm}\hspace{6mm}\hspace{2mm}\frac{n+\mu-2}{2}\leq\mu\leq4,\hspace{2mm}|z_2|\leq\frac{1}{2}\mathscr{R}_{12}.
\end{aligned}
\right.
\end{equation}
\begin{itemize}
\item[$(4)$]
\textbf{Outside the Core region of $W_2$}: in this case we consider $|z_2+\beta_{21}|\geq \frac{1}{2}\mathscr{R}_{12}$.
\end{itemize}
Coupling
$$
|\big(W_1+W_2)^{p}- W_1^{p}-W_2^{p}\big|\lesssim W_2^{p} ~~\text{and}~~\hspace{4mm}\big|(W_1+W_2)^{p-1}- W_1^{p-1}-W_2^{p-1}\big|\lesssim W_2^{p-1},
$$
we get that
\begin{equation}\label{hh0}
\aligned
\hbar&\lesssim
\Big[\frac{1}{|x|^{\mu}}\ast \Big(
\frac{\lambda_2^{\frac{n-2}{2}}}{\tau(z_2)^{n-2}}\Big)^{p}\Big]\Big(
\frac{\lambda_2^{\frac{n-2}{2}}}{\tau(z_2)^{n-2}}\Big)^{p-1}
\lesssim\frac{\lambda_{2}^{(n+2)/2}}{\mathscr{R}_{12}^{(n-\mu+2)/2}\tau( z_2)^{(n+\mu+2)/2}},\hspace{6mm}\mbox{for}\hspace{2mm}|z_2|\geq\frac{1}{2}\mathscr{R}_{12}.
\endaligned
\end{equation}
by the Proposition \ref{B4-1}. Since
$$n-\mu+\Theta>\frac{1}{2}(n-\mu+2),\hspace{2mm}n-2-\mu+\frac{1}{2}(n-\mu+2)>\frac{1}{2}(n-\mu+2),$$
then as a consequence, we can use a unified estimates to control all these situations. By \eqref{5-h}, \eqref{3h-8}, \eqref{6-h}, and \eqref{hh0},
so that we eventually have
\begin{equation}\label{7-h}
\hbar\lesssim
\left\lbrace
\begin{aligned}
&
\sum_{i=1}^{2}\bigg[\frac{\lambda_i^{\frac{n+2}{2}}}{\mathscr{R}^{n-2}\tau( z_i)^{4}}\chi_{\{|z_i|\leq\frac{\mathscr{R}}{2}\}}
+\frac{1}{\mathscr{R}_{12}^{(n-\mu+2)/2}}\frac{\lambda_{i}^{\frac{n+2}{2}}}{\tau( z_i)^{(n+\mu+2)/2}}\chi_{\{|z_i|\geq\frac{\mathscr{R}}{2}\}}\bigg],\hspace{2mm}\hspace{4mm}\hspace{2mm}\hspace{2mm}0<\mu<\frac{n+\mu-2}{2},\\
&
\sum_{i=1}^{2}\bigg[\frac{\lambda_i^{\frac{N+2}{2}}}{\mathscr{R}^{2n-4-\mu}\tau( z_i)^{4}}\chi_{\{|z_i|\leq\frac{\mathscr{R}}{2}\}}+\frac{\lambda_{i}^{\frac{n+2}{2}}}{\mathscr{R}^{n-\mu}\tau( z_i)^{\min\{n,n-\frac{\mu}{2}+2-\Theta\}}}\chi_{\{|z_i|\geq\frac{\mathscr{R}}{2}\}}\bigg],\hspace{3mm}\hspace{1mm}\frac{n+\mu-2}{2}\leq\mu\leq4.
\end{aligned}
\right.
\end{equation}
$\mathbf{Case\ 2.}$ Bubble cluster: \\
If  $\sigma=W[\xi_1,\lambda_1]+W[\xi_2,\lambda_2]$ forming a bubble cluster $\lambda_1\lambda_2|\xi_1-\xi_2|^2\geq\frac{\lambda_1}{\lambda_2}$ but with $\lambda_1\gg\lambda_2$.  Assume that $\mathscr{R}_{12}=\sqrt{\lambda_1/\lambda_2}|\xi_1-\xi_2|=\Re_{12}^{-1}\gg1$. Then $\mathcal{Q}\approx \mathscr{R}_{12}^{2-n}$. Let $z_1=\lambda_1(x-\xi_1)$, by scaling we find that $W_1(x)=\lambda_1^{(n-2)/2}W(z_1)$,  and direct calculation gives
\begin{equation*}
W_2(x)=\tilde{\alpha}_{n,\mu}\frac{\lambda_1^{(n-2)/2}}{\mathscr{R}_{12}^{n-2}}\Big[\frac{1}{\lambda_2^2|\xi_1-\xi_2|^2}+\big|\frac{z_1}{\lambda_1|\xi_1-\xi_2|}+\gamma_{12}\big|^2\Big]^{-\frac{n-2}{2}}
\end{equation*}
with $|\gamma_{12}|=\frac{\xi_1-\xi_2}{|\xi_1-\xi_2|}$.
\begin{itemize}
\item[$(1)$] \textbf{Core region of $W_1$:}
in this case we consider $|z_1|\leq\frac{1}{2}\mathscr{R}_{12}$.
\end{itemize}
Then we have $W_2\lesssim W_1$.
Combining the decomposition \eqref{R-0}, we get
\begin{equation}\label{R-1}
\hbar\lesssim\frac{\lambda_{2}^{(n+2)/2}}{\mathscr{R}_{12}^{n-2}\tau(z_1)^{4}}, \hspace{2mm}\mbox{if}\hspace{2mm}0<\mu<\frac{n+\mu-2}{2};\hspace{2mm}\hbar\lesssim\hspace{2mm}\frac{\lambda_{1}^{(n+2)/2}}{\mathscr{R}_{12}^{2n-4-\mu}\tau (z_1)^{4}},\hspace{2mm}\mbox{if}\hspace{2mm}\frac{n+\mu-2}{2}\leq\mu\leq4.
\end{equation}
\begin{itemize}
\item[$(2)$] \textbf{Outside the Core region of $W_1$}:
we consider $\frac{1}{3}\mathscr{R}_{12}\leq|z_1|\leq\frac{1}{2}\lambda_1|\xi_1-\xi_2|=\frac{1}{2}\sqrt{\frac{\lambda_1}{\lambda_2}}\mathscr{R}_{12}$.
\end{itemize}
We have the estimate
\begin{equation*}
W_1\approx\frac{\lambda_1^{\frac{n-2}{2}}}{|z_1|^{n-2}}~~\text{and} ~~W_2\approx\frac{\lambda_1^{\frac{n-2}{2}}}{\mathscr{R}_{12}^{n-2}}
\leq\big(\frac{1}{2\mathscr{R}_{12}}\sqrt{\frac{\lambda_1}{\lambda_2}}\big)^{\frac{n-2}{2}}\frac{\lambda_1^{\frac{n-2}{2}}}{|z_1|^{(n-2)/2}}\lesssim\frac{\lambda_1^{\frac{n-2}{2}}}{|z_1|^{(n-2)/2}}.
\end{equation*}
 We further decompose $\hbar\lesssim I+II+III$ by \eqref{III}. Then all estimates $I$, $II$ and $III$ are derived by a simply computation similar to \eqref{0h}, \eqref{1h}, \eqref{2h} and \eqref{3h}. Adding up estimates for $I$, $II$ and $III$ we also obtain \eqref{3h-8} for $R$.

Let $z_2=\lambda_2(x-\xi_2)$, we have that $W_2(x)=\lambda_2^{(n-2)/2}W(z_2)$ and
\begin{equation*}
W_1(x)=\alpha_{n,\mu}\frac{\lambda_2^{(n-2)/2}}{\mathcal{R}_{12}^{n-2}}\Big[\frac{1}{\lambda_2^2|\xi_1-\xi_2|^2}+\big|\frac{z_2}{\lambda_1|\xi_1-\xi_2|}-\gamma_{12}\big|^2\Big]^{-\frac{n-2}{2}}.
\end{equation*}
\begin{itemize}
\item[$(3)$] \textbf{Core region of $W_2$}: in this case we consider $|z_2|\leq\frac{1}{2}\mathscr{R}_{12}$ and $\Big|\frac{z_2}{\lambda_1|\xi_1-\xi_2|}-\gamma_{12}\Big|\geq\frac{1}{3}$.
\end{itemize}
Then similar to the estimate of \eqref{R-1}, $\hbar$ is controlled by
\begin{equation}\label{R-2}
\hbar\lesssim\frac{\lambda_{2}^{(n+2)/2}}{\mathscr{R}_{12}^{n-2}\tau (z_2)^{4}}, \hspace{2mm}\mbox{if}\hspace{2mm}0<\mu<\frac{n+\mu-2}{2};\hspace{2mm}\hbar\lesssim\frac{\lambda_{1}^{(n+2)/2}}{\mathscr{R}_{12}^{2n-4-\mu}\tau( z_2)^{4}},\hspace{2mm}\mbox{if}\hspace{2mm}\frac{n+\mu-2}{2}\leq\mu\leq4.
\end{equation}
\begin{itemize}
\item[$(4)$] \textbf{Outside the Core region of $W_2$}:
in this case we consider $|z_2|\geq \frac{1}{3}\mathscr{R}_{12}$ and $|z_1|\geq\lambda_{1}|\xi_1-\xi_2|\geq\frac{1}{3}\mathscr{R}_{12}.$ We evaluate
\end{itemize}
\begin{equation*}
|\big(W_1+W_2)^{p}- W_1^{p}-W_2^{p}\big|\lesssim W_1^{p} +W_2^{p}\approx\frac{\lambda_{1}^{(2n-\mu)/2}}{\tau( z_1)^{2n-\mu}}+\frac{\lambda_{2}^{(2n-\mu)/2}}{\tau(z_2)^{n-\mu/2}},
\end{equation*}
\begin{equation*}
\big|(W_1+W_2)^{p-1}- W_1^{p-1}-W_2^{p-1}\big|\lesssim W_1^{p-1} +W_2^{p-1}\approx\frac{\lambda_{1}^{(n-\mu+2)/2}}{\tau( z_1)^{n-\mu+2}}+\frac{\lambda_{2}^{(n-\mu+2)/2}}{\tau(z_2)^{(n-\mu+2)/2}}.
\end{equation*}
Combining these two estimates we get that
\begin{equation}\label{R-3}
\begin{split}
\hbar&\lesssim\big(|x|^{-\mu}\ast W_1^{p}\big)W_1^{p-1}+\big(|x|^{-\mu}\ast W_2^{p}\big)W_2^{p-1}\\&
\lesssim\frac{\lambda_{1}^{(n+2)/2}}{\mathscr{R}_{12}^{(n-\mu+2)/2}\tau( z_1)^{(n+\mu+2)/2}}+\frac{\lambda_{2}^{(n+2)/2}}{\mathscr{R}_{12}^{(n-\mu+2)/2}\tau( z_2)^{(n+\mu+2)/2}}.
\end{split}
\end{equation}
by Lemma \ref{B4-1}. From \eqref{3h-8} and \eqref{R-3}, we conclude that for $|z_i|\geq\frac{1}{2}\mathscr{R}$
\begin{equation}\label{R-4-1}
\hbar\lesssim
\left\lbrace
\begin{aligned}
&\sum_{i=1}^{2}\frac{1}{\mathscr{R}_{12}^{(n-\mu+2)/2}}\frac{\lambda_{i}^{\frac{n+2}{2}}}{\tau( z_i)^{(n+\mu+2)/2}},\hspace{6mm}\hspace{4mm}\hspace{1mm}\hspace{2mm}\hspace{2mm}0<\mu\leq\frac{n+\mu-2}{2},\\
&\sum_{i=1}^{2}\frac{1}{\mathscr{R}_{12}^{n-\mu}}\frac{\lambda_{i}^{\frac{n+2}{2}}}{\tau( z_i)^{n}},\hspace{6mm}\hspace{6mm}\hspace{6mm}\hspace{6mm}\hspace{6mm}\hspace{6mm}\hspace{2mm}\frac{n+\mu-2}{2}\leq\mu\leq4.
\end{aligned}
\right.
\end{equation}

In order to get the desired estimates of $\hbar$ for $\sigma=W[\xi_1,\lambda_1]+W[\xi_2,\lambda_2]$, either a bubble tower or a bubble cluster,  by the computations in \eqref{7-h} and \eqref{R-4-1},
we eventually get from \eqref{7-h}, \eqref{R-1}, \eqref{R-2} and \eqref{R-4-1} that $\hbar$ is controlled by
\begin{equation*}
\hbar\lesssim
\left\lbrace
\begin{aligned}
&
\sum_{i=1}^{2}\bigg[\frac{\lambda_i^{\frac{n+2}{2}}}{\mathscr{R}^{n-2}\tau( z_i)^{4}}\chi_{\{|z_i|\leq\frac{\mathscr{R}}{2}\}}
+\frac{1}{\mathscr{R}_{12}^{(n-\mu+2)/2}}\frac{\lambda_{i}^{\frac{n+2}{2}}}{\tau( z_i)^{(n+\mu+2)/2}}\chi_{\{|z_i|\geq\frac{\mathscr{R}}{2}\}}\bigg],\hspace{6mm}\hspace{2mm}\hspace{2mm}0<\mu\leq\frac{n+\mu-2}{2},\\
&
\sum_{i=1}^{2}\bigg[\frac{\lambda_i^{\frac{n+2}{2}}}{\mathscr{R}^{2n-4-\mu}\tau( z_i)^{4}}\chi_{\{|z_i|\leq\frac{\mathscr{R}}{2}\}}+\frac{\lambda_{i}^{\frac{n+2}{2}}}{\mathscr{R}^{n-\mu}\tau( z_i)^{\min\{n,n-\frac{\mu}{2}+2-\Theta\}}}\chi_{\{|z_i|\geq\frac{\mathscr{R}}{2}\}}\bigg],\hspace{2mm}\hspace{2mm}\frac{n+\mu-2}{2}\leq\mu\leq4.
\end{aligned}
\right.
\end{equation*}
$\bullet$ \textbf{The case of $\kappa\geq3$.}\\
In order to make sure that the estimates of $\hbar$ in the case $\kappa=2$ can be used to control the $\hbar$ for any finite number of bubbles, it will need to carry out a check that the following inequality holds
\begin{equation}\label{ZUI-2}
\begin{split}
&~\big(|x|^{-\mu}\ast\sigma^{p}\big)\sigma^{p-1}-\sum_{i=1}^{\kappa}\big(|x|^{-\mu}\ast W_{i}^{p}\big)W_{i}^{p-1}\\&
\lesssim \sum\limits_{i\neq j}\Big[\big(|x|^{-\mu}\ast (W_{i}+W_{j})^{p}\big)(W_{i}+W_{j})^{p-1}-\big(|x|^{-\mu}\ast W_{i}^{p}\big)W_{i}^{p-1}-\big(|x|^{-\mu}\ast W_{j}^{p}\big)W_{j}^{p-1}\Big]
\end{split}
\end{equation}
We show that this is indeed true by direct computations the above inequality.
Combining estimates of $\hbar$ in the two bubbles and \eqref{ZUI-2} yields the conclusion. Summarizing, we conclude that
\begin{equation}\label{zv5-0}
\hbar\lesssim\left\lbrace
\begin{aligned}
&\sum_{i=1}^{\kappa}\big[t_{i,1}(x,\frac{\mathscr{R}}{2})
+t_{i,2}(x,\frac{\mathscr{R}}{2})\big]\lesssim\sum_{i=1}^{\kappa}\big[t_{i,1}(x,\frac{\mathscr{R}}{2})
+t_{i,2}(x,\mathscr{R})\big],\hspace{3mm}\quad\quad 0<\mu<\frac{n+\mu-2}{2},\\
&\sum_{i=1}^{\kappa}\big[\hat{t}_{i,1}(x,\mathscr{R})
+\hat{t}_{i,2}(x,\frac{\mathscr{R}}{2})\big]\lesssim\sum_{i=1}^{\kappa}\big[\hat{t}_{i,1}(x,\mathscr{R})
+\hat{t}_{i,2}(x,\mathscr{R})\big],\quad\quad\quad  \frac{n+\mu-2}{2}\leq\mu<4,\\
&\sum_{i=1}^{\kappa}\big[\tilde{t}_{i,1}(x,\frac{\mathscr{R}}{2})
+\tilde{t}_{i,2}(x,\frac{\mathscr{R}}{2})\big]\lesssim\sum_{i=1}^{\kappa}\big[\tilde{t}_{i,1}(x,\mathscr{R})
+\tilde{t}_{i,2}(x,\mathscr{R})\big],\quad\quad\quad  \mu=4.
\end{aligned}
\right.
\end{equation}
Hence the lemma follows.  \qed

\subsection{Proof of Lemma \ref{ww10}}
In the following, in order to estimate $c_a^i$ for $i = 1, 2, \cdot\cdot\cdot,\kappa$ and $a=1,\cdots,n+1$ in the system \eqref{c1}, we establish now the following key estimates:
\begin{lem}\label{cll-0-0}
Assume that $n\geq6-\mu$, $\mu\in(0,n)$ and $0<\mu\leq4$,  we have that
\begin{equation}\label{selqa}
\left\lbrace
\begin{aligned}
&\big\|S_1\big\|_{L^{2^{\ast}}}\lesssim\frac{1}{\mathscr{R}^{(n+2)/2}},  \hspace{4mm}\hspace{4mm}\hspace{3mm}\hspace{8mm}\hspace{2mm}0<\mu<\frac{n+\mu-2}{2},\\&
\big\|S_2\big\|_{L^{2^{\ast}}}\lesssim\frac{1}{\mathscr{R}^{2n-4-\mu}},
\hspace{6mm}\hspace{6mm}\hspace{3mm}\hspace{4mm}\hspace{2mm}\frac{n+\mu-2}{2}\leq\mu<4,
\end{aligned}
\right.
\end{equation}
and that
\begin{equation}\label{eqa-2-1-0}
\left\lbrace
\begin{aligned}
&\big\|T_1\big\|_{L^{(2^{\ast})^{\prime}}}\lesssim\frac{1}{\mathscr{R}^{(n+2)/2}}, \hspace{4mm}\hspace{4mm}\hspace{2mm}\hspace{8mm}\hspace{2mm}0<\mu<\frac{n+\mu-2}{2},\\&
\big\|T_2\big\|_{L^{(2^{\ast})^{\prime}}}\lesssim\frac{1}{\mathscr{R}^{n-\mu+(n-2)/2}},
\hspace{6mm}\hspace{4mm}\hspace{2mm}\frac{n+\mu-2}{2}\leq\mu<4.
\end{aligned}
\right.
\end{equation}
Here we denote by $(2^{\ast})^{\prime}=\frac{2n}{n+2}$ the H\"{o}lder conjugate of $2^{\ast}$ and $0<\epsilon_0<\frac{(n-2)p-n}{p}$ is parameter.
 \end{lem}
 \begin{proof}
We consider separately the following two cases.

\textbf{Case 1.}
$\frac{n+\mu-2}{2}\leq\mu<4$. In this case, we have
$$
\int[\hat{s}_{i,1}(x,\mathscr{R})]^{2^{\ast}}=\frac{1}{\mathscr{R}^{(2n-4-\mu)2^{\ast}}}\int_{|z_i|\leq\mathscr{R}}\frac{\lambda_i^{n}}{\tau( z_i)^{4n/(n-2)}}dx\lesssim\frac{1}{\mathscr{R}^{(2n-4-\mu)2^{\ast}}}.
$$
$$
\int[\hat{s}_{i,2}(x,\mathscr{R})]^{2^{\ast}}=\frac{1}{\mathscr{R}^{(n-\mu+\epsilon_0)2^{\ast}}}\int_{|z_i|\geq\mathscr{R}}\frac{\lambda_i^{n}}{\tau( z_i)^{(n-2-\epsilon_0)2^{\ast}}}dx\lesssim\frac{1}{\mathscr{R}^{(2n-\mu-2)2^{\ast}-n}}.
$$
Similar arguments give that
$$
\int[\hat{t}_{i,1}(x,\mathscr{R})]^{(2^{\ast})^{\prime}}=\frac{1}{\mathscr{R}^{2n(2n-4-\mu)/(n+2)}}\int_{|z_i|\leq\mathscr{R}}\frac{\lambda_i^{n}}{\tau( z_i)^{8n/(n+2)}}dx\lesssim\frac{1}{\mathscr{R}^{(2n-\mu)(2^{\ast})^{\prime}-n}}.
$$
$$
\int[\hat{t}_{i,2}(x,\mathscr{R})]^{(2^{\ast})^{\prime}}=\frac{1}{\mathscr{R}^{(n-\mu+\epsilon_0)(2^{\ast})^{\prime}}}\int_{|z_i|\geq\mathscr{R}}\frac{\lambda_i^{n}}{\tau( z_i)^{(n-\epsilon_0)(2^{\ast})^{\prime}}}dx\lesssim\frac{1}{\mathscr{R}^{(2n-\mu)(2^{\ast})^{\prime}-n}}.
$$
Coming back to the definition \eqref{zv4} and \eqref{zv5}, so that, summing over $i$, we get
\begin{equation}\label{1eq-2-1}
\int (S_2)^{2^{\ast}}\lesssim\frac{1}{\mathscr{R}^{(2n-4-\mu)2^{\ast}}}\hspace{2mm}\mbox{and}\hspace{2mm} \int (T_2)^{(2^{\ast})^{\prime}}\lesssim\frac{1}{\mathscr{R}^{(2n-\mu)(2^{\ast})^{\prime}-n}}\hspace{4mm}\mbox{for}\hspace{2mm} \frac{n+\mu-2}{2}\leq\mu<4.
\end{equation}
\textbf{Case 2.} $0<\mu<\frac{n+\mu-2}{2}$. Then
$$
\int[s_{i,1}(x,\mathscr{R})]^{2^{\ast}}=\frac{1}{\mathscr{R}^{2n}}\int_{|z_i|\leq\mathscr{R}}\frac{\lambda_i^{n}}{\tau( z_i)^{4n/(n-2)}}dx\lesssim\frac{1}{\mathscr{R}^{n(2^{\ast}-1)}}.
$$
$$
\int[s_{i,2}(x,\mathscr{R})]^{2^{\ast}}=\frac{1}{\mathscr{R}^{n(n-\mu+2)/(n-2)}}\int_{|z_i|\geq\mathscr{R}}\frac{\lambda_i^{n}}{\tau( z_i)^{n(n+\mu-2)/(n-2)}}dx\lesssim\frac{1}{\mathscr{R}^{n(2^{\ast}-1)}}.
$$
Similar to the above calculations, we also get that
$$
\int[t_{i,1}(x,\mathscr{R})]^{(2^{\ast})^{\prime}}=\frac{1}{\mathscr{R}^{2n(n-2)/(n+2)}}\int_{|z_i|\leq\mathscr{R}}\frac{\lambda_i^{n}}{\tau( z_i)^{8n/(n+2)}}dx\lesssim\frac{1}{\mathscr{R}^{n}}.
$$
$$
\int[t_{i,2}(x,\mathscr{R})]^{(2^{\ast})^{\prime}}=\frac{1}{\mathscr{R}^{n(n-\mu+2)/(n+2)}}\int_{|z_i|\geq\mathscr{R}}\frac{\lambda_i^{n}}{\tau( z_i)^{n(n+\mu+2)/(n+2)}}dx\lesssim\frac{1}{\mathscr{R}^{n}}.
$$
Coming back to the definition \eqref{zv4} and \eqref{zv5}, so that we eventually have
\begin{equation}\label{e2qa}
\int (S_1)^{2^{\ast}}\lesssim\frac{1}{\mathscr{R}^{n(2^{\ast}-1)}}\hspace{2mm}\mbox{and}\hspace{2mm} \int (T_1)^{(2^{\ast})^{\prime}}\lesssim\frac{1}{\mathscr{R}^{n}}\hspace{4mm}\mbox{for}\hspace{2mm} 0<\mu<\frac{n+\mu-2}{2}.
\end{equation}
Combining \eqref{1eq-2-1} and \eqref{e2qa} yields the conclusion.
 \end{proof}

 \begin{lem}\label{cll}
 Assume that $n\geq6-\mu$, $\mu\in(0,n)$ and $0<\mu\leq4$ satisfying $(\sharp)$. Let $\phi$, $h$ and $c_b^j$ satisfy the system \eqref{c1} for $\mathscr{R}=\mathscr{R}_m$ and  $\sigma=\sum_{i=1}^{\kappa}W_i$ is a family of $\delta-$interacting bubbles, there holds:\\
$(i).$ $n=4$ and $\mu\in[2,4)$, or $n=5$ and $\mu\in(3,4)$,
 $$|c_b^j|\lesssim\mathscr{Q}^{2-\frac{\mu}{n-2}}\|\hbar\|_{\ast\ast}+\mathscr{Q}^{2(2-\frac{\mu}{n-2})}\|\phi\|_{\ast},\hspace{3mm}j=1,\cdots, \kappa~\text{and}~b=1,\cdots,n+1.
$$
$(ii).$ $n=5$ and $\mu\in[1,3)$, or $n=6$ and $\mu\in(0,4)$, or $n=7$ and $\mu\in(\frac{7}{3},4)$,
$$|c_b^j|\lesssim\mathscr{Q}^{\min\{1,\frac{n-\mu}{n-2}\}}\|\hbar\|_{\ast\ast}+\mathscr{Q}^{2^{\ast}-1}\|\phi\|_{\ast},\hspace{3mm}j=1,\cdots, \kappa~\text{and}~b=1,\cdots,n+1.$$
 \end{lem}
 \begin{proof}
 Multiplying \eqref{c1} by $\Xi_{j}^{b}$, $(j = 1, 2, \cdot\cdot\cdot,\kappa)$ and integrating, we see that $c_{b}^{j}$ satisfies
\begin{equation}\label{a1}
     \begin{split}
&\int\big(\Delta \phi+\Phi_{n,\mu}[\sigma,\phi]\big)\Xi_{j}^{b}
		=\int \mathscr{R}\Xi_{j}^{b}+		\int\sum_{i=1}^{\kappa}\sum_{a=1}^{n+1}c_{a}^{i}\Phi_{n,\mu}[W_{i},\Xi^{a}_i]\Xi_{j}^{b}.
\end{split}
\end{equation}
We are led to estimate each term in the left and right hand side of \eqref{a1}. Since $\phi$ satisfy
$\int_{\mathbb{R}^{N}}\Phi_{n,\mu}[W_{j},\Xi^{b}_j](z)\phi=0$
for $j=1,\cdots, \kappa$ and $a=1,\cdots,n+1$, we get
\begin{equation}\label{zb1}
\begin{split}
\int_{\mathbb{R}^{N}}\Phi_{n,\mu}[\sigma,\phi]\Xi^{b}_j&=p\int\Big(\Big(|x|^{-\mu}\ast\sigma^{p-1}\phi\Big)
\sigma^{p-1}\Xi_{j}^{b}-\Big(|x|^{-\mu}\ast W_{j}^{p-1}\Xi_{j}^{b}\Big)W_{j}^{p-1}\phi\Big)\\&+(p-1)\int\Big(\Big(|x|^{-\mu}\ast\sigma^{p}\Big)
\sigma^{p-2}\phi\Xi_{j}^{b}-\Big(|x|^{-\mu}\ast W_{j}^{p}\Big)W_{j}^{p-2}\phi\Xi_{j}^{b}\Big).
\end{split}
\end{equation}
Observe there holds
\begin{equation*}
\Big(\sigma^{p-1}-W_{j}^{p-1}\Big)W_{j}\leq\sum\limits_{i}\Big(\sigma^{p-1}-W_{i}^{p-1}\Big)W_i=\sigma^{p}-\sum\limits_{i}W_{i}^{p}~~\text{for every}~~i.
\end{equation*}
If $0<\mu<\frac{n+\mu-2}{2}$,
using $|\Xi_{j}^{b}|\lesssim W_j$, we have
\begin{equation}\label{zb2}
\begin{split}
&\Big|\int\Big(\Big(|x|^{-\mu}\ast\sigma^{p-1}\phi\Big)
\sigma^{p-1}\Xi_{j}^{b}-\Big(|x|^{-\mu}\ast W_{j}^{p-1}\Xi_{j}^{b}\Big)W_{j}^{p-1}\phi\Big)\Big|\\&
\leq\int\Big|\Big(|x|^{-\mu}\ast\sigma^{p-1}\Xi_{j}^{b}\Big)\Big(
\sigma^{p-1}-W_{j}^{p-1}\Big)\phi\Big|+\int\Big|\Big(|x|^{-\mu}\ast \Big(\sigma^{p-1}-W_{j}^{p-1}\Big)\Xi_{j}^{b}\Big)
W_{j}^{p-1}\phi\Big|\\&
\lesssim\|\phi\|_{\ast}\int\Big|\Big(|x|^{-\mu}\ast \sigma^{p}\Big)
\sigma^{p-1}S_1-\sum\limits_{i}\Big(|x|^{-\mu}\ast W_{i}^{p}\Big)W_{i}^{p-1}\Big)S_1\Big|\\&+\|\phi\|_{\ast}\int\Big|\Big(|x|^{-\mu}\ast \sigma^{p}\Big)
W_{j}^{p-1}S_1-\sum_{i}\Big(|x|^{-\mu}\ast W_{i}^{p}\Big)
W_{j}^{p-1}S_1\Big|\lesssim\|\phi\|_{\ast}\int S_1(x)T_1(x).
\end{split}
\end{equation}
Moreover, we notice that
$\sum_{i=1}W_{i}^{p-1}\leq (\sum_{i=1}W_i)^{p-1}~~\text{for every}~~i,$
so we get from the
Proposition \ref{estimate1} and $|\Xi_{j}^{b}|\lesssim W_j$ that
\begin{equation}\label{zb3}
\begin{split}
&\Big|\int\Big(|x|^{-\mu}\ast|\sigma|^{p}\Big)
\sigma^{p-2}\phi\Xi_{j}^{b}-\Big(|x|^{-\mu}\ast W_{j}^{p}\Big)W_{j}^{p-2}\phi\Xi_{j}^{b}\Big|\\&
\lesssim\|\phi\|_{\ast}\int\Big(|x|^{-\mu}\ast\sigma^{2_{\mu}^{\ast}}\Big)
\Big(\sigma^{p-1}-\sum\limits_{i}W_i^{p-1}\Big)S_1\\&+\|\phi\|_{\ast}\int\Big(\Big(|x|^{-\mu}\ast \sigma^{p}\big) \sigma^{p-1}-\sum\limits_{i}\big(|x|^{-\mu}\ast W_{i}^{p}\big)W_{i}^{p-1}\Big)S_1\lesssim\|\phi\|_{\ast}\int S_1(x)T_1(x).
\end{split}
\end{equation}
When $\frac{n+\mu-2}{2}<\mu<4$, we have
\begin{equation}\label{zb2-00}
\begin{split}
&\Big|\int\Big(\Big(|x|^{-\mu}\ast\sigma^{p-1}\phi\Big)
\sigma^{p-1}\Xi_{j}^{b}-\Big(|x|^{-\mu}\ast W_{j}^{p-1}\Xi_{j}^{b}\Big)W_{j}^{p-1}\phi\Big)\Big|
\lesssim\|\phi\|_{\ast}\int S_2(x)T_2(x).
\end{split}
\end{equation}
\begin{equation}\label{zb3-00}
\begin{split}
&\Big|\int\Big(|x|^{-\mu}\ast|\sigma|^{p}\Big)
\sigma^{p-2}\phi\Xi_{j}^{b}-\Big(|x|^{-\mu}\ast W_{j}^{p}\Big)W_{j}^{p-2}\phi\Xi_{j}^{b}\Big|
\lesssim\|\phi\|_{\ast}\int S_2(x)T_2(x).
\end{split}
\end{equation}
Concerning the right-hand side of \eqref{zb1}-\eqref{zb3-00}, by Lemma~\ref{cll-0-0}, the following estimates hold:
\begin{itemize}
\item[$(1)$]
If $n=4$ and $\mu\in[2,4)$, or $n=5$ and $\mu\in[3,4)$, we get that
\begin{equation}\label{eqa-2-1}
     \begin{split}
\|\phi\|_{\ast}\int S_2(x)T_2(x)\leq\big\|S_2\big\|_{L^{2^{\ast}}}\big\|T_2\big\|_{L^{(2^{\ast})^{\prime}}}\|\phi\|_{\ast}\lesssim\frac{1}{\mathscr{R}^{2(2n-4-\mu)}}\approx\mathscr{Q}^{2(2-\frac{\mu}{n-2})}\|\phi\|_{\ast}.
\end{split}
\end{equation}
\item[$(2)$]
If $n=5$ and $\mu\in[1,3)$, or $n=6$ and $\mu\in(0,4)$, or $n=7$ and $\mu\in(\frac{7}{3},4)$, we get that
\begin{equation}\label{eqa-2-3}
     \begin{split}
\|\phi\|_{\ast}\int S_1(x)T_1(x)\leq\big\|S_1\big\|_{L^{2^{\ast}}}\big\|T_1\big\|_{L^{(2^{\ast})^{\prime}}}\|\phi\|_{\ast}\lesssim\frac{1}{\mathscr{R}^{n+2}}\approx\mathscr{Q}^{2^{\ast}-1}\|\phi\|_{\ast}.
\end{split}
\end{equation}
\end{itemize}
Recall that, the first term of right hand side of \eqref{a1}, when $n=4$ and $\mu\in[2,4)$, or $n=5$ and $\mu\in[3,4)$, we have
\begin{equation*}
\Big|\int\hbar\Xi_{j}^{b}\Big|\lesssim \big\|\hbar\big\|_{\ast\ast}\sum_{j=1}^\kappa\int T_2W_j=\big\|\hbar\big\|_{\ast\ast}\sum_{j=1}^\kappa \Big(\int \hat{t}_{i,1}w_{i,1}+\int \hat{t}_{i,2}w_{i,2}+\int \hat{t}_{i,1}w_{j,1}+\int \hat{t}_{i,2}w_{j,2}\Big).
\end{equation*}
Using Lemma \ref{4B4-0} we get that the RHS is bounded by (up to a constant) $\mathscr{R}^{-(2n-4-\mu)}\big\|\hbar\big\|_{\ast\ast}\approx\mathscr{Q}^{2-\frac{\mu}{n-2}}\big\|\hbar\big\|_{\ast\ast}$.
When $n=5$ and $\mu\in[1,3)$, or $n=6$ and $\mu\in(0,4)$, or $n=7$ and $\mu\in(\frac{7}{3},4)$, we obtain
\begin{equation*}
\Big|\int\hbar\Xi_{j}^{b}\Big|\lesssim \big\|\hbar\big\|_{\ast\ast}\sum_{j=1}^\kappa\int T_1W_j=\big\|\hbar\big\|_{\ast\ast}\sum_{j=1}^\kappa \Big(\int t_{i,1}w_{i,1}+\int t_{i,2}w_{i,2}+\int t_{i,1}w_{j,1}+\int t_{i,2}w_{j,2}\Big)
\end{equation*}
by $|\Xi_{j}^{b}|\lesssim W_j$. Applying Lemma \ref{4B4-0} again, we deduced that the RHS is bounded by (up to a constant) $\mathscr{R}^{-\min\{n-2,n-\mu\}}\big\|\hbar\big\|_{\ast\ast}\approx\mathscr{Q}^{\min\{1,\frac{n-\mu}{n-2}\}}\big\|\hbar\big\|_{\ast\ast}$.

The estimate for the last integral in \eqref{a1} is the most delicate.
\begin{equation}\label{a4}
\begin{split}
\int\sum_{i=1}^{\kappa}\sum_{a=1}^{n+1}c_{a}^{i}\Phi_{n,\mu}[W_{i},\Xi^{a}_i]\Xi_{j}^{b}
=\sum_{j=1}^{\kappa}\sum_{a=1}^{n+1}c_{a}^{j}\int\Phi_{n,\mu}[W_{j},\Xi^{a}_j]\Xi_{j}^{b}
+\sum_{i\neq j}^{\kappa}\sum_{a=1}^{n+1}c_{a}^{i}\int\Phi_{n,\mu}[W_{i},\Xi^{a}_i]\Xi_{j}^{b}.
\end{split}
\end{equation}
Let us estimate each term of the righ hand side of \eqref{a4}.
Thanks to the Proposition~\ref{wwc101}, as is easily
checked, there exists some positive constant $\Gamma_{0}^{a}$ such that for $1\leq a=b\leq n+1$,
\begin{equation*}
(p-1)\int\Big(|x|^{-\mu}\ast W_{j}^{p}\Big)W_{j}^{p-2}\Xi_{j}^{a}\Xi_{j}^{b}+p\int\Big(|x|^{-\mu}\ast (W_{j}^{p-1}\Xi_{j}^{a})\Big)W_{j}^{p-1}\Xi_{j}^{b}=\Gamma_{0}^{a},
\end{equation*}
On the other hand, for $a\neq b$,
\begin{equation*}
(p-1)\int\Big(|x|^{-\mu}\ast W_{j}^{p}\Big)W_{j}^{p-2}\Xi_{j}^{a}\Xi_{j}^{b}+p\int\Big(|x|^{-\mu}\ast (W_{j}^{p-1}\Xi_{j}^{a})\Big)W_{j}^{p-1}\Xi_{j}^{b}=0.
\end{equation*}
In conclusion, there exists a positive constant $\Gamma^{b}_0$ such that
\begin{equation}\label{ww-3}
\begin{split}
\sum_{j=1}^{\kappa}\sum_{a=1}^{n+1}c_{a}^{j}&\int\Big((p-1)\Big(|x|^{-\mu}\ast W_{j}^{p}\Big)W_{j}^{p-2}\Xi_{j}^{a}+p\Big(|x|^{-\mu}\ast (W_{i}^{p-1}\Xi_{j}^{a})\Big)W_{j}^{p-1}\Big)\Xi_{j}^{b}
=c_{b}^{j}\Gamma_0^b,
\end{split}
\end{equation}
and again thanks to the Proposition~\ref{wwc101} for $i\neq j$ and $1\leq a,b\leq n+1$, we have
\begin{equation}\label{ww-2}
\begin{split}
\sum_{i\neq j}^{\kappa}\sum_{a=1}^{n+1}c_{a}^{i}&\int\Phi_{n,\mu}[W_{i},\Xi^{a}_i]\Xi_{j}^{b}
=\sum_{i\neq j}^{\kappa}\sum_{a=1}^{n+1}c_{a}^{i}O(Q_{ij}).
\end{split}
\end{equation}
Combining \eqref{a4}, \eqref{ww-3} and \eqref{ww-2}, we get
\begin{equation*}
\begin{split}
\int&\sum_{i=1}^{\kappa}\sum_{a=1}^{n+1}c_{a}^{i}\Phi_{n,\mu}[W_{i},\Xi^{a}_i]\Xi_{j}^{b}
=c_{b}^{j}\Gamma_0^b+\sum_{i\neq j}^{\kappa}\sum_{a=1}^{n+1}c_{a}^{i}O(\mathcal{Q}_{ij})
=\int\Phi_{n,\mu}[\sigma,\phi]\Xi_{j}^{b}
		+\int \hbar\Xi_{j}^{b}.
\end{split}
\end{equation*}
Combining this equality with \eqref{zb1}-\eqref{eqa-2-3} yields the conclusion.
\end{proof}

We consider
\begin{equation}\label{0}
	\left\{\begin{array}{l}
		\displaystyle \Delta \phi+\Phi_{n,\mu}[\sigma,\phi]
		=\hbar\hspace{4.14mm}\mbox{in}\hspace{1.14mm} \mathbb{R}^n,\\
		\displaystyle \int\Phi_{n,\mu}[W_{i},\Xi^{a}_i]\phi=0,\hspace{4mm}i=1,\cdots, \kappa; ~a=1,\cdots,n+1.
	\end{array}
	\right.
\end{equation}
In conclusion, we deduce the following important a-priori estimate.
\begin{lem}\label{estimate2}
Assume that $n\geq6-\mu$, $\mu\in(0,n)$ and $0<\mu\leq4$ satisfying $(\sharp)$. Let $\phi$ be the solution to problem \eqref{0}. Then it holds that
$$\|\phi\|_{\ast}\leq C \|\hbar\|_{\ast\ast}, $$
where $\|\phi\|_{\ast}=\sup_{x\in\mathbb{R}^n}\big|\phi(x)\big|S_j^{-1}(x)$ and $\|\hbar\|_{\ast\ast}=\sup_{x\in\mathbb{R}^N}\big|\hbar(x)\big|T_j^{-1}(x)$ with weight functions from \eqref{f-ai}-\eqref{h}.
\end{lem}
The key point here is to prove the priori estimate for $\|\phi\|_{\ast}$.
Once this is done, a well-known standard argument (cf. \cite{Manuel}) shows that Lemma \ref{ww10}. However, for clarity and coherence, the proof of Lemma \ref{estimate2} will be deferred and presented in Section 5.

We next give the following proof.\\
\emph{\textbf{Proof of Lemma \ref{ww10}.}}
We now proceed similarly to the Proposition $4.1$ in \cite{Manuel}. Let $\mathscr{M}$ be given by
\begin{equation*}
\begin{split}
\mathscr{M}=&\Big\{\phi\in D^{1,2}(\mathbb{R}^n):\big\langle\phi,\Xi_{i}^{a}\big\rangle_{D^{1,2}(\mathbb{R}^n)}=\int\Phi_{n,\mu}[W_{i},\Xi^{a}_i]\phi=0,\hspace{2mm}\mbox{for}\hspace{2mm}i=1,\cdots, \kappa; ~a=1,\cdots,n+1\Big\},
\end{split}
\end{equation*}
and equipped with the following inner product$
\big\langle\phi,\psi\big\rangle_\mathscr{M}=\big\langle\phi,\psi\big\rangle_{D^{1,2}(\mathbb{R}^n)}
=\int \nabla\varphi\nabla\psi.$
We solving the system \eqref{c1} in the weak form is equivalent to finding a function $\phi\in \mathscr{M}$ such that
$$
\big\langle\phi,\psi\big\rangle_{\mathscr{M}}=\Big\langle \Phi_{n,\mu}[\sigma,\phi]-\hbar,\psi\Big\rangle_{L^2}, \hspace{2mm}\mbox{for all}\hspace{2mm}\psi\in \mathscr{M},
$$
which in operate form can be written as
\begin{equation}\label{adeta}
\begin{split}
\phi=\mathcal{T}_{\delta}+\tilde{\hbar},
\end{split}
\end{equation}
where $\mathcal{T}_{\delta}$ is a compact operator on $\mathscr{M}$ and $\tilde{\hbar}\in \mathscr{M}$ depends linearly in $\hbar$.
So that, by Fredholm's alternative theorem, there exists $\phi\in \mathscr{M}$ is the unique solution of \eqref{adeta} provided that where the only solution $\phi$ solves $\phi=\mathcal{T}_{\delta}(\phi)$
is $\phi\equiv0$ in $\mathscr{M}$. Our main goal now is to show that the equation
\begin{equation*}
\begin{split}
 \Delta \phi&+\Phi_{n,\mu}[\sigma,\phi]=
		\sum_{i=1}^{\kappa}\sum_{a=1}^{n+1}c_{a}^{i}\Phi_{n,\mu}[W_{i},\Xi^{a}_i](z).
\end{split}
\end{equation*}
has a trival solution in $\mathscr{M}$. Assume from now on by contradiction that there exists a non-trivial solution $\phi=\phi_{\delta}$ for $\delta$ small so that we may choose $\|\phi_\delta\|_{\ast}=1$. On the other hand, from the Lemma \ref{cll} and Lemma \ref{estimate2}, we observe that
\begin{equation*}
\begin{split}
\|\phi_\delta\|_{\ast}&\leq C\sum_{i=1}^{\kappa}\sum_{a=1}^{n+1}|c_{a}^{i}|\\&\lesssim
\left\lbrace
\begin{aligned}
& \mathscr{Q}^{2^{\ast}-1}\|\phi_\delta\|_{\ast},\hspace{4mm}n=5\hspace{2mm}\mbox{and}\hspace{2mm}\mu\in[1,3),\mbox{or}\hspace{2mm}n=6\hspace{2mm}\mbox{and}\hspace{2mm}\mu\in(0,4),\mbox{or}\hspace{2mm}n=7\hspace{2mm}\mbox{and}\hspace{2mm}\mu\in(\frac{7}{3},4),\\
& \mathscr{Q}^{2(2-\frac{\mu}{n-2})}\|\phi_{\delta}\|_{\ast},\hspace{4mm}n=4\hspace{2mm}\mbox{and}\hspace{2mm}\mu\in[2,4),\mbox{or}\hspace{2mm}n=5\hspace{2mm}\mbox{and}\hspace{2mm}\mu\in[3,4).
\end{aligned}
\right.
\end{split}
\end{equation*}
However this estimate gives a contradiction for $\mathscr{Q}$ is small enough. Therefore, the system \eqref{c1} admits a unique solution in $\mathscr{M}$ and the conclusion follow by the Lemma \ref{cll}. \qed

\subsection{Proof of Theorem \ref{Figalli}}
Recalling that the hight order term $\mathscr{N}$ in Lemma \ref{nnn1}, we have
\begin{equation*}
\mathscr{N}(\sigma,\phi)\lesssim \mathscr{N}_1(\sigma,\phi)+\mathscr{N}_2(\sigma,\phi)+\cdots+\mathscr{N}_9(\sigma,\phi).
\end{equation*}
Then we multiply the (LHS) of the equation (\ref{u-0}) by $\Xi_{m}^{n+1}$ and integrating by
parts, we get
\begin{equation}\label{0-h-1}
\begin{split}
&\Big|\int \Xi_{m}^{n+1}\hbar\Big|=\Big|\int\Big(\Delta \rho+\mathscr{N}(\sigma,\phi)+\hat{f}+\Phi_{n,\mu}[\sigma,\phi]\Big)\Xi_{m}^{n+1}\Big|\\&
\lesssim\big\|\hat{f}\big\|_{(D^{1,2}(\mathbb{R}^n))^{-1}}+\Big|\Phi_{n,\mu}[\sigma,\phi]\Xi_{m}^{n+1}\Big|+\int \Big|\mathscr{N}_1(\sigma,\phi)\Xi_{m}^{n+1}\Big|+\cdots+\int \Big|\mathscr{N}_9(\sigma,\phi)\Xi_{m}^{n+1}\Big|,
\end{split}
\end{equation}
where we have used the orthogonality conditions, $|\Xi_{m}^{n+1}|\lesssim W_m$ and $\|W_{m}\|_{(D^{1,2}(\mathbb{R}^n))^{-1}}\leq C(n)$.
\begin{lem}\label{p2}
Assume that $n\geq6-\mu$, $\mu\in(0,n)$ and $0<\mu\leq4$. Letting $W_m$ be given by $\sigma=\sigma_m+W_m \hspace{2mm}\mbox{where}\hspace{2mm} \sigma_{m}=\sum_{i=1,i\neq m}^{\kappa}W_i.$  Then we have
\begin{equation*}
\begin{split}
\int \Xi_{m}^{n+1}\hbar&=\sum_{i=1,i\neq m}^{\kappa}\int\Big(|x|^{-\mu}\ast W_i^{p}\Big)
W_i^{p-1}\lambda_{m}\partial_{\lambda_m} W_m+o(\mathcal{\mathscr{Q}}).
\end{split}
\end{equation*}
\end{lem}
\begin{proof}
From the definition of $\hbar$ implies that
\begin{equation*}
\begin{split}
\int \Xi_{m}^{n+1}\hbar&=\int \Big[\Big(|x|^{-\mu}\ast\sigma^{p}\Big)\sigma^{p-1}-\sum_{i=1}^{\kappa}\Big(|x|^{-\mu}\ast W_{i}\Big)W_{i}^{p-1}\Big]\lambda_{m}\partial_{\lambda_m} W_m\\&=\mathscr{I}_1+\mathscr{I}_2+\mathscr{I}_3+\mathscr{I}_4,
\end{split}
\end{equation*}
with
\begin{equation*}
\begin{split}
&\mathscr{I}_1:=\int_{\{\kappa W_m\geq\sigma_m\}} \Big[\Big(|x|^{-\mu}\ast\sigma^{p}\Big)\sigma^{p-1}-\Big(|x|^{-\mu}\ast W_m^{p}\Big)W_m^{p-1},
\\&\hspace{6mm}\hspace{4mm}-p
\Big(|x|^{-\mu}\ast W_m^{p-1}\sigma_m\Big)
W_m^{p-1}-(p-1)
\Big(|x|^{-\mu}\ast W_m^{p}\Big)
W_m^{p-2}\sigma_{m}\Big]\lambda_{m}\partial_{\lambda_m} W_m,
\end{split}
\end{equation*}
\begin{equation*}
\begin{split}
&\mathscr{I}_2:=\int_{\{\kappa W_m\geq\sigma_m\}} \Big[p
\Big(|x|^{-\mu}\ast W_m^{p-1}\sigma_m\Big)
W_m^{p-1}+(p-1)
\Big(|x|^{-\mu}\ast W_m^{p}\Big)
W_m^{p-2}\sigma_{m}\\&\hspace{6mm}\hspace{4mm}-\sum_{i=1,i\neq m}^{\kappa}\Big(|x|^{-\mu}\ast W_{i}^{p}\Big)W_{i}^{p-1}\Big]\lambda_{m}\partial_{\lambda_m} W_m,
\end{split}
\end{equation*}
\begin{equation*}
\begin{split}
&\mathscr{I}_3:=\int_{\{\sigma_m>\kappa W_m\}} \Big[\Big(|x|^{-\mu}\ast\sigma^{p}\Big)\sigma^{p-1}-\Big(|x|^{-\mu}\ast\sigma_m^{p}\Big)\sigma_m^{p-1}
\\&\hspace{6mm}\hspace{4mm}-p
\Big(|x|^{-\mu}\ast\sigma_m^{p-1}W_m\Big)
\sigma_m^{p-1}-(p-1)
\Big(|x|^{-\mu}\ast\sigma_m^{p}\Big)
\sigma_m^{p-2}W_{m}\Big]\lambda_{m}\partial_{\lambda_m} W_m,
\end{split}
\end{equation*}
\begin{equation*}
\begin{split}
&\mathscr{I}_4:=\int_{\{\sigma_m>\kappa W_m\}} \Big[\Big(|x|^{-\mu}\ast\sigma_m^{p}\Big)\sigma_m^{p-1}
+p
\Big(|x|^{-\mu}\ast\sigma_m^{p-1}W_m\Big)
\sigma_m^{p-1}\\&\hspace{6mm}\hspace{4mm}+(p-1)
\Big(|x|^{-\mu}\ast\sigma_m^{p}\Big)
\sigma_m^{p-2}W_{m}-\sum_{i=1}^{\kappa}\Big(|x|^{-\mu}\ast W_{i}^{p}\Big)W_{i}^{p-1}\Big]\lambda_{m}\partial_{\lambda_m} W_m.
\end{split}
\end{equation*}
Now we evaluate separately the various terms.
A similar to the decomposition of \eqref{nnn1-0},
\begin{equation*}
\begin{split}
\big|\mathscr{I}_1\big|\lesssim \int_{\{\kappa W_m\geq\sigma_m\}}\Big|\Big(\mathscr{N}_1(W_m,\sigma_m)+\mathscr{N}_2(W_m,\sigma_m)+\cdots+\mathscr{N}_9(W_m,\sigma_m)\Big)\lambda_{m}\partial_{\lambda_m} W_m\Big|
\end{split}
\end{equation*}
Combining Lemma \ref{p1-00}
and Lemma~\ref{FPU1},
we get that
\begin{equation*}
\begin{split}
\int_{\{\kappa W_m\geq\sigma_m\}}\Big|\mathcal{I}_{\mu}\{W_m^{p}\}\sigma_{m}^{p-1}(x)\lambda_{m}\partial_{\lambda_m} W_m\Big|&\lesssim \widetilde{\alpha}_{n,\mu}\int_{\{\kappa W_m\geq\sigma_m\}}\Big|W_m^{2^{\ast}-1}\sigma_{m}^2\lambda_{m}\partial_{\lambda_m} W_m\Big|
\\&\lesssim\int_{\{\kappa W_m\geq\sigma_m\}}W_m^{2^{\ast}-2}\sigma_{m}^2\approx \mathscr{Q}^{1+\varepsilon}.
\end{split}
\end{equation*}
We similarly compute and get
\begin{equation*}
\begin{split}
\int_{\{\kappa W_m\geq\sigma_m\}}\Big|\big(\mathcal{I}_{\mu}\{W_{m}^{p-2}\sigma_m^2\}W_{m}^{p-1}+\mathcal{I}_{\mu}\{\sigma_{m}^{p}\}W_{m}^{p-1}\big)\lambda_{m}\partial_{\lambda_m} W_m\Big|\lesssim\mathscr{Q}^{1+\varepsilon},
\end{split}
\end{equation*}
\begin{equation*}
\begin{split}
\Big|\int_{\{\kappa W_m\geq\sigma_m\}}\int_{\{\kappa W_m\geq\sigma_m\}}\frac{W_m^{p-1}\sigma_{m}\sigma_{m}^{p-1}}{|x-y|^{\mu}}\lambda_{m}\partial_{\lambda_m} W_m\Big|
\lesssim\mathscr{Q}^{1+\varepsilon},
\end{split}
\end{equation*}
and by HLS inequality, we deduce that
\begin{equation*}
\begin{split}
&\Big|\int_{\{\kappa W_m\geq\sigma_m\}}\int_{\{\kappa W_m<\sigma_m\}}\frac{W_m^{p-1}\sigma_{m}\sigma_{m}^{p-1}}{|x-y|^{\mu}}\lambda_{m}\partial_{\lambda_m} W_m\Big|\\&
\lesssim\Big(\int_{\{\kappa W_m\geq\sigma_m\}}W_m^{\frac{2n(p-1)}{2N-\mu}}\sigma_{m}^{\frac{2n}{2n-\mu}}\Big)^{(2n-\mu)/(2n)}\Big(\int W_m^{\frac{2n(p-\varepsilon)}{2n-\mu}}\sigma_{m}^{\frac{2n\varepsilon}{2n-\mu}}\Big)^{(2n-\mu)/(2n)}\lesssim\mathscr{Q}^{1+\varepsilon},
\end{split}
\end{equation*}
Therefore,
$$\int_{\{\kappa W_m\geq\sigma_m\}}\Big|\big(\mathcal{I}_{\mu}\{W_{m}^{p-1}\sigma_m\}\sigma_{m}^{p-1}+\mathcal{I}_{\mu}\{W_{m}^{p-2}\sigma_m^2\}\sigma_{m}^{p-1}+\mathcal{I}_{\mu}\{\sigma_m^p\}\sigma_{m}^{p-1}\big)\lambda_{m}\partial_{\lambda_m} W_m\Big|\lesssim\mathscr{Q}^{1+\varepsilon}.$$
Moreover, by HLS inequality and Lemma~\ref{FPU1}, we get
\begin{equation*}
\begin{split}
\Big|\int_{\{\kappa W_m\geq\sigma_m\}}\int_{\{\kappa W_m\geq\sigma_m\}}\frac{W_m^{p-1}\sigma_{m}W_{m}^{p-2}\sigma_{m}}{|x-y|^{\mu}}\lambda_{m}\partial_{\lambda_m} W_m\Big|&\leq\Big(\int_{\{\kappa W_m\geq\sigma_m\}}W_m^{\frac{2n(p-1)}{2n-\mu}}\sigma_{m}^{\frac{2n}{2n-\mu}}\Big)^{\frac{2n-\mu}{n}}
\\&\approx \mathscr{Q}^{2}
\end{split}
\end{equation*}
and
\begin{equation*}
\begin{split}
&\Big|\int_{\{\kappa W_m\geq\sigma_m\}}\int_{\{\kappa W_m\leq\sigma_m\}}\frac{W_m^{p-1}\sigma_{m}W_{m}^{p-2}\sigma_{m}}{|x-y|^{\mu}}\lambda_{m}\partial_{\lambda_m} W_m\Big|\\&\leq\Big[\int_{\{\kappa W_m\geq\sigma_m\}}W_m^{\frac{2n(p-1)}{2n-\mu}}\sigma_{m}^{\frac{2n}{2n-\mu}}\Big]^{(2n-\mu)/(2n)}\Big[\int_{\{\kappa W_m\leq\sigma_m\}}W_m^{\frac{2n(p-1)}{2n-\mu}}\sigma_{m}^{\frac{2n}{2n-\mu}}\Big]^{(2n-\mu)/(2n)}
\approx \mathscr{Q}^{2}
\end{split}
\end{equation*}
Similar to the above argument
\begin{equation*}
\begin{split}
\int_{\{\kappa W_m\geq\sigma_m\}}\Big|\big(\mathcal{I}_{\mu}\{W_{m}^{p-2}\sigma_m^2\}W_{m}^{p-2}\sigma_m+\mathcal{I}_{\mu}\{\sigma_m^p\}W_{m}^{p-2}\sigma_m\big)\lambda_{m}\partial_{\lambda_m} W_m\Big|\lesssim\mathscr{Q}^{2}.
\end{split}
\end{equation*}
As a consequence, we are able to conclude that
\begin{equation}\label{I1}
\big|\mathscr{I}_1\big|\lesssim\mathscr{Q}^{1+\varepsilon}.
\end{equation}
Similarly, it holds that $\mathscr{I}_3\lesssim\mathscr{Q}^{1+\varepsilon}.$
For $\mathscr{I}_4$, let us first collect some simple calculations.
We first get that
\begin{equation}\label{i4-0}
\aligned
\int_{\{\sigma_m>\kappa W_m\}} &\Big(\int_{\{\sigma_m>\kappa W_m\}}
\frac{\sigma_m^{p-1}(y)W_m(y)
\sigma_m^{p-1}W_m}{|x-y|^{\mu}}+
\int_{\{\kappa W_m\geq\sigma_m\}}\frac{\sigma_m^{p-1}(y)
W_m(y)\sigma_m^{p-1}W_{m}}{|x-y|^{\mu}}\Big)\\&
\lesssim\Big(\int_{\{\sigma_m>\kappa W_m\}}\sigma_m^{\frac{2n(p-1)}{2n-\mu}}W_{m}^{\frac{2n-\mu}{2n}}\Big)^{(2n-\mu)/n}+\int_{\{\sigma_m>\kappa W_m\}}
|\sigma_m|^{p-1}
W_m^{2^{\ast}+1-p}\\&
\lesssim \Big(\int \sigma_m^{\frac{2n(p-1)+\mu}{2n-\mu}}W_{m}^{1+\varepsilon}\Big)^{(2n-\mu)/n}+\int
\sigma_m^{2^{\ast}-1-\varepsilon}
W_m^{1+\varepsilon}.
\endaligned
\end{equation}
And
\begin{equation}\label{i4-1}
\aligned
\int_{\{\sigma_m>\kappa W_m\}} &\Big(\int_{\{\sigma_m>\kappa W_m\}}
\frac{\sigma_m^{p}(y)
\sigma_m^{p-2}W_m^2}{|x-y|^{\mu}}+
\int_{\{\kappa W_m\geq\sigma_m\}}\frac{\sigma_m^{p}(y)
\sigma_m^{p-2}W_{m}^2}{|x-y|^{\mu}}\Big)
\lesssim \int
\sigma_m^{2^{\ast}-1-\varepsilon}
W_m^{1+\varepsilon}.
\endaligned
\end{equation}
Moreover, we use the elementary inequalities \eqref{pmr3}-\eqref{pmr4} from Lemma \ref{p1}, we have that
\begin{equation}\label{pmr5}
\aligned
&\int_{\{\sigma_m>\kappa W_m\}} \Big[\Big(|x|^{-\mu}\ast\sigma_m^{p}\Big)\sigma_m^{p-1}
-\sum_{i=1, i\neq m}^{\kappa}\Big(|x|^{-\mu}\ast W_{i}^{p}\Big)W_{i}^{p-1}\Big]W_m\\&
\lesssim\int_{\{\sigma_m>\kappa W_m\}} \Big[|x|^{-\mu}\ast\Big(\sigma_m^{p}-\sum_{i=1, i\neq m}^{\kappa}W_{i}^{p}\Big)\Big]\sigma_m^{p-1}W_m
\\&\hspace{6mm}\hspace{6mm}+\sum_{i=1, i\neq m}^{\kappa}\int_{\{\sigma_m>\kappa W_m\}} \Big(|x|^{-\mu}\ast W_{i}^{p}\Big)\Big(\sigma_m^{p-1}-W_{i}^{p-1}\Big)W_m\\&
\lesssim\sum\limits_{\substack{1\leq i<j\leq\kappa\\ i\neq m, j\neq m}}\int_{\{\sigma_m>\kappa W_m\}} \Big(|x|^{-\mu}\ast\Big(W_{i}^{p-1}W_{j}+W_{i}^{p-2}W_{j}^2\Big)\Big)|\sigma_m|^{p-1}W_m\\&\hspace{6mm}\hspace{6mm}+\int_{\{\sigma_m>\kappa W_m\}}\Big(\sum_{i=1, i\neq m}^{\kappa}|x|^{-\mu}\ast W_{i}^{p}\Big)\Big(\sum\limits_{\substack{1\leq i<j\leq\kappa\\ i\neq m, j\neq m}}W_{i}^{p-2}W_{j}+\sum_{l\neq i, l\neq m}W_l^{p-1}\Big)W_m.
\endaligned
\end{equation}
Calculating the first term on the right-hand side, we get that
\begin{equation}\label{pmr5-1}
\aligned
\int_{\{\sigma_m>\kappa W_m\}} \Big(|x|^{-\mu}\ast\Big(W_{i}^{p-1}W_{j}\Big)\Big)\sigma_m^{p-1}W_m&\lesssim\big\|W_{i}^{p-1}W_{j}\big\|_{L^{r}}\big\||\sigma_m|^{p-\varepsilon}W_m^{\varepsilon}\big\|_{L^{r}}
\lesssim\mathscr{Q}^{1+\varepsilon},\hspace{2mm}\mbox{for}\hspace{2mm}i\neq j, 
\endaligned
\end{equation}
\begin{equation}\label{pmr5-1-0}
\aligned
\int_{\{\sigma_m>\kappa W_m\}} \Big(|x|^{-\mu}\ast\Big(W_{i}^{p-2}W_{j}^2\Big)\Big)\sigma_m^{p-1}W_m&\lesssim\big\|W_{i}^{p-2}W_{j}^2\big\|_{L^{r}}\big\||\sigma_m|^{p-1}W_m\big\|_{L^{r}}
\lesssim \mathscr{Q}^{p-1},\hspace{2mm}\mbox{for}\hspace{2mm}i\neq j.
\endaligned
\end{equation}
Similar to the calculation of \eqref{pmr5-1}-\eqref{pmr5-1-0}, we also get that
\begin{equation}\label{pmr5-2}
\aligned
\int_{\{\sigma_m>\kappa W_m\}}\Big(\sum_{i=1, i\neq m}^{\kappa}\Big(|x|^{-\mu}\ast W_{i}^{p}\Big)\Big)&\Big(\sum\limits_{\substack{1\leq i<j\leq\kappa\\ i\neq m, j\neq m}}W_{i}^{p-2}W_{j}W_{m}\Big)\lesssim \mathscr{Q}^{\frac{(4-\mu)r}{n-2}+\frac{n\mu}{2(2n-\mu)}}\big|\log{(\frac{1}{\mathscr{Q}})}\big|^{\frac{\mu(n-2)}{2(2n-\mu)}}.
\endaligned
\end{equation}
By Lemma \ref{FPU3} and
choosing $\varepsilon>$ small such that $1+\varepsilon<p-1-\varepsilon$, we have
\begin{equation}\label{pmr5-3}
\aligned
\int_{\{\sigma_m>\kappa W_m\}}\Big(\sum_{i=1, i\neq m}^{\kappa}&\Big(|x|^{-\mu}\ast W_{i}^{p}\Big)\Big)\Big(\sum_{l\neq i, l\neq m}^{\kappa}W_l^{p-1}W_m\Big)\\&=\Big(\mathscr{Q}^{1+\varepsilon}\hspace{2mm}\mbox{if}\hspace{2mm}0<\mu<4;\hspace{2mm}\mathscr{Q}^{\frac{n}{n-2}}\log{\frac{1}{\mathscr{Q}}}\hspace{2mm}\mbox{if}\hspace{2mm}\mu=4\Big)
\endaligned
\end{equation}
by Lemmas \ref{FPU2} and \ref{FPU1}. We get from \eqref{pmr5}-\eqref{pmr5-2} and \eqref{pmr5-3} that
\begin{equation}\label{pmr5-4}
LHS~\text{of}~\eqref{pmr5}\lesssim\mathscr{Q}^{1+\varepsilon}
\end{equation}
Substituting this estimate and \eqref{i4-0}-\eqref{i4-1} into the right-hand side of $\mathscr{I}_4$, we have that
\begin{equation}\label{I4-0}
\aligned
|\mathscr{I}_4|&=\big|\int_{\{\sigma_m>\kappa W_m\}} \Big[\Big(|x|^{-\mu}\ast\sigma_m^{p}\Big)\sigma_m^{p-1}
+p\Big(|x|^{-\mu}\ast\sigma_m^{p-1}W_m\Big)
\sigma_m^{p-1}\\&\hspace{2mm}+(p-1)
\Big(|x|^{-\mu}\ast\sigma_m^{p}\Big)
\sigma_m^{p-2}W_{m}-\sum_{i=1}^{\kappa}\Big(|x|^{-\mu}\ast W_{i}^{p}\Big)W_{i}^{p-1}\Big]\lambda_{m}\partial_{\lambda_m} W_m\big|\\&
\lesssim \mathscr{Q}^{1+\varepsilon}+\int_{\{\sigma_m>\kappa W_m\}}\Big(|x|^{-\mu}\ast W_{m}^{p}\Big)W_{m}^{p}.
\endaligned
\end{equation}
Observe that by Lemma \ref{FPU2}
\begin{equation}\label{I4-1}
\aligned
\int_{\{\sigma_m>\kappa W_m\}}\Big(|x|^{-\mu}\ast W_{m}^{p}\Big)W_{m}^{p}&
\leq\widetilde{\alpha}_{n,\mu}\sum_{i=1,i\neq m}\int W_m^{2^{\ast}-1-\varepsilon}\inf(W_i^{1+\varepsilon},W_m^{1+\varepsilon})\\&
=O\big(Q_{im}^{\frac{n}{n-2}}|\log Q_{im}|\big)=o(\mathscr{Q}).
\endaligned
\end{equation}
Recalling that $\mathscr{I}_2$,
it is clear that
\begin{equation}\label{I11}
\aligned
\mathscr{I}_2&-p\int
\Big(|x|^{-\mu}\ast(W_m^{p-1}\sum_{i=1,i\neq m}^{\kappa}W_i)\Big)
W_m^{p-1}\lambda_{m}\partial_{\lambda_m} W_m\\&\hspace{4mm}+(p-1)\int\sum_{i=1,i\neq m}^{\kappa}
\Big(|x|^{-\mu}\ast W_m^{p}\Big)
W_m^{p-2}W_i\lambda_{m}\partial_{\lambda_m} W_m\\&
=- \int_{\{\kappa W_m<\sigma_m\}}
\Big[p\Big(|x|^{-\mu}\ast W_m^{p-1}\sigma_m\Big)
W_m^{p-1}\lambda_{m}\partial_{\lambda_m} W_m-(p-1)
\Big(|x|^{-\mu}\ast W_m^{p}\Big)\\&\hspace{4mm}\times
W_m^{p-2}\sigma_{m}\Big]\lambda_{m}\partial_{\lambda_m} W_m-\int_{\{\kappa W_m\geq\sigma_m\}}\sum_{i=1,i\neq m}^{\kappa}\Big(|x|^{-\mu}\ast W_{i}^{p}\Big)W_{i}^{p-1}\lambda_{m}\partial_{\lambda_m} W_m.
\endaligned
\end{equation}
An analogous argument as  $\mathscr{I}_1$ and $\mathscr{I}_2$ tell us that
\begin{equation}\label{I4-2}
\aligned
RHS~\text{of}~\eqref{I11}\lesssim\mathscr{Q}^{1+\varepsilon}.
\endaligned
\end{equation}
From \eqref{I4-0}, \eqref{I4-1} and \eqref{I4-2} we conclude that
\begin{equation}\label{I4-3}
\aligned
\int \Xi_{m}^{n+1}\hbar=\int \lambda_{m}\partial_{\lambda_m} W_m\hbar&=p\sum_{i=1,i\neq m}^{\kappa}\int
\Big(|x|^{-\mu}\ast W_m^{p-1}W_i\Big)
W_m^{p-1}\lambda_{m}\partial_{\lambda_m} W_m\\&+(p-1)\sum_{i=1,i\neq m}^{\kappa}\int
\Big(|x|^{-\mu}\ast W_m^{p}\Big)
W_m^{p-2}W_i\lambda_{m}\partial_{\lambda_m} W_m+o(\mathscr{Q}).
\endaligned
\end{equation}
Furthermore, according to Lemma \ref{wwc101} we have
\begin{equation*}
\begin{split}
p\int
&\Big(|x|^{-\mu}\ast W_m^{p-1}W_i\Big)
W_m^{p-1}\lambda_{m}\partial_{\lambda_m} W_m+(p-1)\int
\Big(|x|^{-\mu}\ast W_m^{p}\Big)
W_m^{p-2}W_i\lambda_{m}\partial_{\lambda_m} W_m\\&
=\int\Big(|x|^{-\mu}\ast W_i^{p}\Big)
W_i^{p-1}\lambda_{m}\partial_{\lambda_m} W_m.
\end{split}
\end{equation*}
Combining this equality with \eqref{I4-3} yields the conclusion.
\end{proof}

\begin{lem}\label{FPU2-1}
There exists a dimensional constant $C$ such that
$$\int W_i^{2^{\ast}-1}\lambda_m\partial_{\lambda_m}W_m=-C(n)\big(Q_{im}^{-\frac{2}{n-2}}-2\frac{\lambda_i}{\lambda_m}\big)Q_{im}^{\frac{n}{n-2}}+O(Q_{im}^{\frac{n}{n-2}}\log Q_{im}^{-1}).$$
\end{lem}
\begin{proof}
The proof of can be found in \cite{Bahri-1989}.
\end{proof}
GLOG, we may assume that
\begin{equation}\label{lamta}
\lambda_1\leq\lambda_1\leq,\dots,\leq\lambda_\kappa.
\end{equation}
Then, using Lemma \ref{FPU2-1} implies that for $i=1,\cdots, \kappa$
\begin{equation}\label{Qij}
\big|\int W_i^{2^{\ast}-1}\lambda_m\partial_{\lambda_m}W_m\big|\lesssim Q_{im}, \hspace{2mm}i\neq m;\hspace{2mm}-\int W_i^{2^{\ast}-1}\lambda_m\partial_{\lambda_m}W_m\approx Q_{im},\hspace{2mm}i<m.
\end{equation}

As a consequence of Lemma \ref{Ni-1} and Lemma \ref{Ni-2}, we deduce the following important estimate.
\begin{lem}\label{QQ-1}
Assume that $n\geq6-\mu$, $\mu\in(0,n)$ and $0<\mu\leq4$. If $\delta$ is small we have
\begin{equation}\label{Qf}
\mathscr{Q}\lesssim\big\|\hat{f}\big\|_{(D^{1,2}(\mathbb{R}^n))^{-1}},
\end{equation}
where $\hat{f}:=-\Delta u-\big(|x|^{-\mu}\ast u^{p}\big)u^{p-1}$.
 \end{lem}
\begin{proof}
We are going to show that the statement of Lemma holds by induction. For $n=4$ and $\mu\in[2,3)$, or $n=5$ and $\mu\in[1,4)$, or $n=6$ and $\mu\in(0,4)$, or $n=7$ and $\mu\in(\frac{7}{3},4)$, and $2\leq\varsigma\leq\kappa$, introducing the induction hypothesis
$$(\Pi_{\varsigma}):\hspace{4mm}\sum_{j=\varsigma}^{\kappa}\sum_{i=1}^{j-1}Q_{ij}\lesssim
\big\|\hat{f}\big\|_{(D^{1,2}(\mathbb{R}^n))^{-1}}+o(\mathscr{Q}).$$
Directly computing, one has
\begin{equation}\label{www-q}
\sum_{i=1,i\neq m}^{\kappa}\int\Big(|x|^{-\mu}\ast W_i^{p}\Big)
W_i^{p-1}\lambda_{m}\partial_{\lambda_m} W_m=\widetilde{\alpha}_{n,\mu}\sum_{i=1,i\neq m}^{\kappa}\int W_i^{2^{\ast}-1}\lambda_{m}\partial_{\lambda_m} W_m.
\end{equation}
Therefore, using \eqref{0-h-1}, \eqref{Qij}, Lemma \ref{p2}, Lemma \ref{Ni-1} and Lemma \ref{Ni-2} and taking $m=\kappa$, we get the following two cases holds:\\
$\bullet$ $n=4$ and $\mu\in[2,3)$, or $n=5$ and $\mu\in[1,4)$, or $n=6$ and $\mu\in(0,4)$, or $n=7$ and $\mu\in(\frac{7}{3},4)$,
\begin{equation*}
\begin{split}
\sum_{i=1}^{\kappa-1}Q_{i\kappa}&\approx\widetilde{\alpha}_{n,\mu}^{-1}\int \Xi_{m}^{n+1}\hbar+o(\mathscr{Q})\lesssim
\big\|\hat{f}\big\|_{(\mathcal{D}^{1,2}(\mathbb{R}^n))^{-1}}+o(\mathscr{Q}).
\end{split}
\end{equation*}
The above estimates tell us $(\Pi_{\kappa})$ holds true. Now we can assume that the statement of $(\Pi_{\varsigma+1})$ holds, then we claim that $(\Pi_{\varsigma})$ is true. Indeed, choosing $m=\varsigma$ in this case, it follows by \eqref{0-h-1}, \eqref{Qij}, \eqref{www-q}, Lemma \ref{p2}, Lemma \ref{Ni-1} and Lemma \ref{Ni-2} that
\begin{equation*}
\begin{split}
\sum_{i=1}^{\varsigma-1}Q_{i\kappa}&\lesssim\widetilde{\alpha}_{n,\mu}^{-1}\sum_{i=\varsigma+1}^{\kappa}\big|\int\Big(|x|^{-\mu}\ast W_i^{p}\Big)
W_i^{p-1}\lambda_{m}\partial_{\lambda_m} W_m\big|+\big\|\hat{f}\big\|_{(D^{1,2}(\mathbb{R}^n))^{-1}}+o(\mathscr{Q})\\&
\lesssim\big\|\hat{f}\big\|_{(D^{1,2}(\mathbb{R}^n))^{-1}}+o(\mathscr{Q}).
\end{split}
\end{equation*}
Then the claim follows by induction. Hence, we have the bound
$$\mathscr{Q}\lesssim\big\|\hat{f}\big\|_{(\mathcal{D}^{1,2}(\mathbb{R}^n))^{-1}}+o(\mathscr{Q}).$$
$\bullet$ If $n=4$ and $\mu\in(3,4)$, and $2\leq\varsigma\leq\kappa$, we set the induction hypothesis
$$(\widetilde{\Pi}_{\varsigma}):\hspace{4mm}\sum_{j=\varsigma}^{\kappa}\sum_{i=1}^{j-1}Q_{ij}\lesssim
\big\|\hat{f}\big\|_{(D^{1,2}(\mathbb{R}^n))^{-1}}+o(\mathscr{Q}^{2-\frac{\mu}{2}}).$$
Analogously, it follows by induction that
$$\mathscr{Q}\lesssim\big\|\hat{f}\big\|_{(\mathcal{D}^{1,2}(\mathbb{R}^n))^{-1}}+o(\mathscr{Q}^{2-\frac{\mu}{2}}).$$
As $\delta$ is small, which concludes the proof.
\end{proof}
We are now in position to conclude the proof of Theorem \ref{Figalli}.

\emph{\textbf{Proof of Theorem \ref{Figalli}.}} We first note that, by Lemma \ref{p00} and Lemma \ref{rr-1-00}, we immediately get the bound
\begin{equation}\label{popo}
\big\|\nabla\rho_0\big\|_{L^2}\lesssim\tau_{n,\mu}(\mathscr{Q}),
\end{equation}
and when $n=4$ and $\mu\in[2,4)$, or $n=5$ and $\mu\in[3,4)$,
$$\big\|\nabla\rho_1\big\|_{L^2}\lesssim\big\|\hat{f}\big\|_{(D^{1,2}(\mathbb{R}^n))^{-1}}
+\mathscr{Q}^{3-\frac{\mu}{n-2}}.$$
Combining these bounds with \eqref{Qf} and using that $\tau_{n,\mu}$ is non-decreasing near $0$, we deduce that
$$\big\|\nabla\rho\big\|_{L^2}\leq\big\|\nabla\rho_0\big\|_{L^2}+\big\|\nabla\rho_1\big\|_{L^2}\lesssim\tau_{n,\mu}\big(\big\|\hat{f}\big\|_{(D^{1,2}(\mathbb{R}^n))^{-1}}\big).$$
When $n=5$ and $\mu\in[1,3)$, or $n=6$ and $\mu\in(0,4)$, or $n=7$ and $\mu\in(\frac{7}{3},4)$, we have the bound
\begin{equation*}	
\big\|\nabla\rho_1\big\|_{L^2}\lesssim\big\|\hat{f}\big\|_{(D^{1,2}(\mathbb{R}^n))^{-1}}
+\mathscr{Q}^{\min\{2,\frac{n-\mu}{n-2}+1\}}.
\end{equation*}
Thus, combining \eqref{Qf} and \eqref{popo}, we conclude that
\begin{equation*}	
\big\|\nabla\rho_1\big\|_{L^2}\lesssim\big\|\hat{f}\big\|_{(D^{1,2}(\mathbb{R}^n))^{-1}}
+\mathscr{Q}^{\min\{2,\frac{n-\mu}{n-2}+1\}}.
\end{equation*}
which concludes the proof of \eqref{tnu}.

Finally, \eqref{moreover} follows by \eqref{Qf} and
\begin{equation*}
\int\Big(|x|^{-\mu}\ast W_i^{p}\Big)
W_i^{p-1}W_j=\int W_i^{2^{\ast}-1}W_j\approx Q_{ij}\leq\mathscr{Q}.
\end{equation*}
Theorem \ref{Figalli} is proven. \qed

\emph{\textbf{Proof of  Corollary~\ref{Figalli2}.}}
We can now prove Corollary~\ref{Figalli2} by using Theorems \ref{Figalli} and a nonlocal version of profile decompositions to \eqref{ele-1.1} in Theorem A.
Then there exists $\epsilon>0$ such that
$$\big\|\Delta u+\left(|x|^{-\mu}\ast u^{p}\right)u^{p-1}\big\|_{(D^{1,2}(\mathbb{R}^n))^{-1}}\leq\epsilon.$$
Therefore, for any $\delta>0$, we have
$$
\big\|\nabla u-\sum_{i=1}^{\kappa}\nabla W_i\big\|_{L^2}\leq\delta,
$$
where $(W_i)_{1\leq i\leq\kappa}$ is a $\delta$-interacting family of Talenti bubbles, which concludes the proof.

\section{Estimates of the weak interaction bubbles and weights}
This section is devoted to derive a rough $C^0$ estimate of solution to problem \eqref{0} by exploiting Green's function representation of the solution in Lemma \ref{cll-1}, which will be concluded by some technical Lemmas \ref{B3}-\ref{cll-0} and Lemmas \ref{cll-1-0}-\ref{cll-3}.
\begin{lem}\label{B3}For any constant $2<\varsigma_0<n$, there is a constant $C>0$, such that
\begin{equation*}
\int_{\mathbb{R}^{N}}\frac{1}{|\hat{y}-x|^{n-2}}\frac{1}{(1+|x|^2)^{\varsigma_0/2}}dx
\leq C\left\lbrace
\begin{aligned}
&\frac{1}{(1+|\hat{y}|^2)^{\varsigma_0/2-1}},\hspace{5mm}\hspace{2mm}2<\varsigma_0<n,\\&
\frac{1+\log\sqrt{1+|\hat{y}|^2}}{(1+|\hat{y}|^2)^{(n-2)/2}},\hspace{2mm}\hspace{2mm}\varsigma_0=n,
\\&\frac{1}{(1+|\hat{y}|^2)^{(n-2)/2}},\hspace{4mm}\hspace{2mm}\varsigma_0>n.
\end{aligned}
\right.
\end{equation*}
\end{lem}
\begin{proof}
The conclusion follows the same proof as in Appendix B of in \cite{FG20} except minor modifications.
\end{proof}
For this, we need the following useful approximation result.
\begin{lem}\label{cll-0}
 Under the assumptions of Lemma \ref{estimate2}. If $n=5$ and $\mu\in[1,3)$, or $n=6$ and $\mu\in(0,4)$, or $n=7$ and $\mu\in(\frac{7}{3},4)$, then
 \begin{equation}\label{zz-0}
(|\hat{y}|^{-(n-2)}\ast T_1)(\hat{y})\approx S_{1}(\hat{y}).
 \end{equation}
If $n=4$ and $\mu\in[2,4)$, or $n=5$ and $\mu\in[3,4)$, then
  \begin{equation}\label{zz-0-0}
(|\hat{y}|^{-(n-2)}\ast T_2)(\hat{y})dx\approx S_{2}(\hat{y}).
 \end{equation}
\end{lem}
\begin{proof}

We split the proof into two cases, depending choice of $n$ and $\mu$, for $t_{i,2}$ and $\hat{t}_{i,2}$.\\
Case $1$. Assume first $n=5$ and $\mu\in[1,3)$, or $n=6$ and $\mu\in(0,4)$, or $n=7$ and $\mu\in(\frac{7}{3},4)$.\\
We now choose $z_i=\lambda_{i}(x-\xi_i)$, $\hat{z}_i=\lambda_{i}(\hat{y}-\xi_i)$.
Then, by a direct computation and using Lemma \ref{B3}, we deduce that
\begin{equation}\label{zi2}
\begin{split}
\int\frac{1}{|\hat{y}-x|^{n-2}}t_{i,2}(x)dx&=\frac{\lambda_{i}^{(n-2)/2}}{\mathscr{R}^{(n-\mu+2)/2}}\int_{|z_i|\geq\mathscr{R}}\frac{1}{|\hat{z}_i-z_i|^{n-2}}
\frac{1}{(1+|z_{i}|^2)^{(n+\mu+2)/4}}dz_i.
\end{split}
\end{equation}
\begin{itemize}
\item[$\bullet$] If $|\hat{z}_i|\leq\frac{2}{3}\mathscr{R}$, then $|\hat{z}_i-z_i|\approx|z_i|$ on $\{x: |z_{i}|\geq\mathscr{R}\}$.
Then, by a simple computation
\begin{equation*}
\begin{split}
\int_{|z_i|\geq\mathscr{R}}\frac{1}{|\hat{z}_i-z_i|^{n-2}}
\frac{1}{(1+|z_{i}|^2)^{(n+\mu+2)/4}}dz_i\approx\mathscr{R}^{2-\frac{n+\mu+2}{2}}
\chi_{\{|\hat{z}_i|\leq\frac{2\mathcal{R}}{3}\}}.
\end{split}
\end{equation*}
\item[$\bullet$] If $|\hat{z}_i|\geq\frac{2}{3}\mathscr{R}$. Then using Lemma \ref{B3} and we deduce that
  \begin{equation*}
\begin{split}
\int_{|z_i|\geq\mathscr{R}}\frac{1}{|\hat{z}_i-z_i|^{n-2}}
\frac{1}{(1+|z_{i}|^2)^{(n+\mu+2)/4}}dz_i\approx\frac{1}{(1+|\hat{z}_{i}|^2)^{(n+\mu-2)/4}}\chi_{\{|\hat{z}_i|\geq\frac{2\mathscr{R}}{3}\}}
\end{split}
\end{equation*}
then as a consequence
\begin{equation}\label{zi2-1}
\begin{split}
\int\frac{1}{|\hat{y}-x|^{n-2}}t_{i,2}(x)dx&\approx\frac{\lambda_{i}^{(n-2)/2}}{\mathscr{R}^{n}}
\chi_{\{|\hat{z}_i|\leq\frac{2\mathscr{R}}{3}\}}+\frac{\lambda_{i}^{(n-2)/2}}{\mathscr{R}^{(n-\mu+2)/2}}\frac{1}{(1+|\hat{z}_{i}|^2)^{(n+\mu-2)/4}}\chi_{\{|\hat{z}_i|\geq\frac{2\mathscr{R}}{3}\}}\\&
=s_{i,1}(\hat{y})+s_{i,2}(\hat{y}).
\end{split}
\end{equation}
An analogous argument tell us that
\begin{equation}\label{zi2-2}
\begin{split}
\int\frac{1}{|\hat{y}-x|^{n-2}}t_{i,1}(x)dx\lesssim s_{i,1}(\hat{y})+s_{i,2}(\hat{y}).
\end{split}
\end{equation}
\end{itemize}
Case $2$. Assume now $n=4$ and $\mu\in[2,4)$, or $n=5$ and $\mu\in[3,4)$.\\
Then. repeating the argument of case 1, we get
\begin{itemize}
\item[$\bullet$] If $|\hat{z}_i|\leq\frac{2}{3}\mathscr{R}$, then $|\hat{z}_i-z_i|\approx|z_i|$ on $\{x:|z_{i}|\geq\mathscr{R}\}$.
Then, we get
\begin{equation*}
\begin{split}
\int_{|z_i|\geq\mathscr{R}}\frac{1}{|\hat{z}_i-z_i|^{n-2}}
\frac{1}{(1+|z_{i}|^2)^{(n-\epsilon_0)/2}}dz_i\approx\frac{1}{\mathscr{R}^{n-2-\epsilon_0}}
\chi_{\{|\hat{z}_i|\leq\frac{2\mathscr{R}}{3}\}}.
\end{split}
\end{equation*}
\item[$\bullet$] If $|\hat{z}_i|\geq\frac{2}{3}\mathscr{R}$. We deduce that by Lemma \ref{B3}
  \begin{equation*}
\begin{split}
\int_{|z_i|\geq\mathscr{R}}\frac{1}{|\hat{z}_i-z_i|^{n-2}}
\frac{1}{(1+|z_{i}|^2)^{(n-\epsilon_0)/2}}dz_i\approx\frac{1}{(1+|\hat{z}_{i}|^2)^{(n-2-\epsilon_0)/2}}\chi_{\{|\hat{z}_i|\geq\frac{2\mathscr{R}}{3}\}}
\end{split}
\end{equation*}
and therefore
\begin{equation}\label{zi2-3}
\begin{split}
\int\frac{1}{|\hat{y}-x|^{n-2}}\hat{t}_{i,2}(x)dx&\approx\frac{\lambda_{i}^{(n-2)/2}}{\mathscr{R}^{2n-2-\mu}}
\chi_{\{|\hat{z}_i|\leq\frac{2\mathscr{R}}{3}\}}+\frac{\lambda_{i}^{(n-2)/2}}{\mathscr{R}^{n-\mu+\epsilon_0}}\frac{1}{(1+|\hat{z}_{i}|^2)^{(n-2-\epsilon_0)/2}}\chi_{\{|\hat{z}_i|\geq\frac{2\mathscr{R}}{3}\}}\\&
= \hat{s}_{i,1}(\hat{y})+\hat{s}_{i,2}(\hat{y}).
\end{split}
\end{equation}
Similarly
\begin{equation}\label{zi2-4}
\begin{split}
\int\frac{1}{|\hat{y}-x|^{n-2}}\hat{t}_{i,1}(x)dx\lesssim \hat{s}_{i,1}(\hat{y})+\hat{s}_{i,2}(\hat{y}).
\end{split}
\end{equation}
\end{itemize}
Combining \eqref{zi2-1}-\eqref{zi2-2} and \eqref{zi2-3}-\eqref{zi2-4}
yielding the result.
\end{proof}
From now on, we set $\xi_{ij}=\lambda_{i}(\xi_j-\xi_i)$. Then
the assumption \eqref{lamta} tell us that $\mathscr{R}_{ij}\sqrt{\lambda_i/\lambda_j}\approx \tau(\xi_{ij})$. Furthermore, we define $\widetilde{H}_{j}(x):=\frac{\lambda_j^{2}}{\tau(z_{j})^{4}}$,
and we will deduce the following a series of technical results that will be useful also later.
\begin{lem}\label{cll-1-0-1}
For each $i\neq j$, assume that $\lambda_i\leq\lambda_j$. When $n\geq6-\mu$ and  $0<\mu<\frac{n+\mu-2}{2}$. Then we have that
\begin{equation}\label{zzz-0-1}
\widetilde{H}_{j}(x)s_{i,1}\lesssim\mathscr{R}_{ij}^{-2}t_{j,1}+\mathscr{R}^{-2}(t_{i,1}+t_{j,2}),
\end{equation}
\begin{equation}\label{zzz-1-1}
\widetilde{H}_{j}(x)s_{i,2}\lesssim\mathscr{R}_{ij}^{-2}t_{j,1}+\mathscr{R}^{-2}(v_{i,2}+t_{j,2}),
\end{equation}
\begin{equation}\label{zzz-2-1}
\widetilde{H}_{i}(x)s_{j,1}\lesssim\mathscr{R}_{ij}^{-2}t_{j,1},
\end{equation}
\begin{equation}\label{zzz-3-1}
\widetilde{H}_{i}(x)s_{j,2}\lesssim\tau( \xi_{ij})^{-2}t_{i,1}+\tau(\xi_{ij})^{-2}\big(t_{i,2}+t_{j,2}\big),
\end{equation}
and then that
\begin{equation}\label{zzz-4-1}
\widetilde{H}_{i}(x)s_{j,2}\lesssim\big(\big(\frac{\lambda_{i}}{\lambda_j}\big)^2+\varepsilon^2\big)t_{j,2}
\hspace{4mm}\mbox{on the set}\hspace{2mm}\{x:|z_i-\xi_{ij}|\leq\varepsilon\},
\end{equation}
\begin{equation}\label{zzz-5-1}
s_{j,2}\lesssim
\left\lbrace
\begin{aligned}
&\tau(\xi_{ij})^{n+\mu-4}\varepsilon^{-\frac{n+\mu-2}{2}}
\big(s_{i,1}+s_{i,2}\big),\hspace{2mm}n>6-\mu\hspace{2mm}\mbox{and}\hspace{2mm}\mbox{on the set}\hspace{2mm}\{x:|z_i-\xi_{ij}|>\varepsilon\},\\&
\tau(\xi_{ij})^{\frac{n+\mu-2}{2}}\varepsilon^{-\frac{n+\mu-2}{2}}
\big(s_{i,1}+s_{i,2}\big),\hspace{2mm}n=6-\mu\hspace{2mm}\mbox{and}\hspace{2mm}\mbox{on the set}\hspace{2mm}\{x:|z_i-\xi_{ij}|>\varepsilon\},
\end{aligned}
\right.
\end{equation}
for all $0<\varepsilon<1$. Furthermore, for $\widetilde{A}>1$,
\begin{equation}\label{zzz-6-1}
\widetilde{H}_{i}(x)s_{j,2}\lesssim\big(|\xi_{ij}|+\widetilde{A}\big)^{-2}t_{j,2}\hspace{4mm}\mbox{on the set}\hspace{2mm}\{x:|z_i|\geq2|\xi_{ij}|+\widetilde{A}\},
\end{equation}
and
\begin{equation}\label{zzz-7-1}
\widetilde{H}_{i}(x)s_{j,2}\lesssim\tau(\xi_{ij})^{\varrho_1}\big(\varepsilon^{-\frac{4}{n+\mu-2}}t_{i,1}+\varepsilon t_{j,2}\big)\hspace{2mm}\mbox{with}\hspace{2mm}\varrho_1=\frac{2(n+\mu-6)}{n+\mu+2},\hspace{2mm}\mbox{on the set}\hspace{2mm}\{x:|z_i|\leq \mathscr{R}\}
\end{equation}
for all $0<\varepsilon<1$.
\end{lem}
The proof will make us of the following lemma, which is proven in
\cite[Lemma 2.3]{DSW21}.
\begin{lem}\label{B3-0} For each $k\neq j$, let
	$$	f_{k,j}(x)=\frac{1}{(1+|z_{k}|^2)^{\alpha_1/2}}\frac{1}{(1+|z_{j}|^2)^{\beta_1/2}}=
\big(\frac{1}{\tau(z_k)}\big)^{\alpha_1}\big(\frac{1}{\tau( z_j)}\big)^{\beta_1},
	$$
	where $\alpha_1\geq 0$ and $\beta_1\geq0$ are two constants. Assume that $\lambda_k\leq\lambda_j$, $1\leq\mathscr{R}_{kj}/2$.
Then, for any constants $\alpha_1$ and $\beta_1\geq\gamma_1\geq0$, there is a constant $C>0$, such that
\begin{equation}\label{R-4-1-00}
f_{k,j}(x)\lesssim
\left\lbrace
\begin{aligned}
&\frac{1}{\mathscr{R}_{kj}^{\alpha_{1}}}(\frac{\lambda_{j}}{\lambda_k })^{\alpha_{1}/2}\frac{1}{\tau( z_j)^{\beta_{1}}},\hspace{6mm}\hspace{6mm}\hspace{6mm}\hspace{8mm}\hspace{6mm}\hspace{2mm}\hspace{2mm}|z_j|\leq\sqrt{\frac{\lambda_j}{\lambda_k}}\frac{\mathscr{R}_{kj}}{2},\\
&\frac{1}{\mathscr{R}_{kj}^{\gamma_{1}}}(\frac{\lambda_{j}}{\lambda_k })^{\beta_{1}-\gamma_{1}/2}\frac{1}{\tau( z_i)^{\alpha_{1}+\beta_{1}-\gamma_{1}}},\hspace{6mm}\hspace{6mm}\hspace{6mm}\hspace{2mm}|z_j|\geq\sqrt{\frac{\lambda_j}{\lambda_k}}\frac{\mathscr{R}_{kj}}{2}.
\end{aligned}
\right.
\end{equation}
\end{lem}
Our main goal is to show that the comparative relationship between
$\widetilde{H}_{j}(x)s_{i,1}$, $\widetilde{H}_{j}(x)s_{i,2}$, and $\widetilde{H}_{j}(x)\hat{s}_{i,1}$, $\widetilde{H}_{j}(x)\hat{s}_{i,2}$ and some Hartree type cross terms (as described in Lemma \ref{cll-2} and Lemma \ref{cll-3} below) and $t_{i,1}$, $t_{i,2}$, $\hat{t}_{i,1}$, $\hat{t}_{i,2}$ separately. We proceed similarly to \cite{DSW21}. However, due to the different definitions of $s_{i,2}$, $t_{i,2}$, $\hat{s}_{i,1}$, $\hat{t}_{i,1}$, $\hat{s}_{i,2}$ and $\hat{t}_{i,2}$ with the parameter $\mu$, there are still significant differences and difficulties in obtaining the desired estimate compared to \cite{DSW21}. Moreover, in the cases involving Hartree type cross terms, we must borrow some estimates of the convolution terms and Young's inequality to find the suitable weight function (see Lemma \ref{cll-2} and Lemma \ref{cll-3}). In particular, when $\mu$ is sufficiently close $0$ or $4$, it is the most difficult cases in our proof, which means that many novel skills are required.

We next give the following proof.

\emph{\textbf{Proof of Lemma \ref{cll-1-0-1}}}.
First note that
$$\widetilde{H}_{j}(x)s_{i,1}=\frac{1}{\mathscr{R}^{n-2}}\frac{\lambda_j^{2}}{(1+|z_{j}|^2)^{2}} \frac{\lambda_i^{(n-2)/2}}{(1+|z_{i}|^2)}=\frac{1}{\mathscr{R}^{n-2}}\frac{\lambda_j^{2}}{\tau( z_{j})^4}\frac{\lambda_i^{(n-2)/2}}{\tau(z_i)^{(n-2)/2}}.$$
then we obtain
$$
\widetilde{H}_{j}(x)s_{i,1}\lesssim \mathscr{R}_{ij}^{-2}t_{j,1}+\mathscr{R}^{-2}(t_{i,1}+t_{j,2}),\hspace{2mm}\widetilde{\mathscr{S}}_{j}(x)s_{i,1}\lesssim \mathscr{R}_{ij}^{-2}t_{j,1},
$$
that is \eqref{zzz-0-1} and \eqref{zzz-2-1}.

$\bullet$
To prove \eqref{zzz-1-1},  we distinguish two cases, depending on whether $|z_j|\leq\mathscr{R}$ or not.

Case 1. On the set $\{x: |z_j|\leq\mathscr{R}\}$.

Observe now that, since Lemma \ref{B3-0}, $n>6-\mu$, and $0<\mu\leq4$, we have
\begin{equation*}
\begin{split}
\widetilde{H}_{j}(x)s_{i,2}&\approx\frac{1}{\mathscr{R}^{(n-\mu+2)/2}}\frac{\lambda_j^{2}}{\tau( z_{j})^4}\frac{\lambda_i^{(n-2)/2}}{\tau(z_i)^{(n+\mu-2)/2}}\\&
\lesssim\mathscr{R}^{\frac{n+\mu-6}{2}}\frac{1}{\mathscr{R}_{ij}^{(n+\mu-2)/2}}\Big(\frac{\lambda_i}{\lambda_j}\Big)^{\frac{n-\mu+2}{2}}t_{j,1}\lesssim\frac{1}{\mathscr{R}_{ij}^2}t_{j,1}.
\end{split}
\end{equation*}

Case 2. On the set $\{x: |z_i|\geq\mathscr{R},\hspace{2mm}|z_j|\geq\mathscr{R}\}$.

We can take the exponent $\theta$, $p_1^{\star}$ and $q_1^{\star}$ satisfying
$$\theta=\frac{8}{n+\mu+6},\hspace{2mm}p_1^{\star}=\frac{2(\mu+2-n)}{n+\mu+6},\hspace{2mm}\hspace{2mm}q_1^{\star}=\frac{2(n-\mu+2)}{n+\mu+6}.$$
Then using H\"{o}lder inequality, we obtain
\begin{equation*}
\begin{split}
\widetilde{H}_{j}(x)s_{i,2}&\approx\frac{1}{\mathscr{R}^{(n-\mu+2)/2}}\frac{\lambda_j^{2}}{\tau( z_{j})^4}\frac{\lambda_i^{(n-2)/2}}{\tau( z_i)^{(n+\mu-2)/2}}=\frac{\lambda_i^{p_1^{\star}}}{\lambda_j^{q_1^{\star}}}\Big(\frac{t_{i,2}}{\tau( z_i)^{2}}\Big)^{1-\theta}\Big(\frac{t_{j,2}}{\tau(z_j)^{2}}\Big)^{\theta}\\&
\lesssim\frac{1}{\mathscr{R}^2}t_{i,2}+\frac{1}{\mathscr{R}^2}t_{j,2}.
\end{split}
\end{equation*}

$\bullet$
To prove \eqref{zzz-3-1}. We consider three cases.

Case 1. On the set $\{x: \mathscr{R}\leq|z_j|\leq\sqrt{\lambda_j/\lambda_i}\mathscr{R}_{ij}/2\}$.

By Lemma \ref{B3-0}, we have
$\mathscr{R}_{ij}\sqrt{\lambda_{i}/\lambda_j}\approx\langle\xi_{ij}\rangle$, we have
\begin{equation*}
\begin{split}
\widetilde{H}_{i}(x)s_{j,2}&\approx\frac{1}{\mathscr{R}^{(n-\mu+2)/2}}\frac{\lambda_i^{2}}{\tau( z_{i})^4}\frac{\lambda_j^{(n-2)/2}}{\tau(z_j)^{(n+\mu-2)/2}}
\lesssim\frac{1}{\tau(\xi_{ij})^2}t_{i,2}.
\end{split}
\end{equation*}

Case 2. On the set $\{x: |z_j|\geq\sqrt{\lambda_j/\lambda_i}\mathscr{R}_{ij}/2,\hspace{2mm}|z_j|\leq\mathscr{R}\}$.

Now, noticing that by $n\leq6+\mu$
$$\Big(\frac{\lambda_i}{\lambda_j}\Big)^{\frac{n-\mu+2}{2}}\leq\Big(\frac{\lambda_i}{\lambda_j}\Big)^{\frac{n-4}{2}},$$
so that, we get for $n>6-\mu$,
\begin{equation*}
\begin{split}
\widetilde{H}_{i}(x)s_{j,2}&\approx\frac{1}{\mathscr{R}^{(n+\mu-2)/4}}\frac{\lambda_i^{2}}{\tau( z_{i})^4}\frac{\lambda_j^{(n-2)/2}}{\tau(z_j)^{(n+\mu-2)/2}}
\lesssim\mathscr{R}^{\frac{n+\mu-6}{2}}\frac{1}{\mathscr{R}_{ij}^{(n+\mu-2)/2}}\frac{\lambda_j}{\lambda_i}t_{i,1}\lesssim\frac{1}{\tau( \xi_{ij})^2}t_{i,1}.
\end{split}
\end{equation*}
by choosing $\gamma_1=\frac{n+\mu-2}{2}$ in Lemma \ref{B3-0}. On the other hand, for $n=6-\mu$
\begin{equation*}
\begin{split}
\widetilde{H}_{i}(x)s_{j,2}&\approx\frac{1}{\mathscr{R}}\frac{\lambda_i^{2}}{\tau( z_{i})^4}\frac{\lambda_j^{(4-\mu)/2}}{\tau(z_j)^{2}}
\lesssim\frac{1}{\mathscr{R}_{ij}^{2}}\frac{\lambda_j}{\lambda_i}t_{i,1}\lesssim\frac{1}{\tau( \xi_{ij})^2}t_{i,1}.
\end{split}
\end{equation*}
Case 3. On the set $\{x: |z_j|\geq\sqrt{\lambda_j/\lambda_i}\mathscr{R}_{ij}/2,\hspace{2mm}|z_j|\geq\mathscr{R}\}$.

Note that
$$\Big(\frac{\lambda_j}{\lambda_i}\Big)^{\varpi-1}\leq\frac{\lambda_j}{\lambda_i}\hspace{2mm}\mbox{with}\hspace{2mm}\varpi=\frac{2(n-\mu+2)}{n+\mu+2},$$
then, combining H\"{o}lder inequality gives
\begin{equation*}
\begin{split}
\widetilde{H}_{i}(x)s_{j,2}&\approx\frac{1}{\mathscr{R}^{(n-\mu+2)/2}}\frac{\lambda_i^{2}}{\tau( z_{i})^4}\frac{\lambda_j^{(n-2)/2}}{\tau( z_j)^{(n+\mu-2)/2}}=\Big(\frac{\lambda_j}{\lambda_i}\Big)^{\varpi}\frac{1}{\tau( z_j)^2}\big(t_{i,2}\big)^{p_2^{\star}}\big(t_{j,2}\big)^{q_2^{\star}}\\&
\lesssim\Big(\frac{\lambda_j}{\lambda_i}\Big)^{\varpi}\frac{1}{\tau( z_j)^2}\big(t_{i,2}+t_{j,2}\big)\leq
\frac{1}{\tau(\xi_{ij})^2}\big(t_{i,2}+t_{j,2}\big).
\end{split}
\end{equation*}
by choosing
$$p_2^{\star}=\frac{8}{n+\mu+2},\hspace{2mm}\hspace{2mm}q_2^{\star}=\frac{n+\mu-6}{n+\mu+2}.$$

$\bullet$
To prove \eqref{zzz-4-1}. A direct computation shows that, on the set $\{x:|z_i-\xi_{ij}|\leq\varepsilon\}$,
\begin{equation*}
\begin{split}
\widetilde{H}_{i}(x)s_{j,2}\approx\frac{1}{\mathscr{R}^{(n-\mu+2)/2}}\frac{\lambda_i^{2}}{\tau( z_{i})^4}\frac{\lambda_j^{(n-2)/2}}{\tau( z_j)^{(n+\mu-2)/2}}&=\Big(\frac{\lambda_{i}}{\lambda_j}\Big)^2\frac{\tau( z_j)^2}{\tau( z_i)^4}\frac{1}{\mathscr{R}^{(n-\mu+2)/2}}\frac{\lambda_j^{\frac{n+2}{2}}}{\tau( z_j)^{(n+\mu+2)/2}}\\&
\lesssim\Big(\big(\frac{\lambda_{i}}{\lambda_j}\big)^2+\varepsilon^2\Big)t_{j,2}.
\end{split}
\end{equation*}

$\bullet$
To prove \eqref{zzz-5-1}. On the set $\{x:|z_i-\xi_{ij}|>\varepsilon\}$, we have $\tau(z_j)\geq(\lambda_j\varepsilon\tau(z_i))/(\lambda_i\tau(\xi_{ij}))$, and for $n\geq6-\mu$,
$$
s_{i,2}\lesssim\big(\mathscr{R}_{ij}\big)^{\frac{n+\mu-6}{2}}\big(s_{i,1}+s_{i,2}\big).
$$
Therefore, using $0<\mu<(n+\mu-2)/2$, we have that for $6-\mu<n\leq6+\mu$,
\begin{equation*}
\begin{split}
s_{j,2}\lesssim\Big(\frac{\lambda_{i}}{\lambda_j}
\Big)^\frac{\mu}{2}\Big(\frac{\tau(\xi_{ij})}{\varepsilon}\Big)^{\frac{n+\mu-2}{2}}s_{i,2}
&\lesssim \big(\mathscr{R}_{ij}\big)^{\frac{n+\mu-6}{2}}\Big(\frac{\lambda_{i}}{\lambda_j}\Big)^\frac{\mu}{2}
\Big(\frac{\tau(\xi_{ij})}{\varepsilon}\Big)^{\frac{n+\mu-2}{2}}
\big(s_{i,1}+s_{i,2}\big)
\\&\lesssim \tau(\xi_{ij})^{n+\mu-4}\varepsilon^{-\frac{n+\mu-2}{2}}
\big(s_{i,1}+s_{i,2}\big),
\end{split}
\end{equation*}
as desired by $\mathscr{R}_{ij}\sqrt{\lambda_{i}/\lambda_j}\approx\tau(\xi_{ij})$.
When $n=6-\mu$, we have
$$s_{j,2}\lesssim\tau(\xi_{ij})^{\frac{n+\mu-2}{2}}\varepsilon^{-\frac{n+\mu-2}{2}}
\big(s_{i,1}+s_{i,2}\big).$$

$\bullet$
To prove \eqref{zzz-6-1}. We use $\tau(z_j)^2\leq2(\lambda_j/\lambda_i)^2|z_i|^2$ for $\{x:|z_i|>2|\xi_{ij}|+\widetilde{A}\}$, and therefore
\begin{equation*}
\begin{split}
\widetilde{H}_{i}(x)s_{j,2}\approx\frac{1}{\mathscr{R}^{(n-\mu+2)/2}}\frac{\lambda_i^{2}}{\tau( z_{i})^4}\frac{\lambda_j^{(n-2)/2}}{\tau( z_j)^{(n+\mu-2)/2}}&=\Big(\frac{\lambda_{i}}{\lambda_j}\Big)^2\frac{\tau( z_j)^2}{\tau( z_i)^4}\frac{1}{\mathscr{R}^{(n-\mu+2)/2}}\frac{\lambda_j^{\frac{n+2}{2}}}{\tau( z_j)^{(n+\mu+2)/2}}\\&
\lesssim\Big(\frac{1}{|\xi_{ij}|+\widetilde{A}}\Big)^2t_{j,2}.
\end{split}
\end{equation*}

$\bullet$
To prove \eqref{zzz-7-1}. On the set $\{x:|z_i|\leq \mathscr{R}\}$, there are two cases.

Case 1. For $n>6-\mu$.
Writing
$$ \varrho=\frac{2(n+\mu-6)}{n+\mu+2},\hspace{2mm}p_{3}^{\star}=\frac{4}{n+\mu+2},\hspace{2mm}q_{3}^{\star}=\frac{n+\mu-2}{n+\mu+2},$$
and by $6-\mu<n<6+\mu$ and $\mathscr{R}_{ij}\sqrt{\lambda_{i}/\lambda_j}\approx\tau(\xi_{ij})$, as a consequence, by H\"{o}lder inequality,
\begin{equation*}
\begin{split}
\widetilde{H}_{i}(x)s_{j,2}\approx\frac{1}{\mathscr{R}^{(n-\mu+2)/2}}\frac{\lambda_i^{2}}{\tau( z_{i})^4}\frac{\lambda_j^{(n-2)/2}}{\tau( z_j)^{(n+\mu-2)/2}}&\lesssim\Big(\sqrt{\frac{\lambda_{i}}{\lambda_j}}\Big)^{\varrho} \mathscr{R}^{\varrho}\big(t_{i,1})^{p_{3}^{\star}}\big(t_{j,2})^{q_{3}^{\star}}\\&
\lesssim\tau(\xi_{ij})^{\varrho}\Big(\varepsilon^{-\frac{4}{n+\mu-2}}t_{i,1}+\varepsilon t_{j,2}\Big).
\end{split}
\end{equation*}

Case 2. For $n=6-\mu$, we have $\varrho=0$.
\begin{equation*}
\widetilde{H}_{i}(x)s_{j,2}\lesssim\varepsilon^{-1}t_{i,1}+\varepsilon t_{j,2},
\end{equation*}
and concluding the proof. \qed

\begin{lem}\label{cll-1-0}
For each $i\neq j$, assume that $\lambda_i\leq\lambda_j$. If $n\geq6-\mu$ and  $\frac{n+\mu-2}{2}\leq\mu<4$. Then we have that
\begin{equation}\label{zzz-0-1-1}
\widetilde{H}_{j}(x)\hat{s}_{i,1}\lesssim\mathscr{R}_{ij}^{-2}\hat{t}_{j,1}+\mathscr{R}^{-2}(\hat{t}_{i,1}+\hat{t}_{j,2}),
\end{equation}
\begin{equation}\label{zzz-1-1-1}
\widetilde{H}_{j}(x)\hat{s}_{i,2}\lesssim\mathscr{R}_{ij}^{-2}\hat{t}_{j,1}+\mathscr{R}^{-2}(\hat{t}_{i,2}+\hat{t}_{j,2}),
\end{equation}
\begin{equation}\label{zzz-2-1-1}
\widetilde{H}_{i}(x)\hat{s}_{j,1}\lesssim\mathscr{R}_{ij}^{-2}\hat{t}_{j,1},
\end{equation}
\begin{equation}\label{zzz-3-1-1}
\widetilde{H}_{i}(x)\hat{s}_{j,2}\lesssim\tau( \xi_{ij})^{-\epsilon_0}\hat{t}_{i,1}+\tau(\xi_{ij})^{-2}\hat{t}_{i,2}+\mathscr{R}^{-2}\big(\hat{t}_{i,2}+\hat{t}_{j,2}\big)\hspace{2mm}\mbox{with}\hspace{2mm}0<\epsilon_0<1,
\end{equation}
and then that
\begin{equation}\label{zzz-4-1-1}
\widetilde{H}_{i}(x)\hat{s}_{j,2}\lesssim\big(\big(\frac{\lambda_{i}}{\lambda_j}\big)^2+\varepsilon^2\big)\hat{t}_{j,2}
\hspace{4mm}\mbox{on the set}\hspace{2mm}\{x:|z_i-\xi_{ij}|\leq\varepsilon\},
\end{equation}
\begin{equation}\label{zzz-5-1-1}
\hat{s}_{j,2}\lesssim
\left\lbrace
\begin{aligned}
&\tau(\xi_{ij})^{2-\epsilon_0}\varepsilon^{\epsilon_0-2}
\big(\hat{s}_{i,1}+\hat{s}_{i,2}\big),\hspace{4.5mm}n=4\hspace{2mm}\mbox{and}\hspace{2mm}\mbox{on the set}\hspace{2mm}\{x:|z_i-\xi_{ij}|>\varepsilon\},\\&
\tau(\xi_{ij})^{4-2\epsilon_0}\varepsilon^{\varepsilon_0-2}
\big(\hat{s}_{i,1}+\hat{s}_{i,2}\big),\hspace{3.2mm}n=5\hspace{2mm}\mbox{and}\hspace{2mm}\mbox{on the set}\hspace{2mm}\{x:|z_i-\xi_{ij}|>\varepsilon\},
\end{aligned}
\right.
\end{equation}
for all $0<\varepsilon<1$ and where $0<\epsilon_0<1$ is a parameter. Furthermore, for $\bar{A}>1$,
\begin{equation}\label{zzz-6-1-1}
\widetilde{H}_{i}(x)\hat{s}_{j,2}\lesssim\big(|\xi_{ij}|+\bar{A}\big)^{-2}\hat{t}_{j,2}\hspace{4mm}\hspace{2mm}\mbox{and}\hspace{2mm}\mbox{on the set}\hspace{2mm}\{x:|z_i|\geq2|\xi_{ij}|+\bar{A}\},
\end{equation}
and
\begin{equation}\label{zzz-7-1-1}
\widetilde{H}_{i}(x)\hat{s}_{j,2}\lesssim
\left\lbrace
\begin{aligned}
&\varepsilon^{-\frac{2-\epsilon_0}{2}}\hat{t}_{i,1}+\varepsilon \hat{t}_{j,2},\hspace{6mm}\hspace{6mm}\hspace{8mm}\hspace{6mm}n=4\hspace{2mm}\mbox{and}\hspace{2mm}\mbox{on the set}\hspace{2mm}\{x:|z_i|\leq \mathscr{R}\},\\&
\tau(\xi_{ij})^{\frac{2(1-\epsilon_0)}{5-\epsilon_0}}\Big(\varepsilon^{-\frac{3-\epsilon_0}{2}}\hat{t}_{i,1}+\varepsilon \hat{t}_{j,2}\Big),\hspace{2mm}n=5\hspace{2mm}\mbox{and}\hspace{2mm}\mbox{on the set}\hspace{2mm}\{x:|z_i|\leq \mathscr{R}\},
\end{aligned}
\right.
\end{equation}
for all $0<\varepsilon<1$ and $0<\epsilon_0<\frac{(n-2)p-n}{p}$ is a parameter.
\end{lem}
\begin{proof}
We conclude the proofs by some tedious argument similar to Lemma \ref{cll-1-0} and so we skip it.

\end{proof}

We denote the integral quantities in what follows:
\begin{equation}\label{hhstar-1}
\widehat{H}_{ji,1}(x)=:\frac{\lambda_i^{\frac{2n-\mu}{2}}\mathscr{R}^{2-n}}{\tau(z_{i})^{2p}},
\hspace{2mm}\widehat{H}_{ji,2}(x)=:\frac{\lambda_i^{\frac{2n-\mu}{2}}\mathscr{R}^{-\frac{n-\mu+2}{2}}}{\tau(z_{i})^{(n+\mu-2)p/2}},
\hspace{2mm}\widehat{K}_{j}(x):=		\frac{\lambda_j^{\frac{n-\mu+2}{2}}}{\tau(z_{j})^{n-\mu+2}}.
\end{equation}
\begin{lem}\label{cll-2}
For each $i\neq j$, let $\lambda_i\leq\lambda_j$.  If $n\geq6-\mu$ and  $0<\mu<\frac{n+\mu-2}{2}$, we have that
\begin{equation}\label{zzzz-0}
\Big(|x|^{-\mu}\ast\widehat{H}_{ji,1}(y)\Big)\widehat{K}_j(x)\lesssim
\left\lbrace
\begin{aligned}
&\mathscr{R}_{ij}^{-\min\{\mu,(5+\mu)/3\}}t_{j,1}+\mathscr{R}^{-\mu}\big(t_{i,1}+t_{j,2}\big),\hspace{1mm}\hspace{2mm}n=5\hspace{2mm}\mbox{and}\hspace{2mm}1\leq\mu<3,
\\&\mathscr{R}_{ij}^{-(2p-n+\mu-\theta_1)}t_{j,1}+\mathscr{R}^{-\mu}\big(t_{i,1}+t_{j,2}\big),\hspace{4.5mm}\hspace{2mm}n=6\hspace{2mm}\mbox{and}\hspace{2mm}0<\mu<4,
\\&
\mathscr{R}_{ij}^{-(2p-n+\mu-\theta_2)}t_{j,1}+\mathscr{R}^{-\mu}\big(t_{i,1}+t_{j,2}\big),\hspace{4.5mm}\hspace{2mm}n=7\hspace{2mm}\mbox{and}\hspace{2mm}\frac{7}{3}<\mu\leq4.
\end{aligned}
\right.
\end{equation}
\begin{equation}\label{zzzz-1}
\Big(|x|^{-\mu}\ast\widehat{H}_{ji,2}(y)\Big)\widehat{K}_j(x)\lesssim
\mathscr{R}_{ij}^{-(6+\mu-n)/2}t_{j,1}+\mathscr{R}^{-\frac{n+2-\mu}{2}}\big(t_{i,2}+t_{j,2}\big),
\end{equation}
\begin{equation}\label{zzzz-2}
\Big(|x|^{-\mu}\ast\widehat{H}_{ij,1}(y)\Big)\widehat{K}_i(x)\lesssim
\left\lbrace
\begin{aligned}
&
\mathscr{R}_{ij}^{-(n-2-\mu+\min\{\mu,(5+\mu)/3\})}t_{j,1},\hspace{5mm}\hspace{2mm}n=5\hspace{2mm}\mbox{and}\hspace{2mm}1\leq\mu<3,
\\&
\mathscr{R}_{ij}^{-(2p-2-\theta_1)}t_{j,1},\hspace{11mm}\hspace{7.5mm}\hspace{8mm}\hspace{2mm}n=6\hspace{2mm}\mbox{and}\hspace{2mm}0<\mu<4,
\\&
\mathscr{R}_{ij}^{-(2p-2-\theta_2)}t_{j,1},\hspace{11mm}\hspace{7.5mm}\hspace{8mm}\hspace{2mm}n=7\hspace{2mm}\mbox{and}\hspace{2mm}\frac{7}{3}<\mu\leq4.
\end{aligned}
\right.
\end{equation}
\begin{equation}\label{zzzz-3}
\Big(|x|^{-\mu}\ast\widehat{H}_{ij,2}(y)\Big)\widehat{K}_i(x)\lesssim
\tau(\xi_{ij})^{-(n-\mu+2)/2}\big(t_{i,1}+t_{j,2}\big)+
\tau(\xi_{ij})^{-\varpi^{\star}_1}\big(t_{i,2}+t_{j,2}\big),
\end{equation}
where
$$\varpi^{\star}_1=\frac{(n-\mu+2)(\bar{\eta}+\mu)}{2(n+2-\bar{\eta})}\hspace{2mm}\mbox{with}\hspace{2mm}0<\bar{\eta}<\mu,$$
and that
\begin{equation}\label{zzzz-4}
\Big(|x|^{-\mu}\ast\widehat{H}_{ij,2}(y)\Big)\widehat{K}_i(x)\lesssim \Big((\frac{\lambda_{i}}{\lambda_j})^2+\varepsilon^2\Big)^{\frac{n-\mu+2}{4}}t_{j,2},
\hspace{2mm}\mbox{on the set}\hspace{2mm}\{x:|z_i-\xi_{ij}|\leq\varepsilon\},
\end{equation}
for all $0<\varepsilon<1$. Moreover, we have that for\hspace{2mm} $\widetilde{A}>1$
\begin{equation}\label{zzzz-5}
\Big(|x|^{-\mu}\ast\widehat{H}_{ij,2}(y)\Big)\widehat{K}_i(x)\lesssim\Big(|\xi_{ij}|+\widetilde{A}\Big)^{-(n-\mu+2)/2}t_{j,2}
,\hspace{2mm}\mbox{on the set}\hspace{2mm}\{x:|z_i|\geq2|\xi_{ij}|+\widetilde{A}\},
\end{equation}
and that
\begin{equation}\label{zzzz-6}
\Big(|x|^{-\mu}\ast\widehat{H}_{ij,2}(y)\Big)\widehat{K}_i(x)\lesssim
\tau(\xi_{ij})^{\zeta_1^{\star}}\Big(\varepsilon^{-\frac{2\mu}{n-\mu+2}}t_{i,1}+\varepsilon t_{j,2}\Big)\hspace{2mm}\mbox{with}\hspace{2mm}\zeta_1^{\star}=\frac{(n-\mu+2)\mu}{n+\mu+2}
\end{equation}
for $x\in\{x:|z_i|\leq \mathscr{R}\}$ and for all $0<\varepsilon<1$.
\end{lem}
\begin{proof}
Recalling that $\widehat{H}_{ji,1}$, $\widehat{H}_{ji,2}$ and $\widehat{K}_j$.
Then, due to Lemmas \ref{B4}-\ref{B4-1}, we obtain
\begin{equation}\label{fai-1}
\Big(|x|^{-\mu}\ast\widehat{H}_{ji,1}(y)\Big)\widehat{K}_j(x)\lesssim
\frac{1}{\mathscr{R}^{n-2}}\widehat{F}_{ji,1}(x)
\end{equation}
with
\begin{equation*}
\widehat{F}_{ji,1}(x)=
\left\lbrace
\begin{aligned}
& \frac{\lambda_{i}^{\mu/2}}{\tau( z_i)^{\min\{\mu,(5+\mu)/3\}}}\frac{\lambda_j^{(n-\mu+2)/2}}{\tau( z_{j})^{n-\mu+2}},\hspace{4mm}\hspace{1mm}\hspace{2mm}\hspace{2mm}n=5\hspace{2mm}\mbox{and}\hspace{2mm}1\leq\mu<3,\\
& \frac{\lambda_{i}^{\mu/2}}{\tau( z_i)^{2p-n+\mu-\theta_1}}\frac{\lambda_j^{(n-\mu+2)/2}}{\tau( z_{j})^{n-\mu+2}},\hspace{3mm}\hspace{4mm}\hspace{3mm}\hspace{2mm}\hspace{2mm}n=6\hspace{2mm}\mbox{and}\hspace{2mm}0<\mu<4,\\
& \frac{\lambda_{i}^{\mu/2}}{\tau( z_i)^{2p-n+\mu-\theta_2}}\frac{\lambda_j^{(n-\mu+2)/2}}{\tau( z_{j})^{n-\mu+2}},\hspace{3mm}\hspace{4mm}\hspace{3mm}\hspace{2mm}\hspace{2mm}n=7\hspace{2mm}\mbox{and}\hspace{2mm}\frac{7}{3}<\mu\leq4,
	\end{aligned}
\right.
\end{equation*}
where $\theta_1,\theta_2$ are defined in \eqref{ceta} and
\begin{equation}\label{fai-2}
\Big(|x|^{-\mu}\ast\widehat{H}_{ji,2}\Big)\widehat{K}_j(x)\lesssim
\frac{1}{\mathscr{R}^{(n-\mu+2)/2}}\widehat{\mathscr{F}}_{ji,2}(x)
\hspace{2mm}\mbox{with}\hspace{2mm}
\widehat{\mathscr{F}}_{ji,2}(x)=\frac{\lambda_{i}^{\mu/2}}{\tau( z_i)^{\mu}}\frac{\lambda_j^{(n-\mu+2)/2}}{\tau(z_{j})^{n-\mu+2}}.
\end{equation}
$\bullet$
To prove \eqref{zzzz-0}, we distinguish two cases, depending on $n$ and $\mu$.

Case 1. On the set $\{x: |z_j|\leq\mathscr{R}\}$.
A direct computation, due to Lemma \ref{B3-0}, it follows that
\begin{equation*}
\left\lbrace
\begin{aligned}
&\Big(|x|^{-\mu}\ast\widehat{H}_{ji,1}(y)\Big)\widehat{K}_j(x)\lesssim
\mathscr{R}_{ij}^{-\min\{\mu,(5+\mu)/3\}}t_{j,1},\hspace{2mm}\hspace{2mm}n=5\hspace{2mm}\mbox{and}\hspace{2mm}1\leq\mu<3,
\\&\Big(|x|^{-\mu}\ast\widehat{H}_{ji,1}(y)\Big)\widehat{K}_j(x)\lesssim
\mathscr{R}_{ij}^{-(2p-n+\mu-\theta_1)}t_{j,1},\hspace{5mm}\hspace{2mm}n=6\hspace{2mm}\mbox{and}\hspace{2mm}0<\mu<4,
\\&
\Big(|x|^{-\mu}\ast\widehat{H}_{ji,1}(y)\Big)\widehat{K}_j(x)\lesssim
\mathscr{R}_{ij}^{-(2p-n+\mu-\theta_2)}t_{j,1},\hspace{5mm}\hspace{2mm}n=7\hspace{2mm}\mbox{and}\hspace{2mm}\frac{7}{3}<\mu\leq4.
\end{aligned}
\right.
\end{equation*}
Case 2. On the set $\{x: |z_i|\leq\mathscr{R},|z_j|\geq\mathscr{R}\}$. Exploiting Young's inequality with exponents $s^{\star}_1=\frac{n-\mu+2}{n+2}$ and $t^{\star}_1=\frac{\mu}{n+2}$ we have, for $n=5$ and $1\leq\mu\leq\frac{5}{2}$
\begin{equation*}
\begin{split}
\Big(|x|^{-\mu}\ast\widehat{H}_{ji,1}(y)\Big)\widehat{K}_j(x)&\lesssim
\frac{1}{\mathscr{R}^{(n-\mu+2)s^{\star}_1/2+(n-2)t^{\star}_1}}
\frac{1}{\tau(z_j)^{l_1}}\frac{1}{\tau(z_j)^{l_2}}(\frac{t_{j,2}}{\tau( z_j)^{\mu}})^{s^{\star}_1}(t_{i,1})^{t^\star_1}\\&
\lesssim\frac{1}{\mathscr{R}^{(n-\mu+2)s^{\star}_1/2+(n-2)t^{\star}_1}}\big(\varepsilon^{-1}
\frac{t_{j,2}}{\tau( z_j)^{\mu}}+\varepsilon^{\frac{s^{\star}_1}{t^{\star}_1}}t_{i,1}\big)\lesssim\frac{t_{j,2}}{\mathscr{R}^{\mu}}+\frac{t_{i,1}}{\mathscr{R}^{\mu}},
\end{split}
\end{equation*}
where we choosing $l_{1}$, $l_2$ and $\varepsilon$ satisfying
$$
l_1=\frac{n-3\mu+2}{2}s^{\star}_1>0,\hspace{2mm}l_2=\mu-4s^{\star}_1>0,\hspace{2mm}\varepsilon=\mathscr{R}^{-[(n-\mu+2)s^{\star}_1/2+(n-2)t^{\star}_1]}.
$$
Similarly, we obtain, for $n=5$ and $\frac{5}{2}<\mu<3$, $n=6\hspace{2mm}\mbox{and}\hspace{2mm}0<\mu<4$, $n=7\hspace{2mm}\mbox{and}\hspace{2mm}\frac{7}{3}<\mu\leq4$,
\begin{equation*}
\begin{split}
\Big(|x|^{-\mu}\ast\widehat{H}_{ji,1}(y)\Big)\widehat{K}_j(x)&\lesssim
\frac{1}{\mathscr{R}^{(n-\mu+2)s^{\star}_1/2+(n-2)t^{\star}_1}}
\frac{1}{\tau(z_j\rangle^{l_1})}\frac{1}{\tau( z_j)^{l_2}}(\frac{t_{j,2}}{\tau( z_j)^{\mu}})^{s^{\star}_1}(t_{i,1})^{t^\star_1}\\&
\lesssim\frac{1}{\mathscr{R}^{(n-\mu+2)s^{\star}_1/2+(n-2)t^{\star}_1}}\big(\varepsilon^{-1}
\frac{t_{j,2}}{\tau( z_j)^{\mu}}+\varepsilon^{\frac{s^{\star}_1}{t^{\star}_1}}t_{i,1}\big)\lesssim\frac{t_{j,2}}{\mathscr{R}^{\mu}}+\frac{v_{i,1}}{\mathscr{R}^{\mu}},
\end{split}
\end{equation*}

To prove \eqref{zzzz-1}, we consider two cases.

Case 1. On the set $\{x: |z_j|\leq\mathscr{R}\}$.

Observe now that, since Lemma \ref{B3-0} and \eqref{fai-1}, we have
\begin{equation*}
\begin{split}
\Big(|x|^{-\mu}\ast\widehat{H}_{ji,2}\Big)\widehat{K}_j(x)\lesssim\frac{1}{\mathscr{R}^{(n-\mu+2)/2}}\frac{\lambda_{i}^{\mu/2}}{\tau( z_i)^{\mu}}\frac{\lambda_j^{(n-\mu+2)/2}}{\tau(z_{j})^{n-\mu+2}}
\lesssim\frac{1}{\mathscr{R}_{ij}^{(6+\mu-n)/2}}t_{j,1}.
\end{split}
\end{equation*}

Case 2. On the set $\{x: |z_i|\geq\mathscr{R},\hspace{2mm}|z_j|\geq\mathscr{R}\}$.

Then using Young's inequality with the exponents $\theta^{\star}=\frac{\mu}{n+2}$ and $1-\theta^{\star}=\frac{n-\mu+2}{\mu+2}$ we have
\begin{equation*}
\begin{split}
\Big(|x|^{-\mu}\ast\widehat{H}_{ji,2}\Big)\widehat{K}_j(x)&\lesssim\frac{1}{\mathscr{R}^{(n-\mu+2)/2}}\frac{\lambda_{i}^{\mu/2}}{\tau( z_i)^{\mu}}\frac{\lambda_j^{(n-\mu+2)/2}}{\tau( z_{j})^{n-\mu+2}}=\Big(\frac{t_{i,2}}{\tau( z_i)^{\frac{n+2-\mu}{2}}}\Big)^{\theta^{\star}}\Big(\frac{t_{j,2}}{\tau( z_j)^{\frac{n+2-\mu}{2}}}\Big)^{1-\theta^{\star}}\\&
\lesssim\tau(z_i)^{-\frac{n+2-\mu}{2}}t_{i,2}+\tau( z_j)^{-\frac{n+2-\mu}{2}}t_{j,2}
\lesssim\frac{1}{\mathscr{R}^\frac{n+2-\mu}{2}}t_{i,2}+\frac{1}{\mathscr{R}^\frac{n+2-\mu}{2}}t_{j,2}.
\end{split}
\end{equation*}

$\bullet$
To prove \eqref{zzzz-2}. On the set $\{x: \hspace{2mm}|z_j|\leq\mathscr{R}\}$, noticing that
 $$2p-2-\theta_1>0,\hspace{2mm}2p-2-\theta_2>0,\hspace{2mm}n-2-\mu+\min\{\mu,(5+\mu)/3\}>0$$
by \eqref{ceta}. Therefore, thanks to Lemma \ref{B3-0}, $\mathscr{R}\leq\mathscr{R}_{ij}/2$ and \eqref{fai-1}, we obtain
\begin{equation*}
\left\lbrace
\begin{aligned}
&\Big(|x|^{-\mu}\ast\widehat{H}_{ij,1}(y)\Big)\widehat{K}_i(x)\lesssim
\mathscr{R}_{ij}^{-(n-2-\mu+\min\{\mu,(5+\mu)/3\})}t_{j,1},\hspace{2mm}\hspace{2mm}n=5\hspace{2mm}\mbox{and}\hspace{2mm}1\leq\mu<3,
\\&
\Big(|x|^{-\mu}\ast\widehat{H}_{ij,1}(y)\Big)\widehat{K}_i(x)\lesssim
\mathscr{R}_{ij}^{-(2p-2-\theta_1)}t_{j,1},\hspace{9mm}\hspace{6mm}\hspace{9mm}\hspace{2mm}n=6\hspace{2mm}\mbox{and}\hspace{2mm}0<\mu<4,
\\&
\Big(|x|^{-\mu}\ast\widehat{H}_{ij,1}(y)\Big)\widehat{K}_i(x)\lesssim
\mathscr{R}_{ij}^{-(2p-2-\theta_2)}t_{j,1},\hspace{9mm}\hspace{6mm}\hspace{9mm}\hspace{2mm}n=7\hspace{2mm}\mbox{and}\hspace{2mm}\frac{7}{3}<\mu\leq4.
\end{aligned}
\right.
\end{equation*}
$\bullet$
To prove \eqref{zzzz-3}. We split the proof in three cases.

Case 1. On the set $\{x: \mathscr{R}\leq|z_j|\leq\sqrt{\lambda_j/\lambda_i}\mathscr{R}_{ij}/2\}$.
By Lemma \ref{B3-0}, \eqref{fai-2} and $\mathscr{R}_{ij}\sqrt{\lambda_{i}/\lambda_j}\approx\tau(\xi_{ij})$, we get
\begin{equation*}
\begin{split}
\Big(|x|^{-\mu}\ast\widehat{H}_{ij,2}(y)\Big)\widehat{K}_i(x)
&\lesssim\frac{1}{\mathscr{R}^{(n-\mu+2)/2}}\frac{\lambda_j^{\mu/2}}{\tau( z_{j})^\mu}\frac{\lambda_i^{(n-\mu+2)/2}}{\tau(z_i)^{n-\mu+2}}\\&
\lesssim\Big(\frac{\lambda_j}{\lambda_i}\Big)^{(n-\mu+2)/4}\Big(\frac{1}{\mathscr{R}_{ij}}\Big)^{(n-\mu+2)/2}
\frac{\lambda_{j}^{(n+2)/2}\mathscr{R}^{(n-\mu+2)/2}}{\tau( z_j)^{(n+\mu+2)/2}}\lesssim\frac{1}{\tau(\xi_{ij})^{(n-\mu+2)/2}}t_{j,2}.
\end{split}
\end{equation*}

Case 2. On the set $\{x: |z_j|\geq\sqrt{\lambda_j/\lambda_i}\mathscr{R}_{ij}/2,\hspace{2mm}|z_j|\leq\mathscr{R}\}$.
Observe now that, due to Lemma \ref{B3-0}, for some dimensional $n<6+\mu$ we have
\begin{equation*}
\begin{split}
\Big(|x|^{-\mu}\ast\widehat{H}_{ij,2}(y)\Big)\widehat{K}_i(x)
&\lesssim\frac{1}{\mathscr{R}^{(n-\mu+2)/2}}\frac{\lambda_j^{\mu/2}}{\tau( z_{j})^\mu}\frac{\lambda_i^{(n-\mu+2)/2}}{\tau(z_i)^{n-\mu+2}}\\&
\lesssim\mathscr{R}^{\frac{n+\mu-6}{2}}\frac{1}{\mathscr{R}_{ij}^{n-2}}\Big(\frac{\lambda_j}{\lambda_i}\Big)^{(n-\mu-2)/2}t_{i,1}\lesssim\frac{1}{\tau( \xi_{ij})^{(n-\mu+2)/2}}t_{i,1}.
\end{split}
\end{equation*}

Case 3. On the set $\{x: |z_j|\geq\sqrt{\lambda_j/\lambda_i}\mathscr{R}_{ij}/2,\hspace{2mm}|z_j|\geq\mathscr{R}\}$.
First note that,
$$
\mathscr{R}_{ij}^{-\frac{(n+\mu+2)(\bar{\eta}+\mu)}{2(n+2-\bar{\eta})}}\Big(\frac{\lambda_j}{\lambda_i}\Big)^{\frac{(n-\mu+2)}{2(n+2-\bar{\eta})}-\frac{(n-\mu+2)(\bar{\eta}+\mu)}{4(n+2-\bar{\eta})}}
\lesssim\Big(\mathscr{R}_{ij}\sqrt{\lambda_i/\lambda_j}\Big)^{-\frac{(n-\mu+2)(\bar{\eta}+\mu)}{2(n+2-\bar{\eta})}},
$$
where $0<\bar{\eta}<\mu$ is small. Then by Young's inequality impies
\begin{equation*}
\begin{split}
\Big(|x|^{-\mu}\ast\widehat{H}_{ij,2}(y)\Big)\widehat{K}_i(x)
&\lesssim\frac{1}{\mathscr{R}^{(n-\mu+2)/2}}\frac{\lambda_j^{\mu/2}}{\tau( z_{j})^\mu}\frac{\lambda_i^{(n-\mu+2)/2}}{\tau( z_i)^{n-\mu+2}}\lesssim\Big(\frac{\lambda_j}{\lambda_i}\Big)^{\varpi^{\star}_2}\frac{1}{\tau( z_j)^{\varpi^{\star}_1}}\big(t_{i,2}\big)^{s_2^{\star}}\big(t_{j,2}\big)^{t_2^{\star}}\\&
\lesssim\Big(\frac{\lambda_j}{\lambda_i}\Big)^{\varpi^{\star}_2}\frac{1}{\tau( z_j)^{\varpi^{\star}_1}}\big(t_{i,2}+t_{j,2}\big)\leq
\frac{1}{(\xi_{ij})^{\varpi^{\star}_1}}\big(t_{i,2}+t_{j,2}\big).
\end{split}
\end{equation*}
where $\varpi^{\star}_1$, $\varpi^{\star}_2$, $s_2^{\star}$ and $t_2^{\star}$ are given by
$$\varpi^{\star}_1=\frac{(n-\mu+2)(\bar{\eta}+\mu)}{2(n+2-\bar{\eta})},\hspace{2mm}\varpi^{\star}_2=\frac{\bar{\eta}(n-\mu+2)}{2(n+2-\bar{\eta})},\hspace{2mm}s_2^{\star}=\frac{n+2-\mu}{n+2-\bar{\eta}},\hspace{2mm}\hspace{2mm}t_2^{\star}=\frac{\mu-\bar{\eta}}{n+2-\bar{\eta}}.$$
$\bullet$
To prove \eqref{zzzz-4}, for $x\in\{x:|z_i-\xi_{ij}|\leq\varepsilon\}$, we first note that
$$\tau( z_j)^{\frac{n-\mu+2}{2}}\leq\Big(1+(\lambda_j/\lambda_i)^2\varepsilon^2\Big)^{\frac{n-\mu+2}{4}}.$$
Therefore, simple computations give,
\begin{equation*}
\begin{split}
\Big(|x|^{-\mu}\ast\widehat{H}_{ij,2}(y)\Big)\widehat{K}_i(x)
\lesssim\frac{1}{\mathscr{R}^{(n-\mu+2)/2}}\frac{\lambda_j^{\mu/2}}{\tau( z_{j})^\mu}\frac{\lambda_i^{(n-\mu+2)/2}}{\tau( z_i)^{n-\mu+2}}&=\Big(\frac{\lambda_{i}}{\lambda_j}\Big)^{\frac{n-\mu+2}{2}}\frac{\tau( z_j)^{(n-\mu+2)/2}}{\tau(z_i)^{n-\mu+2}}t_{j,2}\\&
\lesssim\Big((\frac{\lambda_{i}}{\lambda_j})^2+\varepsilon^2\Big)^{\frac{n-\mu+2}{4}}t_{j,2}.
\end{split}
\end{equation*}
$\bullet$
To prove \eqref{zzzz-5},
for $x\in\{x:|z_i|>2|\xi_{ij}|+\tilde{A}\}$, since
$$\tau( z_j)^\frac{n-\mu+2}{2}\leq2^{\frac{n-\mu+2}{4}}(\lambda_j/\lambda_i)^\frac{n-\mu+2}{2}|z_i|^\frac{n-\mu+2}{2},$$
so that
\begin{equation*}
\begin{split}
\Big(|x|^{-\mu}\ast\widehat{H}_{ij,2}(y)\Big)\widehat{K}_i(x)
\lesssim\frac{1}{\mathscr{R}^{(n-\mu+2)/2}}\frac{\lambda_j^{\mu/2}}{\tau( z_{j})^\mu}\frac{\lambda_i^{(n-\mu+2)/2}}{\tau( z_i)^{n-\mu+2}}&=\Big(\frac{\lambda_{i}}{\lambda_j}\Big)^{\frac{n-\mu+2}{2}}\frac{\tau( z_j)^{(n-\mu+2)/2}}{\tau(z_i)^{n-\mu+2}}t_{j,2}\\&
\lesssim\Big(|\xi_{ij}|+\tilde{A}\Big)^{-(n-\mu+2)/2}t_{j,2}.
\end{split}
\end{equation*}
$\bullet$
To prove \eqref{zzzz-6}. On the set $\{x:|z_i|\leq \mathscr{R}\}$, for $n<6+\mu$ which gives
$$\frac{n-\mu+2}{2}\geq\frac{(n-2)(n-\mu+2)}{n+\mu+2}$$
and introducing
$$ s_{3}^{\star}=\frac{n-\mu+2}{n+\mu+2},\hspace{2mm}t_{3}^{\star}=\frac{2\mu}{n+\mu+2}.$$
Then Young's inequality implies
\begin{equation*}
\begin{split}
\Big(|x|^{-\mu}\ast\widehat{H}_{ij,2}(y)\Big)\widehat{K}_i(x)
&\lesssim\frac{1}{\mathscr{R}^{(n-\mu+2)/2}}\frac{\lambda_j^{\mu/2}}{\tau( z_{j})^\mu}\frac{\lambda_i^{(n-\mu+2)/2}}{\tau( z_i)^{n-\mu+2}}\\&\lesssim\Big(\frac{\lambda_{i}}{\lambda_j}\Big)^{\frac{1}{2}\zeta_1^{\star}} \mathscr{R}^{\zeta_1^{\star}}\big(t_{i,1})^{s_{3}^{\star}}\big(t_{j,2})^{t_{3}^{\star}}
\lesssim\tau(\xi_{ij})^{\zeta_1^{\star}}\Big(\varepsilon^{-\frac{2\mu}{n-\mu+2}}t_{i,1}+\varepsilon t_{j,2}\Big)
\end{split}
\end{equation*}
with $\zeta_1^{\star}$ is given by
$\zeta_1^{\star}=\frac{(n-\mu+2)}{n+\mu+2}$ and we conclude the proof.
\end{proof}
\begin{lem}\label{cll-0-i}
 We have that
 \begin{equation*}
\int\frac{1}{|\hat{y}-x|^{n-2}}\frac{1}{\mathscr{R}^{2n-4-\mu}}
\frac{\lambda_i^\frac{n+2}{2}}{\tau( z_{i})^6}dx\lesssim \bar{\mathbf{s}}_{i,in}+\bar{\mathbf{s}}_{i,out},\hspace{2mm}\hspace{2mm}n=4\hspace{2mm}\mbox{and}\hspace{2mm}2<\mu<4, \hspace{2mm}\mbox{or}\hspace{2mm}n=5\hspace{2mm}\mbox{and}\hspace{2mm}3<\mu<4,
 \end{equation*}
 \begin{equation*}
\int\frac{1}{|\hat{y}-x|^{n-2}}\frac{1}{\mathscr{R}^{n-\mu+\epsilon_0}}
\frac{\lambda_i^\frac{n+2}{2}}{\tau(z_{i})^6}dx\lesssim \bar{\mathbf{s}}_{i,in}+\bar{\mathbf{s}}_{i,out},\hspace{2mm}\hspace{2mm}n=4\hspace{2mm}\mbox{and}\hspace{2mm}2<\mu<4, \hspace{2mm}\mbox{or}\hspace{2mm}n=5\hspace{2mm}\mbox{and}\hspace{2mm}3<\mu<4,
 \end{equation*}
 where $\bar{\mathbf{s}}_{i,in}$ and $\bar{\mathbf{s}}_{i,out}$ be given by
$$
\bar{\mathbf{s}}_{i,in}:=\frac{1}{\mathscr{R}^{2n-4-\mu}}\frac{\lambda_i^{\frac{n-2}{2}}}{\tau( z_{i})^{4}}\chi_{\{|z_{i}|\leq\mathscr{R}\}},\hspace{2mm}\bar{\mathbf{s}}_{i,out}:=\frac{1}{\mathscr{R}^{n-\mu+\epsilon_0}}\frac{\lambda_i^{\frac{n-2}{2}}}{\tau(  z_{i})^{n-2}}\chi_{\{|z_{i}|\geq\mathscr{R}\}}.
$$.
\end{lem}
\begin{proof}
The conclusion follows by applying Lemma \ref{B3} and the relation of $\hat{z}_i=\lambda_{i}(\hat{y}-\xi_i)$ and $\mathscr{R}$ and we skip it.
\end{proof}
We set
\begin{equation}\label{hstar-00}
H^{\star}_{ji,1}(x):=\frac{\lambda_i^{\frac{2n-\mu}{2}}\mathscr{R}^{4+\mu-2n}}{\tau(z_{i})^{2p}},
\hspace{2mm}H^{\star}_{ji,2}(x):=\frac{\lambda_i^{\frac{2n-\mu}{2}}\mathscr{R}^{-(n-\mu+\epsilon_0)}}{\tau(z_{i})^{(n-2-\epsilon_0)p}}.
\end{equation}
\begin{lem}\label{cll-3}
For each $i\neq j$, let $\lambda_i\leq\lambda_j$. If $n\geq6-\mu$ and  $\frac{n+\mu-2}{2}\leq\mu<4$, we have that
\begin{equation}\label{zzzz-00}
\Big(|x|^{-\mu}\ast H^{\star}_{ji,1}(y)\Big)\widehat{K}_j(x)\lesssim
\mathscr{R}_{ij}^{-2}\hat{t}_{j,1}+\mathscr{R}^{-2}\hat{t}_{j,2}+\mathscr{R}^{-2}\hat{t}_{i,1}+\frac{1}{\mathscr{R}^{2n-4-\mu}}
\frac{\lambda_i^\frac{n+2}{2}}{\tau( z_{i})^6},
\end{equation}
\begin{equation}\label{zzzz-10}
\Big(|x|^{-\mu}\ast H^{\star}_{ji,2}(y)\Big)\widehat{K}_j(x)\lesssim
\mathscr{R}_{ij}^{-2}\hat{t}_{j,1}+\mathscr{R}^{-2}\hat{t}_{j,2}+\mathscr{R}^{-2}\hat{t}_{i,2}+\frac{1}{\mathscr{R}^{2n-4-\mu}}
\frac{\lambda_i^\frac{n+2}{2}}{\tau( z_{i})^6},
\end{equation}
\begin{equation}\label{zzzz-20}
\Big(|x|^{-\mu}\ast H^{\star}_{ij,1}(y)\Big)\widehat{K}_i(x)\lesssim
\mathscr{R}_{ij}^{-(n-2)}\hat{t}_{j,1},
\end{equation}
\begin{equation}\label{zzzz-30}
\Big(|x|^{-\mu}\ast H^{\star}_{ij,2}(y)\Big)\widehat{K}_i(x)\lesssim
\frac{1}{\mathscr{R}_{ij}^{4+\mu+\epsilon_0-n}}\hat{t}_{i,1}+
\big(\frac{1}{\tau(\xi_{ij})^{n-\mu-\epsilon_0}}+\frac{1}{\mathscr{R}_{ij}^{2+\epsilon_0}}\big)\big(\hat{t}_{i,2}+\hat{t}_{j,2}\big)+\frac{1}{\mathscr{R}^{n-\mu+\epsilon_0}}
\frac{\lambda_j^\frac{n+2}{2}}{\tau( z_{j})^6},
\end{equation}
and that
\begin{equation}\label{zzzz-40}
\Big(|x|^{-\mu}\ast H^{\star}_{ij,2}(y)\Big)\widehat{K}_i(x)\lesssim \Big(\big(\frac{\lambda_i}{\lambda_j}\big)^2+\varepsilon^2\Big)^{\frac{n-\epsilon_0-\mu}{4}}\hat{t}_{j,2},
\hspace{2mm}\mbox{on the set}\hspace{2mm}\{x:|z_i-\xi_{ij}|\leq\varepsilon\},
\end{equation}
for all $0<\varepsilon<1$. Moreover, we have that
\begin{equation}\label{zzzz-50}
\Big(|x|^{-\mu}\ast H^{\star}_{ij,2}(y)\Big)\widehat{K}_i(x)\lesssim\Big(|\xi_{ij}|+\bar{A}\Big)^{-\frac{(n-\mu+2)}{2}}\hat{t}_{j,2}
,\hspace{2mm}\mbox{on the set}\hspace{2mm}\{x:|z_i|\geq2|\xi_{ij}|+\bar{A}\},
\end{equation}
and that, for $x\in\{x:|z_i|\leq \mathscr{R}\}$,
\begin{equation}\label{zzzz-60}
\Big(|x|^{-\mu}\ast H^{\star}_{ij,2}(y)\Big)\widehat{K}_i(x)
\lesssim
\left\lbrace
\begin{aligned}
&\varepsilon^{-\frac{\mu}{4-\mu-\epsilon_0}}\hat{t}_{i,1}+\varepsilon \hat{t}_{j,2}, \hspace{16mm}\hspace{16mm}\hspace{4.5mm}\hspace{2mm}n=4\hspace{2mm}\mbox{and}\hspace{2mm}2\leq\mu<4,\\&
\tau(\xi_{ij})^{\frac{(n-\epsilon_0-\mu)(1-\epsilon_0)}{n-\epsilon_0}}\Big(\varepsilon^{-\frac{2\mu}{5-\mu-\epsilon_0}}\hat{t}_{i,1}+\varepsilon \hat{t}_{j,2}\Big),\hspace{2mm}\hspace{2mm}n=5\hspace{2mm}\mbox{and}\hspace{2mm}3\leq\mu<4,
\end{aligned}
\right.
\end{equation}
for all $0<\varepsilon<1$.
\end{lem}
\begin{Rem}
At this stage, for $n=4$ and $\mu=2$, or $n=5$ and $\mu=3$, it is clear that $n-\mu+2=4$ in $\widehat{K}_j$ is a good candidate for low order perturbations compared to $\hat{t}_{j,1}$ in Definition \ref{st-11}. However, for the remaining cases, it is clear that $n-\mu+2<4$, which is the most difficult to estimate that all the left-hand sides of Lemma \ref{cll-3} are comparable to $\hat{t}_{j,1}$ or $\hat{t}_{j,2}$ by exploiting Lemma \ref{B3-0}.
\end{Rem}
Recalling that $H^{\star}_{ji,1}$, $H^{\star}_{ji,2}$ and $\widehat{K}_j$ in \eqref{hhstar-1} and \eqref{hstar-00}.
Then, thanks to Lemma \ref{B4}, we deduce
\begin{equation*}
\Big(|x|^{-\mu}\ast H^{\star}_{ji,1}(y)\Big)\widehat{K}_j(x)\lesssim
\mathscr{R}^{4+\mu-2n}F^{\star}_{ji,1}(x)
\end{equation*}
with
\begin{equation*}
F^{\star}_{ji,1}(x)=
\left\lbrace
\begin{aligned}
& \frac{\lambda_{i}^{\mu/2}}{\tau( z_i)^{\mu}}\frac{\lambda_j^{(n-\mu+2)/2}}{\tau( z_{j})^{n-\mu+2}},\hspace{8mm}\hspace{10mm}\hspace{6mm}\hspace{4mm}\hspace{2mm}n=4\hspace{2mm}\mbox{and}\hspace{2mm}2\leq\mu<4,\\
& \frac{\lambda_{i}^{\mu/2}}{\tau( z_i)^{(5+\mu)/3}}\frac{\lambda_j^{(n-\mu+2)/2}}{\tau( z_{j})^{n-\mu+2}},\hspace{8mm}\hspace{3mm}\hspace{3mm}\hspace{4mm}\hspace{1mm}\hspace{2mm}n=5\hspace{2mm}\mbox{and}\hspace{2mm}3\leq\mu<4,	\end{aligned}
\right.
\end{equation*}
and
\begin{equation}\label{f-star-1}
\Big(|x|^{-\mu}\ast H^{\star}_{ji,2}\Big)\widehat{K}_j(x)\lesssim
\mathscr{R}^{-(n-\mu+\epsilon_0)}F^{\star}_{ji,2}(x),\hspace{2mm}\Big(|x|^{-\mu}\ast H^{\star}_{ij,2}\Big)\widehat{K}_i(x)\lesssim
\mathscr{R}^{-(n-\mu+\epsilon_0)}F^{\star}_{ij,2}(x)
\end{equation}
where for$\hspace{2mm}n=4\hspace{2mm}\mbox{and}\hspace{2mm}2\leq\mu<4,\hspace{2mm}\mbox{or}\hspace{2mm}n=5\hspace{2mm}\mbox{and}\hspace{2mm}3\leq\mu<4$,
\begin{equation}\label{f-star-2}
F^{\star}_{ji,2}(x)=
\frac{\lambda_{i}^{\mu/2}}{\tau(z_i)^{\mu}}\frac{\lambda_j^{(n-\mu+2)/2}}{\tau( z_{j})^{n-\mu+2}},\hspace{2mm}F^{\star}_{ij,2}(x)=
\frac{\lambda_{j}^{\mu/2}}{\tau(z_j)^{\mu}}\frac{\lambda_i^{(n-\mu+2)/2}}{\tau( z_{i})^{n-\mu+2}}.
\end{equation}
Young's inequality tell us that
\begin{equation}\label{HKF-1}
\frac{F^{\star}_{ji,1}(x)}{\mathscr{R}^{4-\mu}}\lesssim\frac{1}{\mathscr{R}^{4-\mu}}
\frac{\lambda_i^3}{\tau( z_{i})^6}+\frac{1}{\mathscr{R}^{4-\mu}}
\frac{\lambda_i\lambda_j^2}{\tau( z_{i})^2\tau( z_{j})^4}, \hspace{6mm}\hspace{6mm}n=4,\hspace{2mm}2<\mu<4,
\end{equation}
and
\begin{equation}\label{HKF-2}
\frac{F^{\star}_{ji,1}(x)}{\mathscr{R}^{6-\mu}}\lesssim\frac{1}{\mathscr{R}^{6-\mu}}
\frac{\lambda_i^\frac{7}{2}}{\tau( z_{i})^6}+\frac{1}{\mathscr{R}^{6-\mu}}
\frac{\lambda_i^{\frac{3}{2}}\lambda_j^{2}}{\tau( z_{i})^{2(37-7\mu)/3(7-\mu)}\tau( z_{j})^4}, \hspace{2mm}n=5,\hspace{2mm}3<\mu<4.
\end{equation}
Similarly, if$\hspace{2mm}n=4\hspace{2mm}\mbox{and}\hspace{2mm}2\leq\mu<4,\hspace{2mm}\mbox{or}\hspace{2mm}n=5\hspace{2mm}\mbox{and}\hspace{2mm}3\leq\mu<4$,
\begin{equation}\label{HKF-1-0}
\frac{F^{\star}_{ji,2}(x)}{\mathscr{R}^{n-\mu+\epsilon_0}}\lesssim\frac{1}{\mathscr{R}^{n-\mu+\epsilon_0}}
\frac{\lambda_i^\frac{n+2}{2}}{\tau( z_{i})^6}+\frac{1}{\mathscr{R}^{n-\mu+\epsilon_0}}
\frac{\lambda_i^{\frac{n-2}{2}}\lambda_j^2}{\tau( z_{i})^2\tau( z_{j})^4},
\end{equation}
and
\begin{equation}\label{HKF-2-1}
\frac{F^{\star}_{ij,2}(x)}{\mathscr{R}^{n-\mu+\epsilon_0}}\lesssim\frac{1}{\mathscr{R}^{n-\mu+\epsilon_0}}
\frac{\lambda_j^\frac{n+2}{2}}{\tau( z_{j})^6}+\frac{1}{\mathscr{R}^{n-\mu+\epsilon_0}}
\frac{\lambda_i^{2}\lambda_j^{\frac{n-2}{2}}}{\tau(z_{i})^4\tau(z_{j})^{2}}.
\end{equation}

We now turn to the proof of Theorem \ref{cll-3} itself. First of all we note that,  to prove \eqref{zzzz-30}, using Lemma \ref{cll-0-i}, we get that $$\tau(z_{i})^{-4}\tau(z_{j})^{-2}\lesssim\mathscr{R}_{ij}^{-4}(\lambda_j/\lambda_i)^2\tau(z_{j})^{-2},\hspace{2mm}\mbox{on the set}\hspace{2mm}\{x: \mathscr{R}\leq|z_j|\leq\sqrt{\lambda_j/\lambda_i}\mathscr{R}_{ij}/2\}.$$
Combining this bound with the second term of RHS of \eqref{HKF-2-1}, we get that for $n=5$ and $3<\mu<4$,
$$\frac{1}{\mathscr{R}^{n-\mu+\epsilon_0}}
\frac{\lambda_i^{2}\lambda_j^{\frac{n-2}{2}}}{\tau(z_{i})^4\tau(z_{j})^{2}}\lesssim\tau(\xi_{ij})^{-(1+\epsilon_0)}\frac{\lambda_j}{\lambda_i}\hat{t}_{i,2}.$$
As a consequence, in order to establish \eqref{zzzz-30}, we distinguish two cases \eqref{f-star-1} and \eqref{f-star-2}, or \eqref{HKF-2-1}.

\emph{\textbf{Proof of Lemma \ref{cll-3}}.}
Arguing in an analogous way to Lemma \ref{cll-1-0-1} for the second quantities of \eqref{HKF-1}, \eqref{HKF-2} and \eqref{HKF-1-0}, and the second terms of \eqref{f-star-1} and \eqref{f-star-2}, or the second term of \eqref{HKF-2-1}, we are able to conclude the proofs by some simple computations and Lemma \ref{cll-0-i}. Therefore, we will just sketch the proof of \eqref{zzzz-30}-\eqref{zzzz-60}.

$\bullet$
To prove \eqref{zzzz-30}. We divide the proof in three cases.

Case 1. On the set $\{x: \mathscr{R}\leq|z_j|\leq\sqrt{\lambda_j/\lambda_i}\mathscr{R}_{ij}/2\}$.
 Thanks to Lemma \ref{B3-0}, \eqref{f-star-2} and $\mathscr{R}_{ij}\sqrt{\lambda_{i}/\lambda_j}\approx\tau(\xi_{ij})$, we get
\begin{equation*}
\begin{split}
\frac{F^{\star}_{ij,2}(x)}{\mathscr{R}^{n-\mu+\epsilon_0}}=\frac{1}{\mathscr{R}^{n-\mu+\epsilon_0}}
\frac{\lambda_{j}^{\mu/2}}{\tau(z_j)^{\mu}}\frac{\lambda_i^{(n-\mu+2)/2}}{\tau( z_{i})^{n-\mu+2}}
\lesssim\Big(\sqrt{\frac{\lambda_j}{\lambda_i}}\Big)^{n-\mu-\epsilon_0}\frac{1}{\mathscr{R}_{ij}^{2+\epsilon_0}}
\hat{t}_{i,2}\lesssim\frac{1}{\tau(\xi_{ij})^{n-\mu-\epsilon_0}}\hat{t}_{i,2}.
\end{split}
\end{equation*}

Case 2. On the set $\{x: |z_j|\geq\sqrt{\lambda_j/\lambda_i}\mathscr{R}_{ij}/2,\hspace{2mm}|z_j|\leq\mathscr{R}\}$.
Using Lemma \ref{B3-0} and \eqref{HKF-2-1}, we obtain
\begin{equation*}
\begin{split}
\frac{1}{\mathscr{R}^{n-\mu+\epsilon_0}}
\frac{\lambda_i^{2}\lambda_j^{\frac{n-2}{2}}}{\tau(z_{i})^4\tau(z_{j})^{2}}
\lesssim\frac{\mathscr{R}^{n-4-\epsilon_0}}{\mathscr{R}_{ij}^{\mu}}\Big(\sqrt{\frac{\lambda_j}{\lambda_i}}\Big)^{n-2-\mu}\hat{t}_{i,1}
\lesssim\frac{1}{\mathscr{R}_{ij}^{4+\mu+\epsilon_0-n}}\hat{t}_{i,1}.
\end{split}
\end{equation*}

Case 3. On the set $\{x: |z_j|\geq\sqrt{\lambda_j/\lambda_i}\mathscr{R}_{ij}/2,\hspace{2mm}|z_j|\geq\mathscr{R}\}$.
Using Young's inequality, we have
\begin{equation*}
\begin{split}
\frac{1}{\mathscr{R}^{n-\mu+\epsilon_0}}
\frac{\lambda_{j}^{\mu/2}}{\tau(z_j)^{\mu}}\frac{\lambda_i^{(n-\mu+2)/2}}{\tau( z_{i})^{n-\mu+2}}
\lesssim\frac{1}{\tau(z_j)^{2+\epsilon_0}}\Big(\frac{\lambda_i}{\lambda_j}\Big)^{\frac{\mu(2+\epsilon_0)}{2(n-\epsilon_0)}}\big(\hat{t}_{i,2}\big)^{\frac{n-\epsilon_0-\mu}{n-\epsilon_0}}
\big(\hat{t}_{j,2}\big)^{\frac{\mu}{n-\epsilon_0}}
\lesssim\frac{1}{\mathscr{R}^{2+\epsilon_0}}\big(\hat{t}_{i,2}+\hat{t}_{j,2}\big).
\end{split}
\end{equation*}
To prove \eqref{zzzz-40}, for $x\in\{x:|z_i-\xi_{ij}|\leq\varepsilon\}$, we first note that $\tau(z_j)^2\leq1+(\lambda_j/\lambda_i)^2\varepsilon^2$.
Then we obtain
\begin{equation*}
\begin{split}
\frac{1}{\mathscr{R}^{n-\mu+\epsilon_0}}
\frac{\lambda_{j}^{\mu/2}}{\tau(z_j)^{\mu}}\frac{\lambda_i^{(n-\mu+2)/2}}{\tau( z_{i})^{n-\mu+2}}
\lesssim\big(\frac{\lambda_{i}}{\lambda_j}\big)^{\frac{n-\mu+2}{2}}\Big(1+\big(\frac{\lambda_j}{\lambda_i}\big)^2\varepsilon^2\Big)^{\frac{n-\epsilon_0-\mu}{2}}t_{j,2}
\lesssim\Big(\big(\frac{\lambda_i}{\lambda_j}\big)^2+\varepsilon^2\Big)^{\frac{n-\epsilon_0-\mu}{4}}\hat{t}_{j,2}.
\end{split}
\end{equation*}
$\bullet$
To prove \eqref{zzzz-50},
for $x\in\{x:|z_i|>2|\xi_{ij}|+\bar{A}\}$, we have
\begin{equation*}
\begin{split}
\frac{1}{\mathscr{R}^{n-\mu+\epsilon_0}}
\frac{\lambda_{j}^{\mu/2}}{\tau(z_j)^{\mu}}\frac{\lambda_i^{(n-\mu+2)/2}}{\tau( z_{i})^{n-\mu+2}}
\lesssim\big(\frac{\lambda_{i}}{\lambda_j}\big)^{\frac{n-\mu+2}{2}}\frac{\tau(z_j)^{n-\epsilon_0-\mu}}{\tau(z_i)^{n-\mu+2}}\hat{t}_{j,2}
\lesssim\Big(|\xi_{ij}|+\bar{A}\Big)^{\frac{n-\mu+2}{2}}\hat{t}_{j,2}.
\end{split}
\end{equation*}
$\bullet$
To prove \eqref{zzzz-60}. On the set $\{x:|z_i|\leq \mathscr{R}\}$,
combining Young's inequality and
$$n-\mu+2>\frac{4(n-\epsilon_0-\mu)}{n-\epsilon_0},\hspace{2mm}\frac{\mu(\epsilon_0+2)}{n-\epsilon_0}>\frac{(n-\epsilon_0-\mu)(1-\epsilon_0)}{n-\epsilon_0},$$
then we get
\begin{equation*}
\begin{split}
\frac{1}{\mathscr{R}^{n-\mu+\epsilon_0}}
\frac{\lambda_{j}^{\mu/2}}{\tau(z_j)^{\mu}}\frac{\lambda_i^{(n-\mu+2)/2}}{\tau( z_{i})^{n-\mu+2}}
\lesssim\Big(\frac{\lambda_{i}}{\lambda_j}\Big)^{\frac{\epsilon_0+2}{2(n-\epsilon_0)}} \big(\frac{1}{\mathscr{R}}\big)^{\frac{(n-\epsilon_0-\mu)(4+\epsilon_0-n)}{n-\epsilon_0}}\big(\hat{t}_{i,1})^{\frac{n-\epsilon_0-\mu}{n-\epsilon_0}}\big(\hat{t}_{j,2})^{\frac{\mu}{n-\epsilon_0}}.
\end{split}
\end{equation*}
Therefore, if $n=4$ and $2\leq\lambda<4$, we have
$$\frac{F^{\star}_{ij,2}(x)}{\mathscr{R}^{n-\mu+\epsilon_0}}\lesssim\varepsilon^{-\frac{\mu}{4-\mu-\epsilon_0}}\hat{t}_{i,1}+\varepsilon \hat{t}_{j,2}.$$
When $n=5$ and $3\leq\lambda<4$, we obtain
Therefore, if $n=4$ and $2\leq\lambda<4$, we have
$$\frac{F^{\star}_{ij,2}(x)}{\mathscr{R}^{n-\mu+\epsilon_0}}\lesssim\tau(\xi_{ij})^{\frac{(n-\epsilon_0-\mu)(1-\epsilon_0)}{n-\epsilon_0}}\Big(\varepsilon^{-\frac{2\mu}{5-\mu-\epsilon_0}}\hat{t}_{i,1}+\varepsilon \hat{t}_{j,2}\Big).$$
We conclude the proof.
\qed

Lemmas \ref{B3}-\ref{cll-0} and Lemmas \ref{cll-1-0}-\ref{cll-3} then imply the following important prior estimates.
\begin{lem}\label{cll-1}
 Under the assumptions of Lemma \ref{estimate2}. There exist a dimensional constant $ C_{n,\mu,\kappa}>0$ with the following property:\\
$(i)$ If $n=5$ and $\mu\in[1,3)$, or $n=6$ and $\mu\in(0,4)$, or $n=7$ and $\mu\in(\frac{7}{3},4)$, then
 \begin{equation*}
 \begin{split}
 &\frac{1}{S_{1}(x)}\int\frac{1}{|\hat{y}-x|^{n-2}}\big[\Phi_{n,\mu}[\sigma,\phi]-\hbar(x)\big]dx
\leq
 C_{n,\mu,\kappa}
\Big(\|\hbar\|_{\ast\ast}+\widetilde{A}^{n+\mu-4+\frac{n+\mu-2}{2}}\frac{\widetilde{S}_1(x)}{S_1(x)}\|\phi\|_{\ast}\Big)
 \\&+ C_{n,\mu,\kappa}
\left\lbrace
\begin{aligned}
&(\mathscr{R}^{-\frac{(\mu+1)}{2}}\widetilde{A}^{n-\mu+2}+\widetilde{A}^{-2}+\widetilde{A}^{-\varpi^{\star}_1})\|\phi\|_{\ast}, \hspace{8mm}\hspace{8mm}\hspace{7mm}\hspace{2mm}n=5\hspace{2mm}\mbox{and}\hspace{2mm}\mu\in[1,3),\\&
(\mathscr{R}^{-\min\{\frac{\mu}{2},2p-6+\mu-\theta_1\}}\widetilde{A}^{n-\mu+2}+\widetilde{A}^{-2}+\widetilde{A}^{-\varpi^{\star}_1})\|\phi\|_{\ast},
\hspace{3.5mm}\hspace{2mm}n=6\hspace{2mm}\mbox{and}\hspace{2mm}\mu\in(0,4),\\&
(\mathscr{R}^{-\min\{\frac{\mu}{2},2p-7+\mu-\theta_2\}}\widetilde{A}^{n-\mu+2}+\widetilde{A}^{-2}+\widetilde{A}^{-\varpi^{\star}_1})\|\phi\|_{\ast},
\hspace{3.5mm}\hspace{2mm}n=7\hspace{2mm}\mbox{and}\hspace{2mm}\mu\in(\frac{7}{3},4),
\end{aligned}
\right.
\end{split}
\end{equation*}
where $\widetilde{S}_1(x)=\sum_{i\neq j}(\tilde{\mathbf{s}}_{i,in}+\tilde{\mathbf{s}}_{i,out})$ is given by
 $$\tilde{\mathbf{s}}_{i,in}:=\lambda_i^{\frac{n-2}{2}}\frac{1}{\mathscr{R}^{n-2}}\frac{1}{\tau( z_{i})^{4}}\chi_{\{|z_{i}|\leq\mathscr{R}\}},\hspace{2mm}\tilde{\mathbf{s}}_{i,out}:=\lambda_i^{\frac{n-2}{2}}\frac{1}{\mathscr{R}^{4}}\frac{1}{\tau( z_{i})^{n-2}}\chi_{\{|z_{i}|\geq\mathscr{R}\}}.
$$
$(ii)$ If $n=4$ and $\mu\in[2,4)$, or $n=5$ and $\mu\in[3,4)$, then
 \begin{equation}\label{zz-2}
 \begin{split}
 &\frac{1}{S_{2}(x)}\int\frac{1}{|\hat{y}-x|^{n-2}}\big[\Phi_{n,\mu}[\sigma,\phi]-\hbar(x)\big]dx
\leq C_{n,\mu,\kappa}
\Big(\|\hbar\|_{\ast\ast}+(\bar{A}^4\mathscr{R}^{-4}+\bar{A}^{-(n-\mu-\epsilon_0)}+\bar{A}^{-\epsilon_0})\|\phi\|_{\ast}\Big)\\
&+C_{n,\mu,\kappa}\|\phi\|_{\ast}\frac{\bar{S}_2(x)}{S_2(x)}\Big(\bar{A}^{4-2\epsilon_0}\hspace{2mm}\mbox{if}\hspace{2mm}n=4\hspace{2mm}\mbox{and}\hspace{2mm}\mu\in[2,4);\hspace{2mm}\bar{A}^{6-3\epsilon_0}\hspace{2mm}\mbox{if}\hspace{2mm}n=5\hspace{2mm}\mbox{and}\hspace{2mm}\mu\in[3,4)\Big),
  \end{split}
 \end{equation}
 where $\bar{S}_2(x)=\sum_{i\neq j}(\bar{\mathbf{s}}_{i,in}+\bar{\mathbf{s}}_{i,out})$ with $\bar{\mathbf{s}}_{i,in}$ and $\bar{\mathbf{s}}_{i,out}$ are found in Lemma \ref{cll-0-i}.
\end{lem}
\begin{proof}
The Green's function representation of the solution yield
$$
\phi(\hat{y})=C\int\frac{1}{|\hat{y}-x|^{n-2}}\Big[\Phi_{n,\mu}[\sigma,\phi]-\hbar(x)\Big]dx
$$
For the last integral above, we have
\begin{itemize}
\item[$\bullet$]
 For $n=5$ and $\mu\in[1,3)$, or $n=6$ and $\mu\in(0,4)$, or $n=7$ and $\mu\in(\frac{7}{3},4)$, from \eqref{zz-0} we infer that
 \begin{equation*}
 \int\frac{1}{|\hat{y}-x|^{n-2}}\hbar(x)dx\leq C\|\hbar\|_{\ast\ast}\int\frac{1}{|\hat{y}-x|^{n-2}}T_1(x)dx\lesssim\|\hbar\|_{\ast\ast}S_{1}(\hat{y}).
 \end{equation*}
 \item[$\bullet$]
For $n=4$ and $\mu\in[2,4)$, or $n=5$ and $\mu\in[3,4)$,
from \eqref{zz-0-0} we deduce that
 \begin{equation*}
 \int\frac{1}{|\hat{y}-x|^{n-2}}\hbar(x)dx\leq C\|\hbar\|_{\ast\ast}\int\frac{1}{|\hat{x}-x|^{n-2}}T_2(x)dx\lesssim\|\hbar\|_{\ast\ast}S_{2}(\hat{y}).
 \end{equation*}
\end{itemize}
Using Lemmas \ref{B3} and \ref{cll-1-0} to handle the remainder terms.
we consider two cases, depending choice of $n$ and $\mu$.\\
$\mathbf{Case\ 1.}$ $n=5$ and $\mu\in[1,3)$, or $n=6$ and $\mu\in(0,4)$, or $n=7$ and $\mu\in(\frac{7}{3},4)$. Then from Lemma \ref{B4-1} we infer that
\begin{equation*}
\begin{split}
&\mathcal{I}_1:=\int\frac{1}{|\hat{y}-x|^{n-2}}\Big(|x|^{-\mu}\ast\sigma^{p}\Big)
\sigma^{p-2}(x)\phi(x)dx\\&
\lesssim\|\phi\|_{\ast}\int\frac{1}{|\hat{y}-x|^{n-2}}		\sum_{i,j=1}^{\kappa}\frac{\lambda_j^{\mu/2}}{(1+|z_{j}|^2)^{\mu/2}}\sum_{i,j=1}^{\kappa}		\frac{\lambda_j^{(4-\mu)/2}}{(1+|z_{j}|^2)^{(4-\mu)/2}}\Big(s_{i,1}+s_{i,2}\Big)dx\\&
:=\|\phi\|_{\ast}(G_{ji,1}+G_{ji,2})(\hat{y}),
\end{split}
\end{equation*}
\begin{equation*}
\begin{split}
&\mathcal{I}_2:=\int\frac{1}{|\hat{y}-x|^{n-2}}\Big(|x|^{-\mu}\ast \sigma^{p-1}\phi\Big)
\sigma^{p-1}(x)dx\\&
\lesssim\|\phi\|_{\ast}\int\frac{1}{|\hat{y}-x|^{n-2}}\Big[\frac{1}{|x|^\mu}\ast		\Big(\sum_{i,j=1}^{\kappa}\widehat{K}_j(y)\big(s_{i,1}+s_{i,2}\big)(y)\Big)\Big]\sum_{i,j=1}^{\kappa}\widehat{K}_j(x)dx:=\|\phi\|_{\ast}(\widehat{G}_{ji,1}+\widehat{G}_{ji,2})(\hat{y}).
\end{split}
\end{equation*}
Letting $\tilde{\mathbf{s}}_{i,in}$ and $\tilde{\mathbf{s}}_{i,out}$ be given by
$$
\tilde{\mathbf{s}}_{i,in}:=\lambda_i^{\frac{n-2}{2}}\frac{1}{\mathscr{R}^{n-2}}\frac{1}{\langle z_{i}\rangle^{4}}\chi_{\{|z_{i}|\leq\mathscr{R}\}},\hspace{2mm}\tilde{\mathbf{s}}_{i,out}:=\lambda_i^{\frac{n-2}{2}}\frac{1}{\mathscr{R}^{4}}\frac{1}{\langle z_{i}\rangle^{n-2}}\chi_{\{|z_{i}|\geq\mathscr{R}\}}.
$$
There are four cases.

$\bullet$ If $i=j$, and denote $z_i=\lambda_{i}(x-\xi_i)$, $\hat{z}_i=\lambda_{i}(\hat{y}-\xi_i)$ then since Lemma \ref{B3} we get
\begin{equation*}
\begin{split}
G_{ii,1}&\lesssim \sum_{i=1}^{\kappa}\lambda_i^{2}\int\frac{1}{|\hat{y}-x|^{n-2}}\frac{s_{i,1}}{(1+|z_{i}|^2)^{2}}dx\\&
\lesssim \sum_{i=1}^{\kappa}\lambda_i^{\frac{n-2}{2}}\Big[\frac{1}{\mathscr{R}^{n-2}}\frac{1}{(1+|\hat{z}_{i}|^2)^2}\chi_{\{|\hat{z}_i|\leq2\mathscr{R}\}}+\frac{1}{\mathscr{R}^{4}}\frac{1}{(1+|\hat{z}_{i}|^2)^{(n-2)/2}}\chi_{\{|\hat{z}_i|\geq2\mathscr{R}\}}\Big]
\lesssim\sum_{i=1}^{\kappa}\big[\tilde{\mathbf{s}}_{i,in}+\tilde{\mathbf{s}}_{i,out}\big]
\end{split}
\end{equation*}
and analogously we can take an exponent $\varsigma_0$ satisfying
$\varsigma_0=\frac{n+\mu+6}{2}<n$ in Lemma \ref{B3}.
Then, from $|\hat{z}_i-z_i|\approx|z_i|$ on $\{|z_{i}|>\mathscr{R}\}$ we deduce that
\begin{equation*}
\begin{split}
G_{ii,2}&
\lesssim\sum_{i=1}^{\kappa}\lambda_i^{2}\int\frac{1}{|\hat{y}-x|^{n-2}}\frac{s_{i,2}}{(1+|z_{i}|^2)^{2}}dx\\&
\lesssim\sum_{i=1}^{\kappa}\lambda_i^{\frac{n-2}{2}}\Big[\frac{1}{\mathscr{R}^{n+2}}\chi_{\{|\hat{z}_i|\leq\frac{\mathscr{R}}{2}\}}+
\frac{1}{\mathscr{R}^{(n-\mu+2)/2}}\frac{1}{(1+|\hat{z}_{i}|^2)^{(n+\mu+2)/4}}\chi_{\{|\hat{z}_i|\geq\frac{\mathscr{R}}{2}\}}\Big]
\lesssim\sum_{i=1}^{\kappa}\big[\tilde{\mathbf{s}}_{i,in}+\tilde{\mathbf{s}}_{i,out}\big].
\end{split}
\end{equation*}
Similar to the above argument and from Lemmas \ref{B4-1}, \ref{B3} and \ref{cll-1-0} that
\begin{equation*}
\begin{split}
\widehat{G}_{ji,1}
&= \int\frac{1}{|\hat{y}-x|^{n-2}}\Big[\frac{1}{|x|^\mu}\ast		\sum_{j=1}^{\kappa}\widehat{K}_i(y)s_{i,1}(y)\Big]\sum_{j=1}^{\kappa}\widehat{K}_i(x)dx\lesssim\sum_{i=1}^{\kappa}\big[\tilde{\mathbf{s}}_{i,in}+\tilde{\mathbf{s}}_{i,out}\big],
\end{split}
\end{equation*}
and $\mu<n-2$ that
\begin{equation*}
\begin{split}
\widehat{G}_{ji,2}
&= \int\frac{1}{|\hat{y}-x|^{n-2}}\Big[\frac{1}{|x|^\mu}\ast		\sum_{j=1}^{\kappa}\widehat{K}_i(y)s_{i,2}(y)\Big]\sum_{j=1}^{\kappa}\widehat{K}_i(x)dx
\lesssim\sum_{i=1}^{\kappa}\big[\tilde{\mathbf{s}}_{i,in}+\tilde{\mathbf{s}}_{i,out}\big].
\end{split}
\end{equation*}
$\bullet$ If $i\neq j$, we assume first $\lambda_i<\lambda_j$. Then, we get from \eqref{zz-0}, \eqref{zzz-0-1} and \eqref{zzz-1-1} that
\begin{equation*}
\begin{split}
(G_{ji,1}+G_{ji,2})(\hat{y})\lesssim\sum_{i,j=1}^{\kappa}\mathscr{R}^{-2}\int\frac{1}{|\hat{y}-x|^{n-2}}T_1(x)
\lesssim\sum_{i,j=1}^{\kappa}\mathscr{R}^{-2}S_1(\hat{y}).
\end{split}
\end{equation*}
Recalling that $\widehat{K}_{j}(x)$, $\widehat{H}_{ji,1}(x)$, $\widehat{H}_{ji,2}(x)$, $s_{i,1}$ and $s_{i,2}$, by Young's inequality we have
\begin{equation}\label{k-z1}
\widehat{K}_j(y)s_{i,1}=\frac{\mathscr{R}^{2-n}\lambda_j^{(n-\mu+2)/2}}{(1+|z_{j}|^2)^{(n-\mu+2)/2}} \frac{\lambda_i^{(n-2)/2}}{(1+|z_{i}|^2)}\lesssim\frac{\mathscr{R}^{2-n}\lambda_j^{(2n-\mu)/2}}{(1+|z_{j}|^2)^{(2n-\mu)/2}}+\widehat{H}_{ji,1}(x),
\end{equation}
\begin{equation}\label{kz2}
\widehat{K}_j(y)s_{i,2}=\frac{\mathscr{R}^{-\frac{n-\mu+2}{2}}\lambda_j^{(n-\mu+2)/2}}{(1+|z_{j}|^2)^{(n-\mu+2)/2}} \frac{\lambda_i^{(n-2)/2}}{(1+|z_{i}|^2)^{(n+\mu-2)/4}}\lesssim\frac{\mathscr{R}^{-\frac{n-\mu+2}{2}}\lambda_j^{(2n-\mu)/2}}{(1+|z_{j}|^2)^{(2n-\mu)/2}}+\widehat{H}_{ji,2}(x).
\end{equation}
where since $(n+\mu-2)p/2>(n-\mu+4)/2$.
Therefore, due to Lemmas \ref{B4}-\ref{B4-1}, we obtain
\begin{equation}\label{kz3}
\begin{split}
\Big(\mathscr{R}^{-\frac{n-\mu+2}{2}}+\mathscr{R}^{2-n}\Big)\int&\frac{1}{|\hat{y}-x|^{n-2}}\Big[\frac{1}{|x|^\mu}\ast	\frac{\lambda_j^{(2n-\mu)/2}}{(1+|z_{j}|^2)^{(2n-\mu)/2}}\Big]\sum_{j=1}^{\kappa}\widehat{K}_j(x)dx\lesssim\sum_{i=1}^{\kappa}\big[\tilde{\mathbf{s}}_{j,in}+\tilde{\mathbf{s}}_{j,out}\big].
\end{split}
\end{equation}
Using Lemmas \ref{B4}-\ref{B4-1}, Lemma \ref{cll-0} and \ref{cll-2} to handle the remainder term, so that we get from \eqref{zzzz-0} and \eqref{zzzz-1} that, for $n=5\hspace{2mm}\mbox{and}\hspace{2mm}1\leq\mu<3$,
\begin{equation*}
\begin{split}
\widetilde{I}&=:\int\frac{1}{|\hat{y}-x|^{n-2}}\Big[\frac{1}{|x|^\mu}\ast		\Big(\sum_{i,j=1}^{\kappa}\big(\widehat{H}_{ji,1}(y)+\widehat{H}_{ji,2}(y)\big)\Big]\sum_{i,j=1}^{\kappa}\widehat{K}_j(x)dx
\\&\lesssim\sum_{i,j=1}^{\kappa}\mathscr{R}^{-\frac{(\mu+1)}{2}}\int\frac{1}{|\hat{y}-x|^{n-2}}T_1(x)\lesssim\sum_{i,j=1}^{\kappa}\mathscr{R}^{-\frac{(\mu+1)}{2}}S_1(\hat{y}).
\end{split}
\end{equation*}
Likewise we have, for $n=6\hspace{2mm}\mbox{and}\hspace{2mm}0<\mu<4$,
\begin{equation*}
\begin{split}
\widetilde{I}\lesssim\sum_{i,j=1}^{\kappa}\mathscr{R}^{-\min\{\mu/2,2p-6+\mu-\theta_1\}}\int\frac{1}{|\hat{y}-x|^{n-2}}T_1(x)\lesssim\sum_{i,j=1}^{\kappa}\mathscr{R}^{-\min\{\mu/2,2p-6+\mu-\theta_1\}}S_1(\hat{y}),
\end{split}
\end{equation*}
and for $n=7\hspace{2mm}\mbox{and}\hspace{2mm}\frac{7}{3}<\mu\leq4$,
\begin{equation*}
\begin{split}
\widetilde{I}\lesssim\sum_{i,j=1}^{\kappa}\mathscr{R}^{-\min\{\mu/2,2p-7+\mu-\theta_1\}}\int\frac{1}{|\hat{y}-x|^{n-2}}T_1(x)\lesssim\sum_{i,j=1}^{\kappa}\mathscr{R}^{-\min\{\mu/2,2p-7+\mu-\theta_2\}}S_1(\hat{y}).
\end{split}
\end{equation*}
On the other hand, noticing that
\begin{equation*}
\widehat{K}_i(y)s_{j,1}\lesssim\frac{\mathscr{R}^{2-n}\lambda_i^{(2n-\mu)/2}}{(1+|z_{i}|^2)^{(2n-\mu)/2}}+\widehat{H}_{ij,1}(x),\hspace{2mm}
\widehat{K}_i(y)s_{j,2}\lesssim\frac{\mathscr{R}^{-\frac{n-\mu+2}{2}}\lambda_i^{(2n-\mu)/2}}{(1+|z_{i}|^2)^{(2n-\mu)/2}}+\widehat{H}_{ij,2}(x).
\end{equation*}
Consequence
\begin{equation*}
\begin{split}
\Big(\mathscr{R}^{-\frac{n+\mu-2}{2}}+\mathscr{R}^{2-n}\Big)\int&\frac{1}{|\hat{y}-x|^{n-2}}\Big[\frac{1}{|x|^\mu}\ast	\frac{\lambda_i^{(2n-\mu)/2}}{(1+|z_{i}|^2)^{(2n-\mu)/2}}\Big]\sum_{i=1}^{\kappa}\widehat{K}_i(x)dx\lesssim\sum_{i=1}^{\kappa}\big[\tilde{\mathbf{s}}_{i,in}+\tilde{\mathbf{s}}_{i,out}\big].
\end{split}
\end{equation*}
$\bullet$ When $i\neq j$, we assume now $\lambda_i\geq\lambda_j$ and $\langle \xi_{ij}\rangle\geq \tilde{A}$. Using Lemma \ref{cll-0}, \eqref{zzz-2-1} and \eqref{zzz-3-1}, we get that
\begin{equation*}
\begin{split}
(G_{ji,1}+G_{ji,2})(\hat{y})\lesssim\sum_{i,j=1}^{\kappa}(\mathscr{R}^{-2}+\tilde{A}^{-2})S_1(\hat{y}).
\end{split}
\end{equation*}
By Lemma \ref{cll-0}, \eqref{zzzz-2} and \eqref{zzzz-3} we have that for $n=5\hspace{2mm}\mbox{and}\hspace{2mm}1\leq\mu<3$, or $n=6\hspace{2mm}\mbox{and}\hspace{2mm}0<\mu<4$, or $n=7\hspace{2mm}\mbox{and}\hspace{2mm}\frac{7}{3}<\mu\leq4$,
\begin{equation*}
\begin{split}
\widetilde{II}&=:\int\frac{1}{|\hat{y}-x|^{n-2}}\Big[\frac{1}{|x|^\mu}\ast		\Big(\sum_{i,j=1}^{\kappa}\big(H_{ij,1}(y)+H_{ij,2}(y)\big)\Big)\Big]\sum_{i=1}^{\kappa}\widehat{K}_i(x)dx
\\&\lesssim\big(\mathscr{R}^{-(3-\mu+\min\{\mu,(5+\mu)/3\})}+\mathscr{R}^{-(2p-2-\theta_1)}+\mathscr{R}^{-(2p-2-\theta_2)}+\tilde{A}^{-(n-\mu+2)/2}+
\tilde{A}^{-\varpi^{\star}_1}\big)S_1(\hat{y}).
\end{split}
\end{equation*}
$\bullet$ When $i\neq j$, we assume finally $\lambda_i\geq\lambda_j$ and $\langle \xi_{ij}\rangle\leq \tilde{A}$. Using \eqref{zzz-2-1}, \eqref{zzz-4-1}, \eqref{zzz-5-1} and the fact that $\lambda_i/\lambda_j\lesssim\mathscr{R}^{-2}\tilde{A}^2$, and choosing $\varepsilon=\tilde{A}^{-1}$ we deduce that, for $n>6-\mu$,
\begin{equation*}
\begin{split}
(G_{ji,1}+G_{ji,2})(\hat{y})&\lesssim\sum_{i,j=1}^{\kappa}\Big[\big[\big(\frac{\lambda_{j}}{\lambda_i}\big)^2+\varepsilon^{2}\big]S_1(\hat{y})
+\tau(\xi_{ij})^{n+\mu-4}\varepsilon^{-\frac{n+\mu-2}{2}}
\int\frac{\lambda_j^{2}}{|\hat{y}-x|^{n-2}}\frac{(s_{j,1}+s_{j,2})}{(1+|z_{j}|^2)^{2}}
		\Big]\\&
\lesssim\sum_{i,j=1}^{\kappa}\big[\mathscr{R}^{-4}\tilde{A}^{4}+\tilde{A}^{-2}\big]S_1(\hat{y})+
\sum_{i,j=1}^{\kappa}\tilde{A}^{n+\mu-4+\frac{n+\mu-2}{2}}\big[\tilde{\mathbf{s}}_{j,in}+\tilde{\mathbf{s}}_{j,out}\big].
\end{split}
\end{equation*}
When $n=6-\mu$,
$$
(G_{ji,1}+G_{ji,2})(\hat{y})\lesssim\sum_{i,j=1}^{\kappa}\big[\mathscr{R}^{-4}\tilde{A}^{4}+\tilde{A}^{-2}\big]S_1(\hat{y})+
\sum_{i,j=1}^{\kappa}\tilde{A}^{n+\mu-2}\big[\tilde{\mathbf{s}}_{j,in}+\tilde{\mathbf{s}}_{j,out}\big].
$$
In view of \eqref{zzzz-2} and \eqref{zzzz-4}, we obtain
for $n=5\hspace{2mm}\mbox{and}\hspace{2mm}1\leq\mu<3$,
\begin{equation}\label{IIIv3-0}
\begin{split}
\widetilde{III}&=:\sum_{i,j=1}^{\kappa}\int\frac{1}{|\hat{y}-x|^{n-2}}\Big[\Big(\frac{1}{|x|^\mu}\ast		\widehat{H}_{ij,1}(y)\Big)\widehat{K}_i(x)+\Big(\frac{1}{|x|^\mu}\ast\widehat{H}_{ij,2}(y)\Big)\widehat{K}_i(x)\chi_{\{|z_i-\xi_{ij}|\leq\varepsilon\}}\Big]dx
\\&\lesssim\sum_{i,j=1}^{\kappa}\big[\mathscr{R}^{-(n-\mu+2)}\tilde{A}^{n-\mu+2}+\tilde{A}^{-2}\big]S_1(\hat{y}).
\end{split}
\end{equation}
where we also taking $\varepsilon=\tilde{A}^{-1}$ and $\lambda_j/\lambda_i\lesssim\mathscr{R}^{-2}\tilde{A}^2$. Moreover applying Lemmas \ref{B4-1}, \ref{B3} and \eqref{zzz-5-1} to get that for $x\in\{x:|z_i-\xi_{ij}|>\varepsilon\}$,
\begin{equation}\label{IIIv3-1}
\begin{split}
\sum_{i,j=1}^{\kappa}&\int\frac{1}{|\hat{y}-x|^{n-2}}\Big[\frac{1}{|x|^\mu}\ast		\widehat{K}_j(y)s_{i,2}(y)\chi_{\{|z_i-\xi_{ij}|>\varepsilon\}}\Big]\widehat{K}_j(x)\\&\lesssim
\sum_{i,j=1}^{\kappa}\tilde{A}^{n+\mu-4+\frac{n+\mu-2}{2}}\int\frac{1}{|\hat{y}-x|^{n-2}}
\Big[\frac{1}{|x|^\mu}\ast		\widehat{K}_j(y)[s_{j,1}(y)+s_{j,2}(y)]\Big]\widehat{K}_j(x)\\&
\lesssim \sum_{i,j=1}^{\kappa}\tilde{A}^{n+\mu-4+\frac{n+\mu-2}{2}}\big[\tilde{\mathbf{z}}_{j,in}+\tilde{\mathbf{s}}_{j,out}\big].
\end{split}
\end{equation}
Therefore by \eqref{k-z1}-\eqref{kz3} and \eqref{IIIv3-0} and \eqref{IIIv3-1} we have that
\begin{equation*}
\begin{split}
(\widehat{G}_{ji,1}+\widehat{G}_{ji,2})(\hat{y})&\lesssim
\sum_{i,j=1}^{\kappa}\big[\mathscr{R}^{-(n-\mu+2)}\tilde{A}^{n-\mu+2}+\tilde{A}^{-2}\big]S_1(\hat{y})\\&+
\sum_{i,j=1}^{\kappa}\big[1+\tilde{A}^{n+\mu-4+\frac{n+\mu-2}{2}}\big]\big[\tilde{\mathbf{s}}_{j,in}+\tilde{\mathbf{s}}_{j,out}\big].
\end{split}
\end{equation*}
When $n=6\hspace{2mm}\mbox{and}\hspace{2mm}0<\mu<4$, or $n=7\hspace{2mm}\mbox{and}\hspace{2mm}\frac{7}{3}<\mu\leq4$, choosing $\theta_1,\theta_2$ small enough, by \eqref{zzzz-2} and \eqref{zzzz-4}we have
$$\widetilde{III}\lesssim\sum_{i,j=1}^{\kappa}\big[\mathscr{R}^{-(n-\mu+2)}\tilde{A}^{n-\mu+2}+\tilde{A}^{-2}\big]S_1(\hat{y}),$$
so we get from \eqref{kz3} and \eqref{IIIv3-1} that
\begin{equation*}
\begin{split}
(\widehat{G}_{ji,1}+\widehat{G}_{ji,2})(\hat{y})&\lesssim
\sum_{i,j=1}^{\kappa}\big[\mathscr{R}^{-(n-\mu+2)}\tilde{A}^{n-\mu+2}+\tilde{A}^{-2}\big]S_1(\hat{y})\\&+
\sum_{i,j=1}^{\kappa}\big[1+\tilde{A}^{n+\mu-4+\frac{n+\mu-2}{2}}\big]\big[\tilde{\mathbf{s}}_{j,in}+\tilde{\mathbf{s}}_{j,out}\big].
\end{split}
\end{equation*}
Combining these estimates, we can easily deduce the desired estimate.

$\mathbf{Case\ 2.}$ $n=4$ and $\mu\in[2,4)$, or $n=5$ and $\mu\in[3,4)$.
Recalling that $\bar{\mathbf{s}}_{i,in}$ and $\bar{\mathbf{s}}_{i,out}$ be given by
$$
\bar{\mathbf{s}}_{i,in}:=\lambda_i^{\frac{n-2}{2}}\frac{1}{\mathscr{R}^{2n-4-\mu}}\frac{1}{\langle z_{i}\rangle^{4}}\chi_{\{|z_{i}|\leq\mathscr{R}\}},\hspace{2mm}\bar{\mathbf{s}}_{i,out}:=\lambda_i^{\frac{n-2}{2}}\frac{1}{\mathscr{R}^{n-\mu+\epsilon_0}}\frac{1}{\langle z_{i}\rangle^{n-2}}\chi_{\{|z_{i}|\geq\mathscr{R}\}}.
$$
The proof is then analogous to the case $1$ and we omit it.

Then from Lemma \ref{B4-1} we infer that
\begin{equation*}
\begin{split}
&\mathcal{I}_1
\lesssim\|\phi\|_{\ast}\int\frac{1}{|\hat{y}-x|^{n-2}}		\sum_{i,j=1}^{\kappa}\widetilde{H}_{j}(x)\Big(\hat{s}_{i,1}+\hat{s}_{i,2}\Big)dx:=\|\phi\|_{\ast}(J_{ji,1}+J_{ji,2})(\hat{y}),
\end{split}
\end{equation*}
\begin{equation*}
\begin{split}
&\mathcal{I}_2
\lesssim\|\phi\|_{\ast}\int\frac{1}{|\hat{y}-x|^{n-2}}\Big[\frac{1}{|x|^\mu}\ast		\Big(\sum_{i,j=1}^{\kappa}\widehat{\mathscr{K}}_j(y)\big(\hat{s}_{i,1}+\hat{s}_{i,2}\big)(y)\Big)\Big]\sum_{i,j=1}^{\kappa}\widehat{K}_j(x)dx:=\|\phi\|_{\ast}(\tilde{J}_{ji,1}+\tilde{J}_{ji,2})(\hat{y}).
\end{split}
\end{equation*}
There are four cases.

$\bullet$ If $i=j$, and denote $z_i=\lambda_{i}(x-\xi_i)$, $\hat{z}_i=\lambda_{i}(\hat{y}-\xi_i)$ then since Lemma \ref{B3} we get
\begin{equation*}
\begin{split}
(J_{ii,1}+J_{ii,2})(\hat{y})&\lesssim \sum_{i=1}^{\kappa}\lambda_i^{2}\int\frac{1}{|\hat{y}-x|^{n-2}}\frac{(\hat{s}_{i,1}+\hat{s}_{i,2})}{(1+|z_{i}|^2)^{2}}dx
\lesssim\sum_{i=1}^{\kappa}\big[\bar{\mathbf{s}}_{i,in}+\bar{\mathbf{s}}_{i,out}\big]
\end{split}
\end{equation*}
Similar to the argument of $\bar{\mathscr{G}}_1$, $\bar{\mathscr{G}}_2$ and from Lemmas \ref{B4-1}, \ref{B3} and \ref{cll-1-0} that
\begin{equation*}
\begin{split}
\tilde{J}_{ii,1}+\tilde{J}_{ii,2}
&= \int\frac{1}{|\hat{y}-x|^{n-2}}\Big[\frac{1}{|x|^\mu}\ast		\sum_{j=1}^{\kappa}\widehat{K}_i(y)(\hat{s}_{i,1}(y)+\hat{s}_{i,2}(y))\Big]\sum_{j=1}^{\kappa}\widehat{K}_i(x)dx\lesssim\sum_{i=1}^{\kappa}\big[\bar{\mathbf{s}}_{i,in}+\bar{\mathbf{s}}_{i,out}\big],
\end{split}
\end{equation*}
$\bullet$ If $i\neq j$, we assume first $\lambda_i<\lambda_j$.
Then, we get from \eqref{zz-0-0}, \eqref{zzz-0-1-1} and \eqref{zzz-1-1-1} that
\begin{equation*}
\begin{split}
(J_{ji,1}+J_{ji,2})(\hat{y})\lesssim\sum_{i,j=1}^{\kappa}\mathscr{R}^{-2}\int\frac{1}{|\hat{y}-x|^{n-2}}T_2(x)
\lesssim\sum_{i,j=1}^{\kappa}\mathscr{R}^{-2}S_2(\hat{y}).
\end{split}
\end{equation*}
Recalling that $\widehat{K}_{j}(x)$, $H^{\star}_{ji,1}(x)$, $H^{\star}_{ji,2}(x)$, $\hat{s}_{i,1}$ and $\hat{s}_{i,2}$, by Young's inequality we have
\begin{equation}\label{kz1}
\widehat{K}_j(y)\hat{s}_{i,1}=\frac{\mathscr{R}^{4+\mu-2n}\lambda_j^{(n-\mu+2)/2}}{(1+|z_{j}|^2)^{(n-\mu+2)/2}} \frac{\lambda_i^{(n-2)/2}}{(1+|z_{i}|^2)}\lesssim\frac{\mathscr{R}^{4+\mu-2n}\lambda_j^{(2n-\mu)/2}}{(1+|z_{j}|^2)^{(2n-\mu)/2}}+H^{\star}_{ji,1}(x),
\end{equation}
\begin{equation}\label{kz-22}
\widehat{K}_j(y)\hat{s}_{i,2}=\frac{\mathscr{R}^{-(n-\mu+\epsilon_0)}\lambda_j^{(n-\mu+2)/2}}{(1+|z_{j}|^2)^{(n-\mu+2)/2}} \frac{\lambda_i^{(n-2)/2}}{(1+|z_{i}|^2)^{n-2-\epsilon_0}}\lesssim\frac{\mathscr{R}^{-(n-\mu+\epsilon_0)}\lambda_j^{(2n-\mu)/2}}{(1+|z_{j}|^2)^{(2n-\mu)/2}}+H^{\star}_{ji,2}(x).
\end{equation}
Therefore, due to Lemmas \ref{B4}-\ref{B4-1}, we obtain
\begin{equation*}
\begin{split}
\Big(\frac{1}{\mathscr{R}^{2n-4-\mu}}+\frac{1}{\mathscr{R}^{n-\mu+\epsilon_0}}\Big)\int&\frac{1}{|\hat{y}-x|^{n-2}}\Big[\frac{1}{|x|^\mu}\ast	\frac{\lambda_j^{(2n-\mu)/2}}{(1+|z_{j}|^2)^{(2n-\mu)/2}}\Big]\sum_{j=1}^{\kappa}\widehat{K}_j(x)dx\lesssim\sum_{i=1}^{\kappa}\big[\bar{\mathbf{s}}_{j,in}+\bar{\mathbf{s}}_{j,out}\big].
\end{split}
\end{equation*}
Using Lemmas \ref{B4}-\ref{B4-1}, Lemma \ref{cll-0} and \ref{cll-2} to handle the remainder term, so that we get from Lemma \ref{cll-0}, \eqref{zzzz-00}-\eqref{zzzz-10} and Lemma \ref{cll-0-i} that
\begin{equation*}
\begin{split}
&\int\frac{1}{|\hat{y}-x|^{n-2}}\Big[\frac{1}{|x|^\mu}\ast		\Big(\sum_{i,j=1}^{\kappa}\big(H^{\star}_{ji,1}(y)+H^{\star}_{ji,2}(y)\big)\Big]\sum_{i,j=1}^{\kappa}\widehat{K}_j(x)dx
\\&\lesssim\sum_{i,j=1}^{\kappa}\mathscr{R}^{-2}\int\frac{1}{|\hat{y}-x|^{n-2}}\big(T_2(x)+\frac{1}{\mathscr{R}^{2n-4-\mu}}
\frac{\lambda_i^\frac{n+2}{2}}{\tau( z_{i})^6}\big)\lesssim\sum_{i,j=1}^{\kappa}\mathscr{R}^{-2}S_2(\hat{y})+\sum_{i,j=1}^{\kappa}\big(\bar{\mathbf{s}}_{i,in}+\bar{\mathbf{s}}_{i,out}\big).
\end{split}
\end{equation*}
$\bullet$ When $i\neq j$, we assume now $\lambda_i\geq\lambda_j$ and $\langle \xi_{ij}\rangle\geq \bar{A}$. Using Lemma \ref{cll-0}, \eqref{zzz-2-1-1} and \eqref{zzz-3-1-1}, we get that
\begin{equation*}
\begin{split}
(J_{ji,1}+J_{ji,2})(\hat{y})\lesssim\sum_{i,j=1}^{\kappa}(\mathscr{R}^{-2}+\bar{A}^{-\epsilon_0}+\bar{A}^{-2})S_2(\hat{y}).
\end{split}
\end{equation*}
By Lemma \ref{cll-0}, Lemma \ref{cll-0-i}, \eqref{zzzz-20} and \eqref{zzzz-30} we have
\begin{equation*}
\begin{split}
&\int\frac{1}{|\hat{y}-x|^{n-2}}\Big[\frac{1}{|x|^\mu}\ast		\Big(\sum_{i,j=1}^{\kappa}\big(H^{\star}_{ij,1}(y)+H^{\star}_{ij,2}(y)\big)\Big)\Big]\sum_{i=1}^{\kappa}\widehat{K}_i(x)dx
\\&\lesssim\sum_{i,j=1}^{\kappa}\big(\bar{A}^{-(n-\mu-\epsilon_0)}+\mathscr{R}_{ij}^{-\min\{n-2,2+\epsilon_0\}}\big)S_2(\hat{y})
+\sum_{i,j=1}^{\kappa}(\bar{\mathbf{s}}_{j,in}+\bar{\mathbf{s}}_{j,out}).
\end{split}
\end{equation*}
$\bullet$ When $i\neq j$, we assume finally $\lambda_i\geq\lambda_j$ and $\langle \xi_{ij}\rangle\leq \bar{A}$. Using \eqref{zzz-2-1-1}, \eqref{zzz-4-1-1}, \eqref{zzz-5-1-1} and the fact that $\lambda_i/\lambda_j\lesssim\mathscr{R}^{-2}\bar{A}^2$, and choosing $\varepsilon=\bar{A}^{-1}$ we deduce that, for $n=4$,
\begin{equation*}
\begin{split}
(J_{ji,1}&+J_{ji,2})(\hat{y})\lesssim\sum_{i,j=1}^{\kappa}\Big[\big[\big(\frac{\lambda_{j}}{\lambda_i}\big)^2+\varepsilon^{2}\big]S_2(\hat{y})
+\tau(\xi_{ij})^{2-\epsilon_0}\varepsilon^{\epsilon_0-2}
\int\frac{\lambda_i^{2}}{|\hat{y}-x|^{n-2}}\frac{(\hat{s}_{j,1}+\hat{s}_{j,2})}{(1+|z_{j}|^2)^{2}}
		\Big]\\&
\lesssim\sum_{i,j=1}^{\kappa}\big[\mathscr{R}^{-4}\bar{A}^{4}+\bar{A}^{-2}\big]S_2(\hat{y})+
\sum_{i,j=1}^{\kappa}[\bar{\mathbf{s}}_{j,in}+\bar{\mathbf{s}}_{j,out}\big]\Big(\bar{A}^{4-2\epsilon_0}\hspace{2mm}\mbox{if}\hspace{2mm}n=4;\hspace{1mm}\bar{A}^{6-3\epsilon_0}\hspace{2mm}\mbox{if}\hspace{2mm}n=5\Big).
\end{split}
\end{equation*}
By Lemmas \ref{B4-1}, Lemma \ref{cll-0}, \eqref{zzz-5-1-1}, \eqref{zzzz-20} and \eqref{zzzz-40}, taking $\varepsilon=\bar{A}^{-1}$ and $\lambda_j/\lambda_i\lesssim\mathscr{R}^{-2}\bar{A}^2$, we obtain
\begin{equation*}
\begin{split}
&\sum_{i,j=1}^{\kappa}\int\frac{1}{|\hat{y}-x|^{n-2}}\Big[\Big(\frac{1}{|x|^\mu}\ast		\widehat{H}_{ij,1}^{\star}(y)\Big)\widehat{K}_i(x)+\Big(\frac{1}{|x|^\mu}\ast\widehat{H}_{ij,2}^{\star}(y)\Big)\widehat{K}_i(x)\Big]dx
\\&\lesssim\sum_{i,j=1}^{\kappa}\Big(\frac{\bar{A}^{4}}{\mathscr{R}^{4}}+\frac{1}{\bar{A}^{2}}\Big)^{\frac{n-\epsilon_0-\mu}{4}}S_2(\hat{y})
+\Big(\bar{A}^{4-2\epsilon_0}[\bar{\mathbf{s}}_{j,in}+\bar{\mathbf{s}}_{j,out}\big]\hspace{2mm}\mbox{if}\hspace{2mm}n=4;\hspace{1mm}\bar{A}^{6-3\epsilon_0}[\bar{\mathbf{s}}_{j,in}+\bar{\mathbf{s}}_{j,out}\big]\hspace{2mm}\mbox{if}\hspace{2mm}n=5\Big).
\end{split}
\end{equation*}
Combining these estimates, we eventually get that the desired estimate and concluding the proof.
\end{proof}

\section{Proof of Lemma \ref{estimate2}}
The whole section is dedicated to the proof of Lemma \ref{estimate2},
which is established by the following four steps with several key Propositions with different right hand sides. To show Lemma \ref{estimate2}, we are going to apply the tree structure of $\delta$-interation bubbles as $\delta\rightarrow0$. The concept of bubble-trees was investigated in \cite{DSW21,Druet,Parker,G-Tian}.
\subsection{Tree structure of weak interaction bubbles }
To reformulate the tree structure by introducing the sequence of $\kappa$ bubbles  $\big\{W_i^{\left(m\right)}:=W[\xi_i^{\left(m\right)},\lambda_i^{\left(m\right)}],~i=1,\cdots,\kappa\big\}_{i=1}^{\infty}$
and satisfying
\begin{equation}\label{TTT}
		\left\lbrace
\begin{aligned}
&\lambda_1^{(m)}\leq\cdots\leq\lambda_{\kappa}^{(m)}\hspace{2mm}\mbox{for all}\hspace{2mm}m\in\mathbb{N},\\&
\mathscr{R}_m:=\frac{1}{2}\min\limits_{i\neq j}\big\{\mathscr{R}_{ij}^{(m)}:~i,j=1,\cdots,\kappa,~i\neq j\big\},\\&
\mbox{either}\hspace{2mm}\lim\limits_{m\rightarrow\infty}\xi_{ij}^{(m)}=\xi_{ij}^{(\infty)}\in\mathbb{R}^n\hspace{2mm}\mbox{or}\hspace{2mm}\lim\limits_{m\rightarrow\infty}\xi_{ij}^{(m)}\rightarrow\infty,
\end{aligned}
\right.
\end{equation}
where $\xi_{ij}^{(m)}=\lambda_{i}^{(m)}(\xi_{j}^{(m)}-\xi_{i}^{(m)})$.
\begin{Def}
Let  $\preceq$ be a strict partial
order on a set $\widetilde{T}$, and $\prec$ the corresponding strict partial order on a set $\widetilde{T}$.
\begin{itemize}
\item[$\bullet$]
We say that a partially ordered set $(\widetilde{T},\prec)$ is a tree if for any $\tilde{t}\in\widetilde{T}$ such that the set $\{\tilde{s}\in\widetilde{T}:~\tilde{s}\prec\tilde{t}\}$ is well-ordered by the relation $\prec$.
\item[$\bullet$]
A descendant of $\tilde{s}\in\widetilde{T}$ is any element $\tilde{t}\in\widetilde{T}$ such that $\tilde{s}\prec\tilde{t}$.
\end{itemize}
\end{Def}
The next lemma, proven in \cite{DSW21}.
\begin{lem}
 For any sequence of $\{W_{i}^{(m)}\}$ satisfying \eqref{TTT}, we set a relation
$\prec$ on $I=\{1,\cdots,\kappa\}$ as
$$i\prec j~\Leftrightarrow~i<j\hspace{2mm}\mbox{and}\hspace{2mm}\hspace{2mm}\lim\limits_{m\rightarrow\infty}\xi_{ij}^{(m)}=\xi_{ij}^{(\infty)}\in\mathbb{R}^n.$$
Then $\prec$ is a strict partial order and there exists $\kappa^{\ast}\in\{1,\cdots,\kappa\}$ such that $I$ can be be expressed as a tree.
\end{lem}
For each $i\in I$, we set $\mathscr{H}(i)$ be the set of descendants
of $i\in I$, that is, $\mathscr{H}(i)=\{j\in I:~i\prec j\}$, or equivalently $$\mathscr{H}(i):=\{j\in I:\hspace{2mm} i<j,\hspace{2mm}\lim\limits_{m\rightarrow\infty}\xi_{ij}^{(m)}=\xi_{ij}^{(\infty)}\in\mathbb{R}^n\}.$$
Moreover, we define the sets
$$B(0,\widetilde{A}):=\{x:|z_i^{(m)}|\leq\widetilde{A}\}\hspace{2mm}\mbox{and}\hspace{2mm}B(0,\bar{A}):=\{x:|z_i^{(m)}|\leq\bar{A}\}.$$
We can find the following.
\begin{lem}\label{cll-11-2}
For any $\widetilde{A}>0$. Assume that $n\geq6-\mu$, $\mu\in(0,n)$ and $0<\mu\leq4$, $j\notin\mathscr{H}(i)$. Then we have that, uniformly for $x\in B(0,\widetilde{A})$,
$W_j^{(m)}(x)=o(1)W_i^{(m)}(x)$, and
\begin{equation}\label{zzz-1-2}
s_{j,1}+s_{j,2}=o(1)s_{i,1};\hspace{2mm}t_{j,1}+t_{j,2}=o(1)t_{i,1}\hspace{2mm}\mbox{for} \hspace{2mm} 0<\mu<\frac{n+\mu-2}{2},
\end{equation}
and that, uniformly for $x\in B(0,\bar{A})$, $W_j^{(m)}(x)=o(1)W_i^{(m)}(x)$,
\begin{equation}\label{zzz-2-2}
\hat{s}_{j,1}+\hat{s}_{j,2}=o(1)\hat{s}_{i,1};\hspace{2mm}\hat{t}_{j,1}+\hat{t}_{j,2}=o(1)\hat{t}_{i,1}\hspace{2mm}\mbox{for} \hspace{2mm}\frac{n+\mu-2}{2}\leq\mu<4.
\end{equation}
\end{lem}
\begin{proof}
For $j\notin\mathscr{H}(i)$, we then consider separately the following two cases.\\
$\mathbf{Case\ 1.}$ $j<i$, $\lim\limits_{m\rightarrow\infty}\xi_{ij}^{(m)}$ exists and $\lim\limits_{m\rightarrow\infty}|\xi_{ij}^{(m)}|<\infty$. Noticing that $W_{i}^{(m)}(x)\approx\lambda_{i}^{\frac{n-2}{2}}\tau(z_i)^{2-n}$ and  $W_{j}^{(m)}(x)\approx\lambda_{j}^{\frac{n-2}{2}}\tau(z_j)^{2-n}$. Passing to the limit as $m\rightarrow\infty$ on the set $B(0,\widetilde{A})$ or $B(0,\bar{A})$ we would get
\begin{equation*}
W_{j}^{(m)}(x)\lesssim o(1)W_{i}^{(m)}(x).
\end{equation*}
A direct computation, it follows that
\begin{equation*}
\begin{split}
&s_{j,1}=o(1)s_{i,1},\hspace{2mm}s_{j,2}=o(1)s_{i,1};\hspace{2mm}t_{j,1}=o(1)t_{i,1},\hspace{2mm}t_{j,2}=o(1)t_{i,1} \hspace{2mm}\mbox{for}\hspace{2mm} 0<\mu<\frac{n+\mu-2}{2},\\&
\hat{s}_{j,1}=o(1)\hat{s}_{i,1},\hspace{2mm}\hat{s}_{j,2}=o(1)\hat{s}_{i,1};\hspace{2mm}\hat{t}_{j,1}=o(1)\hat{t}_{i,1},\hspace{2mm}\hat{t}_{j,2}=o(1)\hat{t}_{i,1} \hspace{2mm}\mbox{for} \hspace{2mm} \frac{n+\mu-2}{2}\leq\mu<4.
\end{split}
\end{equation*}
$\mathbf{Case\ 2.}$ $\lim\limits_{m\rightarrow\infty}|\xi_{ij}^{(m)}|=\infty$ and $\mathscr{R}_{ij}=\sqrt{\lambda_j/\lambda_i}|\xi_{ij}^{(m)}|$. On the set  $B(0,\widetilde{A})$ or $B(0,\bar{A})$, $W_{j}^{(m)}(x)\lesssim o(1)W_{i}^{(m)}(x)$ if $m$ is large enough, plus the fact that
$|z_j|\geq\frac{1}{2}\sqrt{\lambda_j/\lambda_i}\mathscr{R}_{ij}$, and
\begin{equation*}
\begin{split}
&s_{j,1}=o(1)s_{i,1},\hspace{2mm}z_{j,2}=o(1)s_{i,1};\hspace{2mm}t_{j,1}=o(1)t_{i,1},\hspace{2mm}t_{j,2}=o(1)t_{i,1} \hspace{2mm}\mbox{for} \hspace{2mm} 0<\mu<\frac{n+\mu-2}{2},\\&
\hat{s}_{j,1}=o(1)\hat{s}_{i,1},\hspace{2mm}\hat{s}_{j,2}=o(1)\hat{s}_{i,1};\hspace{2mm}\hat{t}_{j,1}=o(1)\hat{t}_{i,1},\hspace{2mm}\hat{t}_{j,2}=o(1)\hat{t}_{i,1} \hspace{2mm}\mbox{for} \hspace{2mm} \frac{n+\mu-2}{2}<\mu<4.
\end{split}
\end{equation*}
For $\lambda_j\leq\lambda_i$, one has
\begin{equation*}
\frac{s_{j,2}}{s_{i,1}}=\mathscr{R}^{\frac{n+\mu-6}{2}}\big(\frac{\lambda_j}{\lambda_i}\big)^{\frac{n-2}{2}}
\frac{\tau(z_i)^2}{\tau(z_{j})^{(n+\mu-2)/2}}\rightarrow0, \hspace{2mm}\mbox{for} \hspace{2mm} 0<\mu<\frac{n+\mu-2}{2},
\end{equation*}
\begin{equation*}
\frac{\hat{s}_{j,2}}{\hat{s}_{i,1}}=\mathscr{R}^{n-4-\epsilon_0}\big(\frac{\lambda_j}{\lambda_i}\big)^{\frac{n-2}{2}}
\frac{\tau(z_i)^2}{\tau(z_{j})^{n-2-\epsilon_0}}\rightarrow0, \hspace{2mm}\hspace{2mm}\hspace{2mm}\mbox{for} \hspace{2mm} \frac{n+\mu-2}{2}\leq\mu<4.
\end{equation*}
If $\lambda_j>\lambda_i$, the support of $s_{j,1}$ does not intersect $B(0,\widetilde{A})$ and the support of $\hat{s}_{j,1}$ does not intersect $B(0,\bar{A})$. Thus, $t_{j,1}=o(1)t_{i,1},\hspace{2mm}t_{j,2}=o(1)t_{i,1}$ and $\hat{t}_{j,1}=o(1)\hat{t}_{i,1},\hspace{2mm}\hat{t}_{j,2}=o(1)\hat{t}_{i,1}$  for $0<\mu<4$. Thus we conclude the proof.
\end{proof}
As a consequence of this Lemma, we find the following.
\begin{lem}\label{cll-11-3}
 Assume that $n\geq6-\mu$, $\mu\in(0,n)$ and $0<\mu\leq4$, $j\notin\mathscr{H}(i)$. Then we have that
\begin{equation}\label{zzz-2}
\sum\limits_{j=1}^{\kappa}W_j^{(m)}=\sum\limits_{j\in\mathscr{H}(i)}W_j^{(m)}+(1+o(1))W_i^{(m)},\hspace{3mm}\mbox{uniformly for}\hspace{2mm}x\in B(0,\widetilde{A})\hspace{2mm}\mbox{or}\hspace{2mm}x\in B(0,\bar{A}),
\end{equation}
and that uniformly for $x\in B(0,\widetilde{A})$,
in the case $0<\mu<\frac{n+\mu-2}{2}$,
\begin{equation}\label{zzz-8-1}
\left\lbrace
\begin{aligned}
&S_1(x)=\sum\limits_{j\in\mathscr{H}(i)}(s_{j,1}(x)+s_{j,2}(x))+(1+o(1))s_{i,1}(x),\\
&T_1(x)=\sum\limits_{j\in\mathscr{H}(i)}(t_{j,1}(x)+t_{j,2}(x))+(1+o(1))t_{i,1}(x),
\end{aligned}
\right.
\end{equation}
and that uniformly for $x\in B(0,\bar{A})$,
in the case $\frac{n+\mu-2}{2}\leq\mu<4$,
\begin{equation}\label{zzz-9-2}
\left\lbrace
\begin{aligned}
&S_2(x)=\sum\limits_{j\in\mathscr{H}(i)}(\hat{s}_{j,1}(x)+\hat{s}_{j,2}(x))+(1+o(1))\hat{s}_{i,1}(x),\\
&T_2(x)=\sum\limits_{j\in\mathscr{H}(i)}(\hat{t}_{j,1}(x)+\hat{t}_{j,2}(x))+(1+o(1))\hat{t}_{i,1}(x).
\end{aligned}
\right.
\end{equation}
\end{lem}
To prove the main results in Lemma \ref{estimate2}, the arguments depend a lot on the nondegeneracy property of the positive solutions of equation \eqref{ele-1.1}.
It is a key ingredient in the stability analysis of functional inequality and Lyapunov-Schmidt reduction method of constructing  blow-up solutions of the equation (cf. \cite{MUSSO1, Manuel, MUSSO2, MUSSO3, DeMuS-pp}).
\begin{Prop}[Non-degeneracy\cite{GMYZ22,XLi}]\label{prondgr}
Assume $n\geq 3$, $0<\mu\leq n$ and $0<\mu\leq4$. Then the solution
 $ W[\xi,\lambda](x)$ of the equation
\begin{equation*}
    -\Delta u=(|x|^{-\mu}\ast u^{p})u^{p-1} \quad \mbox{in}\quad \mathbb{R}^n.
    \end{equation*}
is non-degenerate in the sense that all bounded solutions of linearized equation
\begin{equation*}
-\Delta v-p\left(|x|^{-\mu} \ast (W^{p-1}\phi)\right)W^{p-1}
-(p-1)\left(|x|^{-\mu} \ast W^{p}\right)W^{p-2}\phi=0
\end{equation*}
are the linear combination of the functions
\begin{equation*}
\partial_\lambda W[\xi,\lambda],\hspace{2mm}\partial_{\xi_1}W[\xi,\lambda],\cdots,\partial_{\xi_n}W[\xi,\lambda].
\end{equation*}
\end{Prop}
\subsection{Proof of the Lemma \ref{estimate2}.}

For each $i\in I=\left\{1,\cdots,\kappa\right\}$, recall that
	\begin{equation*}		z_i^{\left(m\right)}=\lambda_i^{\left(m\right)}(x-\xi_i^{\left(m\right)}),\hspace{2mm} \mbox{and}\hspace{2mm} \xi_{ij}^{\left(m\right)}
=\lambda_i^{\left(m\right)}\left(\xi_j^{\left(m\right)}-\xi_i^{\left(m\right)}\right), ~~~\mbox{for all}~ j\neq i.
	\end{equation*}
	and we can assume that
\begin{equation}\label{AA-00}
C^{\star}:=1+\max\limits_{i,j\in I, m\geq0}\{|\xi_{ij}^{(m)}|:\hspace{2mm}i<j\hspace{2mm}\mbox{and}\hspace{2mm}\lim\limits_{m\rightarrow\infty}\xi_{ij}^{(m)}\in\mathbb{R}^n\}<\infty.
\end{equation}
\textbf{Decomposition of $\mathbb{R}^n$.}
Let we define the sets
	\begin{equation*}
		\begin{split}
			\Omega^{\left(m\right)}
			:=\mathop{\bigcup}\limits_{i\in I}\left\{x:\vert y_i^{\left(m\right)}\vert\leq M\right\}\hspace{2mm}\mbox{and}\hspace{2mm}
			\Omega_i^{\left(m\right)}
			:=\left\{x:\vert y_i^{\left(m\right)}\vert\leq M,
\vert y_i^{\left(m\right)}-\xi_{ij}^{\left(m\right)}\vert\geq\varepsilon_0,\forall j\in \mathscr{H}(i)\right\},
		\end{split}
	\end{equation*}
	where $M=M(n,\mu,\kappa)>0$ is a large constant and $\varepsilon_0=\varepsilon_0(n,\mu,\kappa)>0$ is a small constant to be determined later. Then decomposition of $\mathbb{R}^n$ given by
$$\mathbb{R}^n=\mathcal{C}_{ext,M,1}+\mathcal{C}_{core,M,2}+\mathcal{C}_{neck,M,3}$$
with
$$\mathcal{C}_{ext,M,1}:=\bigcap\limits_{i\in I}\left\{x:\vert y_i^{\left(m\right)}\vert>M\right\},\hspace{2mm}\mathcal{C}_{core,M,2}:=\bigcup\limits_{i\in I}\Omega_i^{\left(m\right)}\hspace{2mm}\mbox{and}\hspace{2mm} \mathcal{C}_{neck,M,3}:=\bigcup\limits_{i\in I}D_i^{(m)},$$
where $D_i^{(m)}$ is given by
	\begin{equation*}
		D_i^{(m)}:=\mathop{\bigcup}\limits_{j\in \mathscr{H}(i)}
		\left\{x:
		\vert z_i^{(m)}-\xi_{ij}^{(m)}\vert\leq\varepsilon_0
		\right\}\setminus\Big(\mathop{\bigcup}\limits_{j\in \mathscr{H}(i)}\left\{x:\vert z_j^{(m)}\vert< M\right\}\Big).
	\end{equation*}
\emph{Conclusion of the proof of Lemma \ref{estimate2}}.
We consider two cases:
\begin{itemize}
\item[$(i)$] $n=5$ and $\mu\in[1,3)$, or $n=6$ and $\mu\in(0,4)$, or $n=7$ and $\mu\in(\frac{7}{3},4)$.
\item[$(ii)$] $n=4$ and $\mu\in[2,4)$, or $n=5$ and $\mu\in[3,4)$.
\end{itemize}
Since the proofs of the two case are very similar, we will give a
detailed proof of case $(i)$, and indicate the necessary changes when
proving case $(ii)$. From now on we will concentrate on case $(i)$ and we follow the contradiction-based argument outlined in \cite{DSW21}.

Assume that the conclusion of Lemma \ref{estimate2} does not hold true, in other words that, up to a subsequence, there exists a sequence of functions
$\hbar=\hbar_m$ with $\Vert \hbar_m\Vert_{\ast\ast}\rightarrow0$ as $m\rightarrow\infty$, and $\phi=\phi_m$ with $\Vert\phi_m\Vert_\ast=1$ solving the equation
	\begin{equation}\label{00}
	\left\{\begin{array}{l}
		\displaystyle \Delta \phi_m+\Phi_{n,\mu}[\sigma_m,\phi_m]
		=\hbar_m\hspace{4.14mm}\mbox{in}\hspace{1.14mm} \mathbb{R}^n,\\
		\displaystyle \int\Phi_{n,\mu}[W_{i},\Xi^{a}_i]\phi_m=0,\hspace{4mm}i=1,\cdots, \kappa; ~a=1,\cdots,n+1,
	\end{array}
	\right.
\end{equation}
and $1/m$-interacting bubbles $\big\{W_i^{\left(m\right)}\big\}_{i=1}^{\infty},$
	For simplicity, here and after we denote $W_i:=W_i^{\left(m\right)}$ and $\sigma_m:=\sum_{i=1}^\kappa W_i^{\left(m\right)}$. In order to have the desired contradiction with $\|\phi_{m}\|_{\ast}=1$ for any $m$, we divide the proof of Lemma \ref{estimate2} into several steps. At a first step, we claim that the following holds true.
\begin{step}\label{step5.1}
 When $\left\{x_m\right\}\subset\mathcal{C}_{ext,M,1}$. For $n=5$ and $\mu\in[1,3)$, or $n=6$ and $\mu\in(0,4)$, or $n=7$ and $\mu\in(\frac{7}{3},4)$,  we have that
\begin{equation}\label{n-m-1}
\phi_{m}(x)<S_1(x)\hspace{2mm}\mbox{as}\hspace{2mm}m\rightarrow+\infty.
\end{equation}
For $n=4\hspace{2mm}\mbox{and}\hspace{2mm}\mu\in[2,4),\mbox{or}\hspace{2mm}n=5\hspace{2mm}\mbox{and}\hspace{2mm}\mu\in(3,4]$, we have that
\begin{equation}\label{n-m-1-1}
\phi_{m}(x)<S_2(x)\hspace{2mm}\mbox{as}\hspace{2mm}m\rightarrow+\infty.
\end{equation}
\end{step}
\emph{Proof of step \ref{step5.1}}.
We note that
\begin{equation}\label{s-tital}
\begin{split}
&\tilde{\mathbf{s}}_{i,in}\leq2M^{-2}s_{i,1},\hspace{2mm}\tilde{\mathbf{s}}_{i,out}\leq2M^{-\frac{n-\mu-2}{2}}s_{i,2},
\hspace{2mm}\mbox{for}\hspace{2mm} 0<\mu<\frac{n+\mu-2}{2},\\&
\bar{\mathbf{s}}_{i,in}\leq2M^{-2}\hat{s}_{i,1},\hspace{2mm}\bar{\mathbf{s}}_{i,out}\leq2M^{-\epsilon_0}\hat{s}_{i,2},
\hspace{2mm}\mbox{for} \hspace{2mm} \frac{n+\mu-2}{2}\leq\mu<4.
\end{split}
\end{equation}
Using Lemma \ref{cll-1}-(i), there exists $C_{n,\mu,\kappa}$ such that
 \begin{equation*}
 \begin{split}
|\phi(x)|S_{1}^{-1}(x)&\leq
 C_{n,\mu,\kappa}
\Big(\|\hbar\|_{\ast\ast}+\tilde{A}^{n+\mu-4+\frac{n+\mu-2}{2}}\widetilde{S}_1(x)S_1^{-1}(x)\|\phi\|_{\ast}\Big)
 \\&+ C_{n,\mu,\kappa}
\left\lbrace
\begin{aligned}
&(\mathscr{R}^{-\frac{(\mu+1)}{2}}\tilde{A}^{n-\mu+2}+\tilde{A}^{-2}+\tilde{A}^{-\varpi^{\star}_1})\|\phi\|_{\ast}, \hspace{8mm}\hspace{8mm}\hspace{6mm}n=5\hspace{2mm}\mbox{and}\hspace{2mm}\mu\in[1,3),\\&
(\mathscr{R}^{-\min\{\frac{\mu}{2},2p-6+\mu-\theta_1\}}\tilde{A}^{n-\mu+2}+\tilde{A}^{-2}+\tilde{A}^{-\varpi^{\star}_1})\|\phi\|_{\ast},
\hspace{1.5mm}\hspace{1mm}n=6\hspace{2mm}\mbox{and}\hspace{2mm}\mu\in(0,4),\\&
(\mathscr{R}^{-\min\{\frac{\mu}{2},2p-7+\mu-\theta_2\}}\tilde{A}^{n-\mu+2}+\tilde{A}^{-2}+\tilde{A}^{-\varpi^{\star}_1})\|\phi\|_{\ast},
\hspace{1.5mm}\hspace{1mm}n=7\hspace{2mm}\mbox{and}\hspace{2mm}\mu\in(\frac{7}{3},4),
\end{aligned}
\right.
\end{split}
\end{equation*}
Then, choosing $M=M_{n,\mu,\kappa,\bar{A}}$ sufficiently large such that
$$2C_{n,\mu,\kappa}\tilde{A}^{n+\mu-4+\frac{n+\mu-2}{2}}M^{-\frac{n-\mu-2}{2}}\leq(100\kappa)^{-1},$$ we obtain
$$\tilde{A}^{n+\mu-4+\frac{n+\mu-2}{2}}\widetilde{S}_1(x)<(100\kappa)^{-1}S_1(x),\hspace{2mm}\mbox{on}\hspace{2mm}\mathcal{C}_{ext,M,1}.$$
Moreover, choosing $\tilde{A}=\tilde{A}_{n,\mu,\kappa}$ sufficiently large such that $C_{n,\mu,\kappa}(\tilde{A}^{-2}+\tilde{A}^{-\varpi^{\star}_1}))<(100\kappa)^{-1}$,
so that we eventually have $\|\phi_{m}(x)|S_{1}^{-1}(x)\leq o(1)+(50\kappa)^{-1}$
as $m\rightarrow+\infty$, as desired. The proof of \eqref{n-m-1-1} for $n=4\hspace{2mm}\mbox{and}\hspace{2mm}\mu\in[2,4),\mbox{or}\hspace{2mm}n=5\hspace{2mm}\mbox{and}\hspace{2mm}\mu\in(3,4]$ is completely analogous by Lemma \ref{cll-1}-(ii) and \eqref{s-tital} except some minor modifications. So we omit it. \qed

Let the blow-up sequences $\bar{\phi}_{m}^{(l)}$, $\bar{\mathscr{R}}_m^{(l)}$, $\bar{\sigma}_{m}^{(l)}(l=1,2)$ be given by
	\begin{equation*}
		\left\lbrace
\begin{aligned}
&\bar{\phi}_m^{(l)}(z):=\phi_m(z/\lambda_{i_0}^{(m)}+\xi_{i_0j}^{(m)})S_l^{-1}(x_m); \hspace{2mm}      \bar{\hbar}_m^{(l)}(z):=(\lambda_{i_0}^{(m)})^{-2}\hbar_m(z/\lambda_{i_0}^{(m)}+\xi_{i_0j}^{(m)})S_l^{-1}(x_m);     \\
&\bar{\sigma}_m^{(l)}(z):=W[0,1](z)+\sum_{j\in I\setminus\{i_0\}}W[\xi_{i_0j}^{(m)},\lambda_{j}^{(m)}/\lambda_{i_0}^{(m)}](z),
		\end{aligned}
\right.
\end{equation*}
	with $z=\lambda_{i_0}^{(m)}(x-\xi_{i_0j}^{(m)})$, so that $\bar{\phi}_m^{(l)}$ satisfies
	\begin{equation}\label{fai}
	\left\{\begin{array}{l}
		\displaystyle \Delta \bar{\phi}_m^{(l)}(z)+\Phi_{n,\mu}[\bar{\sigma}_m^{(l)}(\bar{z}),\bar{\phi}_m^{(l)}(\bar{z})]
		=\bar{\hbar}_m^{(l)}(z)\hspace{4.14mm}\mbox{in}\hspace{1.14mm} \mathbb{R}^n,\\
		\displaystyle \int\Big[(p-1)\Big(|z|^{-\mu}\ast W_{i}^{p}(\bar{z})\Big)W_{i}^{p-2}(z)\bar{\phi}_m^{(l)}(z)\Xi^{a}+p\Big(|x|^{-\mu}\ast W_{i}^{p-1}(\tilde{z})\Xi^{a}\Big)W_{i}^{p-1}(z)\bar{\phi}_m^{(l)}(z)\Big]=0,\\i=1,\cdots, \kappa; ~a=1,\cdots,n+1,
	\end{array}
	\right.
\end{equation}
with $\bar{z}=\lambda_{i_0}^{(m)}(\tilde{z}-\xi_{i_0j}^{(m)})$ and
$$\Phi_{n,\mu}[\bar{\sigma}_m^{(l)}(\bar{z}),\bar{\phi}_m^{(l)}(\bar{z})]:=p\Big(|z|^{-\mu}\ast \bar{\sigma}_{m}^{p-1}(\bar{z})\bar{\phi}_m^{(l)}(\bar{z})\Big)
\bar{\sigma}^{p-1}(z)+(p-1)\Big(|z|^{-\mu}\ast\bar{\sigma}_m^{p}(\bar{z})\Big)
\bar{\sigma}_{m}^{p-2}(z)\bar{\phi}_m^{(l)}(z).$$
Here $W=W[0,1](z)$ and $\Xi^a=\Xi^a(z)=\Xi^a(z)$. Let $\xi_{i_0j}^{(\infty)}=\lim_{m\rightarrow\infty}\xi_{i_0j}^{\left(m\right)}$ and define
	\begin{equation}\label{nongfushanquan}
		\begin{split}
			&E_1:=\mathop{\bigcap}\limits_{j\in\mathscr{H}(i)}
			\left\{z:\vert z-\xi_{i_0j}^{(\infty)}\vert\geq1/L\right\}.
		\end{split}
	\end{equation}
Let $K_L:=\left\{z:\vert z\vert\leq L\right\}\cap E_1$. If we choose $L\geq2\max\left\{M,\varepsilon^{-1}\right\}$, it is easy to see that $\Omega_i^{\left(m\right)}\subset\subset K_L$ for $m$ large enough.	
\begin{step}\label{step5.2}
 When $\left\{x_m\right\}\subset\mathcal{C}_{core,M,2}$. For $n=5$ and $\mu\in[1,3)$, or $n=6$ and $\mu\in(0,4)$, or $n=7$ and $\mu\in(\frac{7}{3},4)$,  we have that	
\begin{equation}\label{n-m-2}
|\phi_{m}^{(1)}|(x)=o(1)S_1(x)\hspace{4mm}
\end{equation}
$\mbox{as}\hspace{2mm}m\hspace{2mm}\mbox{large enough}$. For $n=4\hspace{2mm}\mbox{and}\hspace{2mm}\mu\in[2,4),\mbox{or}\hspace{2mm}n=5\hspace{2mm}\mbox{and}\hspace{2mm}\mu\in(3,4]$, we have that
\begin{equation}\label{n-m-2-0}
|\phi_{m}^{(2)}|(x)=o(1)S_2(x)\hspace{4mm}\hspace{2mm}\mbox{as}\hspace{2mm}m\rightarrow+\infty.
\end{equation}
\end{step}
To prove estimate \eqref{n-m-2} by contradiction, up to a subsequence, we choose $\epsilon_1>0$ and $x_m\in\mathcal{C}_{core,M,2}$ such that
\begin{equation}\label{n-m-3}
|\phi_{m}|(x_m)>\epsilon_1S_1(x_m),
\end{equation}
At this point it essentially remains to prove the following result.
\begin{Prop}\label{diyigemingti}
Assume that $n=5$ and $\mu\in[1,3)$, or $n=6$ and $\mu\in(0,4)$, or $n=7$ and $\mu\in(\frac{7}{3},4)$.
In each compact subset $K_L$, we have the following estimate, uniformly for $z\in K_L$
\begin{equation}\label{KM-0}
			\bar{\sigma}_m^{(1)}(z)\rightarrow W[0,1](z)
	\hspace{2mm}\mbox{as}\hspace{2mm} m\rightarrow\infty.
		\end{equation}
 Moreover we have that
		\begin{equation}\label{km-1}
			\vert\bar{\phi}_m^{(1)}\vert(z)\lesssim\sum_{j\in\mathscr{H}(i_0) }\big(\frac{M}{|z-\xi_{i_0j}^{(\infty)}|}\big)^{\frac{n+\mu-2}{2}}+M^2,\hspace{2mm}\mbox{for} \hspace{2mm}z\in K_L,
		\end{equation}
\begin{equation}\label{km-2}
\begin{split}
\vert\bar{\hbar}_m^{(1)}\vert(z)\rightarrow0,\hspace{2mm}\mbox{as}\hspace{2mm} m\rightarrow\infty,\hspace{2mm}\mbox{for}\hspace{2mm}z\in K_L.
\end{split}
		\end{equation}
The above estimates \eqref{KM-0} and \eqref{km-2} also hold for $\bar{\sigma}_m^{(2)}(z)$ and $\vert\bar{\hbar}_m^{(2)}\vert(z)$ in dimension $n=4\hspace{2mm}\mbox{and}\hspace{2mm}\mu\in[2,4),\mbox{or}\hspace{2mm}n=5\hspace{2mm}\mbox{and}\hspace{2mm}\mu\in(3,4]$. Furthermore, we have that
\begin{equation}\label{km-1-kl}
\begin{split}
			\vert\bar{\phi}_m^{(2)}\vert(z)&\lesssim\sum_{j\in\mathscr{H}(i_0) }\big(\frac{M}{|z-\xi_{i_0j}^{(\infty)}|}\big)^{n-2-\epsilon_0}+M^2,\hspace{2mm}\mbox{for} \hspace{2mm}z\in K_L.
\end{split}
		\end{equation}
\end{Prop}

\emph{Proof of Proposition \ref{diyigemingti}}.
Thanks to Lemma \ref{cll-11-3}, to conclude the proof of \eqref{KM-0}, it sufficient to obtain
$$\sum_{j\in I\setminus\{i_0\}}W[\xi_{i_0j}^{(m)},\lambda_{j}^{(m)}/\lambda_{i_0}^{(m)}](z)=o(1)W[0,1](z),\hspace{2mm}\mbox{for}\hspace{2mm}z\in K_{L},\hspace{2mm}j\in\mathscr{H}(i_0).$$
In the case $j\in\mathscr{H}(i_0)$, one must have $\lambda_{j}^{(m)}/\lambda_{i_0}^{(m)}\rightarrow\infty$ as $m\rightarrow\infty$.
Then together with $z\in K_L$, we get that
$$\frac{W[\xi_{i_0j}^{(m)},\lambda_{j}^{(m)}/\lambda_{i_0}^{(m)}]}{W[0,1]}\leq(\lambda_{j}^{(m)}/\lambda_{i_0}^{(m)})^{(2-n)/2}(2L)^{2n-4}\rightarrow0\hspace{2mm}\mbox{as}\hspace{2mm}m\rightarrow\infty.$$
as desired.

First, for all $A_1,\cdots,A_\kappa>0$ and $B_1,\cdots,B_\kappa>0$, we notice that
\begin{equation}\label{AB-0}
(\sum_{i=1}^{\kappa}B_i)^{-1}\sum_{i=1}^{\kappa}A_i\leq\max\Big\{\frac{A_i}{B_i},\cdots,\frac{A_\kappa}{B_\kappa}\Big\}\leq\sum_{i=1}^{\kappa}B_i^{-1}A_i.
\end{equation}
Then we get from \eqref{zzz-8-1}-\eqref{zzz-9-2} that, for $0<\mu<\frac{n+\mu-2}{2}$,
\begin{equation}\label{S-1}
\begin{split}
&\frac{S_1(z/\lambda_{i_0}^{(m)}+\xi_{i_0j}^{(m)})}{S_1(x_m)}\leq\frac{\sum_{j\in\mathscr{H}(i_0)}^{\kappa}(s_{j,1}(x)+s_{j,2}(x))+s_{i_0,1}(x)}{\sum_{j\in\mathscr{H}(i_0)}^{\kappa}(s_{j,1}(x_{m})+s_{j,2}(x_{m}))+s_{i_0,1}(x_{m})}
\\&\lesssim\frac{s_{i_0,1}(x)}{s_{i_0,1}(x_{m})}+\sum_{j\in\mathscr{H}(i_0)}^{\kappa}\bigg[\frac{s_{j,1}(x)}{s_{j,1}(x_{m})}_{\big\{\big|\lambda_{j}^{(m)}/\lambda_{i_0}^{(m)}(z-\xi_{i_0j}^{(\infty)})\big|<\mathscr{R}_m\big\}}
+{\frac{s_{j,2}(x)}{s_{j,2}(x_{m})}}_{\big\{\big|\lambda_{j}^{(m)}/\lambda_{i_0}^{(m)}(z-\xi_{i_0j}^{(\infty)})\big|\geq\mathscr{R}_m\big\}}\bigg],
\end{split}
\end{equation}
\begin{equation}\label{T-1}
\begin{split}
&\frac{T_1(z/\lambda_{i_0}^{(m)}+\xi_{i_0j}^{(m)})}{(\lambda_{i_0}^{(m)})^2S_1(x_m)}\lesssim\frac{(\lambda_{i_0}^{(m)})^{-2}\sum_{j\in\mathscr{H}(i_0)}^{\kappa}(t_{j,1}(x)+t_{j,2}(x))+t_{i_0,1}(x)
}{\sum_{j\in\mathscr{H}(i_0)}^{\kappa}(s_{j,1}(x_{m})+s_{j,2}(x_{m}))+s_{i_0,1}(x_{m})}
\lesssim\frac{(\lambda_{i_0}^{(m)})^{-2}t_{i_0,1}(x)}{s_{i_0,1}(x_{m})}\\&+\sum_{j\in\mathscr{H}(i_0)}^{\kappa}\bigg[\frac{(\lambda_{i_0}^{(m)})^{-2}t_{j,1}(x)}{s_{j,1}(x_{m})}_{\big\{\big|\lambda_{j}^{(m)}/\lambda_{i_0}^{(m)}(z-\xi_{i_0j}^{(\infty)})\big|<\mathscr{R}_m\big\}}
+\frac{(\lambda_{i_0}^{(m)})^{-2}t_{j,2}(x)}{s_{j,2}(x_{m})}_{\big\{\big|\lambda_{j}^{(m)}/\lambda_{i_0}^{(m)}(z-\xi_{i_0j}^{(\infty)})\big|\geq\mathscr{R}_m\big\}}\bigg].     \end{split}
\end{equation}
Here and several times in the sequel, we use the elementary inequality \eqref{AB-0}.
Note that, there exists a large number $m_L\in\mathbb{N}$ such that
\begin{equation*}
z\in K_L,~j\in\mathscr{H}(i_0),~m\geq m_L \Rightarrow
 \big|\lambda_{j}^{(m)}/\lambda_{i_0}^{(m)}(z-\xi_{i_0j}^{(\infty)})\big|\gg\mathscr{R}_m.
\end{equation*}
Therefore, in order to conclude the proof of \eqref{km-1}-\eqref{km-2}, we establish now the following key estimates:
\begin{equation*}
\frac{s_{i_0,1}(x)}{s_{i_0,1}(x_{m})}= \frac{\tau(\zeta_m)^2}{\tau(z)^2}\leq1+M^2,\hspace{2mm}\frac{(\lambda_{i_0}^{(m)})^{-2}v_{i_0,1}(x)}{z_{i_0,1}(x_{m})}=\frac{\tau(\zeta_m)^2}{\tau(z)^4}\leq1+M^2,\hspace{2mm} \hspace{2mm}0<\mu<\frac{n+\mu-2}{2},
\end{equation*}
where $z=\lambda_{i_0}(x-\xi_{i_0})$ and $\zeta_m=\lambda_{i_0}(x_m-\xi_{i_0})$.
In the case $i_0<j$, $\lim\limits_{m\rightarrow\infty}\xi_{ij}^{(m)}$ exists, one have $\lambda_{j}^{(m)}/\lambda_{i_0}^{(m)}\rightarrow\infty$ as $m\rightarrow\infty$. We from \eqref{AA-00} and $|\zeta_{m}|\leq M$ that
\begin{equation*}
\frac{s_{j,2}(x)}{s_{j,2}(x_{m})}= \frac{\tau(\lambda_j(x_m-\xi_j))^{\frac{n+\mu-2}{2}}}{\tau(\lambda_j(x-\xi_j))^{\frac{n+\mu-2}{2}}}\leq\frac{|\zeta_m-\xi_{i_0j}|^{\frac{n+\mu-2}{2}}}{|z-\xi_{i_0j}|^{\frac{n+\mu-2}{2}}}
\leq \big(\frac{M}{|z-\xi_{i_0j}|}\big)^{\frac{n+\mu-2}{2}},\hspace{1mm} \hspace{1mm}0<\mu<\frac{n+\mu-2}{2},
\end{equation*}
\begin{equation*}
\frac{(\lambda_{i_0}^{(m)})^{-2}t_{j,2}(x)}{s_{j,2}(x_{m})}=
\frac{\tau(\lambda_j(x_m-\xi_j))^{\frac{n+\mu-2}{2}}}{\tau(\lambda_j(x-\xi_j))^{\frac{n+\mu+2}{2}}}
\leq\frac{|\zeta_m-\xi_{i_0j}|^{\frac{n+\mu-2}{2}}}{|z-\xi_{i_0j}|^{\frac{n+\mu+2}{2}}}
\leq M^{\frac{n+\mu+2}{2}}L^{\frac{n+\mu-2}{2}},\hspace{1mm} 0<\mu<\frac{n+\mu-2}{2}.
\end{equation*}
Combining these estimates, we conclude.
In the case $n=4\hspace{2mm}\mbox{and}\hspace{2mm}\mu\in[2,4),\mbox{or}\hspace{2mm}n=5\hspace{2mm}\mbox{and}\hspace{2mm}\mu\in(3,4]$,
the conclusion follows directly by computations similar to the above case and we omit it. Proposition \ref{diyigemingti} is proven. \qed

Next, we conclude the proof of step \ref{step5.2} by proving the following useful convergence results in Proposition \ref{converges-0} and Proposition \ref{converges-4}, and the removability of singularities of a solution in Proposition \ref{converges-3}. Here we only give the proof of the statement for case $(i)$, and by same argument, it is easy to check that the results hold for case $(ii)$.

Recall that $\bar{\phi}_m^{(1)}(z):=\phi_m(z/\lambda_{i_0}^{(m)}+\xi_{i_0j}^{(m)})S_1^{-1}(x_m)$. Using \eqref{fai} and the standard elliptic regularity, one can prove that
\begin{equation}\label{1-fai}
\sup\limits_{m\in\mathbb{N}}\|\bar{\phi}_m^{(1)}\|_{C^{1}(\bar{B})}\leq C
\end{equation}
for any open ball $B\subset K_{L}$. Therefore, by the diagonal argument, up to subsequence, we have
\begin{equation}\label{loc}
\bar{\phi}_m^{(1)}\rightarrow\bar{\phi}^{(1)}_{\infty}\hspace{2mm}\mbox{in}\hspace{2mm} C_{loc}^{0}\big(\mathbb{R}^n\setminus\{\xi_{i_0j}^{(\infty)}:j\in\mathscr{H}(i_0)\}\big)
	\hspace{2mm}\mbox{as}\hspace{2mm} m\rightarrow\infty
\end{equation}
for some function $\bar{\phi}^{(1)}_{\infty}$. Furthermore, combining \eqref{km-1} gives
$$\vert\bar{\phi}^{(1)}_{\infty}(z)\vert
\lesssim\sum_{j\in\mathscr{H}(i_0)}^{\kappa}\big(\frac{M}{|z-\xi_{i_0j}^{(\infty)}|}\big)^{\frac{n+\mu-2}{2}}+M^2,                  ~\mbox{in}~\mathbb{R}^n\setminus\{\xi_{i_0j}^{(\infty)}:j\in\mathscr{H}(i_0)\},
$$
As the consequence of \eqref{fai} and \eqref{1-fai}, we have the following result.

\begin{Prop}\label{converges-0}
It holds
\begin{equation}\label{converges-1}					\Delta\bar{\phi}^{(1)}_{\infty}+\Phi_{n,\mu}[W[0,1],\bar{\phi}^{(1)}_{\infty}]
		=0,    \text{in}~\mathbb{R}^n\setminus\{\xi_{i_0j}^{(\infty)}:j\in\mathscr{H}(i_0)\}
		\end{equation}
and
\begin{equation}\label{converges-2}
\displaystyle \int \Phi_{n,\mu}[W[0,1],\Xi^{a}[0,1]]\bar{\phi}^{(1)}_{\infty}=0,\hspace{2mm}\mbox{for all}\hspace{2mm}i=1,\cdots,\kappa\hspace{2mm}\mbox{and}\hspace{2mm}a=1,\cdots,n+1
\end{equation}
\end{Prop}
The proof will make us of the following lemma, which is proven similar to Lemma \ref{B4}.
\begin{lem}\label{5-7}
For any $0<\mu<4$, and $0<\delta^{\prime}<2$ there is a constant $C>0$, such that
$$\int_{\mathbb{R}^n}\frac{1}{|\omega-y|^\mu}\frac{1}{(1+|y|)^{n-\mu+2}}dy\leq\frac{C}{(1+|\omega|)^{\delta^{\prime}}},~\mu\neq2; \hspace{2mm}(\frac{C\log|\omega|}{(1+|\omega|)^2}),~\mu=2.$$
\end{lem}
\emph{Proof of Proposition \ref{converges-0}.}
Fix any $z\in B^{\prime}\subset K_L\subset\mathbb{R}^n\setminus\{\xi_{i_0j}^{(\infty)}:j\in\mathscr{H}(i_0)\}$. Then from \eqref{S-1} get that
\begin{equation}\label{S2}
S_1(z/\lambda_{i_0}^{(m)}+\xi_{i_0j}^{(m)})S_1^{-1}(x_m)\lesssim1.
\end{equation}
Note that \eqref{fai}, we have
$$\bar{\phi}_m^{(1)}=\int\frac{1}{|z-\omega|^{n-2}}\big(\Phi_{n,\mu}[\bar{\sigma}_m^{(1)}(\omega),\bar{\phi}_m^{(1)}(\omega)]
		-\bar{\hbar}_m^{(1)}(\omega)\big)d\omega.$$
Thanks to the assume $\|\hbar_{m}\|_{\ast\ast}\rightarrow0$, then as $m\rightarrow\infty$ we have
$$\int\frac{1}{|z-\omega|^{n-2}}\bar{\hbar}_m^{(1)}(\omega)d\omega\rightarrow0.$$
We show that
\begin{equation}\label{2-fai}
\lim\limits_{m\rightarrow\infty}\int\frac{1}{|z-\omega|^{n-2}}\Phi_{n,\mu}[\bar{\sigma}_m^{(1)}(\omega),\bar{\phi}_m^{(1)}(\omega)]d\omega=\int\frac{1}{|z-\omega|^{n-2}}\Phi_{n,\mu}[W[0,1],\bar{\phi}^{(1)}_{\infty}]d\omega
\end{equation}
Given $\widetilde{A}>4C^{\star}$ and $\epsilon_0>0$, we decompose $\mathbb{R}^n$ as follow:
\begin{equation*}
\begin{split}
\mathbb{R}^n&=\big(B(0,\widetilde{A})\cap\big(\cup_{i\in\mathscr{H}(i_0)}B(\xi_{i_0j}^{(\infty)},\epsilon_0)\big)\big)
\bigcup\big(B(0,\widetilde{A})\setminus\big(\cup_{i\in\mathscr{H}(i_0)}B(\xi_{i_0j}^{(\infty)},\epsilon_0)\big)\big)
\bigcup B(0,\widetilde{A})^{c}\\&
:=\Omega_1\cup\Omega_2\cup\Omega_3.
\end{split}
\end{equation*}
To prove \eqref{2-fai}, we set
\begin{equation*}
\begin{split}
F_1&=p\int_{\mathbb{R}^n\times\mathbb{R}^n}\frac{W^{p-1}(\omega)}{|z-\omega|^{n-2}}\frac{W^{p-1}(y)\bar{\phi}^{(1)}_{m}(y)}{|\omega-y|^\mu}d\omega dy+(p-1)\int_{\mathbb{R}^n\times\mathbb{R}^n}\frac{W^{p-2}(\omega)\bar{\phi}^{(1)}_{m}(\omega)}{|z-\omega|^{n-2}}\frac{W^{p}(y)}{|\omega-y|^\mu}d\omega dy\\&
=p\Big(\int_{\mathbb{R}^n\times\Omega_1}\cdots+\int_{\mathbb{R}^n\times\Omega_2}\cdots+\int_{\mathbb{R}^n\times\Omega_3}\cdots\Big)+
(p-1)\Big(\int_{\Omega_1\times\mathbb{R}^n}\cdots+\int_{\Omega_2\times\mathbb{R}^n}\cdots+\int_{\Omega_3\times\mathbb{R}^n}\cdots\Big)\\&
:=p\Big(F_{11}+F_{12}+F_{13}\Big)+(p-1)\Big(F^{\prime}_{11}+F^{\prime}_{12}+F^{\prime}_{13}\Big).
\end{split}
\end{equation*}
\begin{equation*}
\begin{split}
F_2&=p\int_{\mathbb{R}^n}\frac{1}{|z-\omega|^{n-2}}\Big[\Big(|\omega|^{-\mu}\ast(\bar{\sigma}_m^{(1)})^{p-1}\bar{\phi}^{(1)}_{m}\Big)(\bar{\sigma}_m^{(1)})^{p-1}
-\Big(|\omega|^{-\mu}\ast W^{p-1}\bar{\phi}^{(1)}_{m}\Big)W^{p-1}\Big]d\omega \\&
+(p-1)\int_{\mathbb{R}^n}\frac{1}{|z-\omega|^{n-2}}\bigg\{\Big[\Big(|\omega|^{-\mu}\ast(\bar{\sigma}_m^{(1)})^{p}\Big)(\bar{\sigma}_m^{(1)})^{p-2}
-\Big(|\omega|^{-\mu}\ast W^{p}\Big)W^{p-2}\Big]\bar{\phi}^{(1)}_{m}\bigg\}d\omega\\&
=p\int_{\mathbb{R}^n}\frac{1}{|z-\omega|^{n-2}}\Big[\Big(\int_{\Omega_1}\cdots-\int_{\Omega_1}\cdots\Big)+\Big(\int_{\Omega_2}\cdots-\int_{\Omega_2}\cdots\Big)
+\Big(\int_{\Omega_3}\cdots-\int_{\Omega_3}\cdots\Big)\\&
+(p-1)\int_{\mathbb{R}^n}\frac{1}{|z-\omega|^{n-2}}\Big[\Big(\int_{\Omega_1}\cdots-\int_{\Omega_1}\cdots\Big)+\Big(\int_{\Omega_2}\cdots-\int_{\Omega_2}\cdots\Big)
+\Big(\int_{\Omega_3}\cdots-\int_{\Omega_3}\cdots\Big)\\&
:=p\Big(F_{21}+F_{22}+F_{23}\Big)+(p-1)\Big(F^{\prime}_{21}+F^{\prime}_{22}+F^{\prime}_{23}\Big).
\end{split}
\end{equation*}
We choose $L\geq2\max\left\{\widetilde{A},\epsilon_0^{-1}\right\}$, it is easy to check that $\Omega_2\subset K_L$ for $m$ large enough. Combining Proposition \ref{diyigemingti} and \eqref{loc}, then we get
\begin{equation}\label{F12}
\begin{split}
pF_{12}+(p-1)F^{\prime}_{12}\rightarrow&
p\int_{\mathbb{R}^n}\frac{W^{p-1}(\omega)}{|z-\omega|^{n-2}}\Big(\int_{\Omega_2}\frac{W^{p-1}(y)\bar{\phi}^{(1)}_{\infty}(y)}{|\omega-y|^\mu}dy\Big)d\omega \\&+(p-1)\int_{\Omega_2}\frac{W^{p-2}(\omega)\bar{\phi}^{(1)}_{\infty}(\omega)}{|z-\omega|^{n-2}}\Big(\int_{\mathbb{R}^n}\frac{W^{p}(y)}{|\omega-y|^\mu}dy\Big)d\omega,
\end{split}
\end{equation}
\begin{equation}\label{F22}
pF_{22}+(p-1)F^{\prime}_{22}\rightarrow0.
\end{equation}
It is sufficient to handle remainder terms $F_{11}$, $F_{13}$, $F^{\prime}_{11}$,
$F^{\prime}_{13}$, $F_{21}$, $F_{23}$, $F^{\prime}_{21}$ and $F^{\prime}_{23}$.

For $F_{11}$ and $F_{11}^{\prime}$, it holds that
\begin{equation}\label{ML2}
|\bar{\phi}^{(1)}_{m}|\lesssim M^2+S_1^{-1}(x_m)\sum_{j\in\mathscr{H}(i_0)}(s_{j,1}^{(m)}+s_{j,2}^{(m)})\big((\lambda_{i_0}^{(m)})^{-1}z+\xi_{i_0j}^{(m)}\big)
\hspace{2mm}\mbox{for}\hspace{2mm}x\in\Omega_1,
\end{equation}
\begin{equation}\label{LM2}
\begin{split}
|\bar{\phi}^{(1)}_{m}|&\lesssim M^2\widetilde{A}^{-2}+M^{\frac{n+\mu-2}{2}}\widetilde{A}^{-\frac{n+\mu-2}{2}}\\&+\sum_{j\prec i_0}\big(\frac{\lambda_{i_0}^{(m)}}{\lambda_{i_0}^{(m)}}\big)^{2}+S_1^{-1}(x_m)\sum_{|\xi_{i_0j}^{(m)}|\rightarrow\infty}(s_{j,1}^{(m)}+s_{j,2}^{(m)})\big((\lambda_{i_0}^{(m)})^{-1}z+\xi_{i_0j}^{(m)}\big)
\hspace{2mm}\mbox{for}\hspace{2mm}x\in\Omega_3.
\end{split}
\end{equation}
For any $i\in\mathscr{H}(i_0)$. If $|\omega-\xi_{i_0i}^{(\infty)}|\leq2\epsilon_0$, we have
\begin{equation}\label{jx-1}
\begin{split}
\int_{\mathbb{R}^n}\frac{W^{p-1}(\omega)}{|z-\omega|^{n-2}}\Big(\int_{\Omega_1}\frac{W^{p-1}(y)}{|\omega-y|^\mu} dy\Big)d\omega&\lesssim\int_{\mathbb{R}^n}\frac{1}{|z-\omega|^{n-2}}W[0,1]^{p-1}(\omega)\Big(\int_{B(\xi_{i_0i}^{(\infty)},\epsilon_0)}\frac{W^{p-1}(y)}{|\omega-y|^\mu} dy\Big)d\omega\\&
\lesssim\int_{\mathbb{R}^n}\frac{1}{|z-\omega|^{n-2}}W[0,1]^{p-1}(\omega)\Big(\int_{B(\omega,3\epsilon_0)}\frac{dy}{|\omega-y|^\mu} \Big)d\omega\\&
\lesssim
\left\lbrace
\begin{aligned}
&\frac{\epsilon_0^{n-\mu}}{1+|\omega|^{n-\mu}},\hspace{6.5mm}\hspace{8mm}\hspace{2mm}\mu>2,\\&
\frac{\epsilon_0^{n-\mu}(1+\log|\omega|)}{1+|\omega|^{n-2}},\hspace{4mm}\hspace{2mm}\mu=2,\\&
\frac{\epsilon_0^{n-\mu}}{1+|\omega|^{n-2}},\hspace{7mm}\hspace{8mm}\hspace{2mm}\mu<2.
\end{aligned}
\right.
\end{split}
\end{equation}
If $|\omega-\xi_{i_0i}^{(\infty)}|\geq2\epsilon_0$, we also have
\begin{equation}\label{jx2}
\begin{split}
\int_{\mathbb{R}^n}\frac{W^{p-1}(\omega)}{|z-\omega|^{n-2}}\Big(\int_{\Omega_1}\frac{W^{p-1}(y)}{|\omega-y|^\mu} dy\Big)d\omega
\lesssim \epsilon_0^{n-\mu}.
\end{split}
\end{equation}
For any $i\in\mathscr{H}(i_0)$. Similar to the argument of \eqref{jx-1} and \eqref{jx2}, combining Lemma \ref{5-7} and Lemma \ref{p1-00}, then we get
\begin{equation}\label{jx3}
\begin{split}
\int_{\Omega_1}\frac{W^{p-2}(\omega)}{|z-\omega|^{n-2}}\int_{\mathbb{R}^n}\frac{W^{p}(y)}{|\omega-y|^\mu}d\omega dy
\lesssim\int_{B(\xi_{i_0i}^{(\infty)},\epsilon_0)}\frac{1}{|z-\omega|^{n-2}}W[0,1]^{2^{\ast}-2}(\omega)d\omega\lesssim\epsilon_0^2.
\end{split}
\end{equation}
Furthermore, thanks to Lemma \ref{cll-0}, \eqref{zzzz-2}, \eqref{zzzz-4}, \eqref{S2} and \eqref{ML2}, we have
\begin{equation*}
\begin{split}
&\frac{1}{S_1(x_m)}\sum_{j\in\mathscr{H}(i_0)}\int_{\mathbb{R}^n}\frac{W^{p-1}(\omega)}{|z-\omega|^{n-2}}\int_{B(\xi_{i_0j,}^{(\infty)},\epsilon_0)}\frac{W^{p-1}(y)(s_{j,1}^{(m)}+s_{j,2}^{(m)})
\big((\lambda_{i_0}^{(m)})^{-1}z+\xi_{i_0j}^{(m)}\big)}{|\omega-y|^\mu}d\omega dy\\
\lesssim&\frac{1}{S_1(x_m)}
\left\lbrace
\begin{aligned}
&
\sum_{j\in\mathscr{H}(i_0)}\big[\mathscr{R}_{m}^{-(3-\mu+\min\{\mu,\frac{5+\mu}{3}\})}+\epsilon_{0}^{\frac{n-\mu+2}{2}}+(\frac{\lambda_{i_0}^{(m)}}{\lambda_{j}^{(m)}})^{\frac{n-\mu+2}{2}}\big]
(s_{j,1}^{(m)}+s_{j,2}^{(m)})
\big((\lambda_{i_0}^{(m)})^{-1}z+\xi_{i_0j}^{(m)}\big),\\&\hspace{2mm}\mbox{for}\hspace{2mm}n=5\hspace{2mm}\mbox{and}\hspace{2mm}\mu\in[1,3),
\\&
\sum_{j\in\mathscr{H}(i_0)}\big[\mathscr{R}_{m}^{-(2p-2-\theta_1)}+\epsilon_{0}^{\frac{n-\mu+2}{2}}+(\frac{\lambda_{i_0}^{(m)}}{\lambda_{j}^{(m)}})^{\frac{n-\mu+2}{2}}\big](s_{j,1}^{(m)}+s_{j,2}^{(m)})
\big((\lambda_{i_0}^{(m)})^{-1}z+\xi_{i_0j}^{(m)}\big),\\&\hspace{2mm}\mbox{for}\hspace{2mm}n=6\hspace{2mm}\mbox{and}\hspace{2mm}\mu\in(0,4),
\\&
\sum_{j\in\mathscr{H}(i_0)}\big[\mathscr{R}_{m}^{-(2p-2-\theta_2)}+\epsilon_{0}^{\frac{n-\mu+2}{2}}+(\frac{\lambda_{i_0}^{(m)}}{\lambda_{j}^{(m)}})^{\frac{n-\mu+2}{2}}\big](s_{j,1}^{(m)}+s_{j,2}^{(m)})
\big((\lambda_{i_0}^{(m)})^{-1}z+\xi_{i_0j}^{(m)}\big),\\&\hspace{2mm}\mbox{for}\hspace{2mm}n=7\hspace{2mm}\mbox{and}\hspace{2mm}\mu\in(\frac{7}{3},4),
\end{aligned}
\right.
\\ \lesssim
&\epsilon_{0}^{\frac{n-\mu+2}{2}}+o(1),
\hspace{4mm}\mbox{for}\hspace{2mm}n=5\hspace{2mm}\mbox{and}\hspace{2mm}\mu\in[1,3),
\hspace{2mm}\mbox{or}\hspace{2mm}n=6\hspace{2mm}\mbox{and}\hspace{2mm}\mu\in(0,4),
\hspace{2mm}\mbox{or}\hspace{2mm}n=7\hspace{2mm}\mbox{and}\hspace{2mm}\mu\in(\frac{7}{3},4).
\end{split}
\end{equation*}
Similarly, by Lemma \ref{cll-0}, Lemma \ref{p1-00}, \eqref{zzz-2-1}, \eqref{zzz-4-1}, \eqref{S2} and \eqref{ML2}, we obtain
\begin{equation*}
\begin{split}
&\frac{1}{S_1(x_m)}\sum_{j\in\mathscr{H}(i_0)}\int_{B(\xi_{i_0j,}^{(\infty)},\epsilon_0)}\frac{W^{p-1}(\omega)}{|z-\omega|^{n-2}}\Big(\int_{\mathbb{R}^n}\frac{W^{p}(y)}{|\omega-y|^\mu} dy\Big)(s_{j,1}^{(m)}+s_{j,2}^{(m)})
\big((\lambda_{i_0}^{(m)})^{-1}z+\xi_{i_0j}^{(m)}\big)d\omega\\ \lesssim
&\epsilon_{0}^{2}+o(1),
\hspace{4mm}\mbox{for}\hspace{2mm}n=5\hspace{2mm}\mbox{and}\hspace{2mm}\mu\in[1,3),
\hspace{2mm}\mbox{or}\hspace{2mm}n=6\hspace{2mm}\mbox{and}\hspace{2mm}\mu\in(0,4),
\hspace{2mm}\mbox{or}\hspace{2mm}n=7\hspace{2mm}\mbox{and}\hspace{2mm}\mu\in(\frac{7}{3},4).
\end{split}
\end{equation*}
On the other hand, in virtue of \eqref{S-1},
\eqref{S2} and \eqref{ML2}, and choosing $\epsilon_0$ so small that $\frac{1}{2}|\xi_{i_0i,}^{(\infty)}-\xi_{i_0j,}^{(\infty)}|>\epsilon_0$, we get
\begin{equation*}
\begin{split}
&\sum_{\substack{i,j\in\mathscr{H}(i_0),\\
\xi_{i_0i,}^{(\infty)}\neq\xi_{i_0j,}^{(\infty)}}}\int_{\mathbb{R}^n}\frac{W^{p-1}(\omega)}{|z-\omega|^{n-2}}
\int_{B(\xi_{i_0i,}^{(\infty)},\epsilon_0)}\frac{1}{S_1(x_m)}\frac{W^{p-1}(y)(s_{j,1}^{(m)}+s_{j,2}^{(m)})
\big((\lambda_{i_0}^{(m)})^{-1}z+\xi_{i_0j}^{(m)}\big)}{|\omega-y|^\mu}d\omega dy\\
\lesssim &\sum_{\substack{i,j\in\mathscr{H}(i_0),\\
\xi_{i_0i,}^{(\infty)}\neq\xi_{i_0j,}^{(\infty)}}}\int_{\mathbb{R}^n}\frac{W^{p-1}(\omega)}{|z-\omega|^{n-2}}
\int_{B(\xi_{i_0i,}^{(\infty)},\epsilon_0)}\frac{W^{p-1}(y)
}{|\omega-y|^\mu}\Big[\frac{1}{|\omega-\xi_{i_0j,}^{(\infty)}|^2}+(\frac{1}{|\omega-\xi_{i_0j,}^{(\infty)}|})^{\frac{n+\mu-2}{2}}\Big]d\omega dy\\&
\lesssim\epsilon_{0}^{n-\mu}\sum_{i\in\mathscr{H}(i_0)}\int_{\mathbb{R}^n}\frac{1}{|z-\omega|^{n-2}}W[0,1]^{p-1}(\omega)d\omega\lesssim\epsilon_{0}^{n-\mu}.
\end{split}
\end{equation*}
Similarly, choosing $\epsilon_0$ so small that $\frac{1}{2}|\xi_{i_0i,}^{(\infty)}-\xi_{i_0j,}^{(\infty)}|>\epsilon_0$, and together with Lemma \ref{p1-00}, \eqref{S-1},
\eqref{S2} and \eqref{ML2}, we also obtain
\begin{equation*}
\begin{split}
&\sum_{\substack{i,j\in\mathscr{H}(i_0),\\
\xi_{i_0i,}^{(\infty)}\neq\xi_{i_0j,}^{(\infty)}}}\int_{B(\xi_{i_0i,}^{(\infty)},\epsilon_0)}\frac{1}{S_1(x_m)}\frac{W^{p-1}(y)(s_{j,1}^{(m)}+s_{j,2}^{(m)})
\big((\lambda_{i_0}^{(m)})^{-1}z+\xi_{i_0j}^{(m)}\big)}{|z-\omega|^{n-2}}
\Big(\int_{\mathbb{R}^n}\frac{W^{p-1}(y)}{|\omega-y|^\mu} dy\Big)d\omega\\
\lesssim&\sum_{i\in\mathscr{H}(i_0)}\int_{B(\xi_{i_0i,}^{(\infty)},\epsilon_0)}\frac{1}{|z-\omega|^{n-2}}W[0,1]^{2^{\ast}-2}(\omega)d\omega\lesssim\epsilon_{0}^{2}.
\end{split}
\end{equation*}
As a consequence,  we eventually get
\begin{equation}\label{ff11}
F_{11}+F_{11}^{\prime}\lesssim\epsilon_{0}^{n-\mu}+\epsilon_{0}^{2}+o(1).
\end{equation}

For $F_{13}$ and $F_{13}^{\prime}$, using Lemma \ref{5-7}, we obtain
$$\int_{\mathbb{R}^n}\frac{W^{p-1}(\omega)}{|z-\omega|^{n-2}}\Big(\int_{\Omega_3}\frac{W^{p-1}(y)}{|\omega-y|^\mu} dy\Big)d\omega\lesssim1,\hspace{2mm}\int_{\Omega_3}\frac{W^{p-2}(\omega)}{|z-\omega|^{n-2}}\int_{\mathbb{R}^n}\frac{W^{p}(y)}{|\omega-y|^\mu}d\omega dy\lesssim1.$$
We set
$$h(z):=\int_{\mathbb{R}^n}\frac{1}{|z-\omega|^{n-2}}\frac{1}{S_1(x_m)}\big(t_{i_0,1}^{(m)}+t_{i_0,2}^{(m)}+t_{j,1}^{(m)}+t_{j,2}^{(m)}\big)
\big((\lambda_{i_0}^{(m)})^{-1}z+\xi_{i_0j}^{(m)}\big)d\omega.$$
Also, using Lemma \ref{cll-0}, \eqref{zzzz-0}-\eqref{zzzz-3} together with \eqref{S2} and \eqref{LM2}, we get
\begin{equation*}
\begin{split}
&\frac{1}{S_1(x_m)}\sum_{|\xi_{i_0j}^{(m)}|\rightarrow\infty}\int_{\mathbb{R}^n}\frac{W^{p-1}(\omega)}{|z-\omega|^{n-2}}\int_{\Omega_3}\frac{W^{p-1}(y)(s_{j,1}^{(m)}+s_{j,2}^{(m)})
\big((\lambda_{i_0}^{(m)})^{-1}z+\xi_{i_0j}^{(m)}\big)}{|\omega-y|^\mu}d\omega dy\\
\lesssim&
\left\lbrace
\begin{aligned}
&
\big[\mathscr{R}_{m}^{-\frac{6+\mu-n}{2}}+\sum_{|\xi_{i_0j}^{(m)}|\rightarrow\infty}\big(\xi_{i_0j}^{(m)})^{-\frac{n-\mu+2}{2}}+\tau(\xi_{i_0j}^{(m)})^{-\varpi^{\star}_1}\big)\big]h(z)
,\hspace{2mm}\hspace{4mm}\hspace{9mm}\hspace{9mm}\hspace{3mm}\hspace{8mm}\hspace{2mm}n=5,
\\&
\big[\mathscr{R}_{m}^{-(2p-n+\mu-\theta_1)}+\mathscr{R}_{m}^{-\frac{6+\mu-n}{2}}+\sum_{|\xi_{i_0j}^{(m)}|\rightarrow\infty}\big(\tau(\xi_{i_0j}^{(m)})^{-\frac{n-\mu+2}{2}}+\tau(\xi_{i_0j}^{(m)})^{-\varpi^{\star}_1}\big)\big]
h(z),\hspace{2mm}\hspace{2mm}n=6,
\\&
\big[\mathscr{R}_{m}^{-(2p-n+\mu-\theta_2)}+\mathscr{R}_{m}^{-\frac{6+\mu-n}{2}}+\sum_{|\xi_{i_0j}^{(m)}|\rightarrow\infty}\big(\xi_{i_0j}^{(m)})^{-\frac{n-\mu+2}{2}}+\tau(\xi_{i_0j}^{(m)})^{-\varpi^{\star}_1}\big)\big]
h(z),\hspace{2mm}\hspace{4mm}\hspace{2mm}n=7,
\end{aligned}
\right.
\\ \lesssim
&o(1),
\hspace{4mm}\mbox{for}\hspace{2mm}n=5\hspace{2mm}\mbox{and}\hspace{2mm}\mu\in[1,3),
\hspace{2mm}\mbox{or}\hspace{2mm}n=6\hspace{2mm}\mbox{and}\hspace{2mm}\mu\in(0,4),
\hspace{2mm}\mbox{or}\hspace{2mm}n=7\hspace{2mm}\mbox{and}\hspace{2mm}\mu\in(\frac{7}{3},4).
\end{split}
\end{equation*}
Using  Lemma \ref{cll-0} and Lemma \ref{cll-1-0-1}, we similarly compute and get
\begin{equation*}
\begin{split}
&\frac{1}{S_1(x_m)}\sum_{|\xi_{i_0j}^{(m)}|\rightarrow\infty}\int_{\Omega_3}\frac{W^{p-1}(\omega)}{|z-\omega|^{n-2}}\Big(\int_{\mathbb{R}^n}\frac{W^{p}(y)}{|\omega-y|^\mu} dy\Big)(s_{j,1}^{(m)}+s_{j,2}^{(m)})
\big((\lambda_{i_0}^{(m)})^{-1}z+\xi_{i_0j}^{(m)}\big)d\omega\\ \lesssim
&o(1),
\hspace{4mm}\mbox{for}\hspace{2mm}n=5\hspace{2mm}\mbox{and}\hspace{2mm}\mu\in[1,3),
\hspace{2mm}\mbox{or}\hspace{2mm}n=6\hspace{2mm}\mbox{and}\hspace{2mm}\mu\in(0,4),
\hspace{2mm}\mbox{or}\hspace{2mm}n=7\hspace{2mm}\mbox{and}\hspace{2mm}\mu\in(\frac{7}{3},4).
\end{split}
\end{equation*}
Now, combining these estimates with \eqref{LM2}, we finally get
\begin{equation}\label{f13}
F_{13}+F_{13}^{\prime}\lesssim M^2\widetilde{A}^{-2}+M^{\frac{n+\mu-2}{2}}\widetilde{A}^{-\frac{n+\mu-2}{2}}+o(1).
\end{equation}

Furthermore,
recalling the convergence
\begin{equation*}
\bar{\phi}_m^{(1)}\rightarrow\bar{\phi}^{(1)}_{\infty}\hspace{2mm}\mbox{in}\hspace{2mm} \mathbb{R}^n\setminus\{\xi_{i_0j}^{(\infty)}:j\in\mathscr{H}(i_0)\}
	\hspace{2mm}\mbox{as}\hspace{2mm} m\rightarrow\infty.
\end{equation*}
Thus, applying Fatou's Lemma, \eqref{ML2} and \eqref{LM2}, we deduce that
\begin{equation*}
\begin{split}
&p\int_{\mathbb{R}^n}\frac{W^{p-1}(\omega)}{|z-\omega|^{n-2}}\int_{\Omega_{1}}\frac{\big(W^{p-1}|\bar{\phi}^{(1)}_{\infty}|\big)(y)}{|\omega-y|^\mu}d\omega dy+(p-1)\int_{\Omega_{1}}\frac{\big(W^{p-2}|\bar{\phi}^{(1)}_{\infty}|\big)(\omega)}{|z-\omega|^{n-2}}\int_{\mathbb{R}^n}\frac{W^{p}(y)}{|\omega-y|^\mu}d\omega dy\\
\leq&\liminf\limits_{m\rightarrow\infty}\bigg\{\Big(p\int_{\mathbb{R}^n}\frac{W^{p-1}}{|z-\omega|^{n-2}}\int_{\Omega_{1}}\frac{\big(W^{p-1}|\bar{\phi}^{(1)}_{m}|\big)}{|\omega-y|^\mu} +(p-1)\int_{\Omega_{1}}\frac{\big(W^{p-2}|\bar{\phi}^{(1)}_{m}|\big)}{|z-\omega|^{n-2}}\int_{\mathbb{R}^n}\frac{W^{p}}{|\omega-y|^\mu}\Big)d\omega dy\bigg\}\\
\lesssim&\epsilon_{0}^{n-\mu}+\epsilon_{0}^{2}(\epsilon_{0}^{n-\mu}+\epsilon_{0}^{2}),
\end{split}
\end{equation*}
and
\begin{equation*}
\begin{split}
&p\int_{\mathbb{R}^n}\frac{W^{p-1}(\omega)}{|z-\omega|^{n-2}}\int_{\Omega_{3}}\frac{\big(W^{p-1}|\bar{\phi}^{(1)}_{\infty}|\big)(y)}{|\omega-y|^\mu}d\omega dy+(p-1)\int_{\Omega_{3}}\frac{\big(W^{p-2}|\bar{\phi}^{(1)}_{\infty}|\big)(\omega)}{|z-\omega|^{n-2}}\int_{\mathbb{R}^n}\frac{W^{p}(y)}{|\omega-y|^\mu}d\omega dy\\
\leq&\liminf\limits_{m\rightarrow\infty}\bigg\{\Big(p\int_{\mathbb{R}^n}\frac{W^{p-1}}{|z-\omega|^{n-2}}\int_{\Omega_{3}}\frac{\big(W^{p-1}|\bar{\phi}^{(1)}_{m}|\big)}{|\omega-y|^\mu} +(p-1)\int_{\Omega_{3}}\frac{\big(W^{p-2}|\bar{\phi}^{(1)}_{m}|\big)}{|z-\omega|^{n-2}}\int_{\mathbb{R}^n}\frac{W^{p}}{|\omega-y|^\mu}\Big)d\omega dy\bigg\}\\
\lesssim&M^2\widetilde{A}^{-2}+M^{\frac{n+\mu-2}{2}}\widetilde{A}^{-\frac{n+\mu-2}{2}}.
\end{split}
\end{equation*}
Therefore, combining this bound with \eqref{F12}-\eqref{F22} and \eqref{ff11}-\eqref{f13}, we conclude that
\begin{equation}\label{1-F}
F_1=\int_{\mathbb{R}^n}\frac{1}{|z-\omega|^{n-2}}\Phi_{n,\mu}[W[0,1],\bar{\phi}^{(1)}_{\infty}]d\omega+O\big(M^2\widetilde{A}^{-2}+M^{\frac{n+\mu-2}{2}}\widetilde{A}^{-\frac{n+\mu-2}{2}}
+\epsilon_{0}^{n-\mu}+\epsilon_{0}^{2}\big)+o(1).
\end{equation}
The estimates for $F_2$ are completely analogous except minor modifications by applying Lemma \ref{cll-0}, Lemma \ref{cll-1-0-1}, Lemma \ref{cll-2}, Lemma \ref{cll-1}, Lemma \ref{cll-11-2} and \eqref{S2}. Thus, we deduce
\begin{equation}\label{2-F}
F_2\lesssim\widetilde{A}^{-2}+\widetilde{A}^{-\varpi^{\star}_1})+M^2\widetilde{A}^{-2}+M^{\frac{n+\mu-2}{2}}\widetilde{A}^{-\frac{n+\mu-2}{2}}+o(1).
\end{equation}
By collecting the above the definition of $F_1$ and $F_2$, \eqref{2-fai}, \eqref{1-F} and \eqref{2-F}, we can conclude the proof of \eqref{converges-1} by choosing $\epsilon_0>0$ small enough and $\widetilde{A}>0$ sufficiently large.

To prove \eqref{converges-2}, combining the convergence
$
\bar{\phi}_m^{(1)}\rightarrow\bar{\phi}^{(1)}_{\infty}\hspace{2mm}\mbox{in}\hspace{2mm} \Omega_2\hspace{2mm}\mbox{as}\hspace{2mm} m\rightarrow\infty
$, by Lebesgue's dominated convergence theorem we deduce that
\begin{equation*}
\begin{split}
&p\int_{\Omega_2}\Big[\Big(\int_{\mathbb{R}^n}\frac{W[0,1]^{p-1}\Xi^{a}[0,1]}{|x-y|^{\mu}}\Big)
W[0,1]^{p-1}+(p-1)\Big(\int_{\mathbb{R}^n}\frac{W[0,1]^{p}}{|x-y|^\mu}\Big)
W[0,1]^{p-2}\Xi^{a}[0,1]\Big]\bar{\phi}_m^{(1)}\\&
\rightarrow p\int_{\Omega_2}\Big[\Big(\int_{\mathbb{R}^n}\frac{W[0,1]^{p-1}\Xi^{a}[0,1]}{|x-y|^{\mu}}\Big)
W[0,1]^{p-1}+(p-1)\Big(\int_{\mathbb{R}^n}\frac{W[0,1]^{p}}{|x-y|^\mu}\Big)
W[0,1]^{p-2}\Xi^{a}[0,1]\Big]\bar{\phi}_\infty^{(1)}
\end{split}
\end{equation*}
as $m\rightarrow\infty$.
Also, thanks to \eqref{ML2} and definitions of $s_{j,1}$ and $s_{j,2}$, together with Lemma \ref{p1-00} and Fatou's Lemma, we deduce that
\begin{equation*}
\begin{split}
&p\int_{\Omega_1}\Big[\Big(\int\frac{W[0,1]^{p-1}\Xi^{a}[0,1]}{|x-y|^{\mu}}\Big)
W[0,1]^{p-1}+(p-1)\Big(\int\frac{W[0,1]^{p}}{|x-y|^\mu}\Big)
W[0,1]^{p-2}\Xi^{a}[0,1]\Big]\bar{\phi}_\infty^{(1)}\\
\leq& \liminf\limits_{m\rightarrow\infty}\bigg\{ p\int_{\Omega_1}\Big[\Big(\int\frac{W[0,1]^{p-1}\Xi^{a}[0,1]}{|x-y|^{\mu}}\Big)
W[0,1]^{p-1}+(p-1)\Big(\int\frac{W[0,1]^{p}}{|x-y|^\mu}\Big)
W[0,1]^{p-2}\Xi^{a}[0,1]\Big]\bar{\phi}_m^{(1)}\bigg\}\\
\lesssim&\liminf\limits_{m\rightarrow\infty}\bigg\{M^2 \int_{\Omega_1}W[0,1]^{2^{\ast}-2}+\int_{\Omega_1}W[0,1]^{2^{\ast}-2}\sum_{j\in\mathscr{H}(i_0)}\Big[\big(\frac{M}{|\omega-\xi_{i_0j,}^{(\infty)}|}\big)^2+\big(\frac{M}{|\omega-\xi_{i_0j,}^{(\infty)}|}\big)^{\frac{n+\mu-2}{2}}\Big]\bigg\}\\&\lesssim\epsilon_0^2.
\end{split}
\end{equation*}
Furthermore, applying \eqref{LM2} and Fatou's Lemma, we conclude that
\begin{equation*}
\begin{split}
&p\int_{\Omega_3}\Big[\Big(\int\frac{W[0,1]^{p-1}\Xi^{a}[0,1]}{|x-y|^{\mu}}\Big)
W[0,1]^{p-1}+(p-1)\Big(\int\frac{W[0,1]^{p}}{|x-y|^\mu}\Big)
W[0,1]^{p-2}\Xi^{a}[0,1]\Big]\bar{\phi}_\infty^{(1)}\\
\leq& \liminf\limits_{m\rightarrow\infty}\bigg\{
p\int_{\Omega_3}\Big[\Big(\int\frac{W[0,1]^{p-1}\Xi^{a}[0,1]}{|x-y|^{\mu}}\Big)
W[0,1]^{p-1}+(p-1)\Big(\int\frac{W[0,1]^{p}}{|x-y|^\mu}\Big)
W[0,1]^{p-2}\Xi^{a}[0,1]\Big]\bar{\phi}_m^{(1)}\bigg\}\\
\lesssim&  \liminf\limits_{m\rightarrow\infty}\bigg\{\int_{\Omega_3}W[0,1]^{2^{\ast}-2}\Big[M^2\widetilde{A}^{-2}+M^{\frac{n+\mu-2}{2}}\widetilde{A}^{-\frac{n+\mu-2}{2}}+\sum_{j\prec i_0}\big(\frac{\lambda_{i_0}^{(m)}}{\lambda_{i_0}^{(m)}}\big)^{2}\Big]\\&+
\int_{\Omega_3}W[0,1]^{2^{\ast}-2}\sum_{|\xi_{i_0j,}^{(\infty)}|\rightarrow\infty}\Big[\big(\frac{|\xi_{i_0j,}^{(\infty)}|}{|\omega-\xi_{i_0j,}^{(\infty)}|}\big)^2+
\big(\frac{|\xi_{i_0j,}^{(\infty)}|}{|\omega-\xi_{i_0j,}^{(\infty)}|}\big)^{\frac{n+\mu-2}{2}}\Big]\bigg\} \\& \lesssim\big[M^2\widetilde{A}^{-2}+M^{\frac{n+\mu-2}{2}}\widetilde{A}^{-\frac{n+\mu-2}{2}}\big]\widetilde{A}^{-2}+\widetilde{A}^{-2}.
\end{split}
\end{equation*}
and hence we have that \eqref{converges-2} follows. This concludes the proof.
\qed

The following we derive a result on the removability of singularities of
a solution.
\begin{Prop}\label{converges-3}
If the function $g(z)$ satisfies
	\begin{equation}\label{gz}
		\left\{
		\begin{array}{ll}
			\Delta g+\Phi_{n,\mu}[W[0,1],g]
		=0,        &\text{in}~\mathbb{R}^n\setminus\{\xi_{i_0j}^{(\infty)}:j\in\mathscr{H}(i_0)\},            \\
\vert g(z)\vert			\lesssim\sum_{j\in\mathscr{H}(i_0)}\big(\frac{M}{|z-\xi_{i_0j}^{(\infty)}|}\big)^{\frac{n+\mu-2}{2}}+M^2,                  &\text{in}~\mathbb{R}^n\setminus\{\xi_{i_0j}^{(\infty)}:j\in\mathscr{H}(i_0)\},
		\end{array}
		\right.
	\end{equation}
Then $g\in L^{\infty}(\mathbb{R}^n).$
\end{Prop}
\begin{proof}
Let $A^{\star}:=1+\max_{j\in\mathscr{H}(i_0)}|\xi_{i_0j}^{(\infty)}|$. In view of \eqref{gz}, to conclude the proof of Proposition \ref{gz}, it is sufficient to obtain bound in the set $B(0,4A^{\star})\setminus\{\xi_{i_0j}^{(\infty)}:j\in\mathscr{H}(i_0)\}.$
For any $z\in B(0,4A^{\star})\setminus\{\xi_{i_0j}^{(\infty)}:j\in\mathscr{H}(i_0)\}$. Then we have
\begin{equation*}
\begin{split}
|g(z)|&\lesssim\int\frac{1}{|z-\omega|^{n-2}}\Big[\Big(\int\frac{W[0,1]^{p-1}}{|\omega-y|^{\mu}}dy\Big)
W[0,1]^{p-1}+\Big(\int\frac{W[0,1]^{p}}{|\omega-y|^\mu}dy\Big)
W[0,1]^{p-2}\Big]d\omega\\&
+\sum_{j\in\mathscr{H}(i_0)}\int\frac{1}{|z-\omega|^{n-2}}\Big(\int\frac{W[0,1]^{p-1}}{|\omega-y|^{\mu}}\big(\frac{1}{|y-\xi_{i_0j}^{(\infty)}|}\big)^{\frac{n+\mu-2}{2}}dy\Big)
W[0,1]^{p-1}d\omega\\&+\sum_{j\in\mathscr{H}(i_0)}\int\frac{1}{|z-\omega|^{n-2}}\Big(\int\frac{W[0,1]^{p}}{|\omega-y|^\mu}dy\Big)
W[0,1]^{p-2}\big(\frac{1}{|\omega-\xi_{i_0j}^{(\infty)}|}\big)^{\frac{n+\mu-2}{2}}\Big]d\omega:=H_1+H_2+H_3.
\end{split}
\end{equation*}
In order to estimate $H_2$ and $H_3$, we now establish the following key estimates:
\begin{equation}\label{wamiga}
\frac{1}{|\omega-y|^{\mu}}\big(\frac{1}{|y-\xi_{i_0j}^{(\infty)}|}\big)^{\frac{n+\mu-2}{2}}\leq
\frac{C}{|\omega-\xi_{i_0j}^{(\infty)}|^{\sigma^{\prime}}}\Big(\frac{1}{|y-\omega|^{(n+3\mu-2)/2-\sigma^{\prime}}}
+\frac{1}{|y-\xi_{i_0j}^{(\infty)}|^{(n+3\mu-2)/2-\sigma^{\prime}}}\Big)
\end{equation}
for any a constant $0<\sigma^{\prime}<\min\{\mu,2+\delta^{\prime\prime}\}$ where $\delta^{\prime\prime}<\sigma^{\prime}$. And for any $\sigma^{\prime\prime}<(n-\mu+2)/2$, there is a constant $C$ such that
\begin{equation}\label{z-wamiga}
\frac{1}{|z-\omega|^{n-2}}\big(\frac{1}{|\omega-\xi_{i_0j}^{(\infty)}|}\big)^{\frac{n+\mu-2}{2}}\leq
\frac{C}{|z-\xi_{i_0j}^{(\infty)}|^{\sigma^{\prime\prime}}}\Big(\frac{1}{|\omega-z|^{(3n+\mu+2)/2-\sigma^{\prime\prime}}}
+\frac{1}{|\omega-\xi_{i_0j}^{(\infty)}|^{(3n+\mu+2)/2-\sigma^{\prime\prime}}}\Big).
\end{equation}
By handing the cases: $B^{\prime}:=\{y:|y|\leq\frac{1}{2}|\omega|~\mbox{for}~\frac{1}{2}|\omega|\geq1\}$, $B^{\prime\prime}:=\{y:|y-\omega|\leq\frac{1}{2}|\omega|~\mbox{for}~\frac{1}{2}|\omega|\geq1\}$ and $\{y:~\mathbb{R}^n\setminus(B^{\prime}\cup B^{\prime\prime})\}$ separately,
we deduce that
\begin{equation}\label{w01}
\int_{\mathbb{R}^n} \frac{1}{|y-\omega|^{(n+3\mu-2)/2-\sigma^{\prime}}}W[0,1]^{p-1}dy\leq
\left\lbrace
\begin{aligned}
&\frac{C}{(1+|\omega|)^{(n+\mu)/2-\sigma^{\prime}}},\hspace{10mm}\hspace{2mm}\mu>2,\\&
\frac{C(1+\log|\omega|)}{(1+|\omega|)^{(n+3\mu-2)/2-\sigma^{\prime}}},\hspace{5mm}\hspace{2mm}\mu=2,\\&
\frac{C}{(1+|\omega|)^{(n+3\mu-2)/2-\sigma^{\prime}}},\hspace{5.5mm}\hspace{2mm}\mu<2.
\end{aligned}
\right.
\end{equation}
Analogous, in the cases $C^{\prime}:=\{y:|y|\leq\frac{1}{2}|\xi_{i_0j}^{(\infty)}|,~\frac{1}{2}|\omega|\geq1\}$, $C^{\prime\prime}:=\{y:|y-\omega|\leq\frac{1}{2}|\xi_{i_0j}^{(\infty)}|,~\frac{1}{2}|\xi_{i_0j}^{(\infty)}|\geq1\}$ and $\{y:~\mathbb{R}^n\setminus(C^{\prime}\cup C^{\prime\prime})\}$ separately, we obtain
\begin{equation}\label{10w02}
\int_{\mathbb{R}^n} \frac{1}{|y-\xi_{i_0j}^{(\infty)}|^{(n+3\mu-2)/2-\sigma^{\prime}}}W[0,1]^{p-1}dy\leq
\left\lbrace
\begin{aligned}
&\frac{C}{(1+|\xi_{i_0j}^{(\infty)}|)^{(n+\mu)/2-\sigma^{\prime}}},\hspace{8mm}\hspace{2mm}\mu>2,\\&
\frac{C(1+\log|\xi_{i_0j}^{(\infty)}|)}{(1+|\xi_{i_0j}^{(\infty)}|)^{(n+3\mu-2)/2-\sigma^{\prime}}},\hspace{3mm}\hspace{2mm}\mu=2,\\&
\frac{C}{(1+|\xi_{i_0j}^{(\infty)}|)^{(n+3\mu-2)/2-\sigma^{\prime}}},\hspace{3.5mm}\hspace{2mm}\mu<2.
\end{aligned}
\right.
\end{equation}
Thus, by \eqref{wamiga}, \eqref{w01} and \eqref{10w02}, direct computations and we deduce
\begin{equation}\label{110w02-1}
H_2\leq C\sum_{j\in\mathscr{H}(i_0)}\Big[\int_{B(0,2A^{\star})^c}\frac{1}{|z-\omega|^{n-2}}\frac{1}{|\omega|^{n-\mu+2+\sigma^{\prime}}}dw
+\int_{B(0,2A^{\star})}\frac{1}{|z-\omega|^{n-2}}\frac{1}{|\omega-\xi_{i_0j}^{(\infty)}|^{\sigma^{\prime}}}d\omega\Big]
\end{equation}
On the one hand, similarly to \eqref{wamiga}-\eqref{z-wamiga} we obatin
for any a constant $\delta^{\prime\prime}<\sigma^{\prime}$, there is a constant $C$ such that
\begin{equation}\label{zz10w02}
\frac{1}{|z-\omega|^{n-2}}\frac{1}{|\omega-\xi_{i_0j}^{(\infty)}|^{\sigma^{\prime}}}\leq
\frac{C}{|z-\xi_{i_0j}^{(\infty)}|^{\delta^{\prime\prime}}}\Big(\frac{1}{|\omega-z|^{n-2+\sigma^{\prime}-\delta^{\prime\prime}}}
+\frac{1}{|\omega-\xi_{i_0j}^{(\infty)}|^{n-2+\sigma^{\prime}-\delta^{\prime\prime}}}\Big).
\end{equation}
As a result,
\begin{equation}\label{zy}
\begin{split}
&\frac{C}{|z-\xi_{i_0j}^{(\infty)}|^{\delta^{\prime\prime}}}\int_{B(0,2A^{\star})}\frac{1}{|\omega-\xi_{i_0j}^{(\infty)}|^{n-2+\sigma^{\prime}-\delta^{\prime\prime}}}d\omega
\\&=\frac{C}{|z-\xi_{i_0j}^{(\infty)}|^{\delta^{\prime\prime}}}\Big(\int_{|\omega-\xi_{i_0j}^{(\infty)}|<\frac{|\xi_{i_0j}^{(\infty)}|}{2}}\cdots
+\int_{|\omega|<\frac{|\xi_{i_0j}^{(\infty)}|}{2}}\cdots+\int_{\{\frac{|\xi_{i_0j}^{(\infty)}|}{2}\leq|\omega|<2A^\star\}\cap\{|\omega-\xi_{i_0j}^{(\infty)}|\geq\frac{|\xi_{i_0j}^{(\infty)}|}{2}\}}
\Big)\\&
\leq C\frac{1}{|z-\xi_{i_0j}^{(\infty)}|^{\delta^{\prime\prime}}}\Big(|\xi_{i_0j}^{(\infty)}|^{2+\delta^{\prime\prime}-\sigma^{\prime}}+(A^{\star})^{2+\delta^{\prime\prime}-\sigma^{\prime}}\Big).
\end{split}
\end{equation}
On the other hand, by Lemma \ref{B3}, Lemma \ref{p1-00} and Lemma \ref{5-7}, we have
\begin{equation}\label{H111}
H_1+\int_{B(0,2A^{\star})^c}\frac{1}{|z-\omega|^{n-2}}\frac{1}{|\omega|^{n-\mu+2+\sigma^{\prime}}}dw
\leq \left\lbrace
\begin{aligned}
&\frac{C}{(1+|\omega|)^{2}},\hspace{12mm}\hspace{2mm}\mu\neq2+\delta^{\prime}\hspace{2mm}\mbox{or}2+\sigma^{\prime},\\&
\frac{C(1+\log|\omega|)}{(1+|\omega|)^{2}},\hspace{5mm}\hspace{2mm}\mu=2+\delta^{\prime}\hspace{2mm}\mbox{or}2+\sigma^{\prime}.
\end{aligned}
\right.
\end{equation}
Similarly, combining \eqref{z-wamiga} and  Lemma \ref{p1-00}, we obtain
\begin{equation}\label{zz10wzzy}
\begin{split}H_3&\leq C \sum_{j\in\mathscr{H}(i_0)}\Big[\int_{B(0,2A^{\star})^c}\frac{1}{|z-\omega|^{n-2}}\frac{1}{|\omega|^{(n+\mu+6)/2}}dw
+\int_{B(0,2A^{\star})}\frac{1}{|z-\omega|^{n-2}}\frac{1}{|\omega-\xi_{i_0j}^{(\infty)}|^{(n+\mu-2)/2}}d\omega\Big]\\&
\leq C\Big(1+\sum_{j\in\mathscr{H}(i_0)}\frac{1}{|z-\xi_{i_0j}^{(\infty)}|^{\sigma^{\prime\prime}}}\Big) \hspace{3mm}\text{in}~\mathbb{R}^n\setminus\{\xi_{i_0j}^{(\infty)}:j\in\mathscr{H}(i_0)\}.
\end{split}
\end{equation}
By collecting \eqref{110w02-1}, \eqref{zz10w02}, \eqref{zy}, \eqref{H111} and \eqref{zz10wzzy}, we finally get
$$g(z)\leq C\Big(1+\sum_{j\in\mathscr{H}(i_0)}\frac{1}{|z-\xi_{i_0j}^{(\infty)}|^{\delta^{\prime\prime}}}+\sum_{j\in\mathscr{H}(i_0)}\frac{1}{|z-\xi_{i_0j}^{(\infty)}|^{\sigma^{\prime\prime}}}\Big) \hspace{3mm}\text{in}~\mathbb{R}^n\setminus\{\xi_{i_0j}^{(\infty)}:j\in\mathscr{H}(i_0)\}.
$$
where $C$ are the strictly positive constant which depends only on $n$, $\kappa$, $\mu$, $\delta^{\prime\prime}$ and $\sigma^{\prime\prime}$.
We now turn to the proof of \eqref{wamiga}. Feeding back this bound
into the previous estimate of $g(z)$ and the conclusion then follows from iterating the above argument.

We now turn to establish \eqref{wamiga} and \eqref{z-wamiga}. Let $d=|\omega-\xi_{i_0j}^{(\infty)}|$. If $y\in B(\xi_{i_0j}^{(\infty)},d)$, then we have $|y-\omega|\geq\frac{1}{2}|\omega-\xi_{i_0j}^{(\infty)}|$ and $|y-\omega|\geq\frac{1}{2}|y-\xi_{i_0j}^{(\infty)}|$. Thus, we deduce
\begin{equation}\label{LHS}
\mbox{LHS}~\mbox{of}~\eqref{wamiga}\leq
\frac{C}{|\omega-\xi_{i_0j}^{(\infty)}|^{\sigma^{\prime}}}\Big(\frac{1}{|y-\xi_{i_0j}^{(\infty)}|^{(n+3\mu-2)/2-\sigma^{\prime}}}\Big)\hspace{2mm}\mbox{in}\hspace{2mm}
B(\xi_{i_0j}^{(\infty)},d).
\end{equation}
Similarly, we have
\begin{equation*}
\mbox{LHS}~\mbox{of}~\eqref{wamiga}\leq
\frac{C}{|\omega-\xi_{i_0j}^{(\infty)}|^{\sigma^{\prime}}}\Big(\frac{1}{|y-\omega|^{(n+3\mu-2)/2-\sigma^{\prime}}}\Big)\hspace{2mm}\mbox{in}\hspace{2mm}
B(\omega,d).
\end{equation*}
For $\mathbb{R}^n\setminus\big(B(\xi_{i_0j}^{(\infty)},d)\cup\big)B(\omega,d)$. If in the cases $|y-\xi_{i_0j}^{(\infty)}|\geq2|\xi_{i_0j}^{(\infty)}-\omega|$ and $|y-\xi_{i_0j}^{(\infty)}|\leq2|\xi_{i_0j}^{(\infty)}-\omega|$, we still obtain the estimate \eqref{LHS} holds. Moreover, statement of \eqref{z-wamiga} also hold by the same argument and which concludes the proof of
Proposition \ref{converges-3}.
\end{proof}
\begin{Prop}\label{converges-4}
Up to a subsequence, we have
\begin{equation}\label{converges-5}
			\bar{\phi}_m^{(1)}(z)\rightarrow0 \hspace{2mm}\mbox{in}\hspace{2mm} C_{loc}^{0}\big(\mathbb{R}^n\setminus\{\xi_{i_0j}^{(\infty)}:j\in\mathscr{H}(i_0)\}\big)
	\hspace{2mm}\mbox{as}\hspace{2mm} m\rightarrow\infty.
		\end{equation}
\end{Prop}
\begin{proof}
By Proposition \ref{converges-0}, $\bar{\phi}^{(1)}_{\infty}$ satisfies that \eqref{converges-1} and \eqref{converges-2}. and combining Proposition \ref{converges-3}, it follows that
each singular $\xi_{i_0j}^{(\infty)}$ of $\bar{\phi}^{(1)}_{\infty}$ is removable. The conclusion then follows from the orthogonality condition and non-degeneracy of bubbles $W_i$.
\end{proof}
Now we are in position to prove step \ref{step5.2} by combining the blow-up argument.

\emph{Proof of step \ref{step5.2}}. Since $\vert \zeta_m\vert:=\vert\lambda_{i_0}^{(m)}(x_m-\xi_{i_0j}^{(m)})\vert\leq M$ and $\vert \zeta_m-\xi_{i_0j}^{\left(m\right)}\vert\geq\varepsilon_0$, up to a subsequence, then we have $\lim_{m\rightarrow\infty}\zeta_m=\zeta_\infty\notin\{\xi_{i_0j}^{(\infty)}:j\in\mathscr{H}(i_0)\}$ . However this estimate gives a contradiction with $\bar{\phi}^{(1)}_{\infty}\equiv0$ by Proposition \ref{converges-4}. \qed
\begin{step}\label{step5.3}
 When $\left\{x_m\right\}\subset\mathcal{C}_{neck,M,3}$.
Assume that $n=5$ and $\mu\in[1,3)$, or $n=6$ and $\mu\in(0,4)$, or $n=7$ and $\mu\in(\frac{7}{3},4)$, we have that $\phi_{m}(x)<S_1(x)$ as $m\rightarrow+\infty$. The above statements also hold, namely that $\phi_{m}(x)<S_2(x)$ as $m$ enough large for $n=4\hspace{2mm}\mbox{and}\hspace{2mm}\mu\in[2,4),\mbox{or}\hspace{2mm}n=5\hspace{2mm}\mbox{and}\hspace{2mm}\mu\in(3,4]$.
\end{step}

 For $\theta\in (0,\frac{1}{2})$, we define $G(X,Y)$ as follows:
	\begin{equation*}
			G(X,Y)=\frac{X+Y}{2}-\sqrt{(\frac{X-Y}{2})^2+\theta XY},
	\end{equation*}
which approximates the function	$\min\{X,Y\}$.

\begin{Def}
Define the weighted functions $\bar{\mathbf{s}}_{i,1}^{\star}(x)$ and $\tilde{\mathbf{s}}_{i,1}^{\star}(x)$ as
\begin{equation*}
\begin{split}
&\bar{\mathbf{s}}_{i,1}^{\star}(x):=\sum_{j\in\mathscr{H}(i_0) }\lambda_{i}^{\frac{n-2}{2}}\mathscr{R}^{2-n}\big(1+|\xi_{ij}|^2+\varepsilon_0^{-2}|z_i^{(m)}-\xi_{ij}^{(m)}|^2\big)^{-1},\hspace{6mm}\hspace{3mm}\hspace{3mm}0<\mu<\frac{n+\mu-2}{2},\\&
\tilde{\mathbf{s}}_{i,1}^{\star}(x):=\sum_{j\in\mathscr{H}(i_0) }\lambda_{i}^{\frac{n-2}{2}}\mathscr{R}^{4+\mu-2n}\big(1+|\xi_{ij}|^2+\varepsilon_0^{-2}|z_i^{(m)}-\xi_{ij}^{(m)}|^2\big)^{-1},\hspace{4mm}\hspace{3mm}\frac{n+\mu-2}{2}\leq\mu<4.
\end{split}
\end{equation*}
\end{Def}
	\begin{Prop}\label{mingti2}
	Given $x\in D_{i}^{(m)}$. In the above notation we have that
		\begin{equation}\label{z-1}
\begin{split}
			&\frac{1}{3}s_{i,1}\leq\bar{\mathbf{s}}_{i,1}^{\star}(x)\leq 3\kappa s_{i,1}(x),\hspace{6mm}\hspace{6mm}\hspace{4mm}\hspace{1mm}\hspace{2mm} 0<\mu<\frac{n+\mu-2}{2},\\&
\Delta_{x}\bar{\mathbf{s}}_{i,1}^{\star}(x)\leq-\frac{\kappa(\kappa-1)(n-4)}{16\varepsilon_0^2}t_{i,1}(x),\hspace{3mm}0<\mu<\frac{n+\mu-2}{2},
		\end{split}
\end{equation}
\begin{equation}\label{z-2}
			\frac{1}{3}\hat{s}_{i,1}\leq\tilde{\mathbf{s}}_{i,1}^{\star}(x)\leq 3\kappa \hat{s}_{i,1}(x),\hspace{6mm}\hspace{3mm}\hspace{2mm} \hspace{3mm} \hspace{2mm} \frac{n+\mu-2}{2}\leq\mu<4,
\end{equation}
\begin{equation}\label{z-3}
\Delta_{x}\tilde{\mathbf{s}}_{i,1}^{\star}(x)\leq-\frac{\kappa(\kappa-1)}{9\varepsilon_0^2}\hat{t}_{i,1}(x),\hspace{3mm} \hspace{3mm}\hspace{3mm}\hspace{2mm}  \frac{n+\mu-2}{2}\leq\mu<4.
\end{equation}
	\end{Prop}
	\begin{proof}
The proof of \eqref{z-1} and \eqref{z-2} can follow the same proof in \cite{DSW21} except minor modifications. We aim to show that \eqref{z-3} for the case $\frac{n+\mu-2}{2}\leq\mu<4$. Define $h_j(x)=\big(1+|\xi_{ij}^{(m)}|^2+\varepsilon_0^{-2}|z_i^{(m)}-\xi_{ij}^{(m)}|^2\big)^{-1}$, we get $\tilde{\mathbf{s}}_{i,1}^{\star}(x)=\sum_{j\in\mathscr{H}(i_0) }^{\kappa}\lambda_{i}^{\frac{n-2}{2}}\mathscr{R}^{4+\mu-2n}h_j$ and
$$\Delta_{x}h_j(x)=\lambda_i^2\big(-2(n-4)\varepsilon_0^{-2}h_j^2-8\varepsilon_0^{-2}(1+|\xi_{ij}^{(m)}|^2)h_j^3\big)$$
Hence, if $n=4$ and $2\leq\mu<4$, combining $h_j\geq\frac{1}{3}\tau(z_i)^{-2}$ on the set $\{|z_i^{(m)}-\xi_{ij}^{(m)}|\leq\varepsilon_0\}$ gives
$$\Delta_xh_j=-\frac{1+|\xi_{ij}^{(m)}|^2}{1+|\xi_{ij}^{(m)}|^2+\varepsilon_0^{-2}|z_i^{(m)}-\xi_{ij}^{(m)}|^2}h_j^2\leq-\frac{1}{18}\tau(z_i)^{-4}.$$
Then, a direct computation shows that, on the set $\{|z_i^{(m)}-\xi_{ij}^{(m)}|\leq\varepsilon_0\}$,
$$\Delta_{x}\tilde{\mathbf{s}}_{i,1}^{\star}(x)\leq\sum_{j\in\mathscr{H}(i_0) }\lambda_{i}^{\frac{n-2}{2}}\mathscr{R}^{4+\mu-2n}\Delta_xh_j\leq-\frac{\kappa(\kappa-1)}{9\varepsilon_0^{2}}\hat{t}_{i,1}.$$
When $n=5$ and $3\leq\mu<4$, we obtain that on the set $\{|z_i^{(m)}-\xi_{ij}^{(m)}|\leq\varepsilon_0\}$,
$$\Delta_{x}\tilde{\mathbf{s}}_{i,1}^{\star}(x)\leq\sum_{j\in\mathscr{H}(i_0) }\lambda_{i}^{\frac{n-2}{2}}\mathscr{R}^{4+\mu-2n}\Delta_xh_j\leq-2\sum_{j\in\mathscr{H}(i_0) }\lambda_{i}^{\frac{n-2}{2}}\mathscr{R}^{4+\mu-2n}h_j^2\leq-\frac{\kappa(\kappa-1)}{9\varepsilon_0^{2}}\hat{t}_{i,1},$$
and concluding the proof.
\end{proof}

Define barrier function $\bar{S}_1(x)$ and $\bar{S}_2(x)$ is given by
$$\bar{S}_1(x)=\sum_{j\in\mathscr{H}(i_0) }\lambda_{j}^{\frac{n-2}{2}}G\Big(\frac{\mathscr{R}^{2-n}}{\tau(z_j)^{2}},\frac{\mathscr{R}^{(\mu-2-n)/2}}{\tau(z_i)^{(n+\mu-2)/2}}\Big)+\bar{\mathbf{s}}_{i_0,1}^{\star}(x)
,\hspace{5mm}\hspace{2mm}0<\mu<\frac{n+\mu-2}{2},$$
$$\hat{S}_2(x)=\sum_{j\in\mathscr{H}(i_0) }\lambda_{j}^{\frac{n-2}{2}}G\Big(\frac{\mathscr{R}^{4+\mu-2n}}{\tau(z_j)^{2}},\frac{\mathscr{R}^{\mu-n-\epsilon_0}}{\tau(z_i)^{n-2-\epsilon_0}}\Big)+\tilde{\mathbf{s}}_{i_0,1}^{\star}(x)
,\hspace{6mm}\hspace{2mm}\frac{n+\mu-2}{2}\leq\mu<4.$$
We claim that the following estimate holds true.
	\begin{Prop}\label{mingti3}
In the above notation, we have that in the region $D_i$
		\begin{equation*}
			\begin{split}			\Delta_{x}\bar{S}_1(x)\leq&-\frac{(n+\mu-2)(n-\mu-2)}{8}(1+o(1))\sum_{j\in\mathscr{H}(i_0)}(t_{i,1}+t_{i,2})\\&
-\frac{\kappa(\kappa-1)(n-4)}{16\varepsilon_0^2}t_{i_0,1},\hspace{6mm}\hspace{3mm}0<\mu<n-2,
			\end{split}
		\end{equation*}
\begin{equation*}
			\begin{split}	
\Delta_{x}\hat{S}_2(x)\leq-\frac{\epsilon_0(n-2-\epsilon_0)}{2}(1+o(1))\sum_{j\in\mathscr{H}(i_0)}(\hat{t}_{i,1}+\hat{t}_{i,2})-\frac{\kappa(\kappa-1)}{9\varepsilon_0^2}\hat{t}_{i_0,1},\hspace{3mm}\hspace{3mm}n-2\leq\mu<4.
\end{split}
\end{equation*}	
	\end{Prop}
	\begin{proof}
In case $0<\mu<\frac{n+\mu-2}{2}$,
we first set $X_j=\frac{\mathscr{R}^{2-n}}{\tau(z_i)^{2}}$, $Y_j=\frac{\mathscr{R}^{(\mu-2-n)/2}}{\tau(z_i)^{(n+\mu-2)/2}}$.
By a straightforward computation, we get
\begin{equation*}
			\begin{split}		
\Delta_{x}\bar{S}_1(x)\leq\sum_{j\in\mathscr{H}(i_0) }\lambda_{j}^{\frac{n-2}{2}}\Big(\frac{\partial G}{\partial X}(X_j,Y_j)\Delta_{x}X_j\Big)+\frac{\partial G}{\partial Y}(X_j,Y_j)\Delta_{x}Y_j\Big)+\Delta_{x}\bar{\mathbf{s}}_{i_0,1}^{\star}.
	\end{split}
		\end{equation*}		
It is easy to check that
\begin{equation*}
			\begin{split}	
\Delta_{x}Y_j=\Delta_{x}\frac{\mathscr{R}^{(\mu-2-n)/2}}{\tau(z_j)^{(n+\mu-2)/2}}&=-\frac{(n+\mu-2)(n-\mu-2)}{4}\frac{\lambda_j^2\mathscr{R}^{(\mu-2-n)/2}}{\tau(z_i)^{(n+\mu+2)/2}}-
\frac{(n+\mu)^2-4}{4}\frac{\lambda_j^2\mathscr{R}^{(\mu-2-n)/2}}{\tau(z_i)^{(n+\mu+6)/2}}
\\&\leq-\frac{(n+\mu-2)(n-\mu-2)}{4}\frac{\lambda_j^2Y_j}{\tau(z_j)^{2}},
\end{split}
		\end{equation*}	
$$\Delta_{x}X_j=\Delta_{x}\frac{\mathscr{R}^{2-n}}{\tau(z_j)^{2}}\leq-\frac{(n+\mu-2)(n-\mu-2)}{4}\frac{\lambda_j^2X_j}{\tau(z_j)^{2}}.$$
In view of $G$ is homogeneous degree $1$, we have
\begin{equation}\label{G-X-Y}
\frac{\partial G}{\partial X}(X_j,Y_j)=\frac{\partial G}{\partial X}(X_j/\tau(z_i)^{2},Y_j/\tau(z_i)^{2})\hspace{2mm}\mbox{and}\hspace{2mm} X\frac{\partial G}{\partial X}+Y\frac{\partial G}{\partial Y}=G(X,Y).
\end{equation}	
Together with the fact that $G(X,Y)\geq\frac{1}{2}\min\{X,Y\}$ and \eqref{z-1},
so that we are able to conclude that
\begin{equation*}
			\begin{split}	
\Delta_{x}\bar{S}_1(x)&\leq-\frac{(n+\mu-2)(n-\mu-2)}{4}\sum_{j\in\mathscr{H}(i_0) }\lambda_{j}^{\frac{n+2}{2}}G\Big(\frac{X_j}{\tau(z_j)^{2}},\frac{Y_j}{\tau(z_j)^{2}}\Big)+\Delta_{x}\bar{\mathbf{s}}_{i_0,1}^{\star}\\&
\leq-\frac{(n+\mu-2)(n-\mu-2)}{8}(1+o(1))\sum_{j\in\mathscr{H}(i_0)}(t_{i,1}+t_{i,2})-\frac{\kappa(\kappa-1)(n-4)}{16\varepsilon_0^2}t_{i_0,1}.
\end{split}
\end{equation*}	
In case $\frac{n+\mu-2}{2}\leq\mu<4$. As in the previous case, we take $\tilde{X}_j=\frac{\mathscr{R}^{4+\mu-2n}}{\tau(z_j)^{2}}$, $\tilde{Y}_j=\frac{\mathscr{R}^{\mu-n-\epsilon_0}}{\tau(z_j)^{n-2-\epsilon_0}}$.
Then
$$\Delta_{x}\tilde{X}_j=\Delta_{x}\frac{\mathscr{R}^{4+\mu-2n}}{\tau(z_j)^{2}}=-8\frac{\lambda_j^2\mathscr{R}^{(4+\mu-2n)}}{\tau(z_j)^{6}}\leq-8\frac{\lambda_j^2\tilde{X}_j}{\tau(z_j)^{2}},$$
\begin{equation*}
			\begin{split}	
\Delta_{x}\tilde{Y}_j=\Delta_{x}\frac{\mathscr{R}^{\mu-n-\epsilon_0}}{\tau(z_j)^{n-2-\epsilon_0}}&=-\epsilon_0(n-2-\epsilon_0)\frac{\lambda_j^2\mathscr{R}^{(\mu-n-\epsilon_0)}}{\tau(z_j)^{n-\epsilon_0}}-(n-\epsilon_0)(n-2-\epsilon_0)\frac{\lambda_j^2\mathscr{R}^{(\mu-n-\epsilon_0)}}{\tau(z_j)^{n+2-\epsilon_0}}\\&
\leq-\epsilon_0(n-2-\epsilon_0)\frac{\lambda_j^2\tilde{Y}_j}{\tau(z_j)^{2}}.
\end{split}
		\end{equation*}	
Combining all this together, and using \eqref{z-3}-\eqref{G-X-Y}, we get that
\begin{equation*}
			\begin{split}	
\Delta_{x}\hat{S}_2(x)
\leq-\frac{\epsilon_0(n-2-\epsilon_0)}{2}(1+o(1))\sum_{j\in\mathscr{H}(i_0)}(\hat{t}_{i,1}+\hat{t}_{i,2})-\frac{\kappa(\kappa-1)}{9\varepsilon_0^2}\hat{t}_{i_0,1},
\end{split}
\end{equation*}	
and the proposition follows.
\end{proof}
Now we are in position to prove step \ref{step5.3} by combining the above arguments.

\emph{Proof of step \ref{step5.3}}. We divide the proof of step \ref{step5.3} into two parts.

$\bullet$ We prove now that, for $x\in D_{i_0}^{(m)}$
\begin{equation}\label{eq557}	
-\Delta_{x}\bar{S}_1-\Phi_{n,\mu}[\sigma_m,\bar{S}_1]\geq T_1, \hspace{3mm}\hspace{2mm}0<\mu<\frac{n+\mu-2}{2}.
\end{equation}
\begin{equation}\label{xZ58}	
-\Delta_{x}\hat{S}_2-\Phi_{n,\mu}[\sigma_m,\bar{S}_2]\geq T_2, \hspace{3mm}\hspace{2mm}\frac{n+\mu-2}{2}\leq\mu<4.
\end{equation}
Recalling that $\Phi_{n,\mu}[\sigma_m,\bar{S}_i](i=1,2)$ in \eqref{I-FAI-1}. Then, for all $i>0$, there exist some constants $\textbf{C}_1=\textbf{C}(n,\mu,\kappa)$ and $\textbf{C}_2=\textbf{C}(n,\mu,\kappa)$ such that
\begin{equation*}
\left\lbrace
\begin{aligned}
&\sigma_m^{p}
\leq\sum\limits_{i=1}^{\kappa}W_i^{p}+C_1\sum\limits_{1\leq i\neq l\leq\kappa}W_i^{p-1}W_l+C_2\sum\limits_{1\leq i\neq l\leq\kappa}W_i^{p-2}W_l^2\leq\textbf{C}_1\sum\limits_{i=1}^{\kappa}W_i^{p},\\&
\sigma_m^{p-1}
\leq\sum\limits_{i=1}^{\kappa}W_i^{p-1}+C_3\sum\limits_{1\leq i\neq l\leq\kappa}W_i^{p-2}W_l\leq\textbf{C}_2\sum\limits_{i=1}^{\kappa}W_i^{p-1};\hspace{2mm}\sigma_m^{p-2}\leq\sum\limits_{i=1}^{\kappa}W_i^{p-2},
\end{aligned}
\right.
\end{equation*}
where we drop the superscript in $W_{i}^{(m)}$ if there is no confusion. Moreover, note that, by the the fact that $\frac{1}{2}\min\{X,Y\}\leq G(X,Y)\leq\min\{X,Y\}$ and $\bar{S}_1(x)\approx S_1(x)$ we have
\begin{equation}\label{sgema-0}
\frac{1}{4}S_{1}(x)\leq\bar{S}_1(x)\leq3\kappa S_1(x),\hspace{2mm}\frac{1}{4}S_{2}(x)\leq\hat{S}_2(x)\leq3\kappa S_2(x)\hspace{2mm}\mbox{on}\hspace{2mm}D_{i_0}^{(m)}.
\end{equation}
Thus,
for $m$ large, in the region $D_i^{(m)}\subset\{|z_{i_0}^{(m)}|\leq M\}$,
using Lemma \ref{B4-1} and Lemma \ref{cll-11-3}, we eventually have that
\begin{equation}\label{sgema-1}
			\begin{split}	
\Big(|x|^{-\mu}\ast\sigma_{m}^{p}\Big)
\sigma_{m}^{p-2}\bar{S}_1&\leq3\kappa\textbf{C}_1\sum\limits_{i,j\in\mathscr{H}(i_0)}\widetilde{H}_{i}(x)(s_{j,1}+s_{j,2})
+3\kappa\textbf{C}_1\sum\limits_{j\in\mathscr{H}(i_0)}\widetilde{H}_{j}(x)s_{i_0,1}\\&+(1+o(1))\textbf{C}_1\widetilde{H}_{i_0}(x)\Big(\sum\limits_{j\in\mathscr{H}(i_0)}(s_{j,1}+s_{j,2})+s_{i_0,1}\Big),
\hspace{2mm}\mbox{for}\hspace{2mm}0<\mu<\frac{n+\mu-2}{2},
\end{split}
\end{equation}
We will now use Lemma \ref{cll-1-0-1} to construct bounded for the RHL of \eqref{sgema-1}.
Then, with the help of Lemma \ref{cll-1-0-1}, by \eqref{zzz-0-1}-\eqref{zzz-1-1}, \eqref{zzz-2-1}, \eqref{zzz-6-1} we obtain
\begin{equation*}
\widetilde{H}_{i}(x)(s_{j,1}+s_{j,2})\lesssim[M^{-2}+o(1)](t_{i,1}+t_{i,2}+t_{j,1}+t_{j,2}),
\hspace{2mm}\mbox{for}\hspace{2mm}i\in\mathscr{H}(i_0)\hspace{2mm}\mbox{and}\hspace{2mm}x\in D_{i_0}^{(m)},
\end{equation*}
\begin{equation*}
\widetilde{H}_{j}(x)s_{i_0,1}\lesssim o(1)(t_{i_0,1}+t_{j,1}+t_{j,2}),\hspace{2mm}\mbox{for}\hspace{2mm}j\in\mathscr{H}(i_0),
\end{equation*}
Then, using \eqref{zzz-2-1} and \eqref{zzz-7-1} implies that for all $j\in\mathscr{H}(i_0)$
\begin{equation*}
\widetilde{H}_{i_0}(x)(s_{j,1}+s_{j,2})+\widetilde{H}_{i_0}s_{i_0,1}\lesssim o(1)t_{j,1}+t_{i_0,1}+(C^{\star})^{\varrho_1}\Big(\varepsilon^{-\frac{4}{n+\mu-2}}t_{i,1}+\varepsilon t_{j,2}\Big)\hspace{2mm}\mbox{with}\hspace{2mm}\varrho_1=\frac{2(n+\mu-6)}{n+\mu+2}
\end{equation*}
for all $0<\varepsilon<1$ and where $C^{\star}$ is defined in \eqref{AA-00}. Summarizing, by plugging the above estimates in \eqref{sgema-1}, we get that
\begin{equation}\label{sgema-10}
\begin{split}
\Big(|x|^{-\mu}\ast\sigma_{m}^{p}\Big)
\sigma_{m}^{p-2}\bar{S}_1\lesssim[(C^{\star})^{\varrho_1}\varepsilon+M^{-2}+o(1)]\sum\limits_{j\in\mathscr{H}(i_0)}(t_{j,1}+t_{j,2})
+[(C^{\star})^{\varrho_1}\varepsilon^{-\frac{4}{n+\mu-2}}+o(1)]t_{i_0,1}
\end{split}
\end{equation}
for all $0<\varepsilon<1$ and where $\varrho_1=\frac{2(n+\mu-6)}{n+\mu+2}$.

In order to conclude the proofs of \eqref{eq557} and \eqref{xZ58}, recalling that $\widehat{K}_{i}(x)$, $\widehat{H}_{ij,1}(x)$, $\widehat{H}_{ij,2}(x)$, and further set
\begin{equation}\label{sgema-4}
\mathcal{H}_{i}(x)=\Big(|x|^{-\mu}\ast\big(\frac{\mathscr{R}^{2-n}\lambda_i^{(2n-\mu)/2}}{(1+|z_{i}|^2)^{(2n-\mu)/2}}+
\frac{\mathscr{R}^{-\frac{n-\mu+2}{2}}\lambda_i^{(2n-\mu)/2}}{(1+|z_{i}|^2)^{(2n-\mu)/2}}\big)\Big)\widehat{K}_i(x),
\end{equation}
$$
\mathcal{I}_{ij,1}=\big(|x|^{-\mu}\ast\widehat{H}_{ij,1}(x)\big)\widehat{K}_i(x),\hspace{3mm}\mathcal{L}_{ij,1}:=\big(|x|^{-\mu}\ast\widehat{H}_{ij,2}(x)\big)\widehat{K}_i(x)
$$
and then as a consequence
\begin{equation}\label{sgema-2}
			\begin{split}	
\Big(|x|^{-\mu}\ast\sigma_{m}^{p-1}\bar{S}_1\Big)
\sigma_{m}^{p-1}\lesssim 3\kappa\textbf{C}_2^2\sum\limits_{i,j=1}\Big(\mathcal{I}_{ij,1}+\mathcal{L}_{ij,1}\Big)+3\kappa\textbf{C}_2^2\sum\limits_{i=1}\mathcal{H}_{i}(x),\hspace{2mm}0<\mu<\frac{n+\mu-2}{2}.
\end{split}
\end{equation}
In particular, in the region $D_i^{(m)}\subset\{|z_{i_0}^{(m)}|\leq M\}$, for $m$ large, using \eqref{sgema-0} and Lemma \ref{cll-11-3}, we have the following:
\begin{equation*}
\sum\limits_{i,j=1}^{\kappa}\mathcal{I}_{ij,1}\leq
\left\lbrace
\begin{aligned}
&C\sum\limits_{i,j\in\mathscr{H}(i_0)}\mathscr{R}^{2-n}\lambda_{j}^{\mu/2}\mathcal{Z}_1\big(\widehat{K}_{i}(x)+(1+o(1))\widehat{K}_{i_0}(x)\big),
\hspace{4mm}\hspace{2mm}n=5\hspace{2mm}\mbox{and}\hspace{2mm}1\leq\mu<3,\\& C\sum\limits_{i,j\in\mathscr{H}(i_0)}\mathscr{R}^{2-n}\lambda_{j}^{\mu/2} \mathcal{Z}_2\big(\widehat{K}_i(x)+(1+o(1))\widehat{K}_{i_0}(x)\big),\hspace{4mm} \hspace{2mm}n=6\hspace{2mm}\mbox{and}\hspace{2mm}0<\mu<4,\\& C\sum\limits_{i,j\in\mathscr{H}(i_0)}\mathscr{R}^{2-n}\lambda_{j}^{\mu/2}\mathcal{Z}_3\big(\widehat{K}_i(x)+(1+o(1))\widehat{K}_{i_0}(x)\big),
\hspace{4mm}\hspace{2mm}n=7\hspace{2mm}\mbox{and}\hspace{2mm}\frac{7}{3}<\mu\leq4.
\end{aligned}
\right.
\end{equation*}
where $\mathcal{Z}_1:=\tau(z_j)^{-\min\{\mu,(5+\mu)/3\}}$, $\mathcal{Z}_2:=\tau(z_j)^{-(2p-n+\mu-\theta_1)}$ and $\mathcal{Z}_3:=\tau(z_j)^{-(2p-n+\mu-\theta_2)}$.
Furthermore, we ahve
\begin{equation*}
\begin{split}
\sum\limits_{i,j=1}^{\kappa}\mathcal{L}_{ij,1} \leq C\sum\limits_{i,j\in\mathscr{H}(i_0)}\mathscr{R}^{-\frac{n-\mu+2}{2}}\lambda_{j}^{\mu/2}\tau(z_j)^{\mu}\big(\widehat{K}_i(x)+(1+o(1))\widehat{K}_{i_0}(x)\big).
\end{split}
\end{equation*}
Then, combining the above inequalities with \eqref{zzzz-0} and \eqref{zzzz-2} gives that for $i\in\mathscr{H}(i_0)\hspace{2mm}\mbox{and}\hspace{2mm}x\in D_{i_0}^{(m)}$,
\begin{equation*}
\mathcal{I}_{ij,1}\lesssim o(1)(t_{i,1}+t_{i_0,1}+t_{j,1}+t_{j,2}),\hspace{2mm}\hspace{2mm}n=5\hspace{2mm}\mbox{and} \hspace{2mm}1\leq\mu<3,
\end{equation*}
Similarly we compute and give that
\begin{equation*}
\mathcal{I}_{ij,1}\lesssim o(1)(t_{i,1}+t_{i_0,1}+t_{j,1}+t_{j,2}),\hspace{3mm}\hspace{2mm}n=6\hspace{2mm}\mbox{and} \hspace{2mm}0<\mu<4,
\end{equation*}
\begin{equation*}
\mathcal{I}_{ij,1}\lesssim o(1)(t_{i,1}+t_{i_0,1}+t_{j,1}+t_{j,2}),\hspace{3mm}\hspace{2mm}n=7\hspace{2mm}\mbox{and} \hspace{2mm}\frac{7}{3}<\mu<4,
\end{equation*}
and from \eqref{zzzz-1} and \eqref{zzzz-5} that for $0<\mu<\frac{n+\mu-2}{2}$
\begin{equation*}
\mathcal{L}_{ij,1}\lesssim [M^{-\frac{n-\mu+2}{2}}+o(1)](t_{i,2}+t_{j,1}+t_{j,2})+(C^{\star})^{\zeta_1^{\star}}\Big(\varepsilon^{-\frac{2\mu}{n-\mu+2}}t_{i_0,1}+\varepsilon t_{j,2}\Big)\hspace{2mm}\mbox{with}\hspace{2mm}\zeta_1^{\star}=\frac{(n-\mu+2)\mu}{n+\mu+2}.
\end{equation*}
Recall that \eqref{sgema-4}, by a straightforward computation, it follows that
\begin{equation*}
\mathcal{H}_i\lesssim\Big(\mathscr{R}^{-\frac{n+\mu-2}{2}}+\mathscr{R}^{2-n}\Big)\frac{\lambda_{i}^{(n+2)/2}}{\tau(z_{i})^{n+2}}
\lesssim[M^{-\frac{n-\mu+2}{2}}+M^{2-n}](t_{i,1}+t_{i_0,1}+t_{i,2}),\hspace{2mm}x\in D_{i_0}^{(m)}.
\end{equation*}
Together with the above estimates, \eqref{sgema-1} and \eqref{sgema-2}, we eventually have
\begin{equation}\label{sigema-z1}
\begin{split}	
\Phi_{n,\mu}[\sigma_m,\bar{S}_1]&=(p-1)\Big(|x|^{-\mu}\ast\sigma_{m}^{p}\Big)
\sigma_{m}^{p-2}\bar{S}_1+p\Big(|x|^{-\mu}\ast\sigma_{m}^{p-1}\bar{S}_1\Big)
\sigma_{m}^{p-1}\\&
\lesssim[(C^{\star})^{\varrho_1}\varepsilon+(C^{\star})^{\zeta_1^{\star}}\varepsilon+M^{-2}+M^{-\frac{n-\mu+2}{2}}+o(1)]\sum_{j\in\mathscr{H}(i_0)}(t_{j,1}+t_{j,2})
\\&+[(C^{\star})^{\varrho_1}\varepsilon^{-\frac{4}{n+\mu-2}}+(C^{\star})^{\zeta_1^{\star}}\varepsilon^{-\frac{2\mu}{n-\mu+2}}+M^{-\frac{n-\mu+2}{2}}+o(1)]t_{i_0,1}
\end{split}
\end{equation}
for all $0<\varepsilon<1$ and where $\varrho_1=\frac{2(n+\mu-6)}{n+\mu+2}$ and $\zeta_1^{\star}=\frac{(n-\mu+2)\mu}{n+\mu+2}$. Furthermore, it follows from Proposition \ref{mingti3} that
\begin{equation*}
			\begin{split}	
&-\Delta_{x}\bar{S}_1-\Phi_{n,\mu}[\sigma_m,\bar{S}_1]\\&\geq[\frac{(n+\mu-2)(n-\mu-2)}{8}-\tilde{\textbf{C}}_{n,\mu,\kappa}(\bar{A}^{\varrho_1}\varepsilon+\bar{A}^{\zeta_1^{\star}}\varepsilon+M^{-2}+M^{-\frac{n-\mu+2}{2}})-o(1)]\sum_{j\in\mathscr{H}(i_0)}(t_{i,1}+t_{i,2})\\&
+[\frac{\kappa(\kappa-1)(n-4)}{16\varepsilon_0^2}-\tilde{\textbf{C}}_{n,\mu,\kappa}(\bar{A}^{\varrho_1}\varepsilon^{-\frac{4}{n+\mu-2}}+\bar{A}^{\zeta_1^{\star}}\varepsilon^{-\frac{2\mu}{n-\mu+2}}+M^{-\frac{n-\mu+2}{2}})-o(1)]t_{i_0,1}
\end{split}
\end{equation*}
with $\varrho_1=\frac{2(n+\mu-6)}{n+\mu+2}$ and $\zeta_1^{\star}=\frac{(n-\mu+2)\mu}{n+\mu+2}$, and where  $\tilde{\textbf{C}}_{n,\mu,\kappa}$ is constant hidden in $\lesssim$-notation in inequalities \eqref{sgema-10} and \eqref{sgema-2}-\eqref{sigema-z1}.
With the above inequality at our disposal, we can easily conclude the proof of \eqref{eq557}. Indeed, choosing $\varepsilon(n,\mu,\kappa,C^{\star})$ small, $M=M(n,\mu,\kappa,C^{\star})$ and $m$ large enough such that
\begin{equation*}
			\begin{split}	
\frac{(n+\mu-2)(n-\mu-2)}{8}-\tilde{\textbf{C}}_{n,\mu,\kappa}((C^{\star})^{\varrho_1}\varepsilon+(C^{\star})^{\zeta_1^{\star}}\varepsilon+M^{-2}+M^{-\frac{n-\mu+2}{2}})-o(1)\geq1,
\end{split}
\end{equation*}			
and we may choose $\varepsilon_0$ small such that
\begin{equation*}
			\begin{split}
\frac{\kappa(\kappa-1)(n-4)}{16\varepsilon_0^2}-\tilde{\textbf{C}}_{n,\mu,\kappa}((C^{\star})^{\varrho_1}\varepsilon^{-\frac{4}{n+\mu-2}}+(C^{\star})^{\zeta_1^{\star}}\varepsilon^{-\frac{2\mu}{n-\mu+2}}+M^{-\frac{n-\mu+2}{2}})-o(1)\geq1.
\end{split}
\end{equation*}
The result follows. The proof of \eqref{xZ58} follows by simple computations from Lemma \ref{cll-1-0}, Lemma \ref{cll-3}, Proposition \ref{mingti2} and Proposition \ref{mingti3} similar to \eqref{eq557} and so we omit it.

$\bullet$ Given $x\in D_{i_0}^{(m)}$, we aim to show that
\begin{equation}\label{xZ-3}
|\phi_{m}^{(1)}(x)|\leq \frac{1}{6\kappa}\bar{S}_1(x),\hspace{3mm}\hspace{2mm}0<\mu<\frac{n+\mu-2}{2},
\end{equation}
\begin{equation}\label{xZ-4}
|\phi_{m}^{(2)}(x)|\leq \frac{1}{6\kappa}\hat{S}_2(x),\hspace{3mm}\hspace{2mm}\frac{n+\mu-2}{2}\leq\mu<4.
\end{equation}
We set $g_{\pm}^{(i)}(x)=\frac{1}{6\kappa}\bar{S}_i(x)\pm\phi_{m}^{(i)}$ where $i=1,2$.
Observe first by \eqref{00} that $\Delta_{x} \phi_m+\mathscr{I}[\sigma_m,\phi_m]=\hbar_m\hspace{1.14mm}\mbox{in}\hspace{1.14mm} \mathbb{R}^n$ so that
\begin{equation}\label{xZ-ii-0}
\Delta g_{\pm}^{(1)}+\Phi_{n,\mu}[\sigma_m,g_{\pm}^{(1)}]=\frac{1}{6\kappa}[\Delta \bar{S}_1+\Phi_{n,\mu}[\sigma_m,g_{\pm}^{(1)}]]\pm\hbar_m\leq0\hspace{2mm}\mbox{on}\hspace{2mm}D_{i_0}^{(m)},\hspace{2mm}\mbox{for}\hspace{2mm}0<\mu<\frac{n+\mu-2}{2}.
\end{equation}
Furthermore, we claim that
\begin{equation}\label{xZ-ii-1}
g_{\pm}^{(1)}(x)\geq0\hspace{2mm}\mbox{on}\hspace{2mm}D_{i_0}^{(m)},\hspace{2mm}\mbox{for}\hspace{2mm}0<\mu<\frac{n+\mu-2}{2},
\end{equation}
By contradiction, if \eqref{xZ-ii-1} does not hold true, we may assume that there exists a minimum point $x_0$ of $\tilde{g}_{\pm}^{(1)}(x):=g_{\pm}^{(1)}(x)/\bar{S}_1$ such that $\tilde{g}_{\pm}^{(1)}(x_0)<0$. Then, we have
$$
p\Big(|x|^{-\mu}\ast\sigma_{m}^{p-1}\bar{S}_1\Big)
\sigma_{m}^{p-1}\tilde{g}_{\pm}^{(1)}(x_0)= p\Big(|x|^{-\mu}\ast\sigma_{m}^{p-1}g_{\pm}^{(1)}(x_0)\Big)
\sigma_{m}^{p-1}.
$$
Combining \eqref{eq557} and \eqref{xZ-ii-0}, we conclude that
\begin{equation*}
\begin{split}
\Delta\tilde{g}_{\pm}^{(1)}(x_0)&=-2\nabla \tilde{g}_{\pm}^{(1)}(x_0)\frac{\nabla\bar{S}_1}{\bar{S}_1}+\frac{1}{\bar{S}_1}\big(\Delta_{x}g_{\pm}^{(1)}(x_0)-\frac{\Delta_{x}\bar{S}_1}{\bar{S}_1}g_{\pm}^{(1)}(x_0)\big)\\&
\leq-2\nabla \tilde{g}_{\pm}^{(1)}(x_0)\frac{\nabla\bar{S}_1}{\bar{S}_1}+\frac{1}{\bar{S}_1}\big(\Delta_{x}g_{\pm}^{(1)}(x_0)+\frac{\Phi_{n,\mu}[\sigma_m,\bar{S}_1]}{{\bar{S}_1}}g_{\pm}^{(1)}(x_0)+\frac{T_1}{{\bar{S}_1}}g_{\pm}^{(1)}(x_0)\big) <0.
\end{split}
\end{equation*}
which gives a contradiction and concludes the proof of \eqref{xZ-3} by \eqref{sgema-0}. The proof of \eqref{xZ-4} follows the strategy of \eqref{xZ-3} by exploiting \eqref{xZ58} and so we skip it. \qed

\begin{step}\label{step5.7}
Conclusion.
\end{step}
\emph{Proof of step \ref{step5.7}.} The conclusion follows by Steps \ref{step5.1}-\ref{step5.2} and \ref{step5.3}. Lemma \ref{estimate2} is proven. \qed

\section{Existence of the first approximation function}
We complete here the proof of Theorem \ref{Figalli} by proving Proposition~\ref{estimate1} and Proposition \ref{estimate2}.
Let $c_{a}^i$ be a family of scalar solving,
\begin{equation}\label{c1-00}
	\left\{\begin{array}{l}
		\displaystyle \Delta (\sigma+\phi)+\Big(|x|^{-\mu}\ast \big(\sigma+\phi\big)^{p}\Big)
\big(\sigma+\phi\big)^{p-1}
		\displaystyle =
		\sum_{i=1}^{\kappa}\sum_{a=1}^{n+1}c_{a}^{i}\Phi_{n,\mu}[W_{i},\Xi^{a}_i]\hspace{4.14mm}\mbox{in}\hspace{1.14mm} \mathbb{R}^n,\\
		\displaystyle \int\Phi_{n,\mu}[W_{i},\Xi^{a}_i]\phi=0,\hspace{4mm}i=1,\cdots, \kappa; ~a=1,\cdots,n+1.
	\end{array}
	\right.
\end{equation}
 Recalling $\hbar$ and $\mathscr{N}(\phi)$ from \eqref{u-1}-\eqref{u-2}, then \eqref{c1-00} is written as
\begin{equation}\label{AA}
\begin{split}
	\Delta \psi&+\Phi_{n,\mu}[\sigma,\phi]+\hbar+\mathscr{N}(\phi)=\sum_{i=1}^{\kappa}\sum_{a=1}^{n+1}c_{a}^{i}\Phi_{n,\mu}[W_{i},\Xi^{a}_i].
\end{split}
\end{equation}

\subsection{Lipschitz character of the higher-order term}\
\newline
In order to use the Contraction Mapping Theorem to prove that \eqref{AA} is uniquely solvable in
the set that $\|\phi\|_{\ast}$ is small, we estimate all the terms in
$\mathscr{N}(\phi)=\mathscr{N}_1+\mathscr{N}_2+\cdots+\mathscr{N}_9$
and $\hbar$ in the $\|\cdot\|_{\ast\ast}$-norm, defined in \eqref{h}.
\begin{lem}\label{ww101}
For all $i\in\{1,\cdots,\kappa\}$,
	$n\geq6-\mu$, $\mu\in(0,n)$ and $0<\mu\leq4$, there exist two constants $\delta_0=\delta_0(n,\kappa,\mu)$ and $\nu_0=\nu_0(n,\kappa,\mu)$, depending on $n$, $\kappa$ and $\mu$, such that
\begin{equation}\label{N-15}
\| \mathscr{N}(\phi)\|_{\ast\ast}\leq
\delta_0\mathscr{R}^{(p-2)(\mu-n-\epsilon_0)}\|\phi\|_{\ast}^{p-1},\hspace{2mm}0<\mu<\frac{n+\mu-2}{2},
\end{equation}
where $0<\epsilon_0<((n-2)p-n)/p$ and that
\begin{equation*}
\|\mathscr{N}(\phi)\|_{\ast\ast}\leq
\nu_0(\mathscr{R}^{-(4-\mu-\theta_1)}+\mathscr{R}^{-\frac{n-\mu+2}{2}(p-2)})\|\phi\|_{\ast}^{p-1},\hspace{2mm}\frac{n+\mu-2}{2}\leq\mu<4,
\end{equation*}
where $0<\theta_1<4-\mu$ is defined in \eqref{ceta}.
\end{lem}
\begin{proof}
The point is to estimate each integral of the previous decomposition in $\mathscr{N}(\phi)$ in the $\|\cdot\|_{\ast\ast}$-norm.
The following estimates hold true:
\begin{equation}\label{xigma}
\begin{split}
&\|\sigma\|_{\ast}\leq C\mathscr{R}^{n-2},\hspace{2mm}\mbox{for}\hspace{2mm}|z_i|\leq\mathscr{R}, \hspace{2mm}0<\mu<\frac{n+\mu-2}{2},\\&
\|\sigma\|_{\ast}\leq C\mathscr{R}^{2n-4-\mu},\hspace{2mm}\mbox{for}\hspace{2mm}|z_i|\leq\mathscr{R}, \hspace{2mm}\frac{n+\mu-2}{2}\leq\mu<4.
\end{split}
\end{equation}
\begin{equation}\label{xigma1}
\begin{split}
 &\|\sigma\|_{\ast}\leq C\mathscr{R}^{\frac{n-\mu+2}{2}},\hspace{2mm}\mbox{for}\hspace{2mm}|z_i|\geq\mathscr{R}, \hspace{2mm}0<\mu<\frac{n+\mu-2}{2},\\
&\|\sigma\|_{\ast}\leq C\mathscr{R}^{n-\mu+\epsilon_0},\hspace{2mm}\mbox{for}\hspace{2mm}|z_i|\geq\mathscr{R}, \hspace{2mm}\frac{n+\mu-2}{2}\leq\mu<4.
\end{split}
\end{equation}
\emph{Estimates of $\|\mathscr{N}_{1}\|_{\ast\ast}$.}
We introduce the functions
$$\widehat{G}_{\lambda,\mu}(x):=\bigg[\frac{1}{|x|^{\mu}}\ast \Big(\sum_{i=1}^{\kappa}\frac{\lambda_i^{(n-2)/2}}{\tau( z_i)^{2}}\Big)^{p}\bigg]\bigg[\sum_{i=1}^{\kappa}
\frac{\lambda_i^{(n-2)/2}}{\tau(z_i)^{2}}\bigg]^{p-1},$$
$(1)$ Exploiting the Lemma \ref{B4}, \eqref{zv5} and \eqref{xigma},
the following estimates hold for $\mathscr{N}_{1}$:\\
When $n=4$ and $\mu\in[2,4)$, for $|z_i|\leq\mathscr{R}$,
\begin{equation*}
\aligned
\mathscr{N}_1(\phi)
&\leq \frac{C}{\mathscr{R}^{(2n-4-\mu)(2p-1)}}\|\phi\|_{\ast}^{p-1}\|\sigma\|_{\ast}^{p}\widehat{G}_{\lambda,\mu}(x)
\leq C \|\phi\|_{\ast}^{p-1}\mathscr{R}^{-(2n-\mu-4)(p-2)}\hat{t}_{i,1}(x,\mathscr{R}).
\endaligned
\end{equation*}
$\mbox{When}\hspace{2mm}n=5\hspace{2mm}\mbox{and}\hspace{2mm}\mu\in[1,4)$,
\begin{equation*}
\aligned
\mathscr{N}_1(\phi)
&\leq \big(\frac{C}{\mathscr{R}^{3n-2\mu+2}}+\frac{C}{\mathscr{R}^{(2n-4-\mu)(2p-1)}}\big)\big\|\phi\big\|_{\ast}^{p-1}\big\|\sigma\big\|_{\ast}^{p}\widehat{G}_{\lambda,\mu}(x)
\\&\leq C \big\|\phi\big\|_{\ast}^{p-1}(\mathscr{R}^{\mu-4}t_{i,1}(x,\mathscr{R})+\mathscr{R}^{-(2n-\mu-4)(p-2)}\hat{t}_{i,1}(x,\mathscr{R})),\hspace{2mm}\mbox{for}\hspace{2mm}|z_i|\leq\mathscr{R}.
\endaligned
\end{equation*}
where we used the fact that $$\min\Big\{2(p-1)+\mu,2(p-1)+(5+\mu)/3\Big\}\geq4.$$
$\mbox{If}\hspace{2mm}n=6\hspace{2mm}\mbox{and}\hspace{2mm}0<\mu<4$, or $n=7\hspace{2mm}\mbox{and}\hspace{2mm}\frac{7}{3}<\mu<4$,
\begin{equation*}
\aligned
\mathscr{N}_1(\phi)
&\leq \frac{C}{\mathscr{R}^{3n-2\mu+2}}\big\|\phi\big\|_{\ast}^{p-1}\big\|\sigma\big\|_{\ast}^{p}\widehat{G}_{\lambda,\mu}(x)
\leq C \big\|\phi\big\|_{\ast}^{p-1}\big(\frac{t_{i,1}(x,\mathscr{R})}{\mathscr{R}^{4-\mu-\theta_1}}+\mathscr{R}^{\mu-4}t_{i,1}(x,\mathscr{R})\big),\hspace{2mm}\mbox{for}\hspace{2mm}|z_i|\leq\mathscr{R}.
\endaligned
\end{equation*}
$(2)$
If $n=4$ and $\mu\in[2,4)$, or $n=5\hspace{2mm}\mbox{and}\hspace{2mm}\mu\in(3,4)$, in this case we work in \eqref{zv4}-\eqref{zv5} for $S_2(x)$ and $T_2(x)$. We set
$$\widetilde{G}_{\lambda,\mu}(x):=\bigg[\frac{1}{|x|^{\mu}}\ast \Big(\sum_{i=1}^{\kappa}\frac{\lambda_i^{(n-2)/2}}{\tau( z_i)^{n-2-\epsilon_{0}}}\Big)^{p}\bigg]\bigg[\sum_{i=1}^{\kappa}
\frac{\lambda_i^{(n-2)/2}}{\tau(z_i)^{n-2-\epsilon_{0}}}\bigg]^{p-1},\hspace{2mm}|z_i|\geq\mathscr{R}.$$
Owing to the Lemma \ref{B4} and \eqref{xigma1},
we get
\begin{equation*}
\aligned
\mathscr{N}_1(\phi)
&\leq C\|\phi\|_{\ast}^{p-1}\|\sigma\|_{\ast}^{p}
\mathscr{R}^{-(2p-1)(n-\mu+\epsilon_0)}\widetilde{G}_{\lambda,\mu,\mathscr{R}}(x)
\leq C\mathscr{R}^{-(p-2)(n-\mu+\epsilon_0)}\big\|\phi\big\|_{\ast}^{p-1}\hat{t}_{i,2}(x,\mathscr{R}).
\endaligned
\end{equation*}
$(3)$ If $n=5\hspace{2mm}\mbox{and}\hspace{2mm}\mu\in[1,3)$, or $n=6\hspace{2mm}\mbox{and}\hspace{2mm}\mu\in(0,4)$,  or $n=7\hspace{2mm}\mbox{and}\hspace{2mm}\mu\in(\frac{7}{3},4)$, in this case we work in \eqref{zv4}-\eqref{zv5} for $S_1(x)$ and $T_1(x)$.
Let $\check{Z}_{\lambda,\mu}$ is given by
$$\check{G}_{\lambda,\mu}(x)=\bigg[\frac{1}{|x|^{\mu}}\ast \Big(\sum_{i=1}^{\kappa}\frac{\lambda_i^{(n-2)/2}}{\tau( z_i)^{(n+\mu-2)/2}}\Big)^{p}\bigg]\bigg[\sum_{i=1}^{\kappa}
\frac{\lambda_i^{(n-2)/2}}{\tau(z_i)^{(n+\mu-2)/2}}\bigg]^{p-1}.$$
There are three cases:\\
$\mathbf{(i).}$ If $n=5\hspace{2mm}\mbox{and}\hspace{2mm}\mu\in[1,3)$, by the Lemma \ref{B4} and \eqref{xigma1}, we obtain
\begin{equation*}
\aligned
\mathscr{N}_1(\phi)
&\leq C\|\phi\|_{\ast}^{p-1}\|\sigma\|_{\ast}^{p}
\mathscr{R}^{-\frac{n-\mu+2}{2}(2p-1)}\check{G}_{\lambda,\mu}(x)
\leq C\mathscr{R}^{-\frac{n-\mu+2}{2}(p-2)}\|\phi\|_{\ast}^{p-1}t_{i,2}(x,\mathscr{R}),
\endaligned
\end{equation*}
since \begin{equation}\label{nu1}
\min\Big\{\mu+\frac{(n+\mu-2)}{2}(p-1),\frac{5+\mu}{3}+\frac{(n+\mu-2)}{2}(p-1)\Big\}\geq \frac{n+\mu+2}{2}.
\end{equation}
$\mathbf{(ii).}$
If $n=6\hspace{2mm}\mbox{and}\hspace{2mm}\mu\in(0,4)$, by notice that, since \eqref{ceta}, we deduce that
\begin{equation}\label{ceta1}
2p-n+\mu-\theta_3+\frac{(n+\mu-2)}{2}(p-1)\geq\frac{n+\mu+2}{2}.
\end{equation}
Combining the Lemma \ref{B4} and \eqref{xigma1} gives that
\begin{equation}\label{n-1}
\aligned
\mathscr{N}_1(\phi)
&\leq C\big\|\phi\big\|_{\ast}^{p-1}\big\|\sigma\big\|_{\ast}^{p}
\mathscr{R}^{-\frac{n-\mu+2}{2}(2p-1)}\check{G}_{\lambda,\mu}(x)
\leq C\mathscr{R}^{-\frac{n-\mu+2}{2}(p-2)}\big\|\phi\big\|_{\ast}^{p-1}t_{i,2}(x,\mathscr{R}).
\endaligned
\end{equation}
$\mathbf{(iii).}$ If $n=7$ and $\frac{7}{3}\leq\mu<4$, similar to the argument of \eqref{n-1}, in virtue of $\mu+\frac{(n+\mu-2)}{2}(p-1)>\frac{n+\mu+2}{2}$, we obtain
 \begin{equation}\label{nu6-1}
\aligned
\mathscr{N}_1(\phi)\leq C\mathscr{R}^{-\frac{n-\mu+2}{2}(p-2)}\|\phi\|_{\ast}^{p-1}t_{i,2}(x,\mathscr{R}),
\endaligned
\end{equation}
Putting these estimates together, we eventually get
\begin{equation*}
\|\mathscr{N}_1(\phi)\|_{\ast\ast}\leq C_1\mathscr{R}^{-(p-2)(n-\mu+\epsilon_0)}\big\|\phi\big\|_{\ast}^{p-1}
\end{equation*}
if the parameters $n$ and $\mu$ are chosen in the following range
$$n=4\hspace{2mm}\mbox{and}\hspace{2mm}\mu\in[2,4),\hspace{2mm}\mbox{or} \hspace{2mm}n=5\hspace{2mm}\mbox{and}\hspace{2mm}\mu\in(3,4).$$
And we have
\begin{equation*}
\|\mathscr{N}_1(\phi)\|_{\ast\ast}\leq
C_2(\mathscr{R}^{-(4-\mu-\theta_1)}+\mathscr{R}^{-\frac{n-\mu+2}{2}(p-2)})\|\phi\|_{\ast}^{p-1},
\end{equation*}
if the parameters $n$ and $\mu$ are chosen in the following range
$$n=5\hspace{2mm}\mbox{and}\hspace{2mm}\mu\in[1,3),\hspace{2mm}\mbox{or} \hspace{2mm}n=6\hspace{2mm}\mbox{and}\hspace{2mm}\mu\in(0,4),\hspace{2mm}\mbox{or} \hspace{2mm}n=7\hspace{2mm}\mbox{and}\hspace{2mm}\mu\in(\frac{7}{3},4).$$
Here we have used the elementary inequality \eqref{AB-0}.

In a very similar way one gets the estimate of $N_2,\cdots,N_9$ in the $\|\cdot\|_{\ast\ast}$-norm.

\emph{Estimates of $\|\mathscr{N}_{i}\|_{\ast\ast}(i=2,\dots,9).$} If $n=4\hspace{2mm}\mbox{and}\hspace{2mm}\mu\in[2,4),\hspace{2mm}\mbox{or} \hspace{2mm}n=5\hspace{2mm}\mbox{and}\hspace{2mm}\mu\in[3,4)$, we get that
\begin{equation}\label{N15}
\sum_{i=2}^{9}\| \mathscr{N}_i\|_{\ast\ast}\leq
C_3\mathscr{R}^{-(p-2)(n-\mu+\epsilon_0)}(\|\phi\|_{\ast}^{2}+\|\phi\|_{\ast}^{3}+\|\phi\|_{\ast}^{p}+\|\phi\|_{\ast}^{p+1}+\|\phi\|_{\ast}^{2p+1}),
\end{equation}
and that when the parameters $n$ and $\mu$ are chosen in the following range
$$n=5\hspace{2mm}\mbox{and}\hspace{2mm}\mu\in[1,3),\hspace{2mm}\mbox{or} \hspace{2mm}n=6\hspace{2mm}\mbox{and}\hspace{2mm}\mu\in(0,4),\hspace{2mm}\mbox{or} \hspace{2mm}n=7\hspace{2mm}\mbox{and}\hspace{2mm}\mu\in(\frac{7}{3},4),$$
one has
\begin{equation*}
\sum_{i=2}^{9}\|\mathscr{N}_i\|_{\ast\ast}\leq
C_4(\mathscr{R}^{-(4-\mu-\theta_1)}+\mathscr{R}^{-\frac{n-\mu+2}{2}(p-2)})(\|\phi\|_{\ast}^{2}+\|\phi\|_{\ast}^{3}+\|\phi\|_{\ast}^{p}+\|\phi\|_{\ast}^{p+1}+\|\phi\|_{\ast}^{2p+1}).
\end{equation*}
Here $C_i$ are the strictly positive constant which depends only on $n$, $\kappa$ and $\mu$. Combining the estimates of $\|\mathscr{N}_{1}\|_{\ast\ast}$, $\|\mathscr{N}_{2}\|_{\ast\ast}$, $\cdots$, and $\|\mathscr{N}_{9}\|_{\ast\ast}$ yields the conclusion.
\end{proof}
\subsection{$C^0$ estimates of solutions} First of all, we have the following.
\begin{lem}\label{wwc102}
There exists a constant $\zeta_0=\zeta_0(n,\kappa,\mu)$, depending on $n$, $\kappa$ and $\mu$, such that
\begin{equation*}
\|\hbar\|_{\ast\ast}\leq\zeta_0.
\end{equation*}
\end{lem}
\begin{proof}
It follows from the estimate \eqref{eq3.7} of Lemma \ref{estimate1}.
\end{proof}
\begin{lem}\label{ww102}
Assume that $n\geq6-\mu$, $\mu\in(0,n)$ and $0<\mu\leq4$.
Then there exist $\rho_0 > 0$ and a family of scalars $(c_{a}^{i})$ \eqref{c1} which solve \eqref{AA}
such that for $\delta$ is small enough, there holds
\begin{equation}\label{pp1}
|\rho_{0}|\leq
\left\lbrace
\begin{aligned}
& CS_{1}(x),\hspace{4mm}n=5\hspace{2mm}\mbox{and}\hspace{2mm}\mu\in[1,3),\mbox{or}\hspace{2mm}n=6\hspace{2mm}\mbox{and}\hspace{2mm}\mu\in(0,4),\mbox{or}\hspace{2mm}n=7\hspace{2mm}\mbox{and}\hspace{2mm}\mu\in(\frac{7}{3},4),\\
& CS_{2}(x),\hspace{4mm}n=4\hspace{2mm}\mbox{and}\hspace{2mm}\mu\in[2,4),\mbox{or}\hspace{2mm}n=5\hspace{2mm}\mbox{and}\hspace{2mm}\mu\in[3,4).
\end{aligned}
\right.
\end{equation}
\end{lem}
\begin{proof}
Observe that, \eqref{AA} is equivalent to
\begin{equation}\label{c19}
\phi=\mathcal{A}(\phi)=:-\mathcal{L}_{\delta}(\mathscr{N}(\phi))-\mathcal{L}_{\delta}(\hbar),
\end{equation}
where $\mathcal{L}_\delta$ is defined in Lemma \ref{ww10}.
If $n=5$ and $\mu\in[1,3)$, or $n=6$ and $\mu\in(0,4)$, or $n=7$ and $\mu\in(\frac{7}{3},4)$, clearly,
\begin{equation}\label{nnn-1}
\|\mathscr{N}(\phi)\|_{\ast\ast}\leq
\nu_0(\mathscr{R}^{-(4-\mu-\theta_1)}+\mathscr{R}^{-\frac{n-\mu+2}{2}(p-2)})\|\phi\|_{\ast}^{p-1}:=\varpi_0\|\phi\|_{\ast}^{p-1}.
\end{equation}
By \eqref{nnn-1}, observe that we may choose $L_0>0$ sufficiently large from the beginning to have
$$\|\mathcal{L}_\delta(\mathscr{R})\|_{\ast}\leq L_0\|\mathscr{R}\|_{\ast\ast}.$$
Setting
$$\aligned
\mathcal{E}=\Big\{w:w\in C(\mathbb{R}^n)\cap \mathcal{D}^{1,2}(\mathbb{R}^n),\|w\|_{\ast}\leq \zeta_0L_0+1\Big\}.
\endaligned$$
We will prove that $\mathcal{A}$ is a contraction map from $\mathcal{E}$ to $\mathcal{E}$.

Choosing $\delta>0$ small enough such that $\mathscr{R}$ large, we have that
$$L_0\varpi_0\big(L_0\zeta_0+1\big)^{p-1}\leq1.$$
Then, from Lemmas \ref{ww101} and \ref{wwc102} we get
$$
\|\mathcal{A}\|_{\ast}\leq L_0\varpi_0\big(L_0\zeta_0+1\big)^{p-1}+L_0\zeta_0\leq
L_0\zeta_0+1.
$$
Hence, $\mathcal{A}$ maps $\mathcal{E}$ to $\mathcal{E}$.
On the other hand, taking $\phi_1$ and $\phi_2$ in $\mathcal{E}$, we see that
$$
\|\mathcal{A}(\phi_{1})-\mathcal{A}(\phi_{2})\|_{\ast}
=\|\mathcal{L}_{\delta}(\mathscr{N}(\phi_{1}))-\mathcal{L}_{\delta}(\mathscr{N}(\phi_{2}))\|_{\ast}\leq C\|\mathscr{N}(\phi_{1})-\mathscr{N}(\phi_{2})\|_{\ast\ast}.
$$
The point is thus to estimate each term in the right hand side of the $|\mathscr{N}(\phi_{1})-\mathscr{N}(\phi_{2})|$ decomposition. Recalling \eqref{nnn1},
by the convexity of the functions $f(z)=|z|^{p-1}$, and so
\begin{equation}\label{nn1}\aligned
|\mathscr{N}_1(\phi_{1})-\mathscr{N}_1(\phi_{2})|
\leq C\big(|x|^{-\mu}\ast\sigma^{p}\big)\big(|\phi_1|^{p-2}+|\phi_2|^{p-2}\big)|\phi_1-\phi_2|.
\endaligned
\end{equation}
Similarly, direct computations yield
$$\aligned
\sum_{i=2}^{3}|\mathscr{N}_i(\phi_{1})-\mathscr{N}_i(\phi_{2})|&\leq\big(|x|^{-\mu}\ast|\sigma^{p-1}\phi_1|\big)\big[|\sigma|^{p-2}+|\phi_1|^{p-2}+|\phi_2|^{p-2}\big]|\phi_1-\phi_2|\\&+
\big(|x|^{-\mu}\ast|\sigma|^{p-1}|\phi_1-\phi_2|\big)(|\sigma|^{p-2}|\phi_2|+|\phi_2|^{p-1}\big),
\endaligned$$
$$\aligned
\sum_{i=4}^{6}|\mathscr{N}_i(\phi_{1})-\mathscr{N}_i(\phi_{2})|&\leq C\big(|x|^{-\mu}\ast(|\sigma|^{p-2}|\phi_1+\phi_2\big|\big|\phi_1-\phi_2|)\big)\big(\sigma^{p-1}+\sigma^{p-2}\phi_2+|\phi_2|^{p-1}\big)\\&
+C\big(|x|^{-\mu}\ast(\sigma^{p-2}\phi_1^2)\big)\big(|\sigma|^{p-2}|\phi_1-\phi_2|+\big(|\phi_1|^{p-2}+|\phi_2|^{p-2}\big)|\phi_1-\phi_2|\big).
\endaligned$$
As in \eqref{nn1}, the convexity of the functions $f(z)=|z|^{p}$ yields
$$\aligned
\sum_{i=7}^{9}|\mathscr{N}_i(\phi_{1})-\mathscr{N}_i(\phi_{2})|&\leq C\big(|x|^{-\mu}\ast((|\phi_1|^{p-1}+|\phi_2|^{p-1})|\phi_1-\phi_2|)\big)\big(|\sigma|^{p-1}+|\sigma|^{p-2}|\phi_2|+|\phi_2|^{p-1}\big)\\&+
\big(|x|^{-\mu}\ast|\phi_1|^{p}\big)\big(|\sigma|^{p-2}|\phi_1-\phi_2|+(|\phi_1|^{p-2}+|\phi_2|^{p-2})|\phi_1-\phi_2|\big).
\endaligned$$
Putting these estimates together and similarly to the arguments in Lemma \ref{ww101}, we get
$$\aligned
&\|\mathscr{N}(\phi_{1})-\mathscr{N}(\phi_{2})\|_{\ast\ast}\\&\leq
C\varpi_0\Big(\big(\|\phi_{1}\|_{\ast}^{p-2}+\|\phi_{2}\|_{\ast}^{p-2} \big) +\big(\|\phi_{1}\|_{\ast}^{p-1}+\|\phi_{1}\|_{\ast}\|\phi_{2}\|_{\ast}^{p-2}+
\|\phi_{2}\|_{\ast}^{p-1}\big)\Big)\|\phi_{1}-\phi_{2}\|_{\ast}\\&
+C\varpi_0\Big(\|\phi_{1}\|_{\ast}+\|\phi_{2}\|_{\ast}+\|\phi_{1}+\phi_{2}\|_{\ast} +\big(\|\phi_{1}+\phi_{2}\|_{\ast}\|\phi_{2}\|_{\ast}+\|\phi_{1}\|_{\ast}^2\big)\Big)\|\phi_{1}-\phi_{2}\|_{\ast}\\&
+C\varpi_0\Big(\|\phi_{1}+\phi_{2}\|_{\ast}\|\phi_{2}\|_{\ast}^{p-1}+
\|\phi_{1}\|_{\ast}^{p}+\|\phi_{2}\|_{\ast}^{p}+\|\phi_{1}\|_{\ast}^2\|\phi_{2}\|_{\ast}^{p-2}
+\|\phi_{1}\|_{\ast}^{p-1}\|\phi_{2}\|_{\ast}\Big)\|\phi_{1}-\phi_{2}\|_{\ast}\\&
+C\varpi_0\Big(\|\phi_{1}\|_{\ast}^{p-1}\|\phi_{2}\|_{\ast}^{p-1}+
\|\phi_{1}\|_{\ast}^{2(p-1)}+\|\phi_{2}\|_{\ast}^{2(p-1)}+\|\phi_{1}\|_{\ast}^{p}\|\phi_{2}\|_{\ast}^{p-2}
\Big)\|\phi_{1}-\phi_{2}\|_{\ast}.
\endaligned$$
Therefore, we have that
\begin{equation}\label{A-E-E}
\begin{split}
\|\mathcal{A}(\phi_{1})-\mathcal{A}(\phi_{2})\|_{\ast}\leq C\|\mathscr{N}(\phi_{1})-\mathscr{N}(\phi_{2})\big\|_{\ast\ast}\leq \frac{1}{2}\|\phi_{1}-\phi_{2}\|_{\ast},
\end{split}
\end{equation}
provide that $\delta>0$ is small enough such that $\varpi_0\ll1$,
which means that $\mathcal{A}$ is a contraction mapping from $\mathcal{E}$ into itself. In similar way, \eqref{A-E-E} also hold in dimension region $n=4\hspace{2mm}\mbox{and}\hspace{2mm}\mu\in[2,4),\mbox{or}\hspace{2mm}n=5\hspace{2mm}\mbox{and}\hspace{2mm}\mu\in(3,4]$ and so we skip it.

Now by the contraction mapping theorem, there exists a unique $\rho_0\in \mathcal{E}$ such that \eqref{c19} holds true. Moreover, by Lemmas \ref{ww10} and \ref{wwc102}, we deduce
$$\|\rho_0\|_{\ast}\leq\|\mathcal{L}_{\delta}(\mathscr{N}(\rho_0))\|_{\ast}-\|\mathcal{L}_{\delta}(\mathscr{R})\|_{\ast}\leq C,$$
and the conclusion follows.
\end{proof}
\subsection{Choices of the parameters.}\
\newline
In this section, we shall list all the constraints of the constants $n$, $\mu$ which are sufficient for the reduction argument scheme to work.

First, we indicate all the parameters used in different norms (see Remark \ref{stxqi}).

In order to get the desired estimates for the convolution terms, by the computations in section 2 and Lemma \ref{ww101}, we require the parameters $n$ and $\mu$ satisfying the following restrictions
\begin{equation*}
(\ast\ast)\hspace{2mm}\left\lbrace
\begin{aligned}
& n=4\hspace{3mm}\mbox{and}\hspace{3mm}2\leq\mu<4,\\
& n=5\hspace{3mm}\mbox{and}\hspace{3mm}1\leq\mu<4,\\
& n=6\hspace{3mm}\mbox{and}\hspace{3mm}0<\mu<4,\\
& n=7\hspace{3mm}\mbox{and}\hspace{3mm}\frac{7}{3}<\mu<4.
\end{aligned}
\right.
\end{equation*}
To apply Lemma \ref{B4} in estimates of $\|\mathscr{N}\|_{\ast\ast}$, by the computations we need the parameters $\theta_1$, $\theta_2$ and $\theta_3$ satisfying the restrictions
\begin{equation*}
\left\lbrace
\begin{aligned}
&0<\theta_1<4-\mu,\hspace{4mm}\hspace{4mm}\hspace{5mm}\hspace{2mm}n=6\hspace{2mm}\mbox{and}\hspace{2mm}0<\mu<4,\\
&0<\theta_2\leq\frac{9-4\mu}{5},\hspace{4mm}\hspace{4mm}\hspace{2mm}\hspace{2mm}n=7\hspace{2mm}\mbox{and}\hspace{2mm}\frac{7}{3}<\mu<4,\\
&0<\theta_3<\frac{\mu(4-\mu)}{8},\hspace{4mm}\hspace{3mm}\hspace{2mm}n=6\hspace{2mm}\mbox{and}\hspace{2mm}0<\mu<4.
\end{aligned}
\right.
\end{equation*}
In particular, we observe in $(\ast\ast)$ the aforementioned cancellation of $\mu=4$. That is the technical reason why the approach of this section does
not work for $n=7$ and $0<\mu\leq\frac{7}{3}$ or $n>7$ and $0<\mu\leq4$ and why we assume $(\sharp)$ in Theorem \ref{Figalli} (see also the paragraph above Remark \ref{re1.4}).

\section{Some key estimates of the approximation function}
In this section, we shall establish $L^2$ estimates for the $\nabla\rho$.
\subsection{Energy estimate of the first approximation}

\begin{lem}\label{p00}
Assume that $n\geq6-\mu$, $\mu\in(0,n)$ and $0<\mu\leq4$ satisfying $(\sharp)$. For $\delta$ is small enough, we have the following estimate holds true:
\begin{equation*}
\big\|\nabla\rho_0\big\|_{L^2}\lesssim\tau_{n,\mu}(\mathscr{Q}),
\end{equation*}
where  $\tau_{n,\mu}$ is the piece-wise function as defined in \eqref{7-h-0-0}.
\end{lem}
\begin{proof}
We begin from the equation \eqref{AA}. Multiplying this equation by $\rho_0$
and integrating by parts, we obtain
\begin{equation}\label{0AA}
\begin{split}	\int&|\nabla\rho_0|^2+\underbrace{\int\Big[|x|^{-\mu}\ast(\sigma+\rho_0)^{p}\Big](\sigma+\rho_0)^{p-1}\rho_0-\int\Big(|x|^{-\mu}\ast\sigma^{p}\Big)\sigma^{p-1}\rho_0}\limits_{:=\mathcal{P}_1}+\int\hbar\rho_0\\&=\int\sum_{i=1}^{\kappa}\sum_{a=1}^{n+1}c_{a}^{i}\Phi_{n,\mu}[W_{i},\Xi^{a}_i]\rho_0.
\end{split}
\end{equation}
We denote in what follows
$$\mathcal{P}_{1,1}:=\int\Big(|x|^{-\mu}\ast\sigma^{p}\Big)\sigma^{p-2}\rho_0^2, \hspace{2mm}\mathcal{P}_{1,2}:=\int\Big(|x|^{-\mu}\ast\sigma^{p-1}\rho_0\Big)\sigma^{p-1}\rho_0.$$
Combining the elementary inequalities
 $$
 \Big(\sum\limits_{i=1}^{n}a_i\Big)^t\leq C\sum\limits_{i=1}^{n}a_i^t\hspace{2mm}\hspace{2mm}
 \mbox{for all}\hspace{2mm}a_i\geq0\hspace{2mm}\mbox{and all}\hspace{2mm}t>0,$$
and Hardy-Littlewood Sobolev inequality, Sobolev embedding theorem give that
\begin{equation}\label{2AA}
\begin{split}
\mathcal{P}_1&
\lesssim\mathcal{P}_{1,1}+\mathcal{P}_{1,2}+
\sum\limits_{i=1}^{\kappa}\big\|W_i\big\|_{L^{2^\ast}}^{\frac{n-2\mu+6}{2n}}\big\|\rho_0\big\|_{L^{2^\ast}}^{3}+\sum\limits_{i=1}^{\kappa}\big\|W_i\big\|_{L^{2^\ast}}^{\frac{n-\mu+2}{2n}}\big\|\rho_0\big\|_{L^{2^\ast}}^{p+1}\\&+\sum\limits_{i=1}^{\kappa}\big\|W_i\big\|_{L^{2^\ast}}^{\frac{n-2\mu+6}{2n}}\big\|\rho_0\big\|_{L^{2^\ast}}^{3}+\sum\limits_{i=1}^{\kappa}\big\|W_i\big\|_{L^{2^\ast}}^{\frac{8-2\mu}{2n}}\big\|\rho_0\big\|_{L^{2^\ast}}^{4}
+\sum\limits_{i=1}^{\kappa}\big\|W_i\big\|_{L^{2^\ast}}^{\frac{4-\mu}{2n}}\big\|\rho_0\big\|_{L^{2^\ast}}^{p+2}\\&+\sum\limits_{i=1}^{\kappa}\big\|W_i\big\|_{L^{2^\ast}}^{\frac{n-\mu+2}{2n}}\big\|\rho_0\big\|_{L^{2^\ast}}^{p+1}
+\big\|\rho_0\big\|_{L^{2^\ast}}^{2p}
+\sum\limits_{i=1}^{\kappa}\big\|W_i\big\|_{L^{2^\ast}}^{\frac{2n-\mu}{2n}}\big\|\rho_0\big\|_{L^{2^\ast}}^{p}.
\end{split}
\end{equation}
The only remaining terms is $\mathcal{P}_{1,1}$ and $\mathcal{P}_{1,2}$. Due to Lemma \ref{ww102}, we similar compute and get
\begin{equation*}
\begin{split}
\mathcal{P}_{1,1}\lesssim
\sum\limits_{i=1}^{\kappa}\int\Big(|x|^{-\mu}\ast W_i^{p}\Big)\Big(\sum\limits_{i=1}^{\kappa}W_i^{p-2}\Big)S_j^2(x)
\lesssim \sum\limits_{i=1}^{\kappa}\big\|W_i\big\|_{L^{2^{\ast}}}^{\frac{4-\mu}{n-2}}\big\|S_j\big\|_{L^{2^{\ast}}}^2,\hspace{2mm}j=1,2,
\end{split}
\end{equation*}
and
\begin{equation*}
\begin{split}
\mathcal{P}_{1,2}
\lesssim
\sum\limits_{i=1}^{\kappa}\Big\|W_i^{p-1}S_j\Big\|_{L^{r}}^2\lesssim\sum\limits_{i=1}^{\kappa}\big\|W_i\big\|_{L^{2^{\ast}}}^{\frac{2(n-\mu+2)}{n-2}}\big\|S_j\big\|_{L^{2^{\ast}}}^2,\hspace{2mm}j=1,2\hspace{2mm}\mbox{and}\hspace{2mm}r=\frac{2n}{2n-\mu}.
\end{split}
\end{equation*}
Consequently, we get from Lemma \ref{cll-0} and $\mathscr{R}^{2-n}\approx \mathscr{Q}$ that the following estimates hold true:
\begin{itemize}
\item[$\bullet$]
If $n=4$ and $\mu\in[2,4)$, or $n=5$ and $\mu\in[3,4)$, we get that
\begin{equation}\label{3AA-0}
\begin{split}
\mathcal{P}_{1,1}&\lesssim\big\|S_2\big\|_{L^{2^{\ast}}}^2\lesssim(\tau_{n,\mu}(\mathscr{Q}))^2,\hspace{2mm}\mathcal{P}_{1,2}\lesssim\big\|S_2\big\|_{L^{2^{\ast}}}^2\lesssim(\tau_{n,\mu}(\mathscr{Q}))^2.
\end{split}
\end{equation}
\item[$\bullet$]
If $n=5$ and $\mu\in[1,3)$, or $n=6$ and $\mu\in(0,4)$, or $n=7$ and $\mu\in(\frac{7}{3},4)$, we get that
\begin{equation}\label{4AA-3}
\begin{split}
\mathcal{P}_{1,1}&\lesssim\big\|S_1\big\|_{L^{2^{\ast}}}^{2}\lesssim(\tau_{n,\mu}(\mathscr{Q}))^2,\hspace{2mm}\mathcal{P}_{1,2}\lesssim\big\|S_1\big\|_{L^{2^{\ast}}}^{2}\lesssim(\tau_{n,\mu}(\mathscr{Q}))^2.
\end{split}
\end{equation}
\end{itemize}
Putting \eqref{2AA} and \eqref{4AA-3} together it follows that
\begin{equation}\label{4AA-3-10}
\begin{split}
\mathcal{P}_1&\lesssim\big\|\nabla\rho_0\big\|_{L^{2}}^{3}+\big\|\nabla\rho_0\big\|_{L^{2}}^{4}+\big\|\nabla\rho_0\big\|_{L^{2}}^{p}\\&
\hspace{2mm}+\big\|\nabla\rho_0\big\|_{L^{2}}^{p+1}+\big\|\nabla\rho_0\big\|_{L^{2}}^{p+2}+\big\|\nabla\rho_0\big\|_{L^{2}}^{2p}+(\tau_{n,\mu}(\mathscr{Q}))^2\hspace{4mm}\mbox{for every}\hspace{2mm}0<\mu<4.
\end{split}
\end{equation}

Using the Lemmas \ref{estimate1} and \ref{ww102} to handle the remainder term in \eqref{0AA}, so that we get from Lemma \ref{cll-0} and $\mathscr{R}^{2-n}\approx \mathscr{Q}$ the following estimates hold.
\begin{itemize}
\item[$\bullet$]
If $n=4$ and $\mu\in[2,4)$, or $n=5$ and $\mu\in[3,4)$, we get that
\begin{equation*}
\begin{split}
\int\hbar\rho_0&\lesssim\int T_2(x)S_2(x)\leq\big\|T_2\big\|_{L^{(2^{\ast})^{\prime}}}\big\|S_2\big\|_{L^{2^{\ast}}}\lesssim(\tau_{n,\mu}(\mathscr{Q}))^2.
\end{split}
\end{equation*}
\item[$\bullet$]
If $n=5$ and $\mu\in[1,3)$, or $n=6$ and $\mu\in(0,4)$, or $n=7$ and $\mu\in(\frac{7}{3},4)$, we get that
\begin{equation*}
\begin{split}
\int\hbar\rho_0&\lesssim\int T_1(x)S_1(x)\leq\big\|T_1\big\|_{L^{(2^{\ast})^{\prime}}}\big\|S_1\big\|_{L^{2^{\ast}}}\lesssim(\tau_{n,\mu}(\mathscr{Q}))^2.
\end{split}
\end{equation*}
\end{itemize}
Summing, we get that
\begin{equation*}
\begin{split}
\int\hbar\rho_0\lesssim(\tau_{n,\mu}(\mathscr{Q}))^2\hspace{4mm}\mbox{for every}\hspace{2mm}0<\mu<4.
\end{split}
\end{equation*}
Hence, the result easily follows.
\end{proof}
\subsection{Decomposition of the error function}\
\newline
First we need to take a decomposition of $\rho$. Recall that $\rho$ and satisfy \eqref{u-0} and \eqref{AA}, respectively.
Then $\rho_1=\rho-\rho_0$ solves
\begin{equation}\label{c1-000}
	\left\{\begin{array}{l}
		\displaystyle \Delta \rho_1+\big(|x|^{-\mu}\ast(\sigma+\rho_0+\rho_1)^{p}\big)(\sigma+\rho_0+\rho_1)^{p-1}
-\Big(|x^{-\mu}\ast(\sigma+\rho_0)^{p}\Big)(\sigma+\rho_0)^{p-1}
		\\
		\displaystyle+\sum_{i=1}^{\kappa}\sum_{a=1}^{n+1}c_{a}^{i}\Phi_{n,\mu}[W_{i},\Xi^{a}_i]+\hat{f}=0,\\
		\displaystyle \int\Phi_{n,\mu}[W_{i},\Xi^{a}_i]\rho_1=0,\hspace{2mm}i=1,\cdots, \kappa; ~a=1,\cdots,n+1.
	\end{array}
	\right.
\end{equation}
And there are appropriate constants $\gamma^i$ such that the following decomposition hold.
\begin{equation}\label{c1-001}
\rho_1=\sum_{i=1}^{\kappa}\gamma^iW_i+\rho_2\hspace{2mm}{with}\hspace{2mm}\gamma^i=\int\nabla\rho_1\nabla W_i
\end{equation}
such that for all $i=1,\cdots, \kappa; ~a=1,\cdots,n+1$
$$\int\nabla\rho_2\nabla W_i=\int\nabla\rho_2\nabla \Xi_{i}^{a}=0.$$
Introducing the operator
\begin{equation*}
L[u]:=-\Delta u-\big(|x|^{-\mu} \ast W_i^{p}\big)W_i^{p-2}u,\hspace{2mm}
R[u]:=\big(|x|^{-\mu} \ast \big(W_i^{p-1}u\big)\big)W_i^{p-1}
 +\big(|x|^{-\mu} \ast W_i^{p}\big)W_i^{p-2}u,
\end{equation*}
then the eigenvalue problem is written as
\begin{eqnarray}\label{defanndg}
L[u]=\lambda_0R[u],\quad u\in \mathcal{D}^{1,2}(\mathbb{R}^n),
\end{eqnarray}
In view of non-degeneracy result, we know that the functions $W_i$, $\partial_{\lambda}W_i$ and $\partial_{\xi_i}W_i$ are eigenfunctions of \eqref{defanndg}, and
for any $1\leq i\leq\kappa$ we note that $\rho$ satisfies the orthogonality conditions
\begin{equation*}
\int\nabla\rho\nabla W_i=0,\quad\int\nabla\rho\nabla \partial_{\lambda}W_i=0,\quad\int\nabla\rho\nabla\partial_{\xi_i}W_i=0
\end{equation*}
which is equivalent to
\begin{equation}\label{EP1}
\int\big(|x|^{-\mu} \ast \big(W_i^{p-1}\rho\big)\big)W_i^{p}=0,\hspace{2mm}
\Phi_{n,\mu}[W_{i},\rho]\partial_{\lambda}W_{i}=0,
\hspace{2mm}
\Phi_{n,\mu}[W_{i},\rho]\partial_{\xi_i}W_i=0
\end{equation}
for any $1\leq i\leq n$.

\begin{lem}\label{estimate1-0}
Let $n\geq3$ and $\kappa\in\mathbb{N}$. There exists a positive constant $\delta=\delta(n,\kappa)>0$ such that if $\sigma=\sum_{i=1}^{\kappa}W_i$ is a linear combination of $\delta$-interacting Talenti bubbles and $\rho\in D^{1,2}(\mathbb{R}^n)$ satisfies the orthogonal conditions \eqref{EP1}.
Then we have that
\begin{equation*}
\begin{split}
(p-1)\sum\limits_{i=1}^ {\kappa}\int\big(|x|^{-\mu}\ast W_i^{p}\big)\sigma^{p-2}\rho^2
+p\sum\limits_{i=1}^{\kappa}\int\big(|x|^{-\mu}\ast\big(\sigma^{p-1}\rho\big)\big)W_i^{p-1}\rho
\leq \tilde{c}\int|\nabla\rho|^2,
\end{split}
\end{equation*}
where $\tilde{c}$ is a constant strictly less than $1$ which depends only on $n$, $\kappa$ and $\mu$.
\end{lem}
\begin{proof}
The proof of can be found in \cite{p-y-z24}.
\end{proof}
\begin{lem}\label{Ni-1-3}
 We have that
  $$\big\|\nabla\rho_2\big\|_{L^2}\lesssim\sum_{i=1}^{\kappa}|\gamma^i|+\big\|\hat{f}\big\|_{(D^{1,2}(\mathbb{R}^n))^{-1}}.$$
 \end{lem}
\begin{proof}
Taking $\rho_2$ as test function in \eqref{c1-000}, and exploiting the orthogonality condition yields
$$
\int|\nabla\rho_2|^2=\int\Big(\big(|x|^{-\mu}\ast(\sigma+\rho_0+\rho_1)^{p}\big)(\sigma+\rho_0+\rho_1)^{p-1}
-\big(|x^{-\mu}\ast(\sigma+\rho_0)^{p}\big)(\sigma+\rho_0)^{p-1}\Big)\rho_2+\int |\hat{f}\rho_2|.
$$
Thence
\begin{equation}\label{p-0}
\begin{split}
\int|\nabla\rho_2|^2&\lesssim\int\Phi_{n,\mu}[\sigma+\rho_0,\rho_1]\rho_2 +\int|\hat{f}\rho_2|+a_1+a_2+a_3+a_4,
\end{split}
\end{equation}
where we have:
\begin{equation*}
\begin{split}
&a_1:=\int\big(|x|^{-\mu}\ast(\sigma+\rho_0)^{p-1}\rho_1\big)\big((\sigma+\rho_0)^{p-2}\rho_1\rho_2+\rho_1^{p-1}\rho_2\big),\\&
a_2:=\int\big(|x^{-\mu}\ast(\sigma+\rho_0)^{p-2}\rho_0^2\big)\big((\sigma+\rho_0)^{p-1}+(\sigma+\rho_0)^{p-2}\rho_1+\rho_1^{p-1}\big)\rho_2,\\&
a_3:=\int\big(|x^{-\mu}\ast\rho_1^{p}\big)\big((\sigma+\rho_0)^{p-1}+(\sigma+\rho_0)^{p-2}\rho_1+\rho_1^{p-1}\big)\rho_2,\\&
a_4:=\int\big(|x^{-\mu}\ast(\sigma+\rho_0)^{p}\big)\rho_1^{p-1}\rho_2.
\end{split}
\end{equation*}
Using \eqref{c1-001}, we obtain
\begin{equation*}
\begin{split}
\int\Phi_{n,\mu}[\sigma+\rho_0,\rho_1]\rho_2\leq
\Upsilon_0\int\Phi_{n,\mu}[\sigma+\rho_0,W_i]\rho_2+\int\Phi_{n,\mu}[\sigma+\rho_0,\rho_2]\rho_2,
\end{split}
\end{equation*}
where $
\Upsilon_0:=\sum_{i=1}^{\kappa}|\gamma^i|.$
By H\"{o}lder's inequality and Sobolev's inequality imply
\begin{equation*}
\begin{split}
&\int\Phi_{n,\mu}[\sigma+\rho_0,W_i]\rho_2\\&\lesssim\big\|\big|\sigma+\rho_0\big|^{p-1}W_i\big\|_{L^{r}}\big\|\sigma+\rho_0\big\|_{L^{2^{\ast}}}^{p-1}\big\|\rho_2\big\|_{L^{2^{\ast}}}+\big\|\sigma+\rho_0\big\|_{L^{2^{\ast}}}^{p}\big\||\sigma+\rho_0|^{p-2}W_i\big\|_{L^{r}}\big\|\rho_2\big\|_{L^{2^{\ast}}}^2\\&
\lesssim\big\|\nabla\rho_2\big\|_{L^{2}}.
\end{split}
\end{equation*}
An analogous argument tell us that
\begin{equation*}
\begin{split}
a_1&\lesssim\big\|\sigma+\rho_0\big\|_{L^{2^{\ast}}}^{p-1}\big\|\rho_1\big\|_{L^{2^{\ast}}}\big(\big\|\sigma+\rho_0\big\|_{L^{2^{\ast}}}^{p-2}\big\|\rho_1\big\|_{L^{2^{\ast}}}\big\|\rho_2\big\|_{L^{2^{\ast}}}+
\big\|\rho_1\big\|_{L^{2^{\ast}}}^{p-1}\big\|\rho_2\big\|_{L^{2^{\ast}}}\big)\\&
\lesssim\big(\big(\Upsilon_0+\|\nabla\rho_2\|_{L^{2}}\big)^{2}+\big(\Upsilon_0+\|\nabla\rho_2\|_{L^{2}}\big)^{p}\big)\big\|\nabla\rho_2\big\|_{L^{2}},
\end{split}
\end{equation*}
\begin{equation*}
\begin{split}
a_2&\leq C\big\|\sigma+\rho_0\big\|_{L^{2^{\ast}}}^{p-2}\big\|\nabla\rho_0\big\|_{L^{2}}^2\big(\big\|\sigma+\rho_0\big\|_{L^{2^{\ast}}}^{p-1}+
\big\|\sigma+\rho_0\big\|_{L^{2^{\ast}}}^{p-2}\big\|\rho_1\big\|_{L^{2^{\ast}}}+\big\|\rho_1\big\|_{L^{2^{\ast}}}^{p-1}\big)\big\|\nabla\rho_2\big\|_{L^{2}}\\&
\leq\big(C_0+\widehat{C}_1\big(\Upsilon_0+\|\nabla\rho_2\|_{L^{2}}\big)+\widetilde{C}_2\big(\Upsilon_0+\|\nabla\rho_2\|_{L^{2}}\big)^{p-1}\big)\big\|\nabla\rho_0\big\|_{L^{2}}^2\big\|\nabla\rho_2\big\|_{L^{2}},
\end{split}
\end{equation*}
\begin{equation*}
\begin{split}
a_3\lesssim\big[\big(\Upsilon_0+\|\nabla\rho_2\|_{L^{2}}\big)^{p}\big(1+\Upsilon_0+\|\nabla\rho_2\|_{L^{2}}
+\big(\Upsilon_0+\|\nabla\rho_2\|_{L^{2}}\big)^{p-1}\big]\big\|\nabla\rho_2\big\|_{L^{2}},
\end{split}
\end{equation*}
and
\begin{equation*}
\begin{split}
a_4\lesssim\big(\Upsilon_0+\|\nabla\rho_2\|_{L^{2}}\big)^{p-1}\big\|\nabla\rho_2\big\|_{L^{2}}\hspace{2mm}\mbox{and}\hspace{2mm}\int |\hat{f}\rho_2|\lesssim\big\|\hat{f}\big\|_{(\mathcal{D}^{1,2}(\mathbb{R}^n))^{-1}}\big\|\nabla\rho_2\big\|_{L^{2}}.
\end{split}
\end{equation*}
By Lemma \ref{estimate1-0}, there exists number $\tilde{c}<1$ which depends only on $n$ and $\kappa$, such that
\begin{equation*}
\begin{split}
&\int\Phi_{n,\mu}[\sigma+\rho_0,\rho_2]\rho_2
\leq\big(\tilde{c}+C\big\|\nabla\rho_0\big\|_{L^{2}}^{p-2}+C\big\|\nabla\rho_0\big\|_{L^{2}}^{p-1}+C\big\|\nabla\rho_0\big\|_{L^{2}}^{p}\big)\big\|\nabla\rho_2\big\|_{L^{2}}^2.
\end{split}
\end{equation*}
Combining this estimate, we get
\begin{equation*}
\begin{split}
&\int\Phi_{n,\mu}[\sigma+\rho_0,\rho_1]\rho_2
\leq\big(\tilde{c}+C\big\|\nabla\rho_0\big\|_{L^{2}}^\frac{4-\mu}{2n}+C\big\|\nabla\rho_0\big\|_{L^{2}}^\frac{n-\mu+2}{2n}+C\big\|\nabla\rho_0\big\|_{L^{2}}^{p}\big)\big\|\nabla\rho_2\big\|_{L^{2}}^2
+C\Upsilon_0\big\|\nabla\rho_2\big\|_{L^{2}}.
\end{split}
\end{equation*}
Note that the estimate of $a_2$, we can assume that $\|\nabla\rho_0\|_{L^2}\ll1$ by Lemma \ref{p00} and choosing $\delta$ small enough such that $\|\nabla\rho_2\|_{L^2}<1$ and $\Upsilon_0<1$, and such that
$$\alpha_2\leq C\Upsilon_0\big\|\nabla\rho_2\big\|_{L^{2}}+\widetilde{C}_2\big(\Upsilon_0+\|\nabla\rho_2\|_{L^{2}}\big)^{p-1}\big\|\nabla\rho_2\big\|_{L^{2}}+\widehat{C}_1\big\|\nabla\rho_0\big\|_{L^{2}}^2\big\|\nabla\rho_2\big\|_{L^{2}}^2.$$
In conclusion, recalling \eqref{p-0}, we deduce
\begin{equation*}
\begin{split}
&\|\nabla\rho_2\|_{L^2}^2\lesssim\Upsilon_0\big\|\nabla\rho_2\big\|_{L^{2}}+\big[\big(\Upsilon_0+\|\nabla\rho_2\|_{L^{2}}\big)^{2}+\big(\Upsilon_0+\|\nabla\rho_2\|_{L^{2}}\big)^{p-1}\big]\big\|\nabla\rho_2\big\|_{L^{2}}\\&
+\big[\big(\Upsilon_0+\|\nabla\rho_2\|_{L^{2}}\big)^{p}\big(1+\Upsilon_0+\|\nabla\rho_2\|_{L^{2}}+\big(\Upsilon_0+\|\nabla\rho_2\|_{L^{2}}\big)^{p-1}\big)+\big\|\hat{f}\big\|_{(D^{1,2}(\mathbb{R}^n))^{-1}}\big]\big\|\nabla\rho_2\big\|_{L^{2}},
\end{split}
\end{equation*}
since we can assume that $\|\nabla\rho_0\|_{L^2}\ll1$. Then by choosing $\delta$ small enough such that $\|\nabla\rho_2\|_{L^2}<1$ and $\Upsilon_0<1$, we are able to conclude that desired estimate
\begin{equation*}
\begin{split}
\|\nabla\rho_2\|_{L^2}&\lesssim \Upsilon_0+\big\|\hat{f}\big\|_{(\mathcal{D}^{1,2}(\mathbb{R}^n))^{-1}}.
\end{split}
\end{equation*}
The conclusion follows.
\end{proof}
\begin{lem}\label{rr-1}
If $\delta$ is small we have that:\\
$(i)$ If $n=4$ and $\mu\in[2,4)$, or $n=5$ and $\mu\in[3,4)$,
$$|\gamma^i|\lesssim\big\|\hat{f}\big\|_{(\mathcal{D}^{1,2}(\mathbb{R}^n))^{-1}}
+\mathscr{Q}^{3-\frac{\mu}{n-2}},\hspace{2mm}j=1,\cdots, \kappa.$$
$(ii)$ If $n=5$ and $\mu\in[1,3)$, or $n=6$ and $\mu\in(0,4)$, or $n=7$ and $\mu\in(\frac{7}{3},4)$,
\begin{equation*}	
|\gamma^i|\lesssim\big\|\hat{f}\big\|_{(\mathcal{D}^{1,2}(\mathbb{R}^n))^{-1}}
+\mathscr{Q}^{\min\{2,\frac{n-\mu}{n-2}+1\}},\hspace{2mm}j=1,\cdots, \kappa.
\end{equation*}
\end{lem}
\begin{proof}
Multiplying \eqref{c1-000} by $W_j$ and integrating we have
\begin{equation}\label{rr-2}
\begin{split}	\int\nabla\rho_1\nabla W_j&=\mathcal{F}_{n,\mu}(x)
+\int\sum_{i=1}^{\kappa}\sum_{a=1}^{n+1}c_{a}^{i}\Phi_{n,\mu}[W_{i},\Xi^{a}_i]W_j+\int \hat{f}W_j,
\end{split}
\end{equation}
where we denote in what follows
$$
\mathcal{F}_{n,\mu}(x):=\int\big[\big(|x|^{-\mu}\ast(\sigma+\rho_0+\rho_1)^{p}\big)(\sigma+\rho_0+\rho_1)^{p-1}-\big(|x|^{-\mu}\ast(\sigma+\rho_0)^{p}\big)(\sigma+\rho_0)^{p-1}\big]W_j.
$$
Let us prove some estimates which are need to compute $\mathcal{F}_{n,\mu}$. Combining the elementary inequalities
\begin{equation*}
\left\lbrace
\begin{aligned}
&(\sigma+\rho_0)^{p}-W_j^{p}\lesssim\sum\limits_{1\leq i\neq j\leq\kappa}W_i^{p}+\sum\limits_{1\leq i\neq j\leq\kappa}W_i^{p-1}W_j+\sum\limits_{i}\big(W_i^{p-1}\rho_0+W_i^{p-2}\rho_0^2\big)+|\rho_0|^{p},\\&
(\sigma+\rho_0)^{p-1}-W_j^{p-1}\lesssim\sum\limits_{1\leq i\neq j\leq\kappa}W_i^{p-1}+\sum\limits_{1\leq i\neq j\leq\kappa}W_i^{p-2}W_j+\sum\limits_{i}W_i^{p-2}\rho_0+|\rho_0|^{p-1},\\&
(\sigma+\rho_0)^{p-2}-W_j^{p-2}\lesssim\sum\limits_{1\leq i\neq j\leq\kappa}W_i^{p-2}+|\rho_0|^{p-2},
\end{aligned}
\right.
\end{equation*}
then we have
\begin{equation}\label{rr-3}
\begin{split}
b_1:&=\int\big[|x|^{-\mu}\ast\big((\sigma+\rho_0)^{p-1}-W_j^{p-1}\big)\rho_1\big](\sigma+\rho_0+\rho_1)^{p-1}W_j\\&
\leq\int\big[|x|^{-\mu}\ast\big(\sum\limits_{i\neq j}\big(W_i^{p-1}+W_i^{p-2}W_j\big)+\sum\limits_{i}W_i^{p-2}\rho_0+|\rho_0|^{p-1}\big)\rho_1\big](\sigma+\rho_0+\rho_1)^{p-1}W_j,
\end{split}
\end{equation}
\begin{equation}\label{rr-4}
\begin{split}
b_2:&=\int\big(|x|^{-\mu}\ast W_j^{p-1}\rho_1\big)\big((\sigma+\rho_0+\rho_1)^{p-1}-W_j^{p-1}\big)W_j\\&
\leq\int\big(|x|^{-\mu}\ast W_j^{p-1}\rho_1\big)\big((\sigma+\rho_0)^{p-2}|\rho_1|+|\rho_1|^{p-1}\big)W_j,
\end{split}
\end{equation}
and
\begin{equation}\label{rr-5}
\begin{split}
b_3:&=\int\big(|x|^{-\mu}\ast W_j^{p-1}\rho_1\big)\big((\sigma+\rho_0)^{p-1}-W_j^{p-1}\big)W_j\\&
\leq\int\big(|x|^{-\mu}\ast W_j^{p-1}\rho_1\big)\big(\sum\limits_{1\leq i\neq j\leq\kappa}W_i^{p-1}+\sum\limits_{1\leq i\neq j\leq\kappa}W_i^{p-2}W_j+\sum\limits_{i}W_i^{p-2}\rho_0+|\rho_0|^{p-1}\big)W_j.
\end{split}
\end{equation}
Similarly, we denote
\begin{equation}\label{rr-6}
\begin{split}
b_4:&=\int\big(|x|^{-\mu}\ast(\sigma+\rho_0)^{p}\big)\big((\sigma+\rho_0)^{p-2}-W_j^{p-2}\big)\rho_1W_j\\&
\lesssim\int\big(|x|^{-\mu}\ast(\sigma+\rho_0)^{p}\big)\big(\sum\limits_{1\leq i\neq j\leq\kappa}W_i^{p-2}+|\rho_0|^{p-2}\big)\rho_1W_j,
\end{split}
\end{equation}
and
\begin{equation}\label{rr-7}
\begin{split}
b_5:&=\int\big(|x|^{-\mu}\ast\big((\sigma+\rho_0)^{p}-W_j^{p}\big)\big)W_j^{p-1}\rho_1\\&
\lesssim\int\big(|x|^{-\mu}\ast\big(\sum\limits_{1\leq i\neq j\leq\kappa}\big(W_i^{p}+W_i^{p-1}W_j\big)+\sum\limits_{i}\big(W_i^{p-1}\rho_0+W_i^{p-2}\rho_0^2\big)+|\rho_0|^{p}\big)\big)W_j^{p-1}\rho_1.
\end{split}
\end{equation}
And we write
\begin{equation}\label{i-6}
\begin{split}
b_6:&=\int\big(|x|^{-\mu}\ast\big((\sigma+\rho_0)^{p-2}|\rho_1|^2+|\rho_1|^{p}\big)\big)(\sigma+\rho_0+\rho_1)^{p-1}W_j+\int\big(|x|^{-\mu}\ast(\sigma+\rho_0)^{p}\big)\rho_1^{p-1}W_j.
\end{split}
\end{equation}
Thus, combining \eqref{rr-3}, \eqref{rr-4}, \eqref{rr-5}, \eqref{rr-6}, \eqref{rr-7} and \eqref{i-6} entails that
\begin{equation}\label{rr-8}
\begin{split}
\mathcal{F}_{n,\mu}(x)\lesssim \int\big(|x|^{-\mu}\ast W_j^{p-1}\rho_1\big)W_j^{p}+\int\big(|x|^{-\mu}\ast W_j^{p}\big)W_j^{p-1}\rho_1+b_1+b_2+b_3+b_4+b_5+b_6.
\end{split}
\end{equation}
Using HLS inequality, H\"{o}lder's inequality and Sobolev's inequality, we establish now the following key estimates:
\begin{equation*}
\begin{split}
b_1&\lesssim\int\big(|x|^{-\mu}\ast\sum\limits_{i\neq j}\big(W_i^{p-1}+W_i^{p-2}W_j\big)\rho_1\big)\big(\sum\limits_{i\neq j}W_i^{p-1}W_j+W_j^{p}+|\rho_0|^{p-1}W_j+|\rho_1|^{p-1}W_j\big)\\&
\hspace{2mm}+\int\big(|x|^{-\mu}\ast\big(\sum\limits_{i}W_i^{p-2}\rho_0+|\rho_0|^{p-1}\big)\rho_1\big)\big(\sum\limits_{i }W_i^{p-1}W_j+|\rho_0|^{p-1}W_j+|\rho_1|^{p-1}W_j\big)\\&
\lesssim\sum\limits_{1\leq i\neq j\leq\kappa}\big\|W_i\big\|_{L^{2^{\ast}}}^{p-1}\big\|W_i^{p-1}W_j\big\|_{L^{r}}\big\|\rho_1\big\|_{L^{2^{\ast}}}+
\sum\limits_{1\leq i\neq j\leq\kappa}\big\|W_j^{2^{\ast}-p}W_i^{p-1}\big\|_{L^{(2^{\ast})^{\prime}}}\big\|\rho_1\big\|_{L^{2^{\ast}}}\\&
\hspace{2mm}+\sum\limits_{1\leq i\neq j\leq\kappa}\big\|W_i\big\|_{L^{2^{\ast}}}^{p-1}\big\|W_j\big\|_{L^{2^{\ast}}}^{p-1}\big\|\nabla\rho_0\big\|_{L^{2}}^{{p-1}}\big\|\rho_1\big\|_{L^{2^{\ast}}}+\sum\limits_{1\leq i\neq j\leq\kappa}\big\|W_i\big\|_{L^{2^{\ast}}}^{p-1}\big\|W_j\big\|_{L^{2^{\ast}}}\big\|\rho_1\big\|_{L^{2^{\ast}}}^{p}\\&
\hspace{2mm}+\big(\sum\limits_{i}\big\|W_i\big\|_{L^{2^{\ast}}}^{p-2}\big\|\nabla\rho_0\big\|_{L^{2}}+\big\|\nabla\rho_0\big\|_{L^{2}}^{p-1}\big)\big\|\rho_1\big\|_{L^{2^{\ast}}}
\big(\sum\limits_{i}\big\|W_i^{p-1}W_j\big\|_{L^{r}}+\big\|\nabla\rho_0\big\|_{L^{2}}^{p-1}\big\|W_j\big\|_{L^{2^{\ast}}}\big).
\\&\hspace{2mm}+\big(\sum\limits_{i}\big\|W_i\big\|_{L^{2^{\ast}}}^{p-2}\big\|\nabla\rho_0\big\|_{L^{2}}+\big\|\nabla\rho_0\big\|_{L^{2}}^{p-1}\big)
\big\|\rho_1\big\|_{L^{2^{\ast}}}^{p}\big\|W_j\big\|_{L^{2^{\ast}}}.
\end{split}
\end{equation*}
\begin{equation*}
\begin{split}
b_2\lesssim\sum\limits_{i}\big\|W_j\big\|_{L^{2^{\ast}}}^{p}\big\|W_i\big\|_{L^{2^{\ast}}}^{p-2}\big\|\rho_1\big\|_{L^{2^{\ast}}}^{2}+\big\|W_j\big\|_{L^{2^{\ast}}}^{p}\big\|\nabla\rho_0\big\|_{L^{2}}^{p-2}
\big\|\rho_1\big\|_{L^{2^{\ast}}}^{2}+\big\|W_j\big\|_{L^{2^{\ast}}}^{p}\big\|\rho_1\big\|_{L^{2^{\ast}}}^{p},
\end{split}
\end{equation*}
\begin{equation*}
\begin{split}
b_3&\lesssim\big\|W_j\big\|_{L^{2^{\ast}}}^{p-1}\big(\sum\limits_{1\leq i\neq j\leq\kappa}\big\|W_i^{p-1}W_j\big\|_{L^{r}}+\sum\limits_{1\leq i\neq j\leq\kappa}\big\|W_i^{p-2}W_j^2\big\|_{L^{r}}\big)\big\|\rho_1\big\|_{L^{2^{\ast}}}\\&
\hspace{2mm}+\sum\limits_{i}\big\|W_j\big\|_{L^{2^{\ast}}}^{p-1}\big\|W_i^{p-2}W_j\big\|_{L^{\frac{2n}{n-\mu+2}}}\big\|\nabla\rho_0\big\|_{L^{2}}
\big\|\rho_1\big\|_{L^{2^{\ast}}}+\big\|W_j\big\|_{L^{2^{\ast}}}^{p-1}\big\|\nabla\rho_0\big\|_{L^{2}}^{p-1}\big\|\rho_1\big\|_{L^{2^{\ast}}},
\end{split}
\end{equation*}
\begin{equation*}
\begin{split}
b_4\lesssim\big\|\sigma+\rho_0\big\|_{L^{2^{\ast}}}^{p}\sum\limits_{1\leq i\neq j\leq\kappa}\big\|W_i^{p-2}W_j\big\|_{L^{\frac{2n}{n-\mu+2}}}\big\|\rho_1\big\|_{L^{2^{\ast}}}+\big\|\sigma+\rho_0\big\|_{L^{2^{\ast}}}^{p}\big\|\nabla\rho_0\big\|_{L^{2}}^{p-2}
\big\|W_j\big\|_{L^{2^{\ast}}}\big\|\rho_1\big\|_{L^{2^{\ast}}},
\end{split}
\end{equation*}
\begin{equation*}
\begin{split}
b_5&\lesssim\sum\limits_{1\leq i\neq j\leq\kappa}\big\|W_i^{2^{\ast}-p}W_j^{p-1}\big\|_{L^{(2^{\ast})^{\prime}}}\big\|\rho_1\big\|_{L^{2^{\ast}}}+\sum\limits_{1\leq i\neq j\leq\kappa}\big\|W_i^{p-1}W_j\big\|_{L^{r}}\big\|W_j\big\|_{L^{2^{\ast}}}^{p-1}\big\|\rho_1\big\|_{L^{2^{\ast}}}\\&
+\sum\limits_{i}\big(\big\|W_i\big\|_{L^{2^{\ast}}}^{p-1}\big\|\nabla\rho_0\big\|_{L^{2}}+\big\|W_i\big\|_{L^{2^{\ast}}}^{p-2}\big\|\nabla\rho_0\big\|_{L^{2}}^2\big)
\big\|W_j\big\|_{L^{2^{\ast}}}^{p-1}\big\|\rho_1\big\|_{L^{2^{\ast}}}+\big\|\nabla\rho_0\big\|_{L^{2}}^{p}\big\|W_j\big\|_{L^{2^{\ast}}}^{p-1}\big\|\rho_1\big\|_{L^{2^{\ast}}}.
\end{split}
\end{equation*}
and
\begin{equation*}
\begin{split}
b_6&\lesssim\big(\big\|\sigma+\rho_0\big\|_{L^{2^{\ast}}}^{p-2}\big\|\rho_1\big\|_{L^{2^{\ast}}}^2+\big\|\rho_1\big\|_{L^{2^{\ast}}}^{p}\big)\big\|\sigma+\rho_0+\rho_1\big\|_{L^{2^{\ast}}}^{p-1}\big\|W_j\big\|_{L^{2^{\ast}}}.
\end{split}
\end{equation*}
Note that the above estimates, by Lemma \ref{FPU1} we evaluate
\begin{equation*}
\begin{split}
&\big\|W_i^{2^{\ast}-p}W_j^{p-1}\big\|_{L^{(2^{\ast})^{\prime}}}\approx
\begin{cases}
\mathscr{Q}^{\min((2^{\ast}-p)(2^{\ast})^{\prime},(p-1)(2^{\ast})^{\prime})},\hspace{2mm}\hspace{6mm}\quad\hspace{2mm}\mu\neq\frac{n+2}{2},\\
\mathscr{Q}^{\frac{n}{n-2}}\log(\frac{1}{\mathscr{Q}}),\hspace{2mm}\hspace{10mm}\hspace{10mm}\hspace{4mm}\hspace{6mm}\hspace{2mm}\mu=\frac{n+2}{2},
\end{cases}
\end{split}
\end{equation*}
\begin{equation*}
\begin{split}
\big\|W_i^{p-1}W_j\big\|_{L^{r}}\approx
\begin{cases}
\mathscr{Q}^{r},\hspace{2mm}\mbox{with}\hspace{2mm}r=\frac{2n}{2n-\mu},\hspace{6mm}\hspace{4mm}\hspace{6mm}\hspace{6mm}\quad\hspace{2mm}\mu<4,\\
\mathscr{Q}^{\frac{n}{n-2}}\log(\frac{1}{\mathscr{Q}}),\hspace{6mm}\hspace{6mm}\hspace{6mm}\hspace{6mm}\hspace{6mm}\hspace{8mm}\hspace{2mm}\mu=4,
\end{cases}
\end{split}
\end{equation*}
and
\begin{equation*}
\begin{split}
\big\|W_i^{p-2}W_j^2\big\|_{L^{r}}\approx \mathscr{Q}^{(p-2)r},\hspace{4mm}\big\|W_i^{p-2}W_j\big\|_{L^{\frac{2n}{n-\mu+2}}}\approx\begin{cases}
\mathscr{Q}^{\frac{2n}{n-\mu+2}},\hspace{6mm}\hspace{4mm}\quad\hspace{2mm}n>6-\mu,\\
\mathscr{Q}^{\frac{n}{n-2}}\log(\frac{1}{\mathscr{Q}}),\hspace{6mm}\hspace{2mm}n=6-\mu.
\end{cases}
\end{split}
\end{equation*}
Recalling the estimates $b_1-b_6$, we are able to conclude that
\begin{equation}\label{rr-9}
\begin{split}
b_1\lesssim o(1)\big(\Upsilon_0+\big\|\hat{f}\big\|_{(D^{1,2}(\mathbb{R}^n))^{-1}}\big)+\big(\Upsilon_0^{p}+\big\|\hat{f}\big\|_{(D^{1,2}(\mathbb{R}^n))^{-1}}^{p}\big),\hspace{4mm}\mbox{for}\hspace{2mm}i\neq j,
\end{split}
\end{equation}
\begin{equation}\label{rr-10}
\begin{split}
b_2\lesssim o(1)\big(\Upsilon_0^2+\big\|\hat{f}\big\|_{(D^{1,2}(\mathbb{R}^n))^{-1}}^2\big)+\Upsilon_0^{2}+\big\|\hat{f}\big\|_{(D^{1,2}(\mathbb{R}^n))^{-1}}^{2},
\end{split}
\end{equation}
\begin{equation}\label{rr-11}
\begin{split}
b_3\lesssim o(1)\big(\Upsilon_0+\big\|\hat{f}\big\|_{(D^{1,2}(\mathbb{R}^n))^{-1}}\big),\hspace{4mm}\mbox{for}\hspace{2mm}i\neq j,
\end{split}
\end{equation}
\begin{equation}\label{rr-12}
\begin{split}
b_4\lesssim o(1)\Big(\Upsilon_0+\big\|\hat{f}\big\|_{(D^{1,2}(\mathbb{R}^n))^{-1}}\Big),\hspace{4mm}\mbox{for}\hspace{2mm}i\neq j,
\end{split}
\end{equation}
\begin{equation}\label{rr-13}
\begin{split}
b_5\lesssim o(1)\big(\Upsilon_0^2+\big\|\hat{f}\big\|_{(D^{1,2}(\mathbb{R}^n))^{-1}}^2\Big),\hspace{4mm}\mbox{for}\hspace{2mm}i\neq j,
\end{split}
\end{equation}
and
\begin{equation}\label{rr-13-0}
\begin{split}
b_6&\lesssim o(1)\big(\Upsilon_0^2+\big\|\hat{f}\big\|_{D^{1,2}(\mathbb{R}^n))^{-1}}^2\Big)+\Upsilon_0^2+\big\|\hat{f}\big\|_{(D^{1,2}(\mathbb{R}^n))^{-1}}^2,
\end{split}
\end{equation}
 where we have used Lemma \ref{Ni-1-3}. Summarizing, by plugging \eqref{rr-8} and \eqref{rr-9}-\eqref{rr-13-0} in \eqref{rr-2},
we get that
\begin{equation}\label{rr-14}
\begin{split}	
-\int&\big(|x|^{-\mu}\ast W_j^{p-1}\rho_1\big)W_j^{p}\lesssim o(1)\Upsilon_0+\big\|\hat{f}\big\|_{(D^{1,2}(\mathbb{R}^n))^{-1}}
+\int\sum_{i=1}^{\kappa}\sum_{a=1}^{n+1}c_{a}^{i}\Phi_{n,\mu}[W_{i},\Xi^{a}_i]W_j.
\end{split}
\end{equation}
Observe that for $i=j$, it holds
$$
 \Delta \Xi_i^a+p\Big(|x|^{-\mu}\ast W_{i}^{p-1}\Xi_i^a\Big)W_{i}^{p-1}+(p-1)\Big(|x|^{-\mu}\ast W_{i}^{p}\Big)W_{i}^{p-2}\Xi_i^a=0.
$$
By Lemma \ref{FPU1}, the integral estimates are controlled by
\begin{equation*}
\begin{split}
&\int\big(|x|^{-\mu}\ast W_{i}^{p}\big)W_{i}^{p-2}|\Xi_{i}^{a}|W_j\lesssim\int W_{i}^{p-1}W_j\lesssim \Big(\mathscr{Q}\hspace{2mm}\mbox{if}\hspace{2mm}\mu\in(0,4);\hspace{2mm}\mathscr{Q}^{\frac{n}{n-2}}\log{\frac{1}{\mathscr{Q}}}\hspace{2mm}\mbox{if}\hspace{2mm}\mu=4\Big),\hspace{2mm}i\neq j, \\& \int\big(|x|^{-\mu}\ast (W_{i}^{p-1}|\Xi_{i}^{a}|)\big)W_{i}^{p-1}W_j\lesssim \int W_{i}^{p-1}W_j\lesssim \Big(\mathscr{Q}\hspace{2mm}\mbox{if}\hspace{2mm}\mu\in(0,4);\hspace{1mm}\mathscr{Q}^{\frac{n}{n-2}}\log{\frac{1}{\mathscr{Q}}}\hspace{1mm}\mbox{if}\hspace{1mm}\mu=4\Big),\hspace{1mm}i\neq j.
\end{split}
\end{equation*}
Furthermore, the coefficients $c_{a}^i$ are controlled by
$$|c_a^j|\lesssim\mathscr{Q}^{2-\frac{\mu}{n-2}}, \hspace{2mm}n=4\hspace{2mm}\mbox{and}\hspace{2mm}\mu\in[2,4),\mbox{or}\hspace{2mm}n=5\hspace{2mm}\mbox{and}\hspace{2mm}\mu\in(3,4),
$$
$$|c_a^j|\lesssim\mathscr{Q}^{\min\{1,\frac{n-\mu}{n-2}\}},\hspace{1mm}n=5\hspace{2mm}\mbox{and}\hspace{2mm}\mu\in[1,3),\mbox{or}\hspace{2mm}n=6\hspace{2mm}\mbox{and}\hspace{2mm}\mu\in(0,4),\mbox{or}\hspace{2mm}n=7\hspace{2mm}\mbox{and}\hspace{1mm}\mu\in(\frac{7}{3},4),$$
by Lemma \ref{estimate1} and Lemma \ref{cll}.
Therefore, for all $i=1,\cdots, \kappa; ~a=1,\cdots,n+1$, we get from
Lemma \ref{cll} and Lemma \ref{FPU1}, and decompose of $\rho_1$ along with the orthogonality condition for $\rho_2$ that the following estimates hold.\\
$\bullet$
If $n=4$ and $\mu\in[2,4)$, or $n=5$ and $\mu\in[3,4)$, we have
\begin{equation}\label{rr-15}
\begin{split}	
\gamma^j+\sum_{i\neq j}^{\kappa}\gamma^iO(Q_{ij})\lesssim o(1)\Upsilon_0+\big\|\hat{f}\big\|_{(D^{1,2}(\mathbb{R}^n))^{-1}}
+\mathscr{Q}^{3-\frac{\mu}{n-2}}.
\end{split}
\end{equation}
$\bullet$
If $n=5$ and $\mu\in[1,3)$, or $n=6$ and $\mu\in(0,4)$, or $n=7$ and $\mu\in(\frac{7}{3},4)$, we get that
\begin{equation}\label{rr-16}
\begin{split}	
\gamma^j+\sum_{i\neq j}^{\kappa}\gamma^iO(Q_{ij})\lesssim o(1)\Upsilon_0+\big\|\hat{f}\big\|_{(D^{1,2}(\mathbb{R}^n))^{-1}}
+\mathscr{Q}^{\min\{2,\frac{n-\mu}{n-2}+1\}}.
\end{split}
\end{equation}
Hence, then we eventually obtain with \eqref{rr-15} and \eqref{rr-16} desired thesis as $\delta$ is small,
which concludes the proof of this Lemma.
\end{proof}
As a byproduct of Lemma \ref{Ni-1-3} and Lemma \ref{rr-1}, we easily get the following estimate hold.
\begin{lem}\label{rr-1-00}
If $\delta$ is small we have:\\
$(i).$ If $n=4$ and $\mu\in[2,4)$, or $n=5$ and $\mu\in[3,4)$,
$$\big\|\nabla\rho_1\big\|_{L^2}\lesssim\big\|\hat{f}\big\|_{(D^{1,2}(\mathbb{R}^n))^{-1}}
+\mathscr{Q}^{3-\frac{\mu}{n-2}}.$$
$(ii).$ If $n=5$ and $\mu\in[1,3)$, or $n=6$ and $\mu\in(0,4)$, or $n=7$ and $\mu\in(\frac{7}{3},4)$,
\begin{equation*}	
\big\|\nabla\rho_1\big\|_{L^2}\lesssim\big\|\hat{f}\big\|_{(D^{1,2}(\mathbb{R}^n))^{-1}}
+\mathscr{Q}^{\min\{2,\frac{n-\mu}{n-2}+1\}}.
\end{equation*}
\end{lem}
\begin{lem}\label{Ni-1}
If $\delta$ is small, we have that:\\
$(i).$ If $n=4$ and $\mu\in[2,4)$, or $n=5$ and $\mu\in[3,4)$,
\begin{equation*}
 \big|\int\Phi_{n,\mu}[\sigma,\rho]\Xi_{m}^{n+1}\big|
\approx
\left\lbrace
\begin{aligned}
&
o(\mathscr{Q})+\big\|\hat{f}\big\|_{(D^{1,2}(\mathbb{R}^n))^{-1}},\hspace{2mm}n=4\hspace{2mm} \mbox{and}\hspace{2mm}\mu\in[2,3), \hspace{1mm}\mbox{or}\hspace{2mm}n=5\hspace{2mm}\mbox{and}\hspace{2mm}\mu\in[3,4),\\&
o(\mathscr{Q}^{2-\frac{\mu}{n-2}})+\big\|\hat{f}\big\|_{(D^{1,2}(\mathbb{R}^n))^{-1}}, \hspace{2mm}n=4\hspace{2mm}\mbox{and}\hspace{2mm}\mu\in(3,4).
\end{aligned}
\right.
\end{equation*}
$(ii).$ If $n=5$ and $\mu\in[1,3)$, or $n=6$ and $\mu\in(0,4)$, or $n=7$ and $\mu\in(\frac{7}{3},4)$,
\begin{equation*}
 \big|\int\Phi_{n,\mu}[\sigma,\rho]\Xi_{m}^{n+1}\big|
\approx o(\mathscr{Q})+\big\|\hat{f}\big\|_{(D^{1,2}(\mathbb{R}^n))^{-1}}.
\end{equation*}
 \end{lem}
\begin{proof}
An easy computation shows that
\begin{equation}\label{pp-1}
\begin{split}
\int\Phi_{n,\mu}[\sigma,\rho]\Xi_{m}^{n+1}
=\int\Phi_{n,\mu}[\sigma,\rho_{0}]\Xi_{m}^{n+1}+\int\Phi_{n,\mu}[\sigma,\rho_{1}]\Xi_{m}^{n+1}
=:\tilde{h}_1+\tilde{h}_2.
\end{split}
\end{equation}
Let us estimate each integral in \eqref{pp-1}. Similar to the argument of \eqref{zb1}-\eqref{eqa-2-3}, we get that
$$|\tilde{h}_1|=o(\mathscr{Q}^{2-\frac{\mu}{n-2}}),\hspace{2mm}n=4\hspace{2mm}\mbox{and}\hspace{2mm} \mu\in[2,4),\hspace{2mm}\mbox{or}\hspace{2mm}n=5\hspace{2mm}\mbox{and}\hspace{2mm}\mu\in[3,4),$$
\begin{equation*}
|\tilde{h}_1|=o(\mathscr{Q}),\hspace{2mm}n=5\hspace{2mm}\mbox{and}\hspace{2mm}\mu\in[1,3), \hspace{2mm}\mbox{or}\hspace{2mm}n=6\hspace{2mm}\mbox{and}\hspace{2mm}\mu\in(0,4), \hspace{2mm}\mbox{or}\hspace{2mm}n=7\hspace{2mm}\mbox{and}\hspace{2mm}\mu\in(\frac{7}{3},4),
\end{equation*}
by the orthogonality condition. By H\"{o}lder's, Sobolev's inequalities and Lemma \ref{rr-1-00}, we obtain
\begin{equation*}
\begin{split}
|\tilde{h}_2|&\lesssim\|\sigma\|_{L^{2^{\ast}}}^{p-1}\|\sigma^{p-1}\Xi_{m}^{n+1}\|_{L^{r}}\|\nabla\rho_1\|_{L^{2}}+
\|\sigma\|_{L^{2^{\ast}}}^{p}\|\sigma^{p-2}\Xi_{m}^{n+1}\|_{L^{r}}\|\nabla\rho_1\|_{L^{2}}\\&\lesssim
\|\hat{f}\|_{(D^{1,2}(\mathbb{R}^n))^{-1}}+o(\mathscr{Q}).
\end{split}
\end{equation*}
Hence the result easily follows.
\end{proof}
 \begin{lem}\label{Ni-2}
If $\delta$ is small we have
\begin{equation*}
\int \big|\mathscr{N}_i(\phi)\Xi_{m}^{n+1}\big|=o(\mathscr{Q})+\big\|\hat{f}\big\|_{(D^{1,2}(\mathbb{R}^n))^{-1}},\hspace{2mm}i=1,\dots,9.
\end{equation*}
 \end{lem}
\begin{proof}
The conclusion follows by using H\"{o}lder's inequality and Sobolev's inequality, Lemma \ref{p00} and  Lemma \ref{rr-1-00}.
\end{proof}

\appendix
\section{Technical Lemmata}
In this appendix we give some crucial estimates that have been used in the previous sections.
\begin{lem}\label{p1}
We have that
\begin{equation}\label{pmr3}
\Big|\big|\sum\limits_{i=1}^{\kappa}W_i\big|^{p}
-\sum\limits_{i=1}^{\kappa}\big|W_i\big|^{p}\Big|\lesssim \sum\limits_{1\leq i\neq j\leq\kappa}\big|W_i\big|^{p-1}W_j+\sum\limits_{1\leq i\neq j\leq\kappa}\big|W_i\big|^{p-2}W_j^2,
\end{equation}
and
\begin{equation}\label{pmr4}
\Big|\big|\sum\limits_{i=1}^{\kappa}W_i\big|^{p-2}\sum\limits_{i=1}^{\kappa}W_i
-\sum\limits_{i=1}^{\kappa}\big|W_i\big|^{p-2}W_i\Big|\lesssim \sum\limits_{1\leq i\neq j\leq\kappa}\big|W_i\big|^{p-2}W_j,
\end{equation}
\end{lem}
\begin{proof}
The estimates \eqref{pmr3} and \eqref{pmr4} are derived by simple computations.
\end{proof}
\begin{lem}\label{p1-00}
We have that
\begin{equation*}
|x|^{-\mu}\ast W_i^{p}
=\int_{\mathbb{R}^n}\frac{W_i^{p}(y)}{|x-y|^{\mu}}dy
=\widetilde{\alpha}_{n,\mu}W_i^{2^{\ast}-p}(x),
\end{equation*}
where  $\widetilde{\alpha}_{n,\mu}=I(\gamma)S^{\frac{(n-\mu)(2-n)}{4(n-\mu+2)}}C_{n,\mu}^{\frac{2-n}{2(n-\mu+2)}}[n(n-2)]^{\frac{n-2}{4}}.$
\end{lem}
\begin{proof}
The conclusion follow by the Fourier transforms of the kernels of Riesz and Bessel potentials (cf. \cite{DHQWF}).
\end{proof}

In order to study the higher order term $\mathscr{N}$ for $\phi$, we further define the following functions
\begin{equation*}
\begin{split}
&\mathscr{N}_1(\sigma,\phi):=\mathcal{I}_{\mu}\{\sigma^{p}\}\phi^{p-1},
\hspace{2mm}\mathscr{N}_2(\sigma,\phi):=\mathcal{I}_{\mu}\{\sigma^{p-1}\phi\}\sigma^{p-2}\phi,
\hspace{2mm}\mathscr{N}_3(\sigma,\phi):=\mathcal{I}_{\mu}\{\sigma^{p-1}\phi\}\phi^{p-1},\\&
\mathscr{N}_4(\sigma,\phi):=\mathcal{I}_{\mu}\{\sigma^{p-2}\phi^2\}\sigma^{p-1},
\hspace{2mm}\mathscr{N}_5(\sigma,\phi):=\mathcal{I}_{\mu}\{\sigma^{p-2}\phi^2\}\sigma^{p-2}\phi,\hspace{2mm}
\mathscr{N}_6(\sigma,\phi):=\mathcal{I}_{\mu}\{\sigma^{p-2}\phi^2\}\phi^{p-1},\\&
\mathscr{N}_7(\sigma,\phi):=\mathcal{I}_{\mu}\{\phi^{p}\}\sigma^{p-1},\hspace{2mm}
\mathscr{N}_8(\sigma,\phi):=\mathcal{I}_{\mu}\{\phi^{p}\}\sigma^{p-2}\phi,\hspace{2mm}
\mathscr{N}_9(\sigma,\phi):=\mathcal{I}_{\mu}\{\phi^{p}\}\phi^{p-1}.
\end{split}
\end{equation*}
By using the some elementary inequalities, we have the following estimates.
 \begin{lem}\label{nnn1}
The decomposition holds that
\begin{equation}\label{nnn1-0}
\mathscr{N}(\sigma,\phi)\lesssim \mathscr{N}_1(\sigma,\phi)+\mathscr{N}_2(\sigma,\phi)+\cdots+\mathscr{N}_9(\sigma,\phi).
\end{equation}
\end{lem}

The following is devoted to computations of quantities.
The series of tedious integral estimates is similar to the proof in \cite{DSW21}, but the appearance of nonlocal nonlinearity which makes many estimates significantly different from the integral in \cite{DSW21} and exponents have been modified by the parameter $\mu$ and the definition of $T_1$ and $T_2$. Therefore, we need to reevaluate the different terms with from those in \cite{DSW21}.

\begin{lem}\label{4B4-0}
	For $n\geq6-\mu$, $\mu\in(0,n)$ with $0<\mu<4$, and $1\ll\mathscr{R}\leq\mathscr{R}_{ij}/2$, we have the following estimates hold:
\begin{itemize}
\item[$(1)$]
If $n=4$ and $\mu\in[2,4)$, or $n=5$ and $\mu\in[3,4)$, we have that
\begin{equation}\label{5B4}
\int \hat{t}_{i,1}w_{i,1}=\int_{|z_i|\leq\mathscr{R}} \frac{\lambda_i^{\frac{n+2}{2}}\mathscr{R}^{4+\mu-2n}}{\tau( z_i)^{4}}\frac{\lambda_i^{\frac{n-2}{2}}}{\tau(z_i)^{n-2}}
\lesssim\frac{1}{\mathscr{R}^{2n-4-\mu}},
\end{equation}
\begin{equation}\label{6B4}
\int \hat{t}_{i,2}w_{i,2}=\int_{|z_i|\geq\mathscr{R}}
\frac{\lambda_i^{\frac{n+2}{2}}\mathscr{R}^{\mu-n-\epsilon_0}}{\tau( z_i)^{n-\epsilon_0}}\frac{\lambda_i^{\frac{n-2}{2}}}{\tau(z_i)^{n-2}}
\lesssim\frac{1}{\mathscr{R}^{2n-\mu-2}},
\end{equation}
\begin{equation}\label{7B4}
\int \hat{t}_{i,1}w_{j,1}=\int_{|z_i|\leq\mathscr{R}} \frac{\lambda_i^{\frac{n+2}{2}}\mathscr{R}^{4+\mu-2n}}{\tau( z_i)^{4}}\frac{\lambda_j^{\frac{n-2}{2}}}{\tau(z_j)^{n-2}}
\lesssim\frac{1}{\mathscr{R}^{2n-2-\mu}}\log\mathscr{R},
\end{equation}
\begin{equation}\label{8B4}
\int \hat{t}_{i,2}w_{j,2}=\int_{|z_i|\geq\mathscr{R}}
\frac{\lambda_i^{\frac{n+2}{2}}\mathscr{R}^{\mu-n-\epsilon_0}}{\tau( z_i)^{n-\epsilon_0}}\frac{\lambda_j^{\frac{n-2}{2}}}{\tau( z_j)^{n-2}}
\lesssim\frac{1}{\mathscr{R}^{2n-2-\mu}}\log\mathscr{R}.
\end{equation}
\item[$(2)$] If $n=5$ and $\mu\in[1,3)$, or $n=6$ and $\mu\in(0,4)$, or  $n=7$ and $\mu\in(\frac{7}{3},4)$,
\begin{equation}\label{5B4-0}
\int t_{i,1}w_{i,1}=\int_{|z_i|\leq\mathscr{R}} \frac{\lambda_i^{\frac{n+2}{2}}\mathscr{R}^{2-n}}{\tau( z_i)^{4}}\frac{\lambda_i^{\frac{n-2}{2}}}{\tau( z_i)^{n-2}}
\lesssim\frac{1}{\mathscr{R}^{n-2}},
\end{equation}
\begin{equation}\label{6B4-0}
\int t_{i,2}w_{i,2}=\int_{|z_i|\geq \mathscr{R}}
\frac{\lambda_i^{\frac{n+2}{2}}\mathscr{R}^{\frac{\mu-2-n}{2}}}{\tau( z_i)^{\frac{n+\mu+2}{2}}}\frac{\lambda_i^{\frac{n-2}{2}}}{\tau( z_i)^{n-2}}
\lesssim\frac{1}{\mathscr{R}^{n}},
\end{equation}
\begin{equation}\label{7B4-0}
\int t_{i,1}w_{j,1}=\int_{|z_i|\leq\mathscr{R}} \frac{\lambda_i^{\frac{n+2}{2}}\mathscr{R}^{2-n}}{\tau( z_i)^{4}}\frac{\lambda_j^{\frac{n-2}{2}}}{\tau( z_j)^{n-2}}
\lesssim\frac{1}{\mathscr{R}^{n}},
\end{equation}
\begin{equation}\label{8-B-4}
\int t_{i,2}w_{j,2}=\int_{|z_i|\geq \mathscr{R}}
\frac{\lambda_i^{\frac{n+2}{2}}\mathscr{R}^{\frac{\mu-2-n}{2}}}{\tau( z_i)^{\frac{n+\mu+2}{2}}}\frac{\lambda_j^{\frac{n-2}{2}}}{\tau( z_j)^{n-2}}
\lesssim\frac{1}{\mathscr{R}^{\min\{\frac{(n+\mu+2)}{2},n-\mu\}}}.
\end{equation}
\end{itemize}
\end{lem}
\begin{proof}
The proof of \eqref{5B4}, \eqref{6B4}, \eqref{7B4} and \eqref{6B4-0} are derived by a direct computation. \eqref{5B4-0} and \eqref{7B4-0} can be found in \cite{DSW21} except minor modifications. We only give the proof of \eqref{8-B-4}. In a similar way we prove that \eqref{8B4} except minor modifications. We then consider separately the following two case.\\
$(i)$ For $\lambda_{i}\geq\lambda_j$.
In this case we set  $\{|z_i|\geq\mathscr{R}\}=\tilde{\Omega}_{1}\cup\tilde{\Omega}_{2}\cup\tilde{\Omega}_{3}$ the sets are given by
\begin{equation*}
\begin{split}
\tilde{\Omega}_{1}&=\{\mathscr{R}\leq|z_i|\leq\frac{\sqrt{\lambda_i/\lambda_j}\mathscr{R}_{ij}}{2}\}, \hspace{2mm}\tilde{\Omega}_{2}=\{|z_i|\geq\frac{\sqrt{\lambda_i/\lambda_j}\mathscr{R}_{ij}}{2}, |z_j|\leq\frac{\sqrt{\lambda_j/\lambda_i}\mathscr{R}_{ij}}{2}\}, \\& \hspace{2mm}\tilde{\Omega}_{3}=\tilde{\Omega}_{2}=\{|z_i|\geq\frac{\sqrt{\lambda_i/\lambda_j}\mathscr{R}_{ij}}{2}, |z_j|\geq\frac{\sqrt{\lambda_j/\lambda_i}\mathscr{R}_{ij}}{2}\}.
\end{split}
\end{equation*}
Due to Lemma \eqref{B3-0} and the fact that $\lambda_i/\lambda_{j}\leq\mathscr{R}_{ij}^{2}$, we have
\begin{equation*}
\begin{split}
&\int t_{i,2}w_{j,2}=\int_{|z_i|\geq \mathscr{R}}
\frac{\lambda_i^{\frac{n+2}{2}}\mathscr{R}^{\frac{\mu-2-n}{2}}}{\tau( z_i)^{\frac{n+\mu+2}{2}}}\frac{\lambda_j^{\frac{n-2}{2}}}{\tau( z_j)^{n-2}}
\\&\lesssim\mathscr{R}^{\frac{\mu-2-n}{2}}\Big[\frac{1}{\mathscr{R}_{ij}^{n-2}}\int_{\tilde{\Omega}_{1}}\frac{dz_i}{\tau( z_i)^{\frac{n+\mu+2}{2}}}+\mathscr{R}_{ij}^{-\frac{(n+\mu+2)}{2}}(\frac{\lambda_i}{\lambda_j})^{\frac{n-\mu+2}{4}}\int_{\tilde{\Omega}_{2}}\frac{dz_j}{\tau( z_j)^{n-2}}+(\frac{\lambda_i}{\lambda_j})^{-\frac{\mu}{2}}\int_{\bar{\Omega}_{3}}\frac{dz_j}{\tau( z_j)^{\frac{n+\mu+2}{2}+n-2}}\Big]
\\&\lesssim\mathscr{R}^{\frac{\mu-2-n}{2}}\Big[\mathscr{R}_{ij}^{2-\frac{(n+\mu+2)}{2}}(\frac{\lambda_i}{\lambda_j})^{\frac{n-\mu-2}{4}}
+\mathscr{R}_{ij}^{-\frac{(n+\mu-2)}{2}}(\frac{\lambda_i}{\lambda_j})^{\frac{n-\mu-2}{4}}
+\mathscr{R}_{ij}^{\frac{(n-\mu+2)}{2}-n}(\frac{\lambda_i}{\lambda_j})^{\frac{n-\mu-2}{4}}\Big]\lesssim\mathscr{R}^{-\frac{(n+\mu+2)}{2}}.
\end{split}
\end{equation*}
$(ii)$ For $\lambda_{i}\leq\lambda_j$, In this case we set the sets $\bar{\Omega}_{1}$ and $\bar{\Omega}_{2}$ are given by
\begin{equation*}
\begin{split}
\bar{\Omega}_{1}&=\{|z_i|\geq\mathscr{R},\hspace{2mm}|z_j|\leq\frac{\sqrt{\lambda_j/\lambda_i}\mathscr{R}_{ij}}{2}\}, \hspace{2mm}\bar{\Omega}_{2}=\{|z_i|\geq\mathscr{R}, \hspace{2mm} |z_j|>\frac{\sqrt{\lambda_j/\lambda_i}\mathscr{R}_{ij}}{2}\}.
\end{split}
\end{equation*}
Using Lemma \eqref{B3-0}, we obtain
\begin{equation*}
\begin{split}
&\int t_{i,2}w_{j,2}=\int_{|z_i|\geq \mathscr{R}}
\frac{\lambda_i^{\frac{n+2}{2}}\mathscr{R}^{\frac{\mu-2-n}{2}}}{\tau( z_i)^{\frac{n+\mu+2}{2}}}\frac{\lambda_j^{\frac{n-2}{2}}}{\tau( z_j)^{n-2}}\\&
\lesssim\mathscr{R}^{\frac{\mu-2-n}{2}}\Big[\mathscr{R}_{ij}^{-\frac{(n+\mu+2)}{2}}(\frac{\lambda_i}{\lambda_j})^{\frac{n-\mu+2}{4}}\int_{\bar{\Omega}_{1}}\frac{dz_j}{\tau( z_j)^{n-2}}+(\frac{\lambda_i}{\lambda_j})^{-\frac{2-n}{2}}\int_{\bar{\Omega}_{2}}\frac{dz_i}{\tau( z_i)^{\frac{n+\mu+2}{2}+n-2}}\Big]
\\&\lesssim\mathscr{R}^{\frac{\mu-2-n}{2}}\mathscr{R}_{ij}^{2-\frac{(n-\mu+2)}{2}}
+\mathscr{R}^{-n}\lesssim\mathscr{R}^{-(n-\mu)}.
\end{split}
\end{equation*}

Putting these estimates together, we get that \eqref{8-B-4}.
\end{proof}

\end{document}